\documentclass[a4paper, abstract]{scrreprt}
\newcounter{RELEASE}
\usepackage{ifthen}

\usepackage[]{scrpage2}
\automark[chapter]{section}
\usepackage[T1]{fontenc}
\usepackage[utf8]{inputenc}

\usepackage[german,english]{babel}

\usepackage[bibencoding=ascii, style=ieee-alphabetic]{biblatex}
\DeclareSourcemap{
	\maps[datatype=bibtex]{
		\map{
			\step[fieldsource=doi,final]
			\step[fieldset=url,null]
			\step[fieldset=issn,null]
			\step[fieldset=isbn,null]
		}
		\map{
			\step[fieldsource=url,final]
			\step[fieldset=issn,null]
			\step[fieldset=isbn,null]
		}
		\map{
			\step[fieldsource=isbn,final]
			\step[fieldset=issn,null]
		}
	}
}

\DefineBibliographyStrings{english}{
	url         = URL\addspace ,
}

\usepackage{amssymb, latexsym}
\usepackage{textcomp}
\usepackage{colonequals}
\usepackage{mathtools}
\usepackage{amsthm}
\usepackage{makeidx}

\usepackage[cmtip,arrow,all]{xy}
\CompileMatrices

\usepackage{enumitem}

\ifthenelse{ \equal{\value{RELEASE}}{0} }{}{\PassOptionsToPackage{disable}{todonotes}}
\usepackage{todonotes}

\newcounter{todoListItems}
\newcommand{\rem}[2][]{%
	\stepcounter{todoListItems}
	\todo[size=\scriptsize, #1]{{\tiny \thetodoListItems} #2}
}

\usepackage{csquotes}

\usepackage{prettyref}
\newrefformat{sec}{\hyperref[{#1}]{Section~\ref*{#1}}}
\newrefformat{susec}{\hyperref[{#1}]{Section~\ref*{#1}}}
\newrefformat{chp}{\hyperref[{#1}]{Chapter~\ref*{#1}}}
\newrefformat{app}{\hyperref[{#1}]{Appendix~\ref*{#1}}}

\newrefformat{satz}{\hyperref[{#1}]{Theorem~\ref*{#1}}}
\newrefformat{lem}{\hyperref[{#1}]{Lemma~\ref*{#1}}}
\newrefformat{cor}{\hyperref[{#1}]{Corollary~\ref*{#1}}}
\newrefformat{bem}{\hyperref[{#1}]{Remark~\ref*{#1}}}
\newrefformat{prop}{\hyperref[{#1}]{Proposition~\ref*{#1}}}
\newrefformat{beisp}{\hyperref[{#1}]{Example~\ref*{#1}}}
\newrefformat{def}{\hyperref[{#1}]{Definition~\ref*{#1}}}
\newrefformat{defi}{\hyperref[{#1}]{Definition~\ref*{#1}}}
\newrefformat{fact}{\hyperref[{#1}]{fact~\ref*{#1}}}

\newrefformat{ivp}{\hyperref[{#1}]{initial value problem~(\ref*{#1})}}
\newrefformat{est}{\hyperref[{#1}]{estimate~(\ref*{#1})}}
\newrefformat{id}{\hyperref[{#1}]{identity~(\ref*{#1})}}
\newrefformat{bed}{\hyperref[{#1}]{condition~(\ref*{#1})}}
\newrefformat{char}{\hyperref[{#1}]{characterization~(\ref*{#1})}}
\newrefformat{enum1}{(\ref*{#1})}
\newcommand{\refer}[1]{\prettyref{#1}}
\newcommand{\referto}[1]{\rem{Referenz setzen!}}

\usepackage{hyperref}
\usepackage[toc, nosuper, nolist, notree, style=long]{glossaries}

\newglossarystyle{mylong3col}{%
	{\begin{longtable}{@{}l@{}lp{\glsdescwidth}p{\glspagelistwidth}}}%
	{\end{longtable}}%
	\renewcommand*{\glsgroupheading}[1]{}%
}

\newcommand{\Disk}{\ensuremath{\mathbb{D}}}
\newcommand{\morQuot}[1]{\ensuremath{\pi_{#1} }}
\newcommand{\HomQuot}[2]{\ensuremath{\morQuot{#2} \circ #1}}
\newcommand{\R}{\ensuremath{\mathbb{R}}}
\newcommand{\C}{\ensuremath{\mathbb{C}}}
\newcommand{\N}{\ensuremath{\mathbb{N}}}
\newcommand{\K}{\ensuremath{\mathbb{K}}}
\newcommand{\Z}{\ensuremath{\mathbb{Z}}}
\newcommand{\Q}{\ensuremath{\mathbb{Q}}}

\newcommand{\eps}{\ensuremath{\varepsilon}}

\DeclareMathOperator{\idSymb}{\mathrm{id}}
\newcommand{\id}[1]{\idSymb_{#1}}
\newcommand{\iso}{\ensuremath \cong}
\newcommand{\sub}{\ensuremath{\subseteq}}
\newcommand{\set}[2]{\ensuremath{\{#1 : #2\} }}
\newcommand{\sset}[1]{\ensuremath{\{#1\}}}
\newcommand{\ndef}{\colonequals}
\newcommand{\defn}{\equalscolon}
\newcommand{\rest}[3][]{\ensuremath{#2|_{#3}\ifthenelse{ \equal{#1}{} }{}{^{#1}}} }
\newcommand{\image}[1]{\ensuremath{\mathrm{im}\, #1 }}

\newcommand{\Strecke}[2]{\ensuremath{[#1,#2]}}

\DeclareMathOperator{\distop}{\mathrm{dist}}
\newcommand{\dist}[2]{\ensuremath \distop(#1,#2)}
\newcommand{\abs}[1]{\ensuremath \lvert #1\rvert}
\newcommand{\norm}[1]{\ensuremath \lVert #1\rVert}
\newcommand{\MetaBall}[4][]
{
	\ensuremath
	{
	\ifthenelse{ \equal{#1}{} }
	{
		 #4_{#3}(#2)
	}
	{
		#4_{#1}(#2, #3)
	}
	}
}
\newcommand{\Ball}[3][]{\MetaBall[#1]{#2}{#3}{B}}
\newcommand{\clBall}[3][]{ \MetaBall[#1]{#2}{#3}{\cl{B}} }

\newcommand{\supp}[1]{\ensuremath{\mathrm{supp}(#1)}}
\newcommand{\neigh}[2][]{\ensuremath{ \mathcal{U}\ifthenelse{ \equal{#1}{} }{}{_{#1}}(#2) }}
\newcommand{\neighO}[2][]{\ensuremath{ \mathcal{U}\ifthenelse{ \equal{#1}{} }{}{_{#1}}^\circ(#2) }}
\newcommand{\cl}[1]{\overline{#1}}
\newcommand{\interior}[1]{#1^{\circ}}

\newcommand{\fami}[3]{\ensuremath{(#1)_{#2 \in #3}}}

\newcommand{\Rint}[5][\int]{\ensuremath {#1}_{#2}^{#3} #4\, d #5}
\newcommand{\Mint}[3][\int]{\Rint[#1]{0}{1}{#2}{#3}}
\newcommand{\Lint}[3]{\ensuremath \int_{#1} #2\, d #3}
\newcommand{\dA}[4][]{
	\ensuremath{
	d\ifthenelse{ \equal{#1}{} }{}{^{(#1)}} 
		#2 \ifthenelse{\equal{#3}{} \and \equal{#4}{} }{}{(#3;#4)}
		}
}
\newcommand{\FAbl}[2][]{
	\ensuremath{
	D\ifthenelse{ \equal{#1}{} }{}{^{(#1)}}#2
	}
}
\newcommand{\parAbl}[2]{\ensuremath{ \partial^{#1} #2 }}

\newcommand{\LLA}[2][]{ \ensuremath{\delta_{\ell}\ifthenelse{\equal{#2}{} }{}{(#2)} } }
\newcommand{\RLA}[2][]{ \ensuremath{\delta_{\rho}\ifthenelse{\equal{#2}{} }{}{(#2)} } }

\newcommand{\AblAct}[1]{\ensuremath{\dot{#1}}}

\newcommand{\LieAlg}[1]{\ensuremath{\mathbf{L}(#1)}}
\newcommand{\Tang}[2][]{
\ensuremath{\mathbf{T}\ifthenelse{ \equal{#1}{} }{}{_{#1} }
	\ifthenelse{ \equal{#2}{} }%
	{}%
	{#2}%
	}
}
\newcommand{\VecFields}[2][]{\ensuremath{\mathfrak{X}^{#1}(#2)}}
\newcommand{\VecFieldsLinv}[1]{\VecFields{#1}_\ell}

\DeclareMathOperator{\GLsymb}{\mathrm{GL}}
\DeclareMathOperator{\End}{\mathrm{End}}
\DeclareMathOperator{\DiffSym}{\mathrm{Diff}}
\DeclareMathOperator{\evTwo}{ev}

\newcommand{\ConDiff}[4][]
{
	\ensuremath{
	\mathcal{C}^{#4}_{#1}
	\ifthenelse{ \equal{#2}{} \and \equal{#3}{} }
	{}
	{(#2, #3)}
	}
}

\newcommand{\FC}[4][]
{
	\ensuremath{
	\mathcal{FC}^{#4}
	\ifthenelse{ \equal{#2}{} \and \equal{#3}{} }
	{}
	{(#2, #3)}
	}
}

\newcommand{\GewFunk}{\cW}
\newcommand{\hn}[3]{\ensuremath\norm{#1}_{#2, #3}}

\newcommand{\CF}[4]{\ensuremath\mathcal{C}^{#4}_{#3}(#1, #2)}
\newcommand{\CFo}[4]{\CF{#1}{#2}{#3}{\partial, #4}}

\newcommand{\CFvan}[4]{ \CF{#1}{#2}{#3}{#4}^{o} }
\newcommand{\CFvanK}[4]{ \CF{#1}{#2}{#3}{#4}^{\bullet} }

\newcommand{\CcF}[3]{\CF{#1}{#2}{\GewFunk}{#3}}
\newcommand{\CinfF}[2]{\CcF{#1}{#2}{\infty}}
\newcommand{\CcFo}[3]{\CFo{#1}{#2}{\GewFunk}{#3}}
\newcommand{\CcFzero}[3]{ \CcF{#1}{#2}{#3}_{0} }
\newcommand{\CcFvan}[3]{ \CFvan{#1}{#2}{\GewFunk}{#3} }
\newcommand{\CcFvanK}[3]{ \CFvanK{#1}{#2}{\GewFunk}{#3} }
\newcommand{\CcFvanKBig}[3]{ \CF{#1}{#2}{\GewFunk}{#3}^{\bullet}_{\text{ex}} }

\newcommand{\BC}[3]{\ensuremath{\mathcal{BC}^{#3}(#1, #2)}}
\newcommand{\BCo}[3]{ \BC{#1}{#2}{\partial, #3} }
\newcommand{\BCzero}[3]{ \BC{#1}{#2}{#3}_{0}}

\newcommand{\DCc}[3]{\CF{#1}{#2}{c}{#3}}
\newcommand{\DCcInf}[2]{\DCc{#1}{#2}{\infty}}

\newcommand{\LipCon}[5]{\ensuremath{\mathcal{LC}_{#1, #2}^{#5}(#3, #4)}}
\newcommand{\FakLC}[3]{\ensuremath{\tilde{#3}}}
\newcommand{\GL}[1]{\ensuremath{\GLsymb(#1)}}

\newcommand{\Diff}[3]{
	\ensuremath{\DiffSym^{#2}_{#3}(#1)}
}
\newcommand{\DiffW}{\Diff{\SX}{}{\GewFunk}}
\newcommand{\DiffGeWz}[1]{\Diff{\SX}{}{#1}_0}
\newcommand{\DiffWz}{\Diff{\SX}{}{\GewFunk}_0}
\newcommand{\DiffWvan}{\Diff{\SX}{}{\GewFunk}^\circ}
\newcommand{\Diffc}{\Diff{\SX}{}{c}}

\newcommand{\EndW}{ \ensuremath{ \Endos{\SX}{\GewFunk} }}
\newcommand{\EndWvan}{\Endos{\SX}{\GewFunk}^{\circ}}
\newcommand{\Endos}[2]{\ensuremath{\End_{#2}(#1)}}
\newcommand{\kF}{\ensuremath{\kappa_{\GewFunk}}}
\newcommand{\karte}[1]{\ensuremath{\kappa_{#1}}}

\newcommand{\compIdRaw}{\tilde{\mathfrak{c}}}
\newcommand{\compIdConWeights}[3]{\tilde{\mathfrak{c}}_{#3}^{#1, #2}}
\newcommand{\compIdDiffKLWeights}[4]{\mathfrak{c}_{#4, #3}^{#1, #2}}
\newcommand{\compIdSmoothLWeights}[3]{\mathfrak{c}_{#3}^{#1, #2}}
\newcommand{\compIdWeights}[2]{\mathfrak{c}_{#2}^{#1}}
\newcommand{\compIdConCW}[2]{\compIdConWeights{#1}{#2}{\GewFunk}}
\newcommand{\compIdDiffKLcW}[3]{\compIdDiffKLWeights{#1}{#2}{#3}{\GewFunk}}
\newcommand{\compIdSmoothLcW}[2]{\compIdSmoothLWeights{#1}{#2}{\GewFunk}}
\newcommand{\compIdcW}[1]{\compIdWeights{#1}{\GewFunk}}

\newcommand{\InvIdWeights}[2]{I^{#1}_{#2}}
\newcommand{\InvIdcW}[1]{\InvIdWeights{#1}{\GewFunk}}

\newcommand{\InvertWinfDomRan}[2]{\Omega^{#1, #2}_\GewFunk}
\newcommand{\InvIdAlg}[1]{\tilde{I}_{#1}}
\newcommand{\InvertDomRanAlg}[2]{\widetilde{\Omega}_{#1, #2}}

\newcommand{\cW}{\ensuremath{\mathcal{W}}}

\newcommand{\cF}{\ensuremath{\mathcal{F}}}
\newcommand{\ExtWeights}[1]{\ensuremath{#1_{\mathrm{max}}}}

\newcommand{\noma}[1]{\norm{#1}_\infty}
\newcommand{\Opnorm}[1]{\norm{#1}_{op}}
\newcommand{\OpInf}[1]{\norm{#1}_{op,\infty}}

\DeclareMathOperator{\Id}{\mathrm{Id}}
\newcommand{\idco}{\ensuremath{\Id}}
\newcommand{\MaMu}{\ensuremath{\cdot}}
\newcommand{\eval}{\ensuremath{\cdot}}

\newcommand{\mult}{\ensuremath{\ast}}
\newcommand{\QuasiInv}{\ensuremath{QI}}
\newcommand{\QInvertiblesOf}[1]{\ensuremath{#1^{q}}}

\newcommand{\Hom}[3]{\ensuremath{#1_{*}} }

\newcommand{\Lin}[3][]{
\ensuremath{
\ifthenelse{\equal{#1}{}}
	{
		\ifthenelse{\equal{#2}{#3}}
			{\mathrm{L}(#2)}
			{\mathrm{L}(#2,#3)}
	}
	{
		\mathrm{L}^{#1}(#2, #3)
	}
}
}
\newcommand{\EmbA}[3][]{\ensuremath{\mathcal E}_{#2,#3}\ifthenelse{\equal{#1}{}}{}{^{#1}}}

\newcommand{\Evol}[3]{\ensuremath{\mathrm{Evol}_{#1}^{#3}\ifthenelse{\equal{#2}{}}{}{(#2)}}}
\newcommand{\evol}[3]{\ensuremath{\mathrm{evol}_{#1}^{#3}\ifthenelse{\equal{#2}{}}{}{(#2)}}}

\newcommand{\lEvol}[2][]{\Evol{#1}{#2}{\ell}}
\newcommand{\rEvol}[2][]{\Evol{#1}{#2}{\rho}}

\newcommand{\levol}[2][]{\evol{#1}{#2}{\ell}}
\newcommand{\revol}[2][]{\evol{#1}{#2}{\rho}}

\newcommand{\ex}[1][]{\exp_{#1}}
\newcommand{\SX}{\ensuremath X}
\newcommand{\SY}{\ensuremath Y}
\newcommand{\SZ}{\ensuremath Z}

\newcommand{\UF}{\ensuremath U}
\newcommand{\VF}{\ensuremath V}
\newcommand{\WF}{\ensuremath W}

\newcommand{\one}{\ensuremath{\mathbf{1}}}
\newcommand{\G}{G}

\newcommand{\normsOn}[1]{\ensuremath{\mathcal{N}(#1)}}

\newglossaryentry{basic}{name={Basic notation}, sort=a, description={\nopostdesc}}%
\newglossaryentry{LieG_and_Mani}{name={Lie groups and manifolds}, sort=g, description={\nopostdesc}}%
\newglossaryentry{other}{name={Further notation}, sort=z, description={\nopostdesc}}%
\newglossaryentry{funcspac}{name={Spaces of weighted functions, and related notation}, sort=d, description={\nopostdesc}}%
\newglossaryentry{weighgroups}{name={Groups and monoids of functions}, sort=y, description={\nopostdesc}}%

\newglossaryentry{compactly_supported_diffeos}{name={\ensuremath{\Diff{M}{}{c}}}, sort=Diffc, parent=weighgroups, description={The diffeomorphisms of a manifold $M$ that coincide with the identity outside some compact set}}%
\newglossaryentry{Diffeomorphismen_S_getragen}{name={\ensuremath{\Diff{\R^n}{}{\mathcal{S}}}}, sort=DiffS, parent=weighgroups, description={The set of diffeomorphisms of $\R^n$ differing from $\id{\R^n}$ by a rapidly decreasing $\R^n$-valued map}}%
\newglossaryentry{extrealnumbers}{name={\ensuremath{\cl{\R}}}, sort=r, parent=basic, description={\ensuremath{\R\cup\{-\infty,\infty\}}}}
\newglossaryentry{extnaturalnumbers}{name={\ensuremath{\cl{\N}}}, sort=n, parent=basic, description={\ensuremath{\N \cup\{\infty\} = \sset{\infty, 0, 1, \dotsc}}}}
\newglossaryentry{naturalnumbersOhneNull}{name={\ensuremath{\N^\ast}}, sort=n, parent=basic, description={\ensuremath{\N \setminus\sset{0}}} }
\newglossaryentry{dist}{name={\ensuremath{\dist{A}{B}}}, sort=dist, parent=basic, description={Distance between \ensuremath{A} and \ensuremath{B} }}
\newglossaryentry{openBall}{name={\ensuremath{\Ball{x}{r}}}, sort=ball, parent=basic, description={Open ball with radius \ensuremath{r} around \ensuremath{x}}}%
\newglossaryentry{openBallother}{name={\ensuremath{\Ball[\SX]{x}{r}}}, sort=ballb, parent=basic, description={Like \glshyperlink{openBall}, here with indication of the space \ensuremath{\SX}}}%
\newglossaryentry{closedBall}{name={\ensuremath{\clBall{x}{r}}}, sort=ballclosed, parent=basic, description={Closed ball with radius \ensuremath{r} around \ensuremath{x}}}%
\newglossaryentry{Disk}{name={\ensuremath{\Disk}}, sort=disk, parent=basic, description={The closed unit disk in \ensuremath{\R} or \ensuremath{\C}}}
\newglossaryentry{Koerper}{name={\ensuremath{\K}}, sort=k, parent=basic, description={The field $\R$ or $\C$}}%
\newglossaryentry{k_times_diffable_maps}{name={\ensuremath{\ConDiff{\UF}{\SY}{k}}}, sort=Ck, parent=basic, description={The set of all $k$ times differentiable functions from $\UF$ to $\SY$}}%
\newglossaryentry{kth_iterated_derivative}{name={\ensuremath{\dA[k]{f}{}{}}}, sort=d(k), parent=basic, description={$k$-th iterated derivative of $f$}}%
\newglossaryentry{k_times_Frechet-diffable_maps}{name={\ensuremath{\FC{\UF}{\SY}{k}}}, sort=CkF, parent=basic, description={The set of all $k$ times Fr\'echet differentiable functions from $\UF$ to $\SY$}}
\newglossaryentry{kth_iterated_Frechet-derivative}{name={\ensuremath{ \FAbl[k]{\gamma} }}, sort=d(k)F, parent=basic, description={$k$-th Fr\'echet derivative of $\gamma$}}%
\newglossaryentry{gewichtete_f,k-Halbnorm}{name={\ensuremath{\hn{\cdot}{f}{k}}}, sort=normfk, parent=funcspac, description={Supremum of the \glshyperlink[operator norm]{Operatornorm_einer_k-linear-maps-normed} of the \glshyperlink[$k$-th Fr\'echet derivative]{kth_iterated_Frechet-derivative} multiplied with $f$. This quasinorms define the spaces \glshyperlink{Raeume_gewichteter_Abbildungen}}}%
\newglossaryentry{Raeume_gewichteter_Abbildungen}{name={\ensuremath{\CF{\UF}{\SY}{\cW}{k}}}, sort=CkWa, parent=funcspac, description={The set of $k$-times differentiable functions from $\UF$ to $\SY$ that and \glshyperlink[whose derivatives]{kth_iterated_Frechet-derivative} are $\GewFunk$-bounded}}%
\newglossaryentry{Raeume_beschraenkter_Abbildungen}{name={\ensuremath{\BC{\UF}{\SY}{k}}}, sort=CkB, parent=funcspac, description={The set of $k$-times differentiable functions from $\UF$ to $\SY$ that and \glshyperlink[whose derivatives]{kth_iterated_Frechet-derivative} are bounded}}%
\newglossaryentry{Raeume_gewichteter_Abbildungen_Abstand_Rand}{name={\ensuremath{\CcFo{\UF}{\VF}{k}}}, sort=CkWc, parent=funcspac, description={The functions $\gamma \in \glshyperlink{Raeume_gewichteter_Abbildungen}$ such that $\glshyperlink[\dist{\gamma(\UF)}{\SY \setminus \VF}]{dist} > 0$}}%
\newglossaryentry{beschraenkte_Abbildungen_Abstand_Rand}{name={\ensuremath{\BCo{\UF}{\VF}{k}}}, sort=CkB, parent=funcspac, description={The functions $\gamma \in \glshyperlink{Raeume_beschraenkter_Abbildungen}$ such that $\glshyperlink[\dist{\gamma(\UF)}{\SY \setminus \VF}]{dist} > 0$}}%
\newglossaryentry{Raeume_gwichteter_fallender_Abbildungen}{name={\ensuremath{\CcFvan{\UF}{\SY}{k}}}, sort=CkWv, parent=funcspac, description={The set of functions in \glshyperlink{Raeume_gewichteter_Abbildungen} whose seminorms decay outside of bounded sets}}%
\newglossaryentry{Raeume_beschraenkter_Abb_Null_Nullstelle}{name={\ensuremath{\BCzero{\UF}{\VF}{k}}}, sort=CkBz, parent=funcspac, description={The subset of functions in \glshyperlink{Raeume_beschraenkter_Abbildungen} mapping $0$ to $0$}}%
\newglossaryentry{kompakt_getragene_glatte}{name={\ensuremath{\DCcInf{\UF}{\VF}}}, sort=Ckc, parent=funcspac, description={The compactly supported smooth functions defined on $\UF$ taking values in $\VF$}}%
\newglossaryentry{kompakt_getragene_glatteD}{name={\ensuremath{\mathcal{D}(\UF, \VF)}}, sort=D, parent=funcspac, description={See \glshyperlink{kompakt_getragene_glatte}}, see={kompakt_getragene_glatte}}%
\newglossaryentry{gewichtete_f,k,p-Halbnorm}{name={\ensuremath{\hn{\gamma}{p, f}{k}}}, sort=normfkp, parent=other, description={For a suitable map $\gamma$ taking values in the locally convex space $\SY$ and $p \in \glshyperlink[\normsOn{\SY}]{Menge_der_stetigen_Halbnormen}$, this is \glshyperlink[$\hn{ \glshyperlink{Projektion_normierter_Faktorraum} }{f}{k}$]{gewichtete_f,k-Halbnorm} }}%
\newglossaryentry{Operatornorm_Bild_Lokalkonvex}{name={\ensuremath{\norm{T}_{op, p}}}, sort=normop, parent=other, description={The \glshyperlink[operator topology]{Operatornorm_einer_k-linear-maps-normed} of the linear operator $T : \SX \to \SY$, where $\SX$ is a normed space and $\SY$ a locally convex space, with respect to $p \in \glshyperlink[\normsOn{\SY}]{Menge_der_stetigen_Halbnormen}$}}%
\newglossaryentry{Raeume_gwichteter_fallender_Abbildungen-kompakt}{name={\ensuremath{\CcFvanK{\UF}{\SY}{k}}}, sort=CkWvc, parent=funcspac, description={The set of functions in \glshyperlink{Raeume_gewichteter_Abbildungen} whose seminorms decay outside of compact sets}}%
\newglossaryentry{Raeume_gwichteter_fallender_Abbildungen-kompakt_Werte_Teilmenge}{name={\ensuremath{\CcFvanK{\UF}{\VF}{k}}}, sort=CkWvcU, parent=funcspac, description={The set of functions in \glshyperlink{Raeume_gwichteter_fallender_Abbildungen-kompakt} taking values in $\VF$}}%
\newglossaryentry{Diffeomorphismen_Braum}{name={\ensuremath{\Diff{\SX}{}{}}}, sort=diff, parent=weighgroups, description={The set of all diffeomorphisms of the Banach space $\SX$}}%
\newglossaryentry{gew_Diffeomorphismen_Braum}{name={\ensuremath{\Diff{\SX}{}{\cW}}}, sort=diffW, parent=weighgroups, description={The set of weighted diffeomorphisms of the Banach space $\SX$ to the weights $\cW$. Is a Lie group modelled on \glshyperlink[$\CcF{\SX}{\SX}{\infty}$]{Raeume_gewichteter_Abbildungen}}}%
\newglossaryentry{gew_Endomorphismen_Braum}{name={\ensuremath{\Endos{\SX}{\cW}}}, sort=endW, parent=weighgroups, description={The set of weighted endomorphisms of the Banach space $\SX$ to the weights $\cW$}}%
\newglossaryentry{chart_gewEndos}{name={\ensuremath{\karte{\cW} }}, sort=chartEndWDiffW, parent=weighgroups, description={The inverse of the canonical chart for \glshyperlink{gew_Endomorphismen_Braum} and \glshyperlink{gew_Diffeomorphismen_Braum}}}%
\newglossaryentry{composition_map_id}{name={\ensuremath{\compIdDiffKLcW{\SY}{ k}{ \ell} }}, sort=compositionMapId, parent=funcspac, description={The map $(\gamma,\eta) \mapsto \gamma\circ(\eta + \id{})$, restricted to certain \glshyperlink[weighted function spaces]{Raeume_gewichteter_Abbildungen}}}%
\newglossaryentry{gew_abfallende_Endomorphismen_Braum}{name={\ensuremath{\EndWvan}}, sort=endWv, parent=weighgroups, description={The functions $\phi \in \glshyperlink{gew_Endomorphismen_Braum}$ such that $\phi - \id{\SX} \in \glshyperlink[\CcFvan{\SX}{\SX}{\infty}]{Raeume_gwichteter_fallender_Abbildungen}$}}%
\newglossaryentry{inversion_map_id}{name={\ensuremath{\InvIdcW{V} }}, sort=inversionMapId, parent=funcspac, description={The map $\phi \mapsto (\phi + \id{U})^{-1} - \id{V}$, defined on \glshyperlink{domain_inversion_map_id}}}%
\newglossaryentry{domain_inversion_map_id}{name={\ensuremath{\InvertWinfDomRan{U}{V} } }, sort=omegaDomInversionMapId, parent=funcspac, description={The maps $\phi \in \glshyperlink[\CcF{U}{\SX}{\infty}]{Raeume_gewichteter_Abbildungen}$ for which $\phi + \id{U}$ is injective, and its image contains $V$}}%
\newglossaryentry{gew_abfallende_Diffeomorphismen_Braum}{name={\ensuremath{\DiffWvan}}, sort=diffWv, parent=weighgroups, description={$\glshyperlink{gew_Diffeomorphismen_Braum} \cap \glshyperlink{gew_abfallende_Endomorphismen_Braum}$}}%
\newglossaryentry{Ableitung_Gruppenwirkung}{name={\ensuremath{\AblAct{\omega}}}, sort = GroupActionDerivative, parent=LieG_and_Mani, description={For a group action $\omega$, this denotes some kind of \enquote{derivation} at the unital element}}%
\newglossaryentry{gew_Diffeomorphismen_Braum_ZKo}{name={\ensuremath{\DiffWz}}, sort=diffWz, parent=weighgroups, description={The identity component of \glshyperlink{gew_Diffeomorphismen_Braum}}}
\newglossaryentry{max_extend_weights}{name={\ensuremath{\ExtWeights{\GewFunk} }}, sort=weights_extended, parent=other, description={For $\GewFunk \sub \glshyperlink{extrealnumbers}^\UF$, the set of functions $f : \UF \to \cl{R}$ such that $\glshyperlink[\hn{\cdot}{f}{0}]{gewichtete_f,k-Halbnorm}$ is continuous on $\glshyperlink[\CcF{\UF}{\SY}{0}]{Raeume_gewichteter_Abbildungen}$, for each normed space $\SY$. Called the \emph{maximal extension} of $\GewFunk$.}}%
\newglossaryentry{gew_Abbildungsgruppe_BL-Gruppe}{name={\ensuremath{\CcF{\UF}{\G}{\ell}}}, sort=CkWLG, parent=weighgroups, description={Lie group of weighted mappings that take values in a Banach Lie group, modelled on \glshyperlink[$\CcF{\UF}{\LieAlg{\G}}{\ell}$]{Raeume_gewichteter_Abbildungen}}}%
\newglossaryentry{gew_Abbildungsgruppe_abfallend-kompakt}{name={\ensuremath{\CcFvanK{\UF}{\G}{\ell}}}, sort=CkWvLG, parent=weighgroups, description={Lie group of decaying weighted mappings that take values in a Lie group, modelled on \glshyperlink{Raeume_gwichteter_fallender_Abbildungen-kompakt}}}%
\newglossaryentry{Dicke_gew_Abbildungsgruppe_abfallend-kompakt}{name={\ensuremath{\CcFvanKBig{\UF}{\G}{k}}}, sort=CkWvLGbig, parent=weighgroups, description={Lie group normalizing \glshyperlink{gew_Abbildungsgruppe_abfallend-kompakt}}}%
\newglossaryentry{normierter_Faktorraum}{name={\ensuremath{\SX_p}}, sort=lkvxT2, parent=other, description={For a locally convex space $\SX$ and $p \in \glshyperlink{Menge_der_stetigen_Halbnormen}$, this denotes the quotient space $\SX /p^{-1}(0)$}}%
\newglossaryentry{Projektion_normierter_Faktorraum}{name={\ensuremath{\morQuot{p}}}, sort=lkvxT2quo, parent=other, description={For a locally convex space $\SX$ and $p \in \glshyperlink{Menge_der_stetigen_Halbnormen}$, this denotes the quotient map $\SX \to \glshyperlink{normierter_Faktorraum}$}}%
\newglossaryentry{Menge_der_stetigen_Halbnormen}{name={\ensuremath{\normsOn{\SX}}}, sort=lkvxNorms, parent=other, description={For a locally convex space $\SX$, this denotes the set of continuous seminorms on $\SX$}}%
\newglossaryentry{space_of_k-linear-maps-normed}{name={\ensuremath{\Lin[k]{\SX}{\SY}}}, sort=Lin(k), parent=other, description={Space of $k$-linear maps from the normed space $\SX$ to the normed space $\SY$, endowed with the  \glshyperlink[operator topology]{Operatornorm_einer_k-linear-maps-normed}}}%
\newglossaryentry{Operatornorm_einer_k-linear-maps-normed}{name={$\Opnorm{\cdot}$}, sort=normLin(k), parent=other, description={The operator norm of a $k$-linear map between normed spaces}}%
\newglossaryentry{Tangent_space_basepoint}{name={\ensuremath{\Tang[x]{M}}}, sort = TangentSpaceBasepoint, parent=LieG_and_Mani, description={The tangent space at the point $x$ of the manifold $M$}}%
\newglossaryentry{Tangent_bundle}{name={\ensuremath{\Tang{M}}}, sort = TangentBundle, parent=LieG_and_Mani, description={The tangent bundle of the manifold $M$}}%
\newglossaryentry{Tangent_map}{name={\ensuremath{\Tang{f}}}, sort = TangentMap, parent=LieG_and_Mani, description={For a differentiable map $f : M \to N$ between the manifolds $M$ and $N$, this denotes the tangent map $\glshyperlink{Tangent_bundle} \to \Tang{N} $}}%
\newglossaryentry{Tangent_map_basepoint}{name={\ensuremath{\Tang[x]{f}}}, sort = TangentMapBasepoint, parent=LieG_and_Mani, description={The restriction of \glshyperlink{Tangent_map} to \glshyperlink{Tangent_space_basepoint} and $\Tang[f(x)]{N}$, respectively}}%
\newglossaryentry{partial_Tangent_map}{name={$\Tang[1]{f}$, $\Tang[2]{f}$}, sort = TangentMapPartial, parent=LieG_and_Mani, description={For a $\ConDiff{}{}{1}$-map $f$ defined on a product $M \times N$, these denote the partial \glshyperlink[tangent maps]{Tangent_map}}}%
\newglossaryentry{SetOf_Vectorfields}{name={\ensuremath{\VecFields{M}}}, sort = VectorFields, parent=LieG_and_Mani, description={The set of smooth vector fields of the manifold $M$}}%
\newglossaryentry{LieAlgebra_functor}{name={$\LieAlg{\cdot}$}, sort = LieAlgebraFunktor, parent=LieG_and_Mani, description={The Lie algebra functor}}%
\newglossaryentry{LeftLogDeriv}{name={\ensuremath{\LLA{\cdot}}}, sort = LogarithmicDerivativeLeft, parent=LieG_and_Mani, description={The left logarithmic derivative}}%
\newglossaryentry{Left_Evolution}{name={\ensuremath{\lEvol[G]{}}}, sort = EvolutionLeft, parent=LieG_and_Mani, description={The left evolution}}%
\newglossaryentry{Left_Evolution_endpoint}{name={\ensuremath{\levol[G]{}} }, sort = EvolutionLeftEndpoint, parent=LieG_and_Mani, description={The endpoint of the left evolution}}%
\newglossaryentry{RightLogDeriv}{name={\ensuremath{\RLA{\cdot}}}, sort = LogarithmicDerivativeRight, parent=LieG_and_Mani, description={The right logarithmic derivative}}%
\newglossaryentry{Right_Evolution}{name={\ensuremath{\rEvol[G]{}}}, sort = EvolutionRight, parent=LieG_and_Mani, description={The right evolution}}%
\newglossaryentry{Right_Evolution_endpoint}{name={\ensuremath{\revol[G]{}}}, sort = EvolutionRightEndpoint, parent=LieG_and_Mani, description={The endpoint of the right evolution}}%
\newglossaryentry{Exponentialfunktion}{name={\ensuremath{\ex[G]}}, sort = exp, parent=LieG_and_Mani, description={The exponential function of the Lie group $G$}}%
\newglossaryentry{Quasi-Invertierbare}{name={\ensuremath{\QInvertiblesOf{A}} }, sort=QuasiInverts, parent=other, description={The set of quasi-invertible elements of the algebra $A$}}
\newglossaryentry{Quasi-Inversion}{name={\ensuremath{\QuasiInv_A}}, sort=QuasiInv, parent=other, description={Quasi-Inversion map $\glshyperlink{Quasi-Invertierbare} \to \QInvertiblesOf{A}$ of the algebra $A$}}

\hypersetup{%
	breaklinks,
	plainpages=false,
	pdfpagelabels,
	hyperindex
}
\newtheorem{COUNT}{}[section]
\newtheorem{satz}[COUNT]{Theorem}
\newtheorem{lem}[COUNT]{Lemma}
\newtheorem{cor}[COUNT]{Corollary}
\newtheorem{prop}[COUNT]{Proposition}
\theoremstyle{definition}
\newtheorem{defi}[COUNT]{Definition}
\newtheorem{bem}[COUNT]{Remark}
\newtheorem{beisp}[COUNT]{Example}

\newtheorem*{thmOhneNummer}{Theorem}

\numberwithin{equation}{COUNT}

\newcommand{\BeweisschrittCkCinfty}[9]
{
	From the assertions already established, we derive the commutative diagram
	\[
		\xymatrix{
		{#1} \ar[rr]^-{#2} \ar@{>->}[d] && {#3} \ar@{>->}[d]
		\\
		{#6} \ar[rr]_-{#7} && {#8}
		}
	\]
	for each $#9\in\N$,
	where the vertical arrows represent the inclusion maps.
	With \refer{cor:Topologie_von_CinfF} we easily deduce the continuity of $#2$ from the one of $#7$.
}

\newcommand{\BeweisschrittCkCinftySmooth}[9]
{
	From the assertions already established, we derive the commutative diagram
	\[
		\xymatrix{
		{#1} \ar[rr]^-{#2} \ar@{>->}[d] && {#3} \ar@{>->}[d]
		\\
		{#6} \ar[rr]_-{#7} && {#8}
		}
	\]
	for each $#9\in\N$,
	where the vertical arrows represent the inclusion maps.
	With \refer{cor:Topologie_von_CinfF} we easily deduce the smoothness of $#2$ from the one of $#7$.
}

\makeglossaries
\makeindex

\begin{document}
\ifthenelse{ \equal{\value{RELEASE}}{0} } {\listoftodos}{}

\title{Weighted diffeomorphism groups of Banach spaces and weighted mapping groups} 
\author{Boris Walter}
\date{}
\publishers{
\small{}
Universität Paderborn\\ Institut für Mathematik\\
Warburger Straße 100\\
33098 Paderborn\\
E-Mail: bwalter@math.upb.de
}

\maketitle

\newcommand{\keywords}[1]{\def\mykeywords{#1}}
\newcommand{\mathclass}[1]{\def\mymathclass{#1}}
\newcommand{\acknow}[1]{\def\myacknow{#1}}
\let\oldabs\abstract
\let\endoldabs\endabstract
\renewenvironment{abstract}{\oldabs}{

\hfill
\hangindent=\parindent

\noindent{} \emph{Acknowledgement}. \myacknow

\hfill

2010 \emph{Mathematics Subject Classification}: \mymathclass

\emph{Keywords}: \mykeywords
\endoldabs}

\acknow{The current work comprises results from the author's
Diplomarbeit and his Ph.D.-project (advised by Helge Glöckner). The research was supported
by the German Research Foundation (DFG), grant GL 357/4-1.}
\mathclass{Primary 58D05; Secondary 22E65, 22E67, 26E15, 26E20, 46E10, 46E40, 46E50, 46T05, 46T10, 46T20, 58D15.}

\begin{abstract}
	In this work\footnotemark{}, we construct and study certain classes of infinite dimensional Lie groups
	that are modelled on weighted function spaces.
	In particular, we construct a Lie group $\DiffW$ of diffeomorphisms,
	for each Banach space $\SX$ and set $\cW$ of weights on $\SX$ containing the constant weights.
	We also construct certain types of \enquote{weighted mapping groups}.
	These are Lie groups modelled on weighted function spaces of the form $\CcF{\UF}{\LieAlg{G}}{k}$,
	where $G$ is a given (finite- or infinite dimensional) Lie group.
	Both the weighted diffeomorphism groups and the weighted mapping groups are shown to be
	regular Lie groups in Milnor's sense.

	We also discuss semidirect products of the former groups.
	Moreover, we study the integrability of Lie algebras of vector fields of the form
	$\CcF{\SX}{\SX}{\infty} \rtimes \LieAlg{\G}$, where $\SX$ is a Banach space
	and $G$ a Lie group acting smoothly on $\SX$.
	\footnotetext{ An earlier version was published in \cite{MR2952176}.
	The published text lacks the examination of topologies on spaces of multipliers on page \pageref{lem:Kriterium_simultane_Stetigkeit_bilineare_Operation_Multiplier-gewAbb}ff.
	Further, this text contains a more general treatment of the inversion map of $\DiffW$ in \refer{susec:DiffW_Inversion}
	(which can be used to investigate functions that are defined on a subset of $\SX$)
	which uses a Lipschitz inverse function theorem (stated in \refer{susec:InverseFuncThm-Lipschitz}).}
\end{abstract}

\tableofcontents

\chapter{Introduction}
Diffeomorphism groups of compact manifolds, as well as
groups $\ConDiff{K}{G}{k}$ of Lie group-valued mappings on compact manifolds
are among the most important and well-studied examples of infinite dimensional Lie groups
(see for example \cite{MR0210147}, \cite{MR830252}, \cite{MR656198}, \cite{MR1421572}, \cite{MR900587} and \cite{MR1471480}).
While the diffeomorphism group $\Diff{K}{}{}$ of a compact manifold is modelled
on the Fr\'echet space $\ConDiff{K}{\Tang{K} }{\infty}$ of smooth vector fields on $K$,
for a non-compact smooth manifold $M$, it is not possible to make $\Diff{M}{}{}$
a Lie group modelled on the space of all smooth vector fields in a satisfying way
(see \cite{MilnorPreprint82}).
We mention that the LF-space $\CF{M}{\Tang{ M}}{c}{\infty}$
of compactly supported smooth vector fields can be used as the modelling space
for a Lie group structure on $\Diff{M}{}{}$.
But the topology on this Lie group is too fine for many purposes;
the group $\glstext{compactly_supported_diffeos}$ of compactly supported diffeomorphisms
\index{compactly supported diffeomorphisms}%
\index{diffeomorphisms!compactly supported|see{compactly supported diffeomorphisms}}%
\index{diffeomorphisms!groups of}%
(which coincide with the identity map outside some compact set)
is an open subgroup (see \cite{MR583436} and \cite{MilnorPreprint82}).
Likewise, it is no problem to turn groups $\CF{M}{G}{c}{k}$
of compactly supported Lie group-valued maps into Lie groups
(cf. \cite{MR830252}, \cite{MR1233384}, \cite{MR1934608}).
However, only in special cases there exists a Lie group structure on $\ConDiff{M}{G}{\infty}$,
equipped with its natural group topology, the smooth compact-open topology (see \cite{MR2399649}).

In view of these limitations, it is natural to look for Lie groups of diffeomorphisms
which are larger then $\Diff{M}{}{c}$ and modelled on larger Lie algebras of vector fields
than $\CF{M}{\Tang{ M}}{c}{\infty}$.
In the same vain, one would like to find mapping groups modelled on larger spaces than
$\CF{M}{\LieAlg{G}}{c}{k}$.

In this work, we construct such groups in the important case where the non-compact manifold $M$
is a vector space (or an open subset thereof, in the case of mapping groups).
For most of the results, the vector space is even allowed to be a Banach space $\SX$.
The groups we consider are modelled on spaces of weighted functions on $\SX$.
For example, we are able to construct a Lie group structure on the group
$\glstext{Diffeomorphismen_S_getragen}$
of diffeomorphisms differing from $\id{\R^n}$
by a rapidly decreasing $\R^n$-valued map. Considered as a topological group,
this group has been used in quantum physics (\cite{Goldin2004}).
For $n = 1$, another construction of the Lie group structure (in the setting of convenient differential calculus) has been given by P.~Michor (\cite[\S 6.4]{MR2263211}),
and applied to the Burgers' equation.
The general case was treated in the author's unpublished diploma thesis \cite{Wal2006}.

To explain our results, let $\SX$ and $\SY$ be Banach spaces, $\UF \sub \SX$ open and nonempty,
$k \in \cl{\N} \ndef \N \cup \{\infty\}$, $\cW$ be a set of functions $f$ on $\UF$ taking values
in the extended real line $\cl{\R} \ndef \R \cup \sset{\infty, -\infty}$ called weights.
As usual, we let $\CcF{\UF}{\SY}{k}$ be the set of all $k$-times continuously Fr\'echet-differentiable functions
$\gamma :\UF \to \SY$ such that $f \cdot \Opnorm{\FAbl[\ell]{\gamma}}$ is bounded
for all integers $\ell \leq k$ and all $f \in \cW$.
Then $\CcF{\UF}{\SY}{k}$ is a locally convex topological vector space in a natural way.
We prove (see \refer{satz:DiffW_Lie-Gruppe} and \refer{satz:Gewichtete_Diffeos_regulaere_Liegruppen})
\begin{thmOhneNummer}
	Let $\SX$ be a Banach space
	and $\GewFunk \sub \cl{\R}^\SX$ with $1_\SX \in \GewFunk$.
	Then $\DiffW \ndef \{ \phi \in \Diff{\SX}{}{} :
		\phi - \id{\SX} ,\,\phi^{-1} - \id{\SX}
		\in \CF{\SX}{\SX}{\cW}{\infty}
	\}$
	is a regular Lie group modelled on $\CcF{\SX}{\SX}{\infty}$.
\end{thmOhneNummer}
Replacing $\CcF{\SX}{\SX}{\infty}$ by the subspace of functions $\gamma$
such that $f(x) \cdot \Opnorm{\FAbl[\ell]{\gamma}(x)} \to 0$ as $\norm{x} \to \infty$,
we obtain a subgroup $\DiffWvan$ of $\DiffW$ which also is a Lie group
(see \refer{prop:DiffWvan_UnterLiegruppen_von_DiffW}).

As for mapping groups, we first consider mappings into Banach Lie groups.
In \refer{sec:gewAbbildungsgruppen_Banach-Lie} we show
\begin{thmOhneNummer}
	Let $\SX$ be a normed space, $\UF \subseteq \SX$ an open nonempty subset,
	$\GewFunk \subseteq \cl{\R}^\UF$ with $1_\UF \in \GewFunk$, $k \in \cl{\N}$
	and $\G$ be a Banach Lie group.
	Then there exists a connected Lie group $\CcF{\UF}{\G}{k} \sub \G^\UF$
	modelled on $\CcF{\UF}{\LieAlg{\G} }{k}$, and this Lie group is regular.
\end{thmOhneNummer}
Using the natural action of diffeomorphisms on functions,
we always form the semidirect product $\CcF{\SX}{\G}{\infty} \rtimes \DiffW$
and make it a Lie group.

In the case of finite-dimensional domains, we can even discuss mappings into
arbitrary Lie groups modelled on locally convex spaces.
To this end, given a locally convex space $\SY$ and
an open subset $\UF$ in a finite-dimensional vector space $\SX$
we define a certain space $\CcFvanK{\UF}{\SY}{k}$ of $\ConDiff{}{}{k}$-maps
which decay as we approach the boundary of $\UF$, together with their derivatives
(see \refer{defi:Definition_CFvanK} for details).
We obtain the following result
\begin{thmOhneNummer}
	Let $\SX$ be a finite-dimensional space, $\UF \subseteq \SX$ an open nonempty subset,
	$\GewFunk \subseteq \cl{\R}^\UF$ with $1_\UF \in \GewFunk$, $k \in \cl{\N}$
	and $\G$ be a locally convex Lie group.
	Then there exists a connected Lie group $\CcFvanK{\UF}{\G}{k} \sub \G^\UF$
	modelled on $\CcFvanK{\UF}{\LieAlg{\G} }{k}$.
\end{thmOhneNummer}
We also discuss certain larger subgroups of $\G^\UF$ admitting Lie group structures
that make $\CcFvanK{\UF}{\G}{k}$ an open normal subgroup (see \refer{susec:LArger_group_CWvanK}).

Finally, we consider Lie groups $G$ acting smoothly on a Banach space $\SX$.
We investigate when the $G$-action leaves the identity component $\DiffWz$
of $\DiffW$ invariant and whether $\DiffWz \rtimes G$ can be made a Lie group in this case.
In particular, we show that $\Diff{\R^n}{}{\mathcal{S}}_0 \rtimes \GL{\R^n}$
is a Lie group for each $n$ (\refer{beisp:GL_operiert_glatt_auf_DiffS}).
By contrast, $\GL{\R^n}$ does not leave $\Diff{\R^n}{}{ \sset{1_{\R^n} } }$ invariant
(\refer{beisp:GL_operiert_NICHT_auf_DiffBC}).

We mention that certain weighted mapping groups on finite-dimensional spaces (consisting of smooth mappings)
have already been discussed in \cite[\S 4.2]{MR654676} assuming additional hypotheses
on the range group (cf. \refer{bem:BCR}).
Besides the added generality, we provide a more complete discussion
of superposition operators on weighted function spaces.

In the case where $\cW = \sset{1_\SX}$, our group $\DiffW$ also has a counterpart
in the studies of Jürgen Eichhorn and collaborators (\cite{MR1856078}, \cite{MR1392288}, \cite{MR2343536}),
who studied certain diffeomorphism groups on non-compact manifolds with bounded geometry.

Semidirect products of diffeomorphism groups and function spaces on compact manifolds
arise in Ideal Magnetohydrodynamics (see \cite[\RN{2}.3.4]{MR2456522}).
Further, the group $\mathcal{S}(\R^n) \rtimes \Diff{\R^n}{}{\mathcal{S}}$
and its continuous unitary representations are encountered in Quantum Physics
(see \cite{Goldin2004}; cf.\ also \cite[\S 34]{MR1393939} and the references therein).
\chapter{Preliminaries and notation}
We give some notation and basic definitions.
More details are provided in the appendix,
as is a list of symbols used in this work.
\section{Notation}
We write $\glstext{extrealnumbers} \ndef \R\cup\{-\infty,\infty\}$,%
$\glstext{extnaturalnumbers} \ndef \N \cup\{\infty\}$%
and $\glstext{naturalnumbersOhneNull} \ndef \N \setminus\{0\}$.
Further we denote norms by $\norm{\cdot}$.
\begin{defi}
	Let $A, B$ be subsets of the normed space $\SX$.
	As usual, the \emph{distance} of $A$ and $B$ is defined as
	\begin{equation*}
		\glstext{dist}
		\ndef\inf\{\norm{a-b}: a \in A, b \in B \} \in [0,\infty].
	\end{equation*}
	Thus $\dist{A}{B} = \infty$ iff $A = \emptyset$ or $B = \emptyset$.
	
	Further, for $x \in \SX$ and $r \in \R$ we define
	\[
		\glstext{openBallother} \ndef
		\{y \in \SX : \norm{y - x} < r\}
	\]
	Occasionally, we just write $\glstext{openBall}$ instead of $\Ball[\SX]{x}{r}$.
	For the closed ball, we write $\glstext{closedBall}$ and the like.%
	
	Further, we define
	\[
		\glstext{Disk} \ndef \clBall[\K]{0}{1},
	\]
	where $\glstext{Koerper} \in \sset{\R,\C}$.
	No confusion will arise from this abuse of notation.
\end{defi}

\section{Differential calculus of maps between locally convex spaces}
\label{sec:Helge_diffbar--vorne}
We give basic definitions for the differential calculus for maps between locally convex spaces
that is known as Kellers $C^k_c$-theory.
More results can be found in \refer{sec:Helge_diffbar}.

\begin{defi}[Directional derivatives]
	Let $\SX$ and $\SY$ be locally convex spaces,
	$\UF \subseteq \SX$ an open nonempty set,
	$u\in\UF$, $x\in\SX$ and $f:\UF\to\SY$ a map.
	The \emph{derivative of $f$ at $u$ in the direction $x$}
	is defined as
	\[
		\lim_{\substack{t\to 0\\ t \in \K^\ast}}\frac{f(u + t x) - f(u)}{t}
			\defn (D_x f)(u) \defn \dA{f}{u}{x},
	\]
	whenever that limit exists.
\end{defi}

\begin{defi}
	Let $\SX$ and $\SY$ be locally convex spaces, $\UF \subseteq \SX$ an open nonempty set,
	and $f:\UF\to\SY$ be a map.
	
	We call $f$ a \emph{$\ConDiff[\K]{}{}{1}$-map} or just \emph{$\ConDiff[\K]{}{}{1}$}
	if $f$ is continuous,
	the derivative $\dA{f}{u}{x}$ exists for all $(u,x)\in \UF\times\SX$
	and the map $\dA{f}{}{}: \UF\times\SX \to\SY$ is continuous.
	
	Inductively, for a $k\in\N$ we call $f$ a \emph{$\ConDiff[\K]{}{}{k}$-map}
	or just \emph{$\ConDiff[\K]{}{}{k}$}
	if $f$ is a $\ConDiff[\K]{}{}{1}$-map and
	$d^{1}f \ndef \dA{f}{}{} :\UF\times\SX\to\SY$
	is a $\ConDiff[\K]{}{}{k - 1}$-map.
	In this case, the \emph{$k$-th iterated differential} of $f$ is defined by
	\[
		d^k f\ndef d^{k-1}(d f)
		: \UF \times \SX^{2^k - 1} \to \SY .
	\]
	If $f$ is a \emph{$\ConDiff[\K]{}{}{k}$-map} for each $k\in\N$,
	we call $f$ a \emph{$\ConDiff[\K]{}{}{\infty}$-map} or just
	\emph{$\ConDiff[\K]{}{}{\infty}$} or \emph{smooth}.
	
	Further, for each $k \in \cl{\N}$ we define
	\[
		\ConDiff[\K]{\UF}{\SY}{k}
		\ndef \{f:\UF \to \SY \;|\; \text{$f$ is $\ConDiff[\K]{}{}{k}$}\}.
	\]
	Often, we shall simply write $\glstext{k_times_diffable_maps}$,
	$\ConDiff{}{}{k}$ and the like.
\end{defi}

It is obvious from the definition of differentiability
that iterated directional derivatives exist and depend continuously
on the directions.
The converse of this assertion also holds.
\begin{prop}\label{prop:hohe_Ableitungen_d}
	Let $f:\UF\to\SY$ be a continuous map and $r\in \cl{\N}$.
	Then $f \in \ConDiff{\UF}{\SY}{r}$ iff
	for all $u\in\UF$, $k\in\N$ with $k\leq r$ and $x_1, \dotsc, x_k \in \SX$
	the iterated directional derivative
	\[
		\dA[k]{f}{u}{x_1,\dotsc,x_k} \ndef (D_{x_k}\dotsm D_{x_1} f)(u)
	\]
	exists and the map
	\[
		\UF\times\SX^k \to \SY
		: (u, x_1,\dotsc,x_k)\mapsto \dA[k]{f}{u}{x_1,\dotsc,x_k}
	\]
	is continuous.
	We call $\glstext{kth_iterated_derivative}$ the \emph{$k$-th derivative of $f$}.
\end{prop}

\section{Fr\'echet differentiability}
\label{sec:Frechetbarkeit--vorne}
We give basic definitions for Fr\'echet differentiability for maps between normed spaces.
More results can be found in \refer{sec:Frechetbarkeit}.
\begin{defi}[Fr\'echet differentiability]\label{def:Frechet_Diffbarkeit}
	Let $\SX$ and $\SY$ be normed spaces and
	$\UF$ an open nonempty subset of $\SX$.
	We call a map $\gamma:\UF\to\SY$ \emph{Fr\'echet differentiable}
	or $\FC{}{}{1}$
	if it is a $\ConDiff{}{}{1}$-map and the map
	\[
		\FAbl{\gamma} :\UF\to \Lin{\SX}{\SY}: x\mapsto \dA{\gamma}{x}{\cdot}
	\]
	is continuous.
	Inductively, for $k \in \N^\ast$ we call $\gamma$
	a $\FC{}{}{k + 1}$-map if it is Fr\'echet differentiable
	and $\FAbl{\gamma}$ is a $\FC{}{}{k}$-map.
	We denote the set of all $k$-times Fr\'echet differentiable
	maps from $\UF$ to $\SY$ with $\glstext{k_times_Frechet-diffable_maps}$. Additionally, we define
	the \emph{smooth} maps by
	\[
		\FC{\UF}{\SY}{\infty} \ndef
			\bigcap_{k\in\N^\ast}\FC{\UF}{\SY}{k}
	\]
	and
	$\FC{\UF}{\SY}{0} \ndef \ConDiff{\UF}{\SY}{0}$.
	The map
	\[
		\FAbl{} :
		\FC{\UF}{\SY}{k+1}\to\FC{\UF}{\Lin{\SX}{\SY}}{k}
		:\gamma \mapsto \FAbl{\gamma}
	\]
	is called \emph{derivative operator}.
\end{defi}

\begin{bem}
	Let $\SX$ and $\SY$ be normed spaces,
	$\UF$ an open nonempty subset of $\SX$,
	$k \in \N^\ast$ and $\gamma \in \FC{\UF}{\SY}{k}$.
	Then for each $\ell \in \N^\ast$ with $\ell \leq k$
	there exists a continuous map
	\[
		\glsdisp{kth_iterated_Frechet-derivative}{\FAbl[\ell]{\gamma}} : \UF \to \Lin[\ell]{\SX}{\SY},
	\]%
	where $\Lin[\ell]{\SX}{\SY}$ denotes
	the space of $\ell$-linear maps $\SX^\ell \to \SY$,
	endowed with the operator topology.
	The map $\FAbl[\ell]{\gamma}$ can be described more explicitly.
	If $\gamma \in \FC{\UF}{\SY}{k}$,
	also $\gamma \in \ConDiff{\UF}{\SY}{k}$ holds,
	and for each $x\in\UF$ we have the relation
	\[
		\FAbl[k]{\gamma}(x) = \dA[k]{\gamma}{x}{\cdot}.
	\]
\end{bem}

\chapter{Weighted function spaces}
\label{chp:Modellraeume}
In this chapter we give the definition of some locally convex vector spaces
consisting of weighted functions.
The Lie groups that are constructed in this work will be modelled on these
spaces.
We first discuss maps between normed spaces. In \refer{sec:GewAbb_Bildbereich-lokalkonvex},
we will also look at maps that take values in arbitrary locally convex spaces.
The treatment of the latter spaces requires some rather technical effort.
Since these function spaces are only needed in \refer{sec: Weighted_maps_into_locally_convex_Lie_groups},
the reader may eventually skip this section.
\section{Definition and examples}
\begin{defi}
	Let $\SX$ and $\SY$ be normed spaces
	and $\UF \subseteq \SX$ an open nonempty set.
	For $k \in \N$ and a map $f: \UF \to \cl{\R}$ we define the quasinorm
	\[
		\glstext{gewichtete_f,k-Halbnorm}
		:\FC{\UF}{\SY}{k}\to [0,\infty]
		: \phi \mapsto \sup\{\abs{f(x)}\,\Opnorm{\FAbl[k]{\phi}(x)}:x\in \UF\}.
	\]
	Furthermore, for any nonempty set $\cW \sub \cl{\R}^\UF$
	and $k \in \cl{\N}$ we define the vector space
	\[
		\glstext{Raeume_gewichteter_Abbildungen}
		\ndef
		\{ \gamma\in\FC{\UF}{\SY}{k}:
			(\forall f\in\cW, \ell \in \N, \ell \leq k)\:
			\hn{\gamma}{f}{\ell} < \infty
		\}
	\]
	and notice that the seminorms $\hn{\cdot}{f}{\ell}$ induce a
	locally convex vector space topology on $\CF{\UF}{\SY}{\cW}{k}$.%

	We call the elements of $\cW$ \emph{weights}
	\index{weights}%
	and $\CF{\UF}{\SY}{\cW}{k}$ a \emph{space of weighted maps} or \emph{space of weighted functions}.
	\index{weighted maps!into normed spaces}%
\end{defi}

An important example is the space of bounded functions with bounded derivatives:
\begin{beisp}\label{beisp:BC}
	Let $k \in \cl{\N}$. We define
	\[
		\glstext{Raeume_beschraenkter_Abbildungen} \ndef \CF{\UF}{\SY}{\{1_\UF\}}{k} .
	\]%
	\index{bounded maps}%
\end{beisp}
\begin{bem}
	Let $\UF$ and $\VF$ be nonempty open subsets of a normed space $\SX$
	and $\UF \subseteq \VF$.
	For a set $\cW \subseteq \cl{\R}^{\VF}$, we define
	\[
		\rest{\cW}{\UF} \ndef \{ \rest{f}{\UF} : f \in \cW \}.
	\]
	Further we write with an abuse of notation
	\[
		\CF{\UF}{\SY}{\cW}{k}
		\ndef
		\CF{\UF}{\SY}{ \rest{\cW}{\UF} }{k} .
	\]
\end{bem}

\begin{bem}
As is clear, for any set $T \subseteq 2^\cW$ with $\cW = \bigcup_{\cF \in T}\cF$
we have
\[
	\CF{\UF}{\SY}{\cW}{k}
	= \bigcap_{\substack{\cF \in T\\ \ell \in \N, \ell \leq k}}
	\CF{\UF}{\SY}{\cF}{\ell}.
\]
\end{bem}
We define some subsets of $\CF{\UF}{\SY}{\cW}{k}$:
\begin{defi}
	Let $\SX$ and $\SY$ be normed spaces, $\UF \subseteq \SX$ and
	$\VF\subseteq \SY$ open nonempty sets and $\cW \subseteq \cl{\R}^{\UF}$.
	For $k\in\cl{\N}$ we set
	\[
		\CcF{\UF}{\VF}{k}
		\ndef
		\{ \gamma \in \CcF{\UF}{\SY}{k} : \gamma(\UF) \subseteq \VF \}
	\]
	and
	\[
		\glstext{Raeume_gewichteter_Abbildungen_Abstand_Rand}
		\ndef
		\{
			\gamma \in \CcF{\UF}{\VF}{k} :
			(\exists r > 0)\; \gamma(\UF) + \Ball[\SY]{0}{r} \subseteq \VF
		\}.
	\]
	Obviously
	\[
		\CcFo{\UF}{\VF}{k} \subseteq \CcF{\UF}{\VF}{k},
	\]
	and if $1_{\UF} \in \cW$, then
	$\CcFo{\UF}{\VF}{k}$ is open in $\CcF{\UF}{\SY}{k}$.
	The set $\glstext{beschraenkte_Abbildungen_Abstand_Rand}$
	is defined analogously.
	
	If $\UF \subseteq \SX$ is an open neighborhood of $0$, we set
	\[
		\CcFzero{\UF}{\SY}{k}
		\ndef
		\{
			\gamma \in \CcF{\UF}{\SY}{k} : \gamma(0) = 0
		\}.
	\]
	Analogously, we define $\CcFzero{\UF}{\VF}{k}$, $\CcFo{\UF}{\VF}{k}_{0}$
	and $\glstext{Raeume_beschraenkter_Abb_Null_Nullstelle}$ as the corresponding sets
	of functions vanishing at $0$.
	
	Furthermore, we define the set of \emph{decreasing weighted maps} as
	\index{weighted maps!decreasing}
	\[
		\glstext{Raeume_gwichteter_fallender_Abbildungen} \ndef \{\gamma \in \CcF{\UF}{\SY}{k} : (\forall f\in\cW, \ell \in \N, \ell \leq k, \eps > 0)(\exists r > 0)
		\hn{\rest{\gamma}{\UF\setminus\Ball{0}{r}}}{f}{\ell} < \eps\}.
	\]
	Note that we are primarily interested in the spaces $\CcFvan{\SX}{\SY}{k}$,
	but for technical reasons it is useful to have the spaces $\CcFvan{\UF}{\SY}{k}$
	available for $\UF \subset \SX$.
\end{defi}
We show that $\CcFvan{\UF}{\SY}{k}$ is closed.
\begin{lem}\label{lem:CWvan_closed_CW}
	$\CcFvan{\UF}{\SY}{k}$ is a closed vector subspace of $\CcF{\UF}{\SY}{k}$.
\end{lem}
\begin{proof}
	It is obvious from the definition of $\CcFvan{\UF}{\SY}{k}$ that it is a
	vector subspace. It remains to show that it is closed.
	To this end, let $(\gamma_i)_{i\in I}$ be a net in $\CcFvan{\UF}{\SY}{k}$ that converges
	to $\gamma \in \CcF{\UF}{\SY}{k}$ in the topology of $\CcF{\UF}{\SY}{k}$.
	Let $f \in \GewFunk$, $\ell \in \N$ with $\ell \leq k$ and $\eps > 0$.
	Then there exists an $i_\eps \in I$ such that
	\[
		i \geq i_\eps \implies \hn{\gamma - \gamma_i}{f}{\ell} < \frac{\eps}{2}.
	\]
	Further there exists an $r > 0$ such that
	\[
		\hn{\rest{\gamma_{i_\eps}}{\UF\setminus\Ball{0}{r}}}{f}{\ell} < \frac{\eps}{2} .
	\]
	Hence
	\[
		\hn{\rest{\gamma}{\UF\setminus\Ball{0}{r}}}{f}{\ell}
		\leq
		\hn{\rest{\gamma}{\UF\setminus\Ball{0}{r}} - \rest{\gamma_{i_\eps}}{\UF\setminus\Ball{0}{r}} }{f}{\ell}
		+ \hn{\rest{\gamma_{i_\eps}}{\UF\setminus\Ball{0}{r}}}{f}{\ell}
		< \eps ,
	\]
	and this finishes the proof.
\end{proof}
\paragraph{Examples involving finite-dimensional spaces}
Let $\K \in \{\R,\C\}$ and $n \in \N$.
In the following, let $\UF$ be an open nonempty subset of $\K^n$.
For a map $f:\UF \to \cl{\R}$ and a multiindex $\alpha\in\N^n$
with $\abs{\alpha} \leq k$ we define
\begin{equation*}
	\hn{\cdot}{f}{\alpha}: \ConDiff[\K]{\UF}{\SY}{k} \to [0,\infty]:
	\phi \mapsto \sup\{\abs{f(x)}\,\norm{\partial^\alpha \phi(x)}:x \in \UF\}.
\end{equation*}
We conclude from \refer{id:D^k_und_partielle_Abl} in
\refer{prop:Endlichdimensionale_Diffbarkeit}
that for a set \cW\ of maps $\UF \to \cl{\R}$ and $k \in \cl{\N}$
\begin{equation*}
	\CF{\UF}{\SY}{\cW}{k}
	= \{\phi \in \ConDiff[\K]{\UF}{\SY}{k}:
		(\forall f\in\cW, \alpha\in\N^n_0, \abs{\alpha}\leq k)\:
		\hn{\phi}{f}{\alpha}<\infty\},
\end{equation*}
and the topology defined with the seminorms
$\hn{\cdot}{f}{\alpha}$ coincides with the one defined above
using the seminorms $\hn{\cdot}{f}{\ell}$.
This characterization of $\CF{\UF}{\SY}{\cW}{k}$ allows us to recover well-known spaces as special cases:
\begin{itemize}
	\item
	If $\cW$ is the space $\ConDiff{\UF}{\R^m}{0}$
	of all continuous functions, then
	\[
		\CF{\UF}{\R^m}{\cW}{\infty}
		= \mathcal{D}(\UF, \R^m)
		= \glsdisp{kompakt_getragene_glatte}{\DCcInf{\UF}{\R^m}}
	\]%
	where $\glsdisp{kompakt_getragene_glatteD}{\mathcal{D}(\UF, \R^m)}$ denotes the space of
	compactly supported smooth functions from $\UF$ to $\R^m$;
	it should be noticed that $\CF{\UF}{\R^m}{\ConDiff{\UF}{\R^m}{0}}{\infty}$
	is \emph{not} endowed with the ordinary inductive limit topology
	$\varinjlim_{K} \mathcal{D}_{K}(\UF, \R^m)$,
	but instead the (coarser) topology making it the projective limit
	\[
		\varprojlim_{p\in\N}(\varinjlim_{K} \mathcal{D}_{K}^{p}(\UF, \R^m))
		= \varprojlim_{p\in\N} \mathcal{D}^p(\UF, \R^m),
	\]
	where $\mathcal{D}_{K}^{p}(\UF, \R^m)$ denotes the
	$\ConDiff{}{}{p}$-maps with support in the compact set $K$,
	endowed with the topology of uniform convergence of derivatives up to order $p$;
	and $\mathcal{D}^p(\UF, \R^m)$ the compactly supported $\ConDiff{}{}{p}$-maps
	endowed with the inductive limit topology of the sets $\mathcal{D}_{K}^{p}(\UF, \R^m)$.
	
	\item
	The vector-valued \emph{Schwartz space} $\mathcal{S}(\R^n, \R^n)$. Here
	$\UF=\SY=\R^n$, $k = \infty$ and $\cW$ is the
	set of polynomial functions on $\R^n$.
	
	\item
	The space $\BC{\UF}{\K^m}{k}$ of all bounded $\ConDiff{}{}{k}$-functions
	from $\UF \subseteq \K^n$ to $\K^m$ whose partial derivatives are bounded
	(for $\cW = \{1_\UF\}$);
	see \refer{beisp:BC}.
	
	\item
	If $\cW = \{1_{\SX}, \infty \cdot 1_{\SX\setminus\UF}\}$, then
	the space $\CF{\SX}{\SY}{\cW}{k}$ consists of $\BC{\SX}{\SY}{k}$
	functions that are defined on $\SX$ and vanish on the complement of $\UF$.
\end{itemize}

\section{Topological and uniform structure}
We analyze the topology of the weighted function spaces defined above.
In \refer{prop:topologische_Zerlegung_von_CFk} we shall provide a method
that greatly simplifies the treatment of the spaces; it will be used throughout this work.
We will also describe the spaces $\CF{\UF}{\SY}{\cW}{k}$ as the projective
limits of suitable larger spaces. 
In particular, this will simplify the treatment of the spaces $\CF{\UF}{\SY}{\cW}{\infty}$.
Further we give a sufficient criterion on the set $\cW$ which ensures that $\CF{\UF}{\SY}{\cW}{k}$ is complete.

\subsection{Reduction to lower order}
For $\ell > 1$, it is hard to estimate the seminorms $\hn{\cdot}{f}{\ell}$
because in most cases the higher order derivatives $\FAbl[\ell]{\cdot}$ can not be computed.
We develop a technique that allow us to avoid the computation.

First, we show that $\CcF{\UF}{\SY}{k}$ is endowed with the initial topology of the derivative maps.
\begin{lem}\label{lem:Topologie_auf_CFk_ist_InitialTop_der_Ableitungen}
	Let $\SX$ and $\SY$ be normed spaces,
	$\UF \subseteq \SX$ an open nonempty set, $k \in \cl{\N}$,
	$\GewFunk \subseteq \cl{\R}^\UF$ and $\gamma \in \FC{\UF}{\SY}{k}$.
	Then
	\[
		\gamma \in \CcF{\UF}{\SY}{k}
		\iff
		(\forall \ell\in\N, \ell\leq k)\, \FAbl[\ell]{\gamma} \in
			\CcF{\UF}{\Lin[\ell]{\SX}{\SY}}{0} ,
	\]
	and the map
	\[
		\CcF{\UF}{\SY}{k} \to \prod_{\substack{\ell \in \N\\\ell \leq k}}\CcF{\UF}{\Lin[\ell]{\SX}{\SY}}{0}
		: \gamma \mapsto (\FAbl[\ell]{\gamma})_{\ell \in \N, \ell \leq k}
	\]
	is a topological embedding.
\end{lem}
\begin{proof}
	Both assertions are clear from
	the definition of $\CcF{\UF}{\SY}{k}$ and
	$\CcF{\UF}{\Lin[\ell]{\SX}{\SY}}{0}$.
\end{proof}
The next lemma states a relation between the
higher order derivatives of $\gamma$ and those of $\FAbl{\gamma}$.
\begin{lem}\label{lem:Normbeziehung_zwischen_Ableitungen_und_Ableitungen_der_Ableitung}
	Let $\SX$ and $\SY$ be normed spaces,
	$\UF \subseteq \SX$ an open nonempty set, $k \in \N$ and
	$\gamma\in\FC{\UF}{\SY}{k + 1}$. Then
	\begin{equation}
		\label{id:Normbeziehung_zwischen_Ableitungen_und_Ableitungen_der_Ableitung}
		\Opnorm{\FAbl[\ell]{\FAbl{\gamma}}(x)} = \Opnorm{\FAbl[\ell + 1]{\gamma}(x)}
	\end{equation}
	for each $x\in\UF$ and $\ell < k$.
	In particular, for each map $f \in \cl{\R}^\UF$, $\ell < k$ and subset $\VF \subseteq \UF$
	\begin{equation}
		\label{id:hn(f,k)-Norm_Ableitungen_und_Ableitungen_der_Ableitung}
		\hn{\rest{\gamma}{\VF}}{f}{\ell + 1} = \hn{\rest{(\FAbl{\gamma})}{\VF}}{f}{\ell}.
	\end{equation}
\end{lem}
\begin{proof}
	In \refer{lem:Ableitungen_von_D} the identity
	\begin{equation*}
		\FAbl[\ell + 1]{\gamma}
		= \EmbA{\ell}{1}\circ (\FAbl[\ell]{\FAbl{\gamma}} )
	\end{equation*}
	is proved, where
	$\EmbA{\ell}{1} : \Lin{\SX}{\Lin[\ell]{\SX}{\SY}} \to \Lin[\ell + 1]{\SX}{\SY}$
	is an isometric isomorphism
	(see \refer{lem:Identifizierung_seltsamer_Raeume}).
	The asserted identities follow immediately.
\end{proof}
We can state the main tool for the treatment of weighted function spaces
$\CcF{\UF}{\SY}{k}$ with $k \geq 1$.
It is useful because it allows induction arguments of the following kind:
Suppose we want to show that $\gamma \in \CcF{\UF}{\SY}{k}$.
First, we have to show that $\gamma \in \CcF{\UF}{\SY}{0}$.
Then, we suppose $\gamma \in \CcF{\UF}{\SY}{\ell}$ and show that $\FAbl{\gamma}$
in $\CcF{\UF}{\Lin{\SX}{\SY}}{\ell}$ by expressing it in terms of $\gamma$.
This finishes the induction argument.
\begin{prop}[Reduction to lower order]
\label{prop:topologische_Zerlegung_von_CFk}
\label{prop:Ableitung_ist_stetig}
	Let $\SX$ and $\SY$ be normed spaces,
	$\UF \subseteq \SX$ an open nonempty set, $\GewFunk \subseteq \cl{\R}^\UF$,
	$k \in \N$ and $\gamma\in\FC{\UF}{\SY}{1}$. Then
	\[
		\gamma \in \CcF{\UF}{\SY}{k+1}
		\iff
		(\FAbl{\gamma}, \gamma)\in
		\CcF{\UF}{\Lin{\SX}{\SY}}{k} \times
			\CcF{\UF}{\SY}{0}.
	\]
	Moreover, the map
	\[
		\CcF{\UF}{\SY}{k+1}\to
		\CcF{\UF}{\Lin{\SX}{\SY}}{k} \times \CcF{\UF}{\SY}{0}
		:\gamma \mapsto (\FAbl{\gamma} , \gamma)
	\]
	is a topological embedding. In particular, the map
	\[
		\FAbl{} : \CcF{\UF}{\SY}{k + 1} \to \CcF{\UF}{\Lin{\SX}{\SY}}{k}
	\]
	is continuous.
\end{prop}
\begin{proof}
	The first relation follows immediately from
	the definition of $\FC{}{}{k}$-maps and
	\refer{id:hn(f,k)-Norm_Ableitungen_und_Ableitungen_der_Ableitung}
	in
	\refer{lem:Normbeziehung_zwischen_Ableitungen_und_Ableitungen_der_Ableitung}.
	This identity, together with \refer{lem:Topologie_auf_CFk_ist_InitialTop_der_Ableitungen},
	also implies that $\CcF{\UF}{\SY}{k+1}$ is endowed
	with the initial topology with respect to
	\[
		\FAbl{}: \CcF{\UF}{\SY}{k + 1} \to \CcF{\UF}{\Lin{\SX}{\SY}}{k}
	\]
	and the inclusion map
	\[
		\CcF{\UF}{\SY}{k+1} \to \CcF{\UF}{\SY}{0}.
	\]
	This proves the second assertion.
\end{proof}
The same argument can be made for the vanishing weighted functions.
\begin{cor}\label{cor:topologische_Zerlegung_von_CFkvan}
	Let $\SX$ and $\SY$ be normed spaces,
	$\UF \subseteq \SX$ an open nonempty set, $\GewFunk \subseteq \cl{\R}^\UF$,
	$k \in \N$ and $\gamma\in\FC{\UF}{\SY}{1}$. Then
	\[
		\gamma \in \CcFvan{\UF}{\SY}{k+1}
		\iff
		( \FAbl{\gamma} , \gamma)
		\in
		\CcFvan{\UF}{\Lin{\SX}{\SY}}{k} \times \CcFvan{\UF}{\SY}{0}.
	\]
\end{cor}
\begin{proof}
	This is also an immediate consequence of \refer{prop:topologische_Zerlegung_von_CFk}
	and \refer{id:hn(f,k)-Norm_Ableitungen_und_Ableitungen_der_Ableitung} in
	\refer{lem:Normbeziehung_zwischen_Ableitungen_und_Ableitungen_der_Ableitung}.
\end{proof}
%
\subsection[Projective limits and the topology of \texorpdfstring{$\CcF{\UF}{\SY}{\infty}$}{}]%
{Projective limits and the topology of \texorpdfstring{$\boldsymbol{\CcF{\UF}{\SY}{\infty}}$}{}}
\label{susec:Projektive_Limiten_Gewichtete_Funktionen}
Sometimes it is useful that $\CF{\UF}{\SY}{\cW}{k}$ can be written as the projective
limit of larger weighted functions spaces.
\begin{prop}\label{prop:CW_projektiver_LB-Raum}
	Let $\SX$ and $\SY$ be normed spaces,
	$\UF \subseteq \SX$ an open nonempty set, $k \in \cl{\N}$
	and $\cW \subseteq \cl{\R}^\UF$ a nonempty set.
	Further let $(\cF_{i})_{i \in I}$ be a directed family of nonempty subsets of $\cW$
	such that $\bigcup_{i \in I}\cF_{i} = \cW$.
	Consider $I \times \{\ell \in \N : \ell \leq k\}$ as a directed set via
	\[
		((i_1, \ell_1) \leq (i_2, \ell_2))
		\iff
		i_1 \leq i_2 \text{ and } \ell_1 \leq \ell_2.
	\]
	Then $\CF{\UF}{\SY}{\cW}{k}$ is the projective limit of
	\[
		\{\CF{\UF}{\SY}{\cF_{i}}{\ell} : \ell \in \N, \ell \leq k, i \in I\}
	\]
	in the category of topological (vector) spaces, with the inclusion maps as morphisms.
\end{prop}
\begin{proof}
	Since
	\[
		\CF{\UF}{\SY}{\cW}{k} = \bigcap_{\substack{i \in I\\ \ell \in \N, \ell \leq k}}
		\CF{\UF}{\SY}{\cF_{i}}{\ell},
	\]
	the set
	$\CF{\UF}{\SY}{\cW}{k}$ is the desired projective limit as a set,
	and hence also as a vector space.
	Moreover, it is well known that the initial topology
	with respect to the limit maps
	$
		\CF{\UF}{\SY}{\cW}{k} \to \CF{\UF}{\SY}{\cF_{i}}{\ell}
	$
	makes $\CF{\UF}{\SY}{\cW}{k}$ the projective limit as a topological space,
	and also as a topological vector space.
	But it is clear from the definition that the given topology on
	$\CF{\UF}{\SY}{\cW}{k}$ coincides with this initial topology.
\end{proof}
\begin{cor} \label{cor:Topologie_von_CinfF}
	Let $\SX$ and $\SY$ be normed spaces,
	$\UF \subseteq \SX$ an open nonempty set and $\GewFunk \subseteq \cl{\R}^\UF$.
	The space $\CcF{\UF}{\SY}{\infty}$ is endowed with the initial topology
	with respect to the inclusion maps
	\[
		\CcF{\UF}{\SY}{\infty}\to\CcF{\UF}{\SY}{k}.
	\]
	Moreover, $\CcF{\UF}{\SY}{\infty}$ is the projective limit
	of the spaces $\CcF{\UF}{\SY}{k}$ with $k\in \N$,
	together with the inclusion maps.
\end{cor}
\begin{proof}
	This is an immediate consequence of \refer{prop:CW_projektiver_LB-Raum}.
\end{proof}
\subsection{A completeness criterion}
We describe a condition on $\GewFunk$ that ensures that $\CcF{\UF}{\SY}{k}$
is complete, provided that $\SY$ is a Banach space.
The proof uses \refer{prop:topologische_Zerlegung_von_CFk}.
We start with the following observation concerning the continuity of evaluation maps.
\begin{lem}\label{lem:Kriterium_fuer_Stetigkeit_der_Punktauswertungen}
	Let $\SX$ and $\SY$ be normed spaces, $\UF \subseteq \SX$ an open nonempty set,
	$k \in \cl{\N}$ and $x\in\UF$.
	Suppose that $\GewFunk \subseteq \cl{\R}^{\UF}$ contains
	a weight $f_x$ with $f_x(x) \neq 0$.
	Then
	\[
		\evTwo_x : \CcF{\UF}{\SY}{k} \to \SY
		: \gamma \mapsto \gamma(x)
	\]
	is a continuous linear map.
\end{lem}
\begin{proof}
	If there exists $f \in \GewFunk$ with $f(x) \in \sset{-\infty, \infty}$,
	then for each $\gamma \in \CcF{\UF}{\SY}{k}$ the estimate
	\[
		\norm{\evTwo_{x}(\gamma)} = 0 \leq \hn{\gamma}{f}{0}
	\]
	holds. Otherwise, for each $f \in \GewFunk$ with $f(x) \neq 0$
	and $\gamma \in \CcF{\UF}{\SY}{k}$, we have
	\[
		\norm{\evTwo_{x}(\gamma)}
		= \norm{\gamma(x)}
		\leq \frac{1}{\abs{f(x)}} \hn{\gamma}{f}{0}.
	\]
	In both cases, these estimates ensure the continuity of $\evTwo_{x}$.
\end{proof}
We examine when the image of $\CcF{\UF}{\SY}{k + 1}$ under the embedding
described in \refer{prop:topologische_Zerlegung_von_CFk} is closed.
\begin{prop}\label{prop:Abgeschlossenheit_des_Bildes_unter_der_Einbettung}
	Let $\SX$ and $\SY$ be normed spaces,
	$\UF \subseteq \SX$ an open nonempty set
	and $k\in\N$.
	Further, let $\GewFunk \sub \cl{\R}$ such that for each compact line segment
	$S \subseteq \UF$ there exists $f_{S} \in \GewFunk$ with
	$\inf_{x\in S}\abs{f_{S}(x)} > 0$.
	Then the image of $\CcF{\UF}{\SY}{k+1}$ under the embedding
	described in \refer{prop:topologische_Zerlegung_von_CFk} is closed.
\end{prop}
\begin{proof}
	Let $(\gamma_{i})_{i\in I}$ be a net in $\CcF{\UF}{\SY}{k+1}$ such that
	$(\gamma_{i})_{i\in I}$ converges to $\gamma$ in $\CcF{\UF}{\SY}{0}$
	and the net $(\FAbl{\gamma}_{i})_{i\in I}$ converges to $\Gamma \in \CcF{\UF}{\Lin{\SX}{\SY}}{k}$.
	We have to show that $\gamma \in \CcF{\UF}{\SY}{k+1}$ with
	$\FAbl{\gamma} = \Gamma$.
	
	To this end, consider $x \in \UF$, $h\in \SX$ and $t\in\R^\ast$ such that
	the line segment $S_{x,t,h} \ndef \{x + s t h : s \in [0,1]\}$ is contained in $\UF$.
	Since evaluation maps and the weak integration are continuous
	(see \refer{lem:Kriterium_fuer_Stetigkeit_der_Punktauswertungen} and
	\refer{lem:Integral_stetiges_Funktional}, respectively)
	and the hypothesis on $\GewFunk$ implies that $(\FAbl{\gamma}_{i})_{i\in I}$ converges
	to $\Gamma$ uniformly on $S_{x,t,h}$, we derive
	\begin{equation*}
		\tfrac{\gamma(x + t h) -\gamma(x) }{t}
		=\lim_{i\in I} \tfrac{\gamma_{i}(x + t h) -\gamma_{i}(x) }{t}
		\\
		=\lim_{i\in I}
			\frac{\Mint{ \FAbl{\gamma}_{i}(x + s t h)\eval (t h) }{s}}{t}
		= \Mint{ \Gamma(x + s t h)\eval h}{s}.
	\end{equation*}
	Since $\Gamma$ is continuous, we can apply 
	\refer{prop:Stetigkeit_parameterab_Int} and get
	\[
		\lim_{t \to 0} \frac{\gamma(x + t h) -\gamma(x) }{t}
		= \Rint{0}{1}{ \lim_{t \to 0} \bigl(\Gamma(x + s t h)\eval h\bigr) }{s}
		= \Gamma(x) \eval h .
	\]
	Because $\Gamma$ and the evaluation of linear maps are continuous
	(\refer{lem:Auswertung_linearer_Abb_ist_stetig}),
	$\gamma$ is a $\ConDiff{}{}{1}$-map
	with $\dA{\gamma}{x}{\cdot} = \Gamma(x)$,
	from which we conclude with another application of the continuity
	of $\Gamma$ that $\gamma \in \FC{\UF}{\SY}{1}$ with $\FAbl{\gamma} = \Gamma$.
	Finally we conclude with \refer{prop:topologische_Zerlegung_von_CFk}
	that $\gamma \in \CcF{\UF}{\SY}{k+1}$.
\end{proof}
The last proposition allows us to conclude the completeness of $\CcF{\UF}{\SY}{k}$
from the on of $\CcF{\UF}{\SY}{0}$.
\begin{cor} \label{cor:Kriterium_Vollstaendigkeit_von_CFk}
	In the situation of
	\refer{prop:Abgeschlossenheit_des_Bildes_unter_der_Einbettung},
	assume that $\CcF{\UF}{\SY}{0}$ is complete for each Banach space $\SY$.
	Then also $\CcF{\UF}{\SY}{k}$ is complete, for each $k \in \cl{\N}$.
\end{cor}
\begin{proof}
	The proof for $k < \infty$ is by induction.
	
	$k=0$: This holds by our hypothesis.
	
	$k\to k+1$: We conclude from
	\refer{prop:Abgeschlossenheit_des_Bildes_unter_der_Einbettung}
	that $\CcF{\UF}{\SY}{k+1}$ is isomorphic to a closed vector subspace of
	$\CcF{\UF}{\Lin{\SX}{\SY}}{k} \times \CcF{\UF}{\SY}{0}$.
	Since this space is complete by induction hypothesis,
	so is $\CcF{\UF}{\SY}{k+1}$.
	
	$k=\infty$: This follows from \refer{cor:Topologie_von_CinfF}
	and the fact that $\CcF{\UF}{\SY}{k}$ is complete for all $k \in \N$
	because projective limits of complete topological vector spaces are complete.
\end{proof}
We give a sufficient condition for the completeness of $\CcF{\UF}{\SY}{0}$.
\begin{prop}\label{prop:Kriterium_Vollstaendigkeit_von_CF0}
	Let $\SX$ be a normed space, $\UF \subseteq \SX$ an open nonempty set and $\SY$ a Banach space.
	Further let $\GewFunk \sub \cl{\R}$ such that for each compact set
	$K \subseteq \UF$ there exists $f_{K} \in \GewFunk$ with
	$\inf_{x\in K}\abs{f_{K}(x)} > 0$.
	\index{weights!condition for completeness}
	Then $\CcF{\UF}{\SY}{0}$ is complete.
\end{prop}
\begin{proof}
	Let $(\gamma_{i})_{i\in I}$ be a Cauchy net in $\CcF{\UF}{\SY}{0}$.
	The hypotheses on \GewFunk\ imply that the topology of
	$\CcF{\UF}{\SY}{0}$ is finer than the topology of the uniform convergence
	on compact sets.
	Hence we deduce from the completeness of $\SY$ that there exists a map
	$\gamma: \UF \to \SY$ to which $(\gamma_{i})_{i\in I}$
	converges uniformly on each compact subset of $\UF$;
	and since each $\gamma_{i}$ is continuous, the restriction
	of $\gamma$ to each compact subset is continuous.
	Hence $\gamma$ is sequentially continuous
	since the union of a convergent sequence with its limit is compact;
	but $\UF$ is first countable, so $\gamma$ is continuous.
	
	It remains to show that $\gamma \in \CcF{\UF}{\SY}{0}$ and
	$(\gamma_{i})_{i\in I}$ converges to $\gamma$ in $\CcF{\UF}{\SY}{0}$.
	To see this, let $f \in \GewFunk$ and $\eps > 0$.
	Then there exists an $\ell\in I$ such that
	\begin{equation*}
		(\forall i, j \geq \ell)\; \hn{\gamma_i\ - \gamma_j}{f}{0} \leq \eps,
	\end{equation*}
	which is equivalent to
	\begin{equation*}
		(\forall x\in \UF, i, j \geq \ell) \; \abs{f(x)}\,
		\norm{\gamma_i(x) - \gamma_j(x)} \leq \eps .
	\end{equation*}
	If we fix $i$ in this estimate and let $\gamma_j(x)$ pass
	to its limit, then we get
	\begin{equation*}\label{est:vollstaendig2}
		\tag{\ensuremath{\ast}}
		(\forall x\in \UF, i \geq \ell) \;\abs{f(x)}\,
		\norm{\gamma_i(x) - \gamma(x)} \leq \eps .
	\end{equation*}
	From this estimate, we conclude with the triangle inequality that
	\begin{align*}
		(\forall x\in \UF) \;\abs{f(x)}\, \norm{\gamma(x)}
		\leq \eps + \abs{f(x)}\,\norm{\gamma_\ell(x)}
		\leq \eps + \hn{\gamma_\ell}{f}{0},
	\end{align*}
	so $\gamma \in \CcF{\UF}{\SY}{0}$.
	Finally we conclude from \eqref{est:vollstaendig2} that
	$\hn{\gamma_i -\gamma}{f}{0} \leq \eps$ for all $i\geq \ell$,
	so $(\gamma_i)_{i\in I}$ converges to $\gamma$ in $\CcF{\UF}{\SY}{0}$.
\end{proof}
Finally, we state the derived criterion in a more citable form.
\begin{cor}\label{cor:CkF_komplett_schwache_Bedingung}
	Let $\SX$ be a normed space, $\UF \subseteq \SX$ an open nonempty set,
	$\SY$ a Banach space and $k \in \cl{\N}$.
	Further, let $\GewFunk \sub \cl{\R}$ such that for each compact set
	$K \subseteq \UF$ there exists $f_{K} \in \GewFunk$ with
	$\inf_{x\in K}\abs{f_{K}(x)} > 0$.
	Then $\CcF{\UF}{\SY}{k}$ is complete.
\end{cor}
\begin{proof}
	This is an immediate consequence of
	\refer{cor:Kriterium_Vollstaendigkeit_von_CFk}
	and \refer{prop:Kriterium_Vollstaendigkeit_von_CF0}.
\end{proof}
\begin{cor}\label{cor:CkF_komplett}\label{cor:BC_komplett}
	Let $\SX$ be a normed space, $\UF \subseteq \SX$ an open nonempty set,
	$\SY$ a Banach space and $k \in \cl{\N}$.
	Further, let $\GewFunk \sub \cl{\R}$ with $1_{\UF} \in \GewFunk$.
	Then $\CcF{\UF}{\SY}{k}$ is complete;
	in particular, $\BC{\UF}{\SY}{k}$ is complete.
\end{cor}
\subsubsection{An integrability criterion}
The given completeness criterion entails a criterion for the existence of the
weak integral of a continuous curve to a space $\CcF{\UF}{\SY}{k}$ where $\SY$
is not necessarily complete.
\begin{lem}
\label{lem:Kriterium_Integrierbarkeit_in_CW}
\label{lem:Punktetrennende_Menge_Integrale_CW}
	Let $\SX$ and $\SY$ be normed spaces, $\UF \subseteq \SX$
	a nonempty open set, $k\in\cl{\N}$, $\GewFunk \subseteq \cl{\R}^{\UF}$,
	$\Gamma : [a,b] \to \CcF{\UF}{\SY}{k}$ a map
	and $R \in \CcF{\UF}{\SY}{k}$.
	\begin{enumerate}
		\item\label{enum1:INtegrale_von_Kurven_in_gewichtetenAbb-1}
		Assume that $\Gamma$ is weakly integrable
		and that for each $x\in\UF$ there exists $f_x \in \cW$ with $f_x(x) \neq 0$. Then
		\[
			\Rint{a}{b}{\Gamma(s)}{s} = R
			\iff
			(\forall x \in \UF)\;
			\evTwo_x\left(\Rint{a}{b}{\Gamma(s)}{s}\right)
			= R(x) ,
		\]
		and for each $x \in \UF$ we have
		\begin{equation}\label{id:Vertauschung_Kurvenintegral_Auswertung}\tag{$\ast$}
			\evTwo_x\left(\Rint{a}{b}{\Gamma(s)}{s}\right)
			= \Rint{a}{b}{\evTwo_x(\Gamma(s))}{s} .
		\end{equation}

		\item\label{enum1:INtegrale_von_Kurven_in_gewichtetenAbb-2}
		Assume that for each compact set $K \subseteq \UF$,
		there exists $f_{K} \in \GewFunk$ with
		$
			\inf\limits_{x \in K}\abs{f_{K}(x)} > 0 ,
		$
		that $\Gamma$ is continuous and
		\begin{equation*}\label{id:punktweises_Integral_CW}\tag{$\ast\ast$}
			\Rint{a}{b}{\evTwo_x(\Gamma(s))}{s}
			= \evTwo_x(R)
		\end{equation*}
		holds for all $x \in \UF$. Then $\Gamma$ is weakly integrable with
		\[
			\Rint{a}{b}{\Gamma(s)}{s} = R.
		\]
	\end{enumerate}
\end{lem}
\begin{proof}
	\refer{enum1:INtegrale_von_Kurven_in_gewichtetenAbb-1}
	Since $\{\evTwo_x: x \in \UF\}$ separates the points on $\CcF{\UF}{\SY}{k}$,
	the stated equivalence is obvious.
	Further, we proved in \refer{lem:Kriterium_fuer_Stetigkeit_der_Punktauswertungen}
	that the condition on $\cW$ implies that
	$\{\evTwo_x: x \in \UF\} \subseteq \Lin{\CcF{\UF}{\SY}{k}}{\SY}$,
	so \eqref{id:Vertauschung_Kurvenintegral_Auswertung} follows from
	\refer{lem:stet_lin_Abb_wirken_auf_int_Kurven}.
	
	\refer{enum1:INtegrale_von_Kurven_in_gewichtetenAbb-2}
	Let $\widetilde{\SY}$ be the completion of $\SY$.
	Then $\CcF{\UF}{ \SY}{k} \subseteq \CcF{\UF}{\widetilde{\SY}}{k}$,
	and we denote the inclusion map by $\iota$.
	It is obvious that $\iota$ is a topological embedding.
	In the following, we denote the evaluation of $\CcF{\UF}{\widetilde{\SY}}{k}$
	at $x \in \UF$ with $\widetilde{\evTwo}_{x}$.
	\\
	Since we proved in \refer{cor:CkF_komplett_schwache_Bedingung}
	that the condition on $\cW$ ensures that $\CcF{\UF}{\widetilde{\SY}}{k}$ is complete,
	$\iota \circ \Gamma$ is weakly integrable.
	Since $\widetilde{\evTwo}_{x} \circ \iota = \evTwo_{x}$ for $x \in \UF$,
	we conclude from
	\refer{enum1:INtegrale_von_Kurven_in_gewichtetenAbb-1}
	(using \eqref{id:Vertauschung_Kurvenintegral_Auswertung} and \eqref{id:punktweises_Integral_CW})
	that
	\[
		\Rint{a}{b}{(\iota \circ \Gamma)(s)}{s} = \iota(R).
	\]
	This identity ensures the integrability of $\Gamma$:
	By the Hahn-Banach theorem, each $\lambda \in \CcF{\UF}{\SY}{k}'$ extends
	to a $\widetilde{\lambda} \in \CcF{\UF}{ \widetilde{\SY} }{k}'$, that is
	$
		\widetilde{\lambda} \circ \iota = \lambda
	$.
	Hence
	\begin{equation*}
		\Rint{a}{b}{ (\lambda \circ \Gamma)(s)}{s}
		= \Rint{a}{b}{ (\widetilde{\lambda} \circ \iota \circ \Gamma)(s)}{s}
		= \widetilde{\lambda}( \iota(R))
		= \lambda(R),
	\end{equation*}
	which had to be proved.
\end{proof}
\section{Composition on weighted functions and superposition operators}
\label{sec:Composition_superposition_normedspaces}
In this subsection we discuss the behaviour of weighted functions when they are
composed with certain functions.
In particular, we show that a continuous multilinear or a suitable analytic map
induce a superposition operator between weighted function spaces.
Moreover, we examine the composition between bounded functions and
between bounded functions mapping $0$ to $0$ and weighted functions.
\subsection{Composition with a multilinear map}
\label{susec:Multilineare_Abb}
We prove that a continuous multilinear map
from a normed space $\SY_1 \times \dotsb \times \SY_m$
to a normed space $\SZ$ induces
a continuous multilinear map from
$\CcF{\UF}{\SY_1}{k} \times \dotsb \times \CcF{\UF}{\SY_m}{k}$
to $\CcF{\UF}{\SZ}{k}$.
As a preparation, we calculate the differential of a composition of a multilinear
map and other differentiable maps. The following definition is quite useful to do that.
\begin{defi}\label{def:Bausteine_Differential_multilineareAbb}
	Let $\SY_1,\dotsc,\SY_m$, $\SX$ and $\SZ$ be normed spaces and
	$b:\SY_1 \times\dotsb\times \SY_m \to \SZ$
	a continuous $m$-linear map.
	For each $i \in \{1,\dotsc,m\}$ we define the $m$-linear continuous map
	\begin{align*}
		b^{(i)}:&
		\SY_1 \times \dotsb \times
		\SY_{i-1} \times \Lin{\SX}{\SY_i}\times \SY_{i+1}
		\times \dotsb\times \SY_m
		\to \Lin{\SX}{\SZ}\\
		&(y_1,\dotsc,y_{i-1},T,y_{i+1},\dotsc,y_m)\mapsto
		(h\mapsto b(y_1,\dotsc,y_{i-1},T\eval h, y_{i+1},\dotsc,y_m)).
	\end{align*}
\end{defi}

\begin{lem}\label{lem:Ableitung_Kompo_mit_multilinearer_Abb}
	Let	$\SY_1,\dotsc,\SY_m$ and $\SZ$ be normed spaces,
	$\UF$ be an open nonempty subset of the normed space $\SX$
	and $k\in\cl{\N}$.
	Further let
	$b:\SY_1 \times\dotsb\times \SY_m \to \SZ$
	be a continuous $m$-linear map and
	$\gamma_1 \in \FC{\UF}{\SY_1}{k},
	\dotsc, \gamma_m \in \FC{\UF}{\SY_m}{k}$.
	Then
	\[
		b\circ(\gamma_1,\dotsc,\gamma_m) \in \FC{\UF}{\SZ}{k}
	\]
	with
	\begin{equation}
		\label{id:Ableitung_Kompo_mit_multilinearen_Abb}
		\FAbl{(b\circ(\gamma_1,\dotsc,\gamma_m))}
		= \sum_{i=1}^m
			b^{(i)}\circ (\gamma_1, \dotsc,\gamma_{i-1}, \FAbl{\gamma_i},
				\gamma_{i+1}, \dotsc,\gamma_m).
	\end{equation}
\end{lem}
\begin{proof}
	To calculate the derivative of $b\circ(\gamma_1,\dotsc,\gamma_m)$,
	we apply the chain rule and get
	\begin{align*}
		\FAbl{(b\circ(\gamma_1,\dotsc,\gamma_m))}(x)\eval h
		&= \sum_{i=1}^m
			b(\gamma_1(x),\dotsc,\gamma_{i-1}(x),\FAbl{\gamma_i}(x)\eval h,
				\gamma_{i+1}(x),\dotsc,\gamma_m(x))\\
		&= \sum_{i=1}^m
			b^{(i)} (\gamma_1(x),\dotsc,\gamma_{i-1}(x),\FAbl{\gamma_i}(x),
				\gamma_{i+1}(x),\dotsc,\gamma_m(x))
			\eval h .
	\end{align*}
	This clearly implies \eqref{id:Ableitung_Kompo_mit_multilinearen_Abb}.
\end{proof}
We are ready to prove the result about the superposition.
\begin{prop}\label{prop:multilineare_Abb_und_CF}
	Let $\UF$ be an open nonempty subset of the normed space $\SX$.
	Let $\SY_1,\dotsc,\SY_m$ be normed spaces, $k \in \cl{\N}$ and
	$\GewFunk, \cW_1,\dotsc,\cW_m \subseteq \cl{\R}^{\UF}$ nonempty sets
	such that
	\[
		(\forall f \in \GewFunk) (\exists g_{f,1} \in \cW_1,\dotsc, g_{f,m} \in\cW_m)\,
		\abs{f} \leq \abs{g_{f,1}}\dotsm \abs{g_{f, m}}.
	\]
	Further let $\SZ$ be another normed space and
	$b:\SY_1 \times\dotsb\times \SY_m \to \SZ$
	a continuous $m$-linear map.
	Then
	\[
		b\circ(\gamma_1,\dotsc,\gamma_m) \in \CF{\UF}{\SZ}{\GewFunk}{k}
	\]
	for all $\gamma_1 \in \CF{\UF}{\SY_1}{\cW_1}{k},
	\dotsc, \gamma_m \in \CF{\UF}{\SY_m}{\cW_m}{k}$.
	The map
	\[
		M_k(b):
		\CF{\UF}{\SY_1}{\cW_1}{k} \times \dotsb \times
		\CF{\UF}{\SY_m}{\cW_m}{k}
		\to\CF{\UF}{\SZ}{\GewFunk}{k}:
		(\gamma_1,\dotsc,\gamma_m)\mapsto b\circ(\gamma_1,\dotsc,\gamma_m)
	\]
	is $m$-linear and continuous.
	\index{superposition!with a multilinear map}
\end{prop}
\begin{proof}
	For $k < \infty$, we proceed by induction on $k$.
	
	$k=0$: For $f \in \GewFunk$, $x\in\UF$ and
	$\gamma_1\in\CF{\UF}{\SY_1}{\cW_1}{k},\dotsc,
	\gamma_m\in\CF{\UF}{\SY_m}{\cW_m}{k}$
	we compute
	\begin{align*}
		\abs{f(x)}\,\norm{b\circ(\gamma_1,\dotsc,\gamma_m)(x)}
		\leq \Opnorm{b} \prod_{i=1}^m \abs{g_{f, i}} \norm{\gamma_i(x)}
		\leq \Opnorm{b} \prod_{i=1}^m\hn{\gamma_i}{g_{f, i}}{0},
	\end{align*}
	so $b\circ(\gamma_1,\dotsc,\gamma_m) \in \CF{\UF}{\SZ}{\GewFunk}{0}$.
	From this estimate we also conclude
	\[
		\hn{M_0(b)(\gamma_1,\dotsc,\gamma_m)}{f}{0}
		= \hn{b\circ(\gamma_1,\dotsc,\gamma_m)}{f}{0}
		\leq \Opnorm{b}\prod_{i=1}^m\hn{\gamma_i}{g_{f, i}}{0},
	\]
	so $M_0(b)$ is continuous at $0$. Since	the $m$-linearity of $M_0(b)$ is obvious,
	this implies the continuity	of $M_0(b)$
	(see \cite[Chapter \RN{1}, \S 1, no. 6]{MR910295}).
	
	$k\to k + 1$: From \refer{prop:topologische_Zerlegung_von_CFk}
	(together with the induction base) we know that for
	$\gamma_1\in\CF{\UF}{\SY_1}{\cW_1}{k + 1},\dotsc,
	\gamma_m\in\CF{\UF}{\SY_m}{\cW_m}{k + 1}$
	\[
		b\circ(\gamma_1,\dotsc,\gamma_m) \in \CF{\UF}{\SZ}{\GewFunk}{k + 1}
		\iff
		D(b\circ(\gamma_1,\dotsc,\gamma_m)) 
		\in\CF{\UF}{\Lin{\SX}{\SZ}}{\GewFunk}{k}
	\]
	and that $M_{k+1}(b)$ is continuous if
	\[
		D\circ M_{k+1}(b):
		\CF{\UF}{\SY_1}{\cW_1}{k+1} \times \dotsb \times 
		\CF{\UF}{\SY_m}{\cW_m}{k+1}
		\to\CF{\UF}{\Lin{\SX}{\SZ}}{\GewFunk}{k}
	\]
	is so.
	We know from \eqref{id:Ableitung_Kompo_mit_multilinearen_Abb} in
	\refer{lem:Ableitung_Kompo_mit_multilinearer_Abb} that
	\[
		D(b\circ(\gamma_1,\ldots,\gamma_m))
		= \sum_{i=1}^m
			b^{(i)}\circ (\gamma_1, \dotsc,\gamma_{i-1},D\gamma_i,
				\gamma_{i+1}, \dotsc,\gamma_m).
	\]
	By the inductive hypothesis,
	\[
		b^{(i)}\circ
		(\gamma_1, \dotsc, \gamma_{i-1}, D\gamma_i, \gamma_{i+1}, \dotsc,
			\gamma_m)
		\in \CF{\UF}{\Lin{\SX}{\SZ}}{\GewFunk}{k}
	\]
	and hence
	\[
		D(b\circ(\gamma_1,\ldots,\gamma_m))
		\in\CF{\UF}{\Lin{\SX}{\SZ}}{\GewFunk}{k}.
	\]
	Since
	\[
		M_k(b^{(i)}):
		\CF{\UF}{\SY_1}{\cW_1}{k} \times \dotsb \times
		\CF{\UF}{\Lin{\SX}{\SY_i}}{\cW_i}{k} \times
		\dotsb \times \CF{\UF}{\SY_m}{\cW_m}{k}
		\to \CF{\UF}{\Lin{\SX}{\SZ}}{\GewFunk}{k}
	\]
	is continuous by the inductive hypothesis, it follows that
	$D\circ M_{k+1}(b)$ is continuous as
	\[
		(D\circ M_{k+1}(b))
		(\gamma_1, \dotsc, \gamma_m)
		= \sum_{i=1}^m M_k(b^{(i)})
		(\gamma_1, \dotsc, \gamma_{i-1}, D\gamma_i, \gamma_{i+1}, \dotsc,
			\gamma_m).
	\]	
	Furthermore, $M_{k+1}(b)$ is obviously $m$-linear, so the induction step is finished.	
	
	$k=\infty$:
	\BeweisschrittCkCinfty%
	{\CF{\UF}{\SY_1}{\cW_1}{\infty} \times \dotsb \times
		\CF{\UF}{\SY_m}{\cW_m}{\infty}}
	{M_\infty(b)}
	{\CF{\UF}{\SZ}{\GewFunk}{\infty}}
	{}
	{}
	{\CF{\UF}{\SY_1}{\cW_1}{n} \times \dotsb \times
		\CF{\UF}{\SY_m}{\cW_m}{n}}
	{M_n(b)}
	{\CF{\UF}{\SZ}{\GewFunk}{n}}
	{n}
\end{proof}
We prove an analogous result for decreasing functions.
\begin{cor}\label{cor:multilineare_Abb_und_CFvan}
	\index{superposition!with a multilinear map}
	Let $\SY_1,\dotsc,\SY_m$ be normed spaces,
	$\UF$ an open nonempty subset of the normed space $\SX$, $k \in \cl{\N}$
	and $\GewFunk, \cW_1,\dotsc,\cW_m \subseteq \cl{\R}^{\UF}$ nonempty such that
	\[
		(\forall f \in \GewFunk) (\exists g_{f,1} \in \cW_1,\dotsc, g_{f,m} \in\cW_m)\,
		\abs{f} \leq \abs{g_{f,1}}\dotsm \abs{g_{f, m}}.
	\]
	Further let $\SZ$ be another normed space,
	$b:\SY_1 \times\dotsb\times \SY_m \to \SZ$
	a continuous $m$-linear map and $j \in \{1,\dotsc,m\}$.
	Then
	\[
		\tag{\ensuremath{\dagger}}\label{Superposition_multilinear_CWvan}
		b\circ(\gamma_1,\dotsc,\gamma_j, \dotsc,\gamma_m) \in \CFvan{\UF}{\SZ}{\GewFunk}{k}
	\]
	for all $\gamma_i \in \CF{\UF}{\SY_i}{\cW_i}{k}$ ($i\neq j$)
	and $\gamma_j \in \CFvan{\UF}{\SY_j}{\cW_j}{k}$.
	Moreover, the map
	\begin{align*}
		M_k(b)&:
		\CF{\UF}{\SY_1}{\cW_1}{k}
		\times \dotsb \times
		\CFvan{\UF}{\SY_j}{\cW_j}{k}
		\times \dotsb \times
		\CF{\UF}{\SY_m}{\cW_m}{k}
		\to\CFvan{\UF}{\SZ}{\GewFunk}{k}
		\\
		&:
		(\gamma_1, \dotsc,\gamma_j, \dotsc,\gamma_m)\mapsto b\circ(\gamma_1, \dotsc,\gamma_j, \dotsc, \gamma_m)
	\end{align*}
	is $m$-linear and continuous.
\end{cor}
\begin{proof}
	Using \refer{prop:multilineare_Abb_und_CF} and \refer{lem:CWvan_closed_CW},
	we only have to prove that \eqref{Superposition_multilinear_CWvan} holds.
	This is done by induction on $k$ (which we may assume finite).
	
	$k=0$: For $f\in \GewFunk$, $x\in\UF$ and
	$\gamma_1\in\CF{\UF}{\SY_1}{\cW_1}{0},\dotsc, \gamma_j \in \CFvan{\UF}{\SY_j}{\cW_j}{0}, \dotsc,
	\gamma_m\in\CF{\UF}{\SY_m}{\cW_m}{0}$
	we compute
	\begin{multline*}
		\abs{f(x)}\,\norm{b\circ(\gamma_1, \dotsc, \gamma_j, \dotsc, \gamma_m)(x)}
		\\
		\leq \Opnorm{b} \prod_{i=1}^m \abs{g_{f,i}(x)}\norm{\gamma_i(x)}
		\leq \left(\Opnorm{b} \prod_{i\neq j}\hn{\gamma_i}{g_{f,i}}{0}\right)
			\abs{g_{f,j}(x)}\,\norm{\gamma_j(x)} .
	\end{multline*}
	With this estimate we easily see that $b\circ(\gamma_1, \dotsc, \gamma_j, \dotsc,\gamma_m) \in \CFvan{\UF}{\SZ}{\cW_j}{0}$.

	$k\to k + 1$: From \refer{cor:topologische_Zerlegung_von_CFkvan}
	(together with the induction base) we know that for
	$\gamma_1\in\CF{\UF}{\SY_1}{\cW_1}{k + 1},
	\dotsc,
	\gamma_j \in \CFvan{\UF}{\SY_j}{\cW_j}{k + 1},
	\dotsc,
	\gamma_m\in\CF{\UF}{\SY_m}{\cW_m}{k + 1}$
	\[
		b\circ(\gamma_1,\dotsc, \gamma_j, \dotsc, \gamma_m) \in \CFvan{\UF}{\SZ}{\GewFunk}{k + 1}
		\iff
		D(b\circ(\gamma_1,\dotsc, \gamma_j, \dotsc, \gamma_m)) 
		\in\CFvan{\UF}{\Lin{\SX}{\SZ}}{\GewFunk}{k}.
	\]
	We know from \eqref{id:Ableitung_Kompo_mit_multilinearen_Abb} in
	\refer{lem:Ableitung_Kompo_mit_multilinearer_Abb} that
	\begin{align*}
		D(b\circ(\gamma_1,\dotsc,\gamma_j, \dotsc, \gamma_m))
		&= \sum_{\substack{i=1\\ i\neq j}}^m
			b^{(i)}\circ (\gamma_1, \dotsc,\gamma_j, \dotsc,\gamma_{i-1},D\gamma_i,
				\gamma_{i+1}, \dotsc,\gamma_m)
		\\
		&+ b^{(j)}\circ (\gamma_1, \dotsc, \gamma_{j-1},D\gamma_j,
				\gamma_{j+1}, \dotsc,\gamma_m).
	\end{align*}
	Because $\gamma_j \in \CFvan{\UF}{\SY_j}{\cW_j}{k}$
	and $D\gamma_j \in \CFvan{\UF}{\Lin{\SX}{\SY_j}}{\cW_j}{k}$,
	we can apply the inductive hypothesis to all $b^{(i)}$ and the
	$\ConDiff{}{}{k}$-maps
	$\gamma_1, \dotsc, \gamma_{m}$ and $D\gamma_1, \dotsc, D\gamma_m$
	to see that this is an element of $\CFvan{\UF}{\Lin{\SX}{\SZ}}{\GewFunk}{k}$.
\end{proof}
%
%
We list some applications of \refer{prop:multilineare_Abb_und_CF}.
In the following corollaries, $k \in \cl{\N}$, $\UF$ is an open nonempty subset 
of the normed space $\SX$ and $\GewFunk \subseteq \cl{\R}^\UF$
always contains the constant map $1_{\UF}$.
\begin{cor}\label{cor:Multiplikation_auf_CF()(A)}
	Let $A$ be a normed algebra with the continuous multiplication $\mult$.
	Then $\CcF{\UF}{A}{k}$ is an algebra with the continuous multiplication
	\begin{align*}
		M(\mult):& \CcF{\UF}{A}{k}\times\CcF{\UF}{A}{k}
			\to\CcF{\UF}{A}{k} \\
			& M(\mult)(\gamma,\eta)(x) = \gamma(x)\mult\eta(x) .
	\end{align*}
	We shall often write $\mult$ instead of $M(\mult)$.
\end{cor}
\begin{cor}\label{cor:Komposition_linearer_Abb_und_CF}
	If $E$, $F$ and $G$ are normed spaces, then the composition of
	linear operators
	\begin{align*}
		\MaMu:\Lin{F}{G}\times \Lin{E}{F}\to \Lin{E}{G}
	\end{align*}
	is bilinear and continuous and therefore induces
	the continuous bilinear maps
	\begin{align*}
		M(\MaMu) :& \CcF{\UF}{\Lin{F}{G}}{k}\times\CcF{\UF}{\Lin{E}{F}}{k}
			\to \CcF{\UF}{\Lin{E}{G}}{k}\\
		& M(\MaMu)(\gamma,\eta)(x) = \gamma(x)\MaMu\eta(x)
	\end{align*}
	and
	\begin{align*}
		M_{\cal BC}(\MaMu) :& \CcF{\UF}{\Lin{F}{G}}{k}\times\BC{\UF}{\Lin{E}{F}}{k}
			\to \CcF{\UF}{\Lin{E}{G}}{k}\\
		& M_{\cal BC}(\MaMu)(\gamma,\eta)(x) = \gamma(x)\MaMu\eta(x) .
	\end{align*}
	We shall often denote $M(\MaMu)$ just by $\MaMu$.
\end{cor}
\begin{cor}\label{cor:Auswertung_linearer_abb_und_CF}
	Let $E$ and $F$ be normed spaces.
	Then the evaluation of linear maps
	\begin{align*}
		\eval:\Lin{E}{F} \times E \to F:(T,w)\mapsto T\eval w
	\end{align*}
	is bilinear und continuous
	(see \refer{lem:Auswertung_linearer_Abb_ist_stetig})
	and hence induces the continuous bilinear map
	\begin{align*}
		M(\eval): &\CcF{\UF}{\Lin{E}{F}}{k}\times\CcF{\UF}{E}{k}
		\to \CcF{\UF}{F}{k}\\
		& M(\eval)(\Gamma,\eta)(x) = \Gamma(x)\eval\eta(x).
	\end{align*}
	Instead of $M(\eval)$ we will often write $\eval$.
\end{cor}

\subsection{Composition of weighted functions with bounded functions}
We explore the composition between spaces of bounded functions and spaces of weighted functions.
One case that is of particular interest is the composition between certain
subsets of the spaces $\BC{\UF}{\SY}{k}$.
\subsubsection{Composition of bounded functions}
We discuss under which conditions the composition is continuous or differentiable.
\newcommand{\Hcomp}{\ensuremath{g}}
\begin{lem}
\label{lem:BC_ist_unter_Komposition_abgeschlossen}
\label{lem:Komposition_von_BC_ist_stetig}
	Let $\SX$, $\SY$ and $\SZ$ be normed spaces,
	$\UF\subseteq \SX$ and $\VF\subseteq \SY$ open nonempty subsets
	and $k \in \cl{\N}$.
	Then for $\gamma \in \BC{\VF}{\SZ}{k + 1}$ and
	$\eta \in \BCo{\UF}{\VF}{k}$
	\[
		\gamma\circ\eta \in \BC{\UF}{\SZ}{k},
	\]
	and the map
	\begin{equation*}\label{comp_BC}
	\tag{\ensuremath{\ast}}
		\BC{\VF}{\SZ}{k + 1} \times \BCo{\UF}{\VF}{k}
		\to \BC{\UF}{\SZ}{k}
		: (\gamma,\eta)\mapsto \gamma\circ\eta
	\end{equation*}
	is continuous.
\end{lem}
\begin{proof}
	For $k < \infty$ this is proved by induction:\\
	$k=0$:
	Obviously
	\[
		\BC{\VF}{\SZ}{1} \circ \BCo{\UF}{\VF}{0}
		\subseteq \BC{\UF}{\SZ}{0},
	\]
	so it remains to show that the composition is continuous.
	To this end, let $\gamma ,\gamma_0 \in \BC{\VF}{\SZ}{1}$,
	$\eta ,\eta_0 \in \BCo{\UF}{\VF}{0}$ with
	$\hn{\eta - \eta_{0}}{1_{\UF}}{0} < \dist{\eta_{0}(\UF)}{\partial\VF}$
	and $x\in\UF$. Then
	\begin{align*}
		&\norm{(\gamma \circ \eta)(x) - (\gamma_0 \circ \eta_0)(x)}
		\\
		=& \norm{\gamma(\eta(x)) - \gamma(\eta_0(x))
				+ \gamma(\eta_0(x)) - \gamma_0(\eta_0(x))}\\
		%
		\leq& \left\norm{\Rint{0}{1}{
			D\gamma(t \eta(x) + (1 - t) \eta_0(x) )\eval(\eta(x) - \eta_0(x))}%
			{t}\right}
			+ \norm{(\gamma - \gamma_0)(\eta_0(x))}\\
		\leq& \hn{ \FAbl{\gamma} }{1_{\VF}}{0} \norm{\eta(x) - \eta_0(x)}
			+ \norm{(\gamma - \gamma_0)(\eta_0(x))};
	\end{align*}
	in this estimate we used
	$\hn{\eta - \eta_{0}}{1_{\UF}}{0} < \dist{\eta_{0}(\UF)}{\partial\VF}$
	to ensure that the line segment between $\eta(x)$ und $\eta_0(x)$
	is contained in $\VF$.
	The estimate yields
	\begin{align*}
		\hn{\gamma \circ \eta - \gamma_0 \circ \eta_0}{1_{\UF}}{0}
		\leq \hn{\gamma}{1_{\VF}}{1} \hn{\eta - \eta_0}{1_{\UF}}{0}
			+ \hn{\gamma - \gamma_0}{1_{\UF}}{0},
	\end{align*}
	whence the composition is continuous.
	
	$k\to k + 1$:
	In the following, we denote the composition map \eqref{comp_BC}
	with $\Hcomp_{k,\SZ}$.
	We know from \refer{prop:topologische_Zerlegung_von_CFk}
	(and the induction base) that
	\begin{multline*}
		\Hcomp_{k+1, \SZ}
		(\BC{\VF}{\SZ}{k + 2} \times \BCo{\UF}{\VF}{k+1})
		\subseteq \BC{\UF}{\SZ}{k+1}
		\\
		\iff
		(\FAbl{}\circ \Hcomp_{k+1, \SZ})
			(\BC{\VF}{\SZ}{k + 2} \times \BCo{\UF}{\VF}{k+1})
		\subseteq \BC{\UF}{\Lin{\SX}{\SZ}}{k}
	\end{multline*}
	and $\Hcomp_{k+1, \SZ}$ is continuous iff
	$\FAbl{}\circ \Hcomp_{k+1, \SZ}$ is so,
	as a map to $\BC{\UF}{\Lin{\SX}{\SZ}}{k}$.
	An application of the chain rule gives
	\begin{equation*}\label{id:Ableitung_der_Komposition}
	\tag{\ensuremath{\ast\ast}}
		(D\circ \Hcomp_{k+1, \SZ})(\gamma,\eta)
		=\Hcomp_{k, \Lin{\SY}{\SZ} }(D\gamma, \eta)\MaMu D\eta
	\end{equation*}
	for $\gamma \in \BC{\VF}{\SZ}{k + 2}$ and
	$\eta \in \BCo{\UF}{\VF}{k+1}$,
	where $\MaMu$ denotes the composition of linear maps, see \refer{cor:Komposition_linearer_Abb_und_CF}.
	Since
	$D\gamma \in \BC{\VF}{\Lin{\SY}{\SZ}}{k + 1}$,
	we deduce from the inductive hypothesis that
	\[
		\Hcomp_{k, \Lin{\SY}{\SZ} }(D\gamma, \eta)
		\in \BC{\UF}{\Lin{\SY}{\SZ}}{k},
	\]
	and using \refer{cor:Komposition_linearer_Abb_und_CF} we get
	\[
		(D\circ \Hcomp_{k+1, \SZ})(\gamma,\eta)
			\in \BC{\UF}{\Lin{\SY}{\SZ}}{k}.
	\]
	The continuity of $D\circ \Hcomp_{k+1, \SZ}$ follows with	
	\refer{id:Ableitung_der_Komposition}
	from the continuity of 
	$\Hcomp_{k, \Lin{\SY}{\SZ}}$ (by the inductive hypothesis), 
	$\MaMu$ (by \refer{cor:Komposition_linearer_Abb_und_CF})
	and $D$ (by \refer{prop:Ableitung_ist_stetig}).
	
	$k=\infty$:
	\BeweisschrittCkCinfty%
	{\BC{\VF}{\SZ}{\infty} \times \BCo{\UF}{\VF}{\infty}}
	{\Hcomp_{\infty,\SZ}}
	{\BC{\UF}{\SZ}{\infty}}
	{}
	{}
	{\BC{\VF}{\SZ}{n + 1} \times \BCo{\UF}{\VF}{n}}
	{\Hcomp_{n,\SZ}}
	{\BC{\UF}{\SZ}{n}}
	{n}
\end{proof}
As a preparation for discussing the differentiable properties of the composition,
we prove a nice identity for its differential quotient.
\begin{lem}\label{lem:Differenzenquotient_Komposition}
	Let $\SX$, $\SY$ and $\SZ$ be normed spaces
	and $\UF\subseteq \SX$, $\VF\subseteq \SY$ be
	open subsets.
	Further, let $\gamma \in \FC{\VF}{\SZ}{1}$,
	$\tilde{\gamma} \in  \ConDiff{\VF}{\SZ}{0}$,
	$\tilde{\eta} \in  \BC{\UF}{\SY}{0}$
	and
	$\eta \in \ConDiff{\UF}{\VF}{0}$
	such that $\dist{\eta(\UF)}{\partial \VF} > 0$.
	Then, for all $x \in \UF$ and $t \in \R^\ast$ with
	\[
		\abs{t}
		\leq
		\tfrac{\dist{\eta(\UF)}{\partial \VF}}{\hn{\tilde{\eta}}{1_{\UF}}{0} + 1},
	\]
	the identity
	\begin{equation}\label{id:Differenzenquotient_Komposition}
		\evTwo_{x}\left(
		\tfrac{(\gamma + t \tilde{\gamma} )\circ (\eta + t \tilde{\eta}) - \gamma \circ \eta}{t}
		\right)
		=
		\evTwo_{x}(
			\tilde{\gamma} \circ (\eta + t \tilde{\eta})
		)
		+
		\Rint{0}{1}
		{
			\evTwo_{x}(
				(\FAbl{\gamma}\circ(\eta + s t \tilde{\eta}))\eval \tilde{\eta}
			)
		}
		{s}
	\end{equation}
	holds, where $\evTwo_{x}$ denotes the evaluation at $x$.
\end{lem}
\begin{proof}
	For $t$ as above the identity
	\begin{equation*}
		(\gamma + t \tilde{\gamma}) \circ (\eta + t \tilde{\eta}) - \gamma\circ\eta
		= \gamma \circ (\eta + t \tilde{\eta}) + t \tilde{\gamma} \circ (\eta + t \tilde{\eta}) - \gamma\circ\eta
	\end{equation*}
	holds, and an application of the mean value theorem gives
	\[
		\evTwo_{x}(\gamma \circ (\eta + t \tilde{\eta}) - \gamma\circ\eta)
		=
		\Mint
		{
			\evTwo_{x}(
				(\FAbl{\gamma}\circ(\eta + s t \tilde{\eta}))\eval t \tilde{\eta}
			)
		}
		{s}.
	\]
	Division by $t$ leads to the desired result.
\end{proof}
\newcommand{\comp}[3]{
\ensuremath{
	g_{\ifthenelse{\equal{#1}{BC}}{\mathcal BC}{#1}, #2}^{#3}
}
}
So we are ready to discuss when the composition is differentiable.
\begin{prop}\label{prop:Komposition_von_BCell_ist_diffbar}
	Let $\SX$, $\SY$ and $\SZ$ be normed spaces,
	$\UF\subseteq \SX$ and $\VF\subseteq \SY$ open subsets
	and $k \in \cl{\N}$, $\ell \in \cl{\N}^\ast$.
	Then the continuous map
	\[
		\comp{BC}{\SZ}{k + \ell + 1} :
		\BC{\VF}{\SZ}{k + \ell + 1} \times \BCo{\UF}{\VF}{k}
		\to \BC{\UF}{\SZ}{k}
		: (\gamma,\eta)\mapsto \gamma\circ\eta
	\]
	(cf. \refer{lem:BC_ist_unter_Komposition_abgeschlossen})
	is a $\ConDiff{}{}{\ell}$-map with
	\begin{equation}\label{Ableitung_Comp_BCellxBC}
		\dA{\comp{BC}{\SZ}{k + \ell + 1}}{\gamma_0,\eta_0}{\gamma,\eta}
		= \comp{BC}{\SZ}{k + \ell + 1}(\gamma,\eta_0)
		+ \comp{BC}{\Lin{\SY}{\SZ}}{k + \ell}(D\gamma_0,\eta_0) \eval \eta.
	\end{equation}
	\index{superposition!with a bounded map}%
	\index{bounded maps!composition of}%
	\index{composition!of bounded maps|see{bounded maps, composition of}}%
\end{prop}%
\newcommand{\BeweisKompoCBCW}[1]%
{
	For $k < \infty$, the proof is by induction on $\ell$ which we may assume finite
	because the inclusion maps
	$\ifthenelse{\equal{#1}{BC}}{\BC}{\BCzero}{\VF}{\SZ}{\infty}
	\to \ifthenelse{\equal{#1}{BC}}{\BC}{\BCzero}{\VF}{\SZ}{k + \ell + 1}$
	are continuous linear (and hence smooth).
	
	$\ell = 1$:
	Let $\gamma_0, \gamma \in \ifthenelse{\equal{#1}{BC}}{\BC}{\BCzero}{\VF}{\SZ}{k + \ell + 1}$,
	$\eta_0 \in \ifthenelse{\equal{#1}{BC}}{\BCo}{\CcFo}{\UF}{\VF}{k}$ and $\eta \in \ifthenelse{\equal{#1}{BC}}{\BC}{\CcF}{\UF}{\SY}{k}$.
	From \refer{lem:Differenzenquotient_Komposition} and
	\refer{lem:Kriterium_Integrierbarkeit_in_CW}
	we conclude that for $t \in \K$ with
	$\abs{t} \leq \frac{\dist{\eta_0(\UF)}{\partial \VF}}{\hn{\eta}{1_{\UF}}{0} + 1}$,
	the integral
	\[
		\Mint
		{
			(\FAbl{\gamma_0} \circ(\eta_0 + s t \eta))\eval \eta
		}
		{s}
	\]
	exists in $\ifthenelse{\equal{#1}{BC}}{\BC}{\CcF}{\UF}{\SZ}{k}$.
	Using \refer{id:Differenzenquotient_Komposition} we derive
	\begin{equation*}
		\tfrac{\comp{#1}{\SZ}{k + \ell + 1}(\gamma_0 + t \gamma,\eta_0 + t \eta)
			- \comp{#1}{\SZ}{k + \ell + 1}(\gamma_0,\eta_0)}{t}
		= \comp{#1}{\SZ}{k + \ell + 1}(\gamma,\eta_0 + t \eta) \\ + 
			\Mint{\comp{BC}{\Lin{\SY}{\SZ}}{k + \ell}(\FAbl{\gamma_0}, \eta_0 + s t \eta) \eval \eta}{s}.
	\end{equation*}
	We use \refer{prop:Stetigkeit_parameterab_Int}
	and the continuity of $\comp{#1}{\SZ}{k + \ell + 1}$,
	$\comp{BC}{\Lin{\SY}{\SZ}}{k + \ell}$ and $\eval$
	(cf. \ifthenelse{\equal{#1}{BC}}{}{\refer{lem:BC0_operiert_stetig_auf_CW},}
	\refer{lem:Komposition_von_BC_ist_stetig} and
	\refer{cor:Auswertung_linearer_abb_und_CF})
	to see that the right hand side of this equation converges to
	\[
		\comp{#1}{\SZ}{k + \ell + 1}(\gamma,\eta_0)
		+ \comp{BC}{\Lin{\SY}{\SZ}}{k + \ell}(\FAbl{\gamma_0},\eta_0) \eval \eta
	\]
	in $\ifthenelse{\equal{#1}{BC}}{\BC}{\CcF}{\UF}{\SZ}{k}$ as $t \to 0$. 
	Hence the $\comp{#1}{\SZ}{k + \ell + 1}$ is differentiable and its differential is given by
	\ifthenelse{\equal{#1}{BC}}%
	{\eqref{Ableitung_Comp_BCellxBC}}{\eqref{Ableitung_Comp_BC0ellxCW}}
	and thus continuous.

	$\ell - 1 \to \ell$:
	The map $\comp{#1}{\SZ}{k + \ell + 1}$ is $\ConDiff{}{}{\ell}$
	if $\dA{\comp{#1}{\SZ}{k + \ell + 1}}{}{}$ is $\ConDiff{}{}{\ell - 1}$.
	The latter follows easily from
	\ifthenelse{\equal{#1}{BC}}%
	{\eqref{Ableitung_Comp_BCellxBC}}{\eqref{Ableitung_Comp_BC0ellxCW}},
	since the inductive hypothesis \ifthenelse{\equal{#1}{BC}}%
	{ensures}{respective \refer{prop:Komposition_von_BCell_ist_diffbar} ensure}
	that $\comp{#1}{\SZ}{k + \ell + 1}$
	and $\comp{BC}{\Lin{\SY}{\SZ}}{k + \ell}$
	are $\ConDiff{}{}{\ell - 1}$; and $\eval$ and $\FAbl{}$ are smooth.
	
	If $k = \infty$, then in view of \refer{cor:Topologie_von_CinfF} and \refer{prop:Differenzierbarkeit_Abb_in_projektiven_Limes},
	$\ifthenelse{\equal{#1}{BC}}{ \comp{BC} }{ \comp{\cW} }{\SZ}{\infty}$
	is smooth as a map to $\ifthenelse{\equal{#1}{BC}}{ \BC }{ \CcF }{\UF}{\SZ}{\infty}$
	iff it is smooth as a map to $\ifthenelse{\equal{#1}{BC}}{ \BC }{ \CcF }{\UF}{\SZ}{j}$
	for each $j \in \N$. This was already proved in the case where $k = j$ and $\ell = \infty$.
}%
\begin{proof}
	\BeweisKompoCBCW{BC}
\end{proof}
\subsubsection{Composition of weighted functions with bounded functions}
Generally, we can not expect that the composition of a bounded function
with a weighted function is again a weighted function (to the same weights).
As an example, the composition of the constant $1$ function
and a Schwartz function is not a Schwartz function.
However, if we compose a bounded function mapping $0$ to $0$ with a weighted function,
we get good results.
\begin{lem}\label{lem:BC0_operiert_stetig_auf_CW}
	Let $\SX$, $\SY$ and $\SZ$ be normed spaces,
	$\UF\subseteq \SX$ and $\VF\subseteq \SY$ open subsets
	such that $\VF$ is star-shaped with center~$0$,
	$k \in \cl{\N}$ and $\cW \subseteq \cl{\R}^{\UF}$ with
	$1_{\UF} \in \cW$.
	Then for $\gamma \in \BCzero{\VF}{\SZ}{k + 1}$ and
	$\eta \in \CcFo{\UF}{\VF}{k}$
	\[
		\gamma\circ\eta \in \CcF{\UF}{\SZ}{k},
	\]
	and the composition map
	\begin{equation*}\label{comp_BCxCW}
	\tag{\ensuremath{\ast}}
		\BCzero{\VF}{\SZ}{k + 1} \times \CcFo{\UF}{\VF}{k}
		\to \CcF{\UF}{\SZ}{k}
		: (\gamma,\eta)\mapsto \gamma\circ\eta
	\end{equation*}
	is continuous.
\end{lem}
\begin{proof}
	We distinguish the cases $k < \infty$ and $k = \infty$:\\
	$k < \infty$:
	To prove that for $\gamma \in \BCzero{\VF}{\SZ}{k + 1}$
	and $\eta \in \CcFo{\UF}{\VF}{k}$
	the composition $\gamma \circ \eta$ is in $\CcF{\UF}{\SZ}{k}$,
	in view of \refer{prop:topologische_Zerlegung_von_CFk} it suffices to show that
	\[
		\gamma \circ \eta \in \CcF{\UF}{\SZ}{0}
		\text{ and for $k > 0$ also }
		\FAbl{(\gamma \circ \eta)} \in \CcF{\UF}{\Lin{\SX}{\SZ}}{k - 1}.
	\]
	Similarly the continuity of the composition \eqref{comp_BCxCW}, which is denoted
	by $g_k$ in the remainder of this proof, is equivalent to the continuity of
	$\iota_0 \circ g_k$
		and for $k > 0$ also of
		$\FAbl{} \circ g_k$,
	where
	$\iota_0 : \CcF{\UF}{\SZ}{k} \to \CcF{\UF}{\SZ}{0}$
	denotes the inclusion map.
	
	First we show $\gamma \circ \eta \in \CcF{\UF}{\SZ}{0}$.
	To this end, let $f \in \cW$ and $x\in\UF$. Then
	\begin{multline*}
		\abs{f(x)}\,\norm{\gamma(\eta(x))}
		= \abs{f(x)}\,\norm{\gamma(\eta(x)) - \gamma(0)}\\
		= \abs{f(x)}\,\left\norm{\Mint{\FAbl{\gamma}(t \eta(x)) \eval \eta(x)}{t}\right}
		\leq \hn{\FAbl{\gamma}}{1_\VF}{0} \hn{\eta}{f}{0};
	\end{multline*}
	here we used that the line segment from $0$ to $\eta(x)$ is contained in $\VF$.
	So $\gamma \circ \eta \in \CcF{\UF}{\SZ}{0}$.
	To check the continuity of $\iota_0 \circ g_k$,
	let $\gamma ,\gamma_0 \in \BCzero{\VF}{\SZ}{k + 1}$
	and $\eta ,\eta_0 \in \CcFo{\UF}{\VF}{k}$ such that
	$\hn{\eta - \eta_0}{1_{\UF}}{0} < \dist{\eta_0(\UF)}{\partial\VF}$,
	$f\in\cW$ and $x\in\UF$. Then
	\begin{align*}
		&\abs{f(x)} \,\norm{(\gamma \circ \eta)(x) - (\gamma_0 \circ \eta_0)(x)}\\
		=& \abs{f(x)} \,\norm{\gamma(\eta(x)) - \gamma(\eta_0(x))
				+ \gamma(\eta_0(x)) - \gamma_0(\eta_0(x))}\\
		\leq& \abs{f(x)} \,\norm{\gamma(\eta(x)) - \gamma(\eta_0(x))}
			+ \abs{f(x)} \,\norm{(\gamma - \gamma_0)(\eta_0(x))}\\
		=& \abs{f(x)} \,\left\norm{\Rint{0}{1}{\FAbl{\gamma}(t \eta(x) + (1 - t) \eta_0(x)) \eval (\eta(x) - \eta_0(x))}{t}\right}\\
			&\hspace{3cm} + \abs{f(x)} \,\norm{(\gamma - \gamma_0)(\eta_0(x)) - (\gamma - \gamma_0)(0)}\\
		\leq& \abs{f(x)} \,\hn{\FAbl{\gamma}}{1_\VF}{0} \norm{\eta(x) - \eta_0(x)}
			+ \abs{f(x)} \,\left\norm{\Rint{0}{1}{\FAbl{(\gamma - \gamma_0)}(t \eta_0(x)) \eval \eta_0(x)}{t}\right}\\
		\leq& \abs{f(x)} \,\hn{\FAbl{\gamma}}{1_\VF}{0} \norm{\eta(x) - \eta_0(x)}
			+ \abs{f(x)} \,\hn{\FAbl{(\gamma - \gamma_0)}}{1_\VF}{0} \norm{\eta_0(x)}.
	\end{align*}
	From this, we easily see that $\iota_0 \circ g_k$ is continuous in $(\gamma_0, \eta_0)$.
	
	For $k > 0$, $\gamma \in \BCzero{\VF}{\SZ}{k + 1}$ and
	$\eta \in \CcFo{\UF}{\VF}{k}$ we apply the chain rule to get
	\begin{equation*}\label{id:Ableitung_der_Komposition_CW}
	\tag{\ensuremath{\ast\ast}}
		(\FAbl{} \circ g_k)(\gamma, \eta)
		= \FAbl{(\gamma \circ \eta)}
		=(\FAbl{\gamma} \circ \eta)\MaMu \FAbl{\eta}
		= \comp{BC}{\Lin{\SY}{\SZ}}{k}(\FAbl{\gamma}, \eta)
			\MaMu \FAbl{\eta};
	\end{equation*}
	here we used that $\eta \in \BC{\UF}{\VF}{k}$ because $1_{\UF}$ is in $\cW$.
	Since $\FAbl{\eta} \in \CcF{\UF}{\Lin{\SX}{\SY} }{k - 1}$ 
	and
	$\comp{BC}{\Lin{\SY}{\SZ}}{k}(\FAbl{\gamma}, \eta) \in \BC{\UF}{\Lin{\SY}{\SZ}}{k - 1}$
	hold, (see \refer{lem:BC_ist_unter_Komposition_abgeschlossen}),
	$(\FAbl{} \circ g_k)(\gamma, \eta)$ is in $\CcF{\UF}{\Lin{\SY}{\SZ} }{k - 1}$
	(see \refer{cor:Komposition_linearer_Abb_und_CF}).
	Using that $\FAbl{}$, $\MaMu$ and $\comp{BC}{\Lin{\SY}{\SZ}}{k}$
	are continuous (see \refer{prop:Ableitung_ist_stetig},
	\refer{cor:Komposition_linearer_Abb_und_CF} and
	\refer{lem:Komposition_von_BC_ist_stetig}, respectively),
	we deduce the continuity of
	$\FAbl{} \circ g_k$ from \eqref{id:Ableitung_der_Komposition_CW}.
	
	$k = \infty$:
	\BeweisschrittCkCinfty%
	{\BC{\VF}{\SZ}{\infty}_0 \times \CcFo{\UF}{\VF}{\infty}}
	{\Hcomp_{\infty}}
	{\CcFo{\UF}{\SZ}{\infty}}
	{}
	{}
	{\BC{\VF}{\SZ}{n + 1}_0 \times \CcFo{\UF}{\VF}{n}}
	{\Hcomp_{n}}
	{\CcFo{\UF}{\SZ}{n}}
	{n}
\end{proof}

\begin{prop}\label{prop:BC0_operiert_glatt_auf_CW}
	Let $\SX$, $\SY$ and $\SZ$ be normed spaces,
	$\UF\subseteq \SX$ and $\VF\subseteq \SY$ open subsets
	such that $\VF$ is star-shaped with center~$0$,
	$k \in \cl{\N}$, $\ell \in \cl{\N}^\ast$
	and $\cW \subseteq \cl{\R}^{\UF}$ with $1_{\UF} \in \cW$.
	Then the map
	\begin{equation*}
		\comp{\cW}{\SZ}{k + \ell + 1}
		:
		\BCzero{\VF}{\SZ}{k + \ell + 1} \times \CcFo{\UF}{\VF}{k}
		\to \CcF{\UF}{\SZ}{k}
		: (\gamma,\eta)\mapsto \gamma\circ\eta
	\end{equation*}
	whose existence was stated in \refer{lem:BC0_operiert_stetig_auf_CW}
	is a $\ConDiff{}{}{\ell}$-map with
	\begin{equation}\label{Ableitung_Comp_BC0ellxCW}
		\dA{\comp{\cW}{\SZ}{k + \ell + 1}}{\gamma_0,\eta_0}{\gamma,\eta}
		= \comp{\cW}{\SZ}{k + \ell + 1}(\gamma,\eta_0)
		+ \comp{BC}{\Lin{\SY}{\SZ}}{k + \ell}(D\gamma_0,\eta_0) \eval \eta.
	\end{equation}
	\index{superposition!with a bounded map}%
	\index{composition!of bounded maps and weighted maps}%
\end{prop}
\begin{proof}
	\BeweisKompoCBCW{\GewFunk}
\end{proof}

\subsection{Composition of weighted functions with an analytic map}
\label{susec:Composing_weighted_functions_with_analytic_maps}
We discuss a sufficient criterion for an analytic map to operate on $\CcFo{\UF}{\VF}{k}$
through (covariant) composition.
First, we state a result about the superposition on weighted functions
that is a direct consequence of \refer{prop:BC0_operiert_glatt_auf_CW}.
After that, we have to treat real and complex analytic functions seperately.
While the complex case is straightforward, in the real case we have to deal with complexifications.

We begin with a lemma about star-shaped open sets.
\begin{lem}\label{lem:Auschoepfen_sternfoermiger_Mengen}
	Let $\SX$ be a normed space and $\VF\subseteq \SX$ an open set
	that is star-shaped with center~$0$.
	Then for $d \ndef \dist{\sset{0}}{\partial V}$, there exists an anti-monotone family $\fami{\VF^\partial_r}{r}{]0, d[}$ such that
	\[
		\VF = \bigcup_{d > r > 0} \VF^\partial_r .
	\]
	Further, each $\VF^{\partial,r}$ is an open bounded set that is star-shaped with center $0$
	such that
	\begin{equation}\label{est:Abstand_Rand_Approx_sfMengen}
		\dist{\VF^{\partial,r}}{\partial \VF} \geq \tfrac{d - r}{2} \min(1, r^2).
	\end{equation}
\end{lem}
\begin{proof}
	If $\VF = \SX$, this is obviously true. Otherwise, $\partial \VF \neq \emptyset$
	and $d \in \R$. We set for $r \in ]0, d[$
	\[
		\VF^\partial_r \ndef
		[0,1] \cdot \left( \set{y \in \VF}{ \dist{\sset{y}}{\partial \VF} > r} \cap \Ball{0}{\frac{1}{r}}\right).
	\]
	This set is obviously bounded and star-shaped with center $0$.
	Further, it is open: It is the union of an open set with $\sset{0}$,
	and by the choice of $r$, it contains $\Ball{0}{d - r}$.
	So it remains to show that $\dist{\VF^\partial_r}{\partial \VF} > 0$.
	To this end, let $x \in \VF^\partial_r$. We distinguish two cases.
	
	First case: $x \in \Ball[ ]{0}{\tfrac{d - r}{2}}$.
	Then $\Ball[ ]{x}{\tfrac{d - r}{2}} \sub \VF$.
	
	Otherwise, there exists
	$z \in \set{y \in \VF}{ \dist{\sset{y}}{\partial \VF} > r} \cap \Ball{0}{\frac{1}{r}}$
	and $t \in ]0, 1]$ with $x = t z$.
	We show that $\Ball{x}{t r} \sub \VF$, and use that obviously $\Ball{z}{r} \sub \VF$.
	So let $v \in \Ball{x}{t r}$. We set $h \ndef \tfrac{v}{t} - z$.
	Then $\norm{h} < r$, and hence $v = t (z + h) \in \VF$ since $\VF$ is star-shaped.
	Further, we know that $\norm{z} < \tfrac{1}{r}$ and $\norm{x} \geq \tfrac{d - r}{2}$.
	Hence
	\[
		\tfrac{t}{r} > t \norm{z} = \norm{x}  \geq \tfrac{d - r}{2}.
	\]
	This implies that $\Ball[ ]{x}{\tfrac{d - r}{2} r^2} \sub \VF$.
	
	We deduce that \refer{est:Abstand_Rand_Approx_sfMengen} holds.
\end{proof}

\begin{lem}\label{lem:Wirkung_einer_lokal_beschraenkten_Abb_auf_CWo}
	Let $\SX$, $\SY$ and $\SZ$ be normed spaces,
	$\UF\subseteq \SX$ and $\VF\subseteq \SY$ open subsets
	such that $\VF$ is star-shaped with center~$0$,
	$k \in \cl{\N}$, $\ell \in \cl{\N}^\ast$
	and $\cW \subseteq \cl{\R}^{\UF}$ with $1_{\UF} \in \cW$.
	Suppose further that $\Phi : \VF \to \SZ$ with $\Phi(0) = 0$ satisfies
	\begin{multline*}
		\text{ $\WF$ open in $\VF$, bounded and star-shaped with center~$0$, }
		\dist{\WF}{\partial \VF} > 0
		\\\implies
		\rest{\Phi}{\WF} \in \BC{\WF}{\SZ}{k + \ell + 1}.
	\end{multline*}
	Then
	$
		\Phi \circ \gamma \in \CF{\UF}{\SZ}{\cW}{k}
	$
	for all $\gamma \in \CFo{\UF}{\VF}{\cW}{k}$, and the map
	\[
		\CFo{\UF}{\VF}{\cW}{k} \to \CF{\UF}{\SZ}{\cW}{k}
		:
		\gamma \mapsto \Phi \circ \gamma
	\]
	is $\ConDiff{}{}{\ell}$.
\end{lem}
\begin{proof}
	We let $\fami{\VF^\partial_r}{r}{]0,d[}$ as in \refer{lem:Auschoepfen_sternfoermiger_Mengen}.
	Then for each $r$, we know from \refer{prop:BC0_operiert_glatt_auf_CW} that
	\[
		\CcFo{\UF}{\VF^\partial_r}{k} \to \CcF{\UF}{\SZ}{k}
		: \gamma \mapsto \Phi\circ\gamma
	\]
	is defined and $\ConDiff{}{}{\ell}$ since $\Phi \in \BCzero{M_r}{\SZ}{k + \ell + 1}$
	by our assumption. But
	\[
		\CcFo{\UF}{\VF}{k} = \bigcup_{r>0} \CcFo{\UF}{\VF^\partial_r}{k},
	\]
	and $1_\UF \in \cW$ implies that each $\CcFo{\UF}{\VF^\partial_r}{k}$ is open in $\CcFo{\UF}{\VF}{k}$,
	hence the assertion is proved.
\end{proof}

\begin{lem}\label{lem:analyt-Abb_BC0-large_impliziert_BCinf-klein}
	Let $\SY$ and $\SZ$ be complex normed spaces,
	$\VF \subseteq \SY$ an open nonempty subset
	and $\Phi: \VF \to \SZ$ a complex analytic map.
	Further, let $\WF \subseteq \VF$ with $\dist{\WF}{\partial \VF} > 0$
	and $r > 0$ with $r < \dist{\WF}{\partial \VF}$ such that
	$\rest{\Phi}{\WF + \clBall[\SY]{0}{r}} \in \BC{\WF + \clBall[\SY]{0}{r} }{\SZ}{0}$.
	Then $\rest{\Phi}{\WF} \in \BC{\WF}{\SZ}{\infty}$.
	More explicitly, for each $k \in \N$ we have
	\begin{equation}\label{est:k-norm_durch_0-norm_anaAbb}
		\hn{\Phi}{1_\WF}{k}
		\leq \frac{(2 k)^k}{ r^k } \hn{\Phi}{1_{\WF + \clBall[\SY]{0}{r}}}{0}.
	\end{equation}
\end{lem}
\begin{proof}
	For each $x \in \WF$ and $h \in \SY$ with $\norm{h}\leq 1$,
	we get an analytic map
	\[
		\Phi_{x,h}: \Ball[\C]{0}{r + \eps} \to \SZ
			: z\mapsto \Phi(x + z h)
	\]
	by restricting $\Phi$, see \refer{satz:Char_komplex_analytischer_Abb}
	(here $\eps > 0$ exists because $r < \dist{\WF}{\partial \VF}$).
	By applying the Cauchy estimates (stated in \refer{cor:Cauchy_Abschaetzungen})
	to this map (and the circle with radius $r$ in $\C$), we get the estimate
	\[
		\norm{\Phi_{x,h}^{(k)}(0)}
		\leq
		\frac{k!}{ r^k }\noma{ \rest{\Phi}{\sset{x} + r \cl{\Disk} h  } }
		\leq
		\frac{k!}{ r^k }\noma{ \rest{\Phi}{\clBall[\SY]{x}{r}  } }.
	\]
	From \refer{lem:Komplexe_Kurven_Ableitung} and the chain rule
	we know that $\Phi_{x,h}^{(k)}(0) = \FAbl[k]{\Phi}(x)(h,\dotsc,h)$,
	so we conclude with the Polarization Formula (\refer{prop:Polarisierungsformel})
	that
	\[
		\Opnorm{\FAbl[k]{\Phi}(x)}
		\leq
		\frac{(2 k)^k}{ r^k } \noma{ \rest{\Phi}{ \clBall[\SY]{x}{r} } }.
	\]
	We derive \eqref{est:k-norm_durch_0-norm_anaAbb} and from that
	the assertion.
\end{proof}

\begin{lem}\label{lem:analyt-Abb_lokal-BCinf_wenn_lokal-BC0}
	Let $\SY$ and $\SZ$ be complex normed spaces,
	$\VF \subseteq \SY$ an open nonempty  subset
	and $\Phi: \VF \to \SZ$ a complex analytic map that satisfies
	the following condition:
	\begin{equation}\label{bed:Abb_in_BC^0}
		\WF \subseteq \VF,
		\text{ $\WF$ open in $\VF$, }
		\dist{\WF}{\partial \VF} > 0
		\implies \rest{\Phi}{\WF} \in \BC{\WF}{\SZ}{0}.
	\end{equation}
	Then $\rest{\Phi}{\WF} \in \BC{\WF}{\SZ}{\infty}$ for all open subsets
	$\WF \subseteq \VF$ with $\dist{\WF}{\partial \VF} > 0$.
\end{lem}
\begin{proof}
	Let $\WF \subseteq \VF$ open with $r \ndef \dist{\WF}{\partial \VF} > 0$.
	Then $\widetilde{\WF} \ndef \WF + \clBall{0}{\frac{r}{2}}$ is open in $\VF$
	with $\dist{\widetilde{\WF}}{\partial \VF} > 0$.
	By \eqref{bed:Abb_in_BC^0}, we have that
	$\rest{\Phi}{\widetilde{\WF}} \in \BC{\widetilde{\WF}}{\SZ}{0}$.
	Hence we can apply \refer{lem:analyt-Abb_BC0-large_impliziert_BCinf-klein} to see that
	$\rest{\Phi}{\WF} \in \BC{\WF}{\SZ}{\infty}$.
\end{proof}
\subsubsection{On real analytic maps and good complexifications}
The two previous lemmas would allow us to state the desired result concerning
covariant composition, but only for complex analytic maps.
There are examples of real analytic maps for which
the assertion of \refer{lem:analyt-Abb_lokal-BCinf_wenn_lokal-BC0} is wrong.
We define a class of real analytic maps that is suited to our need.
Before that, we state the following small result concerning complexifications.
\begin{lem}\label{lem:Komplexifizierung_gewichteter_Funktionenraeume}
	Let $\SX$ and $\SY$ be real normed spaces, $\UF \subseteq \SX$ an open nonempty set,
	$k \in \cl{\N}$ and $\GewFunk \subseteq \cl{\R}^{\UF}$.
	Further let $\iota : \SY \to \SY_\C$ denote the canonical inclusion into $\SY_\C$.
	\begin{enumerate}
		\item\label{enum1:Komplexifizierung_gewichteter_Funktionenraeume}
		Then $\CcF{\UF}{\SY_\C}{k}$ is the complexification of $\CcF{\UF}{\SY}{k}$,
		and the canonical inclusion map is given by
		\[
			\CcF{\UF}{\SY}{k} \to \CcF{\UF}{\SY_\C}{k} : \gamma \mapsto \iota \circ \gamma .
		\]

		\item\label{enum1:Komplexifizierung-Kriterium_fuer_Inklusion_CFo}
		Let $\VF \subseteq \SY$ be an open nonempty set and
		$\widetilde{\VF} \subseteq \SY_\C$ an open neighborhood of $\iota(\VF)$
		such that
		\begin{equation}\label{bed:Komplexifizierung-Kriterium_fuer_Inklusion_CFo}
			(\forall M \subseteq \VF) \, \dist{M}{\SY\setminus\VF} > 0
			\implies
			\dist{\iota(M)}{\SY_\C\setminus\widetilde{\VF}} > 0 .
		\end{equation}
		Then
		\[
			\iota \circ \CcFo{\UF}{\VF}{k} \subseteq \CcFo{\UF}{\widetilde{\VF}}{k} .
		\]
	\end{enumerate}
\end{lem}
\begin{proof}
	\refer{enum1:Komplexifizierung_gewichteter_Funktionenraeume}
	It is a well known fact that $\SY_\C \iso \SY \times \SY$ and $\iota(y) = (y,0)$ for each $y\in\SY$.
	Hence
	\[
		\CcF{\UF}{\SY_\C}{k} \iso \CcF{\UF}{\SY \times \SY}{k} \iso \CcF{\UF}{\SY}{k} \times \CcF{\UF}{\SY}{k}
	\]
	by \refer{lem:gewichtete,verschwindende_Abb_Produktisomorphie-lokalkonvex} (and \refer{prop:multilineare_Abb_und_CF}),
	and
	\[
		\iota \circ \gamma = (\gamma, 0) \in \CcF{\UF}{\SY}{k} \times \CcF{\UF}{\SY}{k} \iso \CcF{\UF}{\SY_\C}{k}
	\]
	for $\gamma \in \CcF{\UF}{\SY}{k}$.

	\refer{enum1:Komplexifizierung-Kriterium_fuer_Inklusion_CFo}
	This is an immediate consequence of \refer{enum1:Komplexifizierung_gewichteter_Funktionenraeume}
	and \refer{bed:Komplexifizierung-Kriterium_fuer_Inklusion_CFo}.
\end{proof}
\begin{defi}
	Let $\SY$ and $\SZ$ be real normed spaces, $\VF \subseteq \SY$ an open nonempty set,
	$\Phi: \VF \to \SZ$ a real analytic map.
	We say that $\Phi$ has a \emph{good complexification}%
	\index{complexification!good}\index{good complexification|see{complexification, good}}
	if there exists a complexification $\widetilde{\Phi} : \widetilde{\VF} \subseteq \SY_\C \to \SZ_\C$
	of $\Phi$ which satisfies \refer{bed:Abb_in_BC^0}
	and whose domain satisfies \refer{bed:Komplexifizierung-Kriterium_fuer_Inklusion_CFo}.
	In this case, we call $\widetilde{\Phi}$ a good complexification.
\end{defi}
The following lemma states that good complexifications always exist at least locally.
It is not needed in the further discussion.
\begin{lem}\label{lem:Verkleinerung_reell_analytischer_Abb_haben_gute_Komplexifizierung}
	Let $\SY$ and $\SZ$ be real normed spaces, $\VF \subseteq \SY$ an open nonempty set and
	$\Phi: \VF \to \SZ$ a real analytic map.
	Then for each $x \in \VF$ there exists an open neighborhood
	$\WF_x \subseteq \SY$ of $x$ such that $\rest{\Phi}{\WF_x}$
	has a good complexification.
\end{lem}
\begin{proof}
	Let $\widetilde{\Phi} : \widetilde{\VF} \subseteq \SY_\C \to \SZ_\C$ be a complexification of $\Phi$
	and $\iota : \VF \to \widetilde{\VF}$ the canonical inclusion. Then there exists an $r > 0$
	such that $\Ball[\SY_\C]{\iota(x)}{r} \subseteq \widetilde{\VF}$
	and $\widetilde{\Phi}$ is bounded on $\Ball[\SY_\C]{\iota(x)}{r}$.
	Then it is obvious that
	$\WF_x \ndef \iota^{-1}(\Ball[\SY_\C]{\iota(x)}{r}) = \Ball[\SY]{x}{r}$
	has the stated property.
\end{proof}

\subparagraph{Power series}
We present a class of analytic maps which have good complexifications:
Absolutely convergent power series in Banach algebras.
\begin{lem}\label{lem:Potenzreihen_in_Banachalgebren}
	Let $A$ be a Banach algebra and $\sum_{\ell=0}^{\infty}a_{\ell}z^{\ell}$
	a power series with $a_{\ell} \in \K$ and
	the radius of convergence $R \in ]0,\infty]$.
	We define for $x \in A$
	\[
		P_x
		: \Ball[A]{x}{R} \to A
		: y \mapsto \sum_{\ell=0}^{\infty}a_{\ell}(y-x)^{\ell}.
	\]
	Then the following assertions hold:
	\begin{enumerate}
		\item\label{enum1:Potenzreihen_in_Banachalgebren_a}
		The map $P_x$ is analytic.

		\item\label{enum1:Potenzreihen_in_Banachalgebren_b}
		If $\K = \C$ then $P_x$ satisfies \refer{bed:Abb_in_BC^0}.

		\item\label{enum1:Potenzreihen_in_Banachalgebren_c}
		If $\K = \R$ then $P_x$ has a good complexification.
	\end{enumerate}
	\index{complexification!of power series}%
\end{lem}
\begin{proof}
	The map $P_x$ is defined since
	$\sum_{\ell=0}^{\infty}a_{\ell}(y-x)^{\ell}$ is absolutely convergent on
	$\Ball{x}{R}$ and $A$ is complete.

	\refer{enum1:Potenzreihen_in_Banachalgebren_a}
	This is a special case of \cite[\S 3.2.9]{MR0219078}.

	\refer{enum1:Potenzreihen_in_Banachalgebren_b}
		If $\VF \subseteq \Ball[A]{x}{R}$ is open
		and $\dist{\VF}{\partial \Ball[A]{x}{R}} > 0$,
		there exists $r \in ]0, R[$ such that
		$\VF \subseteq \Ball[A]{x}{r}$ holds.
		So we compute for $y \in \VF$ that
		\[
			\left\norm{\sum_{\ell=0}^{\infty}a_{\ell} (y-x)^{\ell}\right}
			\leq \sum_{\ell=0}^{\infty}\abs{a_{\ell}}\,\norm{y-x}^{\ell}
			\leq \sum_{\ell=0}^{\infty}\abs{a_{\ell}} r^{\ell}
			< \infty.
		\]

		\refer{enum1:Potenzreihen_in_Banachalgebren_c}
		It is well known that the complexification of a Banach algebra is
		a Banach algebra as well, and a complexification of $P_x$ is given by
		\[
			\tilde{P}_x : \Ball[A_\C]{x}{R} \to A :
			y \mapsto \sum_{\ell=0}^{\infty}a_{\ell}(y-x)^{\ell}.
		\]
		This gives the assertion.
\end{proof}

\subsubsection{Main Result}
Finally, we state the desired result for complex analytic maps
and real analytic maps with good complexifications.
\begin{prop}\label{prop:Operation_analytischer_Abb_auf_CWk}
	Let $\SX$, $\SY$ and $\SZ$ be normed spaces, $\UF \subseteq \SX$ and $\VF \subseteq \SY$
	open nonempty sets such that $\VF$ is star-shaped with center~$0$,
	$k \in \cl{\N}$ and	$\GewFunk \subseteq \cl{\R}^{\UF}$ with $1_{\UF} \in \GewFunk$.
	Further, let $\Phi: \VF \to \SZ$ with $\Phi(0) = 0$ be either a
	complex analytic map that satisfies \refer{bed:Abb_in_BC^0}
	or a real analytic map that has a good complexification.
	Then for each $\gamma \in \CcFo{\UF}{\VF}{k}$
	\[
		\Phi \circ \gamma \in \CcF{\UF}{\SZ}{k},
	\]
	and the map
	\[
		\Hom{\Phi}{}{}:
		\CcFo{\UF}{\VF}{k}
		\to \CcF{\UF}{\SZ}{k}
		: \gamma \mapsto \Phi\circ\gamma
	\]
	is analytic.
	\index{superposition!with an analytic map}%
	\index{analytic maps!superposition|see{superposition with an analytic map}}
\end{prop}
\begin{proof}
	If $\Phi$ is complex analytic, this is an immediate consequence of
	\refer{lem:Wirkung_einer_lokal_beschraenkten_Abb_auf_CWo}
	and \refer{lem:analyt-Abb_lokal-BCinf_wenn_lokal-BC0}.

	If $\Phi$ is real analytic, by our assumptions there exists a good complexification
	$\widetilde{\Phi} : \widetilde{\VF} \subseteq \SY_\C \to \SZ$.
	We know from the first part that $\widetilde{\Phi}$
	induces a complex analytic map
	\[
		\widetilde{\Phi}_\ast
		:
		\CcFo{\UF}{\widetilde{\VF}}{k}
		\to \CcF{\UF}{\SZ_\C}{k}
		: \gamma \mapsto \widetilde{\Phi} \circ \gamma .
	\]
	Since $\CcFo{\UF}{\VF}{k} \subseteq \CcFo{\UF}{\widetilde{\VF}}{k}$
	by \refer{lem:Komplexifizierung_gewichteter_Funktionenraeume}
	and $\Phi_\ast$ coincides with the restriction of $\widetilde{\Phi}_\ast$
	to $\CcFo{\UF}{\VF}{k}$, it follows that $\Phi_\ast$ is real analytic.	
\end{proof}
\subsubsection{Quasi-inversion algebras of weighted functions}
As an application, we see that for a set $\GewFunk$ of weights with
$1_\UF \in \GewFunk$ and a Banach algebra $A$, the space $\CcF{\UF}{A}{k}$
can be turned into a topological algebra with continuous quasi-inversion.
Details on algebras with quasi-inversion can be found in \refer{app:Quasi-Inversion}.
\begin{lem}\label{lem:CkF(,A)_topologische_Algebra}
	Let $A$ be a Banach algebra, $\SX$ a normed space,
	$\UF \subseteq \SX$ an open nonempty subset,
	$k \in \cl{\N}$ and $\GewFunk \subseteq \cl{\R}^{\UF}$
	with $1_{\UF} \in \GewFunk$.
	Then the locally convex space $\CcF{\UF}{A}{k}$ endowed
	with the multiplication described in \refer{cor:Multiplikation_auf_CF()(A)}
	becomes a complete topological algebra with continuous quasi-inversion
	in the sense of \refer{def:topologische_Quasi-Inversions-Algebren}.
	For each $\gamma \in \QInvertiblesOf{\CcF{\UF}{A}{k}}$
	\[
		\QuasiInv_{\CcF{\UF}{A}{k}}(\gamma) = \QuasiInv_{A}\circ\gamma ,
	\]
	and
	\[
		\CcFo{\UF}{\Ball[A]{0}{1}}{k}
		= \{\gamma\in\CcF{\UF}{A}{k} : \hn{\gamma}{1_{\UF}}{0} < 1\}
		\subseteq \QInvertiblesOf{\CcF{\UF}{A}{k}} .
	\]
\end{lem}
\begin{proof}
	The relation
	$\QuasiInv_{\CcF{\UF}{A}{k}}(\gamma) = \QuasiInv_{A}\circ\gamma$
	is an immediate consequence of the definition of the multiplication,
	so it only remains to show that $\QInvertiblesOf{\CcF{\UF}{A}{k}}$ is open
	and $\QuasiInv_{\CcF{\UF}{A}{k}}$ is continuous.
	We proved in \refer{lem:Kriterium_fuer_Stetigkeit_der_Adversion}
	that it suffices to find a neighborhood of $0$ that consists
	of quasi-invertible elements such that the restriction of
	$\QuasiInv_{\CcF{\UF}{A}{k}}$ to it is continuous.
	We show that $\CcFo{\UF}{\Ball[A]{0}{1}}{k}$ is such a neighborhood.
	The map
	\[
		G: \Ball{0}{1} \to A : x\mapsto \sum_{i=1}^\infty x^i
	\]
	is given by a power series and maps $0$ to $0$, hence we know from
	\refer{lem:Potenzreihen_in_Banachalgebren} and \refer{prop:Operation_analytischer_Abb_auf_CWk}
	that the map
	\[
		\CcFo{\UF}{\Ball[A]{0}{1}}{k} \to \CcF{\UF}{A}{k} : \gamma \mapsto G \circ \gamma
	\]
	is defined and analytic. Since
	\[
		G \circ \gamma = \sum_{i = 1}^{\infty}\gamma^{i}
	\]
	for each $\gamma \in \CcFo{\UF}{\Ball[A]{0}{1}}{k}$,
	we conclude from
	\refer{lem:Topologisches_Kriterium_fuer_Quasi-Invertierbarkeit}
	that $\gamma$ is quasi-invertible with
	\[
		\QuasiInv_{\CcF{\UF}{A}{k}}(\gamma) = -G \circ \gamma,
	\]
	so the proof is complete.
\end{proof}

\begin{beisp}
	Let $\SY$ be a Banach space, $\UF \subseteq \SX$ an open nonempty subset,
	$k \in \cl{\N}$ and $\GewFunk \subseteq \cl{\R}^{\UF}$
	with $1_{\UF} \in \GewFunk$.
	Then the locally convex space $\CcF{\UF}{\Lin{\SY}{\SY}}{k}$ endowed
	with the multiplication described in \refer{cor:Komposition_linearer_Abb_und_CF}
	becomes a complete topological algebra with continuous quasi-inversion.
\end{beisp}
\section{Weighted maps into locally convex spaces}
\label{sec:GewAbb_Bildbereich-lokalkonvex}
We define and examine weighted functions with values in arbitrary locally convex spaces.
In order to do this, we use tools and definitions that are provided in
\ref{sususec:Lipschitz-stetige_Abb}.
The material of this section is only needed for latter discussions of
weighted mapping groups with values in arbitrary locally convex Lie groups
in \refer{sec: Weighted_maps_into_locally_convex_Lie_groups};
readers primarily interested in diffeomorphism groups may want to skip this section.
\subsection{Definition and topological structure}
The definition of weighted function with values in locally convex spaces
relies on the one with values in normed spaces.
\begin{defi}
	Let $\SX$ be a normed space, $\UF \subseteq \SX$ an open nonempty set,
	$\SY$ a locally convex space, $k\in \cl{\N}$ and $\cW \subseteq \cl{\R}^\UF$ nonempty. We define
	\[
		\glstext{Raeume_gewichteter_Abbildungen}
		\ndef
		\set{ \gamma \in \ConDiff{\UF}{\SY}{k} }%
		{
			(\forall p \in \normsOn{\SY})\, \HomQuot{\gamma}{p} \in \CF{\UF}{\SY_p}{\cW}{k}
		},
	\]
	\index{weighted maps!into locally convex spaces}%
	using notation as in \refer{def:lokalkonvexe_Faktorraeume}.
	For $p \in \normsOn{\SY}$, $f \in \cW$ and $\ell \in \N$ with $\ell \leq k$,
	\[
		\glsdisp{gewichtete_f,k,p-Halbnorm}{\hn{\cdot}{p,f}{\ell}} : \CcF{\UF}{\SY}{k} \to \R : \gamma \mapsto \hn{\HomQuot{\gamma}{p}}{f}{\ell}
	\]
	is a seminorm on $\CF{\UF}{\SY}{\cW}{k}$.
	We endow $\CF{\UF}{\SY}{\cW}{k}$ with the locally convex vector space topology that is generated by these seminorms.
\end{defi}
We show that the structure of $\CcF{\UF}{\SY}{k}$ is already determined by
$\set{\hn{\cdot}{p,f}{\ell} }{p \in \mathcal{P}, f \in \cW, \ell \in \N \text{ with } \ell \leq k}$,
where $\mathcal{P}$ is just a generator of $\normsOn{\SY}$.
This can be useful in some cases.
\begin{lem}\label{lem:lokalkonvexe_Abbildungsraeume_Erzeugendensystem_der_HN_reicht_aus}
\label{lem:gew_Abb_in_Lokalkovexe_ist_Produkt_von_gew_Abb_in_normierte}
	Let $\SX$ be a normed space, $\UF \subseteq \SX$ an open nonempty set,
	$\SY$ a locally convex space, $k \in \cl{\N}$, $\GewFunk \subseteq \cl{\R}^\UF$ nonempty
	and $\mathcal{P} \subseteq \normsOn{\SY}$ a set that generates $\normsOn{\SY}$.
	Then for $\gamma \in \ConDiff{\UF}{\SY}{k}$
	\[
		\gamma \in \CcF{\UF}{\SY}{k}
		\iff
		(\forall p \in \mathcal{P}) \, \HomQuot{\gamma}{p} \in \CcF{\UF}{\SY_p}{k},
	\]
	and the map
	\[
		\CcF{\UF}{\SY}{k}\to
		\prod_{p \in \mathcal{P}}
		\CcF{\UF}{\SY_p}{k}
		:\gamma \mapsto (\HomQuot{\gamma}{p})_{p \in \mathcal{P}}
		\tag{\ensuremath{\dagger}}
		\label{Einbettung_lokalkonvexe_gewichtete_Raeume_Erzeuger}
	\]
	is a topological embedding.
\end{lem}
\begin{proof}
	Let $q \in \normsOn{\SY}$. Then there exist $p_1, \dotsc, p_n \in \mathcal{P}$ and $C > 0$
	such that
	\[
		q \leq C \cdot \max_{i=1,\dotsc,n} p_i .
	\]
	Further we know that for each $\ell \in \N$ with $\ell \leq k$ and $x \in \UF$, $h_1,\dotsc,h_\ell \in \SX$
	\[
		\dA[\ell]{(\HomQuot{\gamma}{q})}{x}{h_1,\dotsc, h_{\ell}}
		= (\HomQuot{\dA[\ell]{\gamma}{}{}}{q})(x, h_1,\dotsc, h_{\ell}),
	\]
	so for $y \in \UF$ we get
	\begin{align*}
		&\norm{\dA[\ell]{ ( \HomQuot{\gamma}{q} ) }{x}{h_1,\dotsc, h_{\ell}}
		- \dA[\ell]{ ( \HomQuot{\gamma}{q} ) }{y}{h_1,\dotsc, h_{\ell}}}_q
		\\
		\leq& \norm{\dA[\ell]{\gamma}{x}{h_1,\dotsc, h_{\ell}} - \dA[\ell]{\gamma}{y}{h_1,\dotsc, h_{\ell}}}_q
		\\
		\leq&\, C \cdot \max_{i=1,\dotsc,n} \norm{\dA[\ell]{\gamma}{x}{h_1,\dotsc, h_{\ell}} - \dA[\ell]{\gamma}{y}{h_1,\dotsc, h_{\ell}}}_{p_{i}}.
	\end{align*}
	Since we assumed that $\HomQuot{\gamma}{p_{i}} \in \FC{\UF}{\SY_{p_{i}}}{k}$,
	from this estimate we conclude with \refer{prop:Char_Frechet_Diff}
	that $\HomQuot{\gamma}{q} \in \FC{\UF}{\SY_q}{k}$ with
	\[
		\norm{D^{(\ell)} (\HomQuot{\gamma}{q})(x)}_{op}
		\leq C \cdot \max_{i=1,\dotsc,n}\norm{D^{(\ell)} (\HomQuot{\gamma}{p_i})(x)}_{op}
	\]
	for all $\ell \in \N$ with $\ell \leq k$ and $x \in \UF$.
	This implies that
	\[
		\hn{\gamma}{q, f}{\ell} \leq C \cdot \max_{i=1,\dotsc,n} \hn{\gamma}{p_i, f}{\ell}
	\]
	for each $f \in \GewFunk$ and $\ell \in \N$ with $\ell \leq k$. Hence
	\[
		\HomQuot{\gamma}{q} \in \CcF{\UF}{\SY_q}{k},
	\]
	and $\hn{\cdot}{q, f}{\ell}$ is continuous with respect to the initial topology
	induced by \eqref{Einbettung_lokalkonvexe_gewichtete_Raeume_Erzeuger}.
	Since $q$ was arbitrary, the proof is complete.
\end{proof}
\paragraph{An integrability criterion}
We generalize the assertion of \refer{lem:Kriterium_Integrierbarkeit_in_CW}.
\begin{lem}
\label{lem:Kriterium_Integrierbarkeit_in_CW_lkvx}
	Let $\SX$  be a normed space, $\UF \subseteq \SX$ a nonempty open set, $\SY$ a locally convex space,
	 $k\in\cl{\N}$, $\GewFunk \subseteq \cl{\R}^{\UF}$ such that for each compact set $K \subseteq \UF$,
	there exists an $f_{K} \in \GewFunk$ with
	$
		\inf_{x \in K}\abs{f_{K}(x)} > 0
	$.
	Further, let $\Gamma : [a,b] \to \CcF{\UF}{\SY}{k}$ a continuous curve
	and $R \in \CcF{\UF}{\SY}{k}$.
	Assume that
	\begin{equation*}\label{id:punktweises_Integral_CW_lkvx}\tag{$\ast$}
		\Rint{a}{b}{\evTwo_x(\Gamma(s))}{s}
		= \evTwo_x(R)
	\end{equation*}
	holds for all $x \in \UF$. Then $\Gamma$ is weakly integrable with
	\[
		\Rint{a}{b}{\Gamma(s)}{s} = R.
	\]
\end{lem}
\begin{proof}
	We derive from \refer{lem:lokalkonvexe_Abbildungsraeume_Erzeugendensystem_der_HN_reicht_aus} that the dual space of $\CcF{\UF}{\SY}{k}$
	coincides with the set of functionals
	$
		\set{\lambda \circ {\morQuot{p}}_\ast}{p \in \normsOn{\SY}, \lambda \in \CcF{\UF}{\SY_p}{k}'}.
	$
	Hence $\Gamma$ is weakly integrable with the integral $R$ iff
	\[
		\Rint{a}{b}{\lambda (\morQuot{p} \circ \Gamma)(s)}{s} = \lambda (\morQuot{p} \circ R)
	\]
	holds for all $p \in \normsOn{\SY}$ and $\lambda \in \CcF{\UF}{\SY_p}{k}'$;
	this is clearly equivalent to the weak integrability of $\morQuot{p} \circ \Gamma$ with integral $\morQuot{p} \circ R$ for all $p \in \normsOn{\SY}$.
	But we derive this assertion from \refer{id:punktweises_Integral_CW_lkvx}
	and \refer{lem:Kriterium_Integrierbarkeit_in_CW}.
\end{proof}

\subsubsection{Reduction to lower order}
We prove a generalization of \refer{prop:topologische_Zerlegung_von_CFk}.
To this end, we need a locally convex topology on $\Lin{\SX}{\SY}$,
where $\SX$ is a normed and $\SY$ a locally convex space.
We define such a topology and show that it arises as the intial topology
with respect to the embedding
$\Lin{\SX}{\SY} \to \prod_{p\in\normsOn{\SY}}\Lin{\SX}{\SY_p}$.
\paragraph{Topology on linear operators}

\begin{defi}[Topology on linear operators]
	Let $\SX$ be a normed space and $\SY$ a locally convex space. For each $p \in \normsOn{\SY}$
	and $T \in \Lin{\SX}{\SY}$, we set
	\[
		\glstext{Operatornorm_Bild_Lokalkonvex} \ndef \sup_{x \neq 0} \frac{ \norm{T x}_p }{ \norm{x} }
		= \Opnorm{\HomQuot{T}{p}}.
	\]
	This obviously defines a seminorm on $\Lin{\SX}{\SY}$, and henceforth we endow
	$\Lin{\SX}{\SY}$ with the locally convex topology that is generated by these seminorms.
	Further we define
	$
		\Lin{\SX}{\SY}_{op, p} \ndef \Lin{\SX}{\SY}_{\norm{\cdot}_{op, p} }
	$.
\end{defi}

\begin{lem}\label{lem:Charakterisierung_Topologie_der_linearen_Operatoren}
	Let $\SX$ be a normed space, $\SY$ a locally convex space and $p \in \normsOn{\SY}$.
	Then the map induced by
	\[
		(\morQuot{p})_\ast : \Lin{\SX}{\SY} \to \Lin{\SX}{\SY_p} : T \mapsto \HomQuot{T}{p}
	\]
	that makes
	\[
		\xymatrix{
		{(\Lin{\SX}{\SY}, \norm{\cdot}_{op, p})} \ar[rr]^{(\morQuot{p})_\ast} \ar@{->>}[dr]^{\morQuot{op, p}} & & {\Lin{\SX}{\SY_p}}\\
		&  {\Lin{\SX}{\SY}_{op, p}} \ar@{-->}[ur]
		}
	\]
	a commutative diagram is an isometric isomorphism onto the image of $(\morQuot{p})_\ast$.
	The map
	\[
		\Lin{\SX}{\SY} \to \prod_{p\in\normsOn{\SY}}\Lin{\SX}{\SY_p}
		: T \mapsto (\HomQuot{T}{p})_{p \in \normsOn{\SY}}
	\]
	is a topological embedding.
\end{lem}
\begin{proof}
	Since
	$
		\norm{T}_{op, p} = \Opnorm{\HomQuot{T}{p}}
	$
	for each $T \in \Lin{\SX}{\SY}$, the induced map is an isometry.
	By the definition of the topology of $\Lin{\SX}{\SY}$,
	\[
		\Lin{\SX}{\SY} \to \prod_{p\in\normsOn{\SY}}\Lin{\SX}{\SY}_{op, p}
		: T \mapsto (\HomQuot{T}{op, p})_{p \in \normsOn{\SY}}
	\]
	is an embedding, so by the transitivity of initial topologies, the proof is finished.
\end{proof}
\paragraph{Weighted maps into spaces of linear operators and the main result}
Before we can prove the main result, we have to take a look at the structure
of $\CcF{\UF}{\Lin{\SX}{\SY}}{k}$.
\begin{lem}\label{lem:Charakterisierung_gewichtete_Abb_mit_Werten_in_Linearen_Operatoren}
	Let $\SX$ be a normed space, $\SY$ a locally convex space, $\UF \subseteq \SX$ an open nonempty subset and $k\in\cl{\N}$.
	Then for $\Gamma \in \ConDiff{\UF}{\Lin{\SX}{\SY}}{k}$, nonempty $\GewFunk \subseteq \cl{\R}^\UF$ and $k \in \cl{\N}$ the equivalence
	\[
		\Gamma \in \CcF{\UF}{\Lin{\SX}{\SY}}{k}
		\iff
		(\forall p \in \normsOn{\SY})\,
		(\morQuot{p})_\ast \circ \Gamma \in \CcF{\UF}{\Lin{\SX}{\SY_p}}{k}
	\]
	holds. More precisely, for $\ell \in \N$ with $\ell \leq k$, $f \in  \cl{\R}^\UF$
	and $p \in \normsOn{\SY}$ we have
	\begin{equation}
		\label{id:Gleichheit_von_Gewichts-Halbnormen_in_Lineare_Operatoren}
		\hn{\Gamma}{\norm{\cdot}_{op, p}, f}{\ell} = \hn{(\morQuot{p})_\ast \circ \Gamma}{f}{\ell} .
	\end{equation}
	This induces that the map
	\[
		\CcF{\UF}{\Lin{\SX}{\SY}}{k}\to
		\prod_{p \in \normsOn{\SY}}
		\CcF{\UF}{\Lin{\SX}{\SY_p}}{k}
		: \Gamma \mapsto ((\morQuot{p})_\ast \circ \Gamma)_{p \in \normsOn{\SY}}
	\]
	is a topological embedding.
\end{lem}
\begin{proof}
	Note first that $\morQuot{op, p} \circ \Gamma$ is $\FC{}{}{k}$
	iff $(\morQuot{p})_\ast \circ \Gamma$ is $\FC{}{}{k}$ as a consequence of
	\refer{lem:Charakterisierung_Topologie_der_linearen_Operatoren} and \refer{prop:Char_Frechet_Diff}.
	Using \refer{lem:Charakterisierung_Topologie_der_linearen_Operatoren}
	it is easy to see that \refer{id:Gleichheit_von_Gewichts-Halbnormen_in_Lineare_Operatoren}
	is satisfied.
	This implies that for each $p\in\normsOn{\SY}$ the equivalence
	\[
		(\morQuot{p})_\ast \circ \Gamma \in \CcF{\UF}{\Lin{\SX}{\SY_p}}{k}
		\iff \morQuot{op, p} \circ \Gamma \in \CcF{\UF}{\Lin{\SX}{\SY}_{op, p}}{k}
	\]
	holds and that the isometry whose existence was stated in
	\refer{lem:Charakterisierung_Topologie_der_linearen_Operatoren}
	induces an embedding
	\[
		\CcF{\UF}{\Lin{\SX}{\SY}_{op, p}}{k} \to \CcF{\UF}{\Lin{\SX}{\SY_p}}{k}.
	\]
	Further we proved in \refer{lem:lokalkonvexe_Abbildungsraeume_Erzeugendensystem_der_HN_reicht_aus}
	that
	\[
		\CcF{\UF}{\Lin{\SX}{\SY}}{k} \to \prod_{p \in \normsOn{\SY}} \CcF{\UF}{\Lin{\SX}{\SY}_{op, p}}{k}
		: \Gamma \mapsto ((\morQuot{op, p})_\ast \circ \Gamma)_{p \in \mathcal{P}}
	\]
	is an embedding, so we are home.
\end{proof}

\begin{prop}[Reduction to lower order]\label{prop:topologische_Zerlegung_von_CFk_lokalkonvex}
	Let $\SX$ be a normed space, $\SY$ a locally convex space,
	$\UF \subseteq \SX$ an open nonempty set, $\GewFunk \subseteq \cl{\R}^\UF$ nonempty and $k \in \N$.
	Let $\gamma \in \ConDiff{\UF}{\SY}{1}$. Then
	\[
		\gamma \in \CcF{\UF}{\SY}{k+1}
		\iff
		(D\gamma , \gamma) \in
		\CcF{\UF}{\Lin{\SX}{\SY}}{k} \times
			\CcF{\UF}{\SY}{0}.
	\]
	Furthermore, the map
	\[
		\CcF{\UF}{\SY}{k+1}\to
		\CcF{\UF}{\Lin{\SX}{\SY}}{k} \times \CcF{\UF}{\SY}{0}
		:\gamma \mapsto (D\gamma , \gamma)
	\]
	is a topological embedding.
\end{prop}
\begin{proof}
	The definition of $\CcF{\UF}{\SY}{k+1}$,
	\refer{prop:topologische_Zerlegung_von_CFk}
	and
	\refer{lem:Charakterisierung_gewichtete_Abb_mit_Werten_in_Linearen_Operatoren}
	give the equivalences
	\begin{align*}
		\gamma \in \CcF{\UF}{\SY}{k+1}
		&\iff
		(\forall p \in \normsOn{\SY})\, \HomQuot{\gamma}{p} \in \CcF{\UF}{\SY_p}{k+1}
		\\
		&\iff
		(\forall p \in \normsOn{\SY})\,
		(D(\HomQuot{\gamma}{p}) , \HomQuot{\gamma}{p}) \in
		\CcF{\UF}{\Lin{\SX}{\SY_p}}{k} \times
			\CcF{\UF}{\SY_p}{0}
		\\
		&\iff
		(D\gamma, \gamma) \in \CcF{\UF}{\Lin{\SX}{\SY}}{k} \times\CcF{\UF}{\SY}{0}.
	\end{align*}
	Furthermore, we have the commutative diagram
	\[
		\xymatrix{
		{\CcF{\UF}{\SY}{k+1}} \ar[rr] \ar@{>->}[d] & & {\CcF{\UF}{\Lin{\SX}{\SY}}{k} \times \CcF{\UF}{\SY}{0}} \ar@{>->}[d]\\
		{\prod_{p \in \normsOn{\SY}} \CcF{\UF}{\SY_p}{k + 1}} \ar@{>->}[rr]
				& &  {\prod_{p \in \normsOn{\SY}} \CcF{\UF}{\Lin{\SX}{\SY_p}}{k} \times \CcF{\UF}{\SY_p}{0}}
		}
	\]
	and since the maps represented by the three lower arrows are embeddings,
	so is the map at the top.
\end{proof}
\subsection{Weighted decreasing maps}
We give another definition for weighted maps that decay at infinity.
Here, the domain of the maps is contained in a finite dimensional vector space.
\begin{defi}\label{defi:Definition_CFvanK}
Let $\SY$ be a normed space, $\UF$ an open nonempty subset of the \emph{finite-dimensional} space $\SX$
and $\GewFunk \subseteq \cl{\R}^\UF$ nonempty. We define for $k\in \cl{\N}$
\begin{multline*}
	\glstext{Raeume_gwichteter_fallender_Abbildungen-kompakt}
	\ndef
	\{
		\gamma \in \CcF{\UF}{\SY}{k} : (\forall f\in\cW, \ell \in \N, \ell \leq k)\,
		\\(\forall \eps > 0)(\exists K \subseteq \UF \text{ compact} )
		\hn{\rest{\gamma}{\UF\setminus K}}{f}{\ell} < \eps
	\}.
\end{multline*}
\index{weighted maps!decreasing}%
For a locally convex space $\SY$ we set
\[
	\CcFvanK{\UF}{\SY}{k}
	\ndef
	\{
		\gamma \in \CcF{\UF}{\SY}{k} :
		(\forall p \in \normsOn{\SY} )\, \HomQuot{\gamma}{p} \in \CcFvanK{\UF}{\SY_p}{k}
	\}.
\]
For a subset $\VF \sub \SY$, we define
\[
	\glstext{Raeume_gwichteter_fallender_Abbildungen-kompakt_Werte_Teilmenge}
	\ndef
	\set{\gamma \in \CcFvanK{\UF}{\SY}{k}}{ \gamma(\UF) \sub \VF}
\]
\end{defi}
As in \refer{lem:CWvan_closed_CW}, we can prove
that $\CcFvanK{\UF}{\SY}{k}$ is closed in $\CcF{\UF}{\SY}{k}$.
\begin{lem}\label{lem:CWvan_(endlichdimensional)_abgeschlossen}
	Let $\SY$ be a locally convex space, $\UF$ an open nonempty subset of the finite-dimensional space $\SX$,
	$\GewFunk \subseteq \cl{\R}^\UF$ nonempty and $k\in \cl{\N}$.
	Then
	$\CcFvanK{\UF}{\SY}{k}$ is a closed vector subspace of $\CcF{\UF}{\SY}{k}$.
\end{lem}
\begin{proof}
	It is obvious from the definition of $\CcFvanK{\UF}{\SY}{k}$ that it is a
	vector subspace. It remains to show that it is closed.
	To this end, let $(\gamma_i)_{i\in I}$ be a net in $\CcFvanK{\UF}{\SY}{k}$ that converges
	to $\gamma \in \CcF{\UF}{\SY}{k}$ in the topology of $\CcF{\UF}{\SY}{k}$.
	Let $p \in \normsOn{\SY}$, $f \in \GewFunk$, $\ell \in \N$ with $\ell \leq k$ and $\eps > 0$.
	Then there exists an $i_\eps \in I$ such that
	\[
		i \geq i_\eps \implies \hn{\gamma - \gamma_i}{p,f}{\ell} < \frac{\eps}{2}.
	\]
	Further there exists a compact set $K$ such that
	\[
		\hn{\rest{\gamma_{i_\eps}}{\UF\setminus K}}{p,f}{\ell} < \frac{\eps}{2} .
	\]
	Hence
	\[
		\hn{\rest{\gamma}{\UF\setminus K}}{p,f}{\ell}
		\leq
		\hn{\rest{\gamma}{\UF\setminus K} - \rest{\gamma_{i_\eps}}{\UF\setminus K} }{p,f}{\ell}
		+ \hn{\rest{\gamma_{i_\eps}}{\UF\setminus K}}{p,f}{\ell}
		< \eps,
	\]
	so $\gamma \in \CcFvanK{\UF}{\SY}{k}$.
\end{proof}
Further, we prove the following convexity criterion.
\begin{lem}\label{lem:Konvexe_Mengen_im_gewichteten_Funktionenraum-lokalkonvex_van}
	Let $\SX$ be a finite-dimensional space, $\UF \subseteq \SX$ an open nonempty subset,
	$\SY$ a locally convex space,
	$\GewFunk \subseteq \cl{\R}^\UF$ with $1_\UF \in \GewFunk$, $\ell \in \cl{\N}$
	and $\VF \subseteq \SY$ convex.
	Then the set $\CcFvanK{\UF}{\VF }{\ell}$ is convex.
\end{lem}
\begin{proof}
	It is obvious that $\CcF{\UF}{\VF }{\ell}$ -- whose definition is straightforward --
	is convex since $\VF$ is so.
	But then
	\[
		\CcFvanK{\UF}{\VF }{\ell} = \CcF{\UF}{\VF }{\ell} \cap \CcFvanK{\UF}{\SY }{\ell}
	\]
	is convex as intersection of convex sets.
\end{proof}
As in \refer{cor:topologische_Zerlegung_von_CFkvan},
we prove a reduction to lower order for $\CcFvanK{\UF}{\SY}{k+1}$.
\begin{prop}\label{prop:topologische_Zerlegung_von_CFkvan-kompakte_Version}
	Let $\SX$ be a finite-dimensional space, $\SY$ a locally convex space,
	$\UF \subseteq \SX$ an open nonempty set, $\GewFunk \subseteq \cl{\R}^\UF$ nonempty, $k \in \N$
	and $\gamma \in \ConDiff{\UF}{\SY}{1}$. Then
	\[
		\gamma \in \CcFvanK{\UF}{\SY}{k+1}
		\iff
		(\FAbl{\gamma} , \gamma)\in
		\CcFvanK{\UF}{\Lin{\SX}{\SY}}{k} \times
			\CcFvanK{\UF}{\SY}{0},
	\]
	and the map
	\[
		\CcFvanK{\UF}{\SY}{k+1} \to \CcFvanK{\UF}{\Lin{\SX}{\SY}}{k} \times
			\CcFvanK{\UF}{\SY}{0}
		: \gamma \mapsto (\FAbl{\gamma} , \gamma)
	\]
	is a topological embedding.
\end{prop}
\begin{proof}
	It is a consequence of
	\refer{id:hn(f,k)-Norm_Ableitungen_und_Ableitungen_der_Ableitung}
	in \refer{lem:Normbeziehung_zwischen_Ableitungen_und_Ableitungen_der_Ableitung}
	that for each $p \in \normsOn{\SY}$
	\[
		\HomQuot{\gamma}{p} \in \CcFvanK{\UF}{\SY_p}{k+1}
		\iff
		(D (\HomQuot{\gamma}{p}) , \HomQuot{\gamma}{p})\in
		\CcFvanK{\UF}{\Lin{\SX}{\SY_p}}{k} \times
			\CcFvanK{\UF}{\SY_p}{0}.
	\]
	Further it is a consequence of \refer{id:Gleichheit_von_Gewichts-Halbnormen_in_Lineare_Operatoren}
	in \refer{lem:Charakterisierung_gewichtete_Abb_mit_Werten_in_Linearen_Operatoren}
	that
	\[
		D \gamma \in \CcFvanK{\UF}{\Lin{\SX}{\SY}}{k}
		\iff (\forall p \in \normsOn{\SY}) D (\HomQuot{\gamma}{p}) \in \CcFvanK{\UF}{\Lin{\SX}{\SY_p}}{k},
	\]
	so the equivalence is proved.
	The assertion on the embedding is a consequence of
	\refer{prop:topologische_Zerlegung_von_CFk_lokalkonvex} and
	\refer{lem:CWvan_(endlichdimensional)_abgeschlossen}.
	So the proof is finished.
\end{proof}
\subsection{Composition and Superposition}
As in \refer{sec:Composition_superposition_normedspaces},
we examine which kind of maps induce superposition operators on $\CcF{\UF}{\SY}{k}$ or $\CcFvanK{\UF}{\SY}{k}$.
We show that continuous multilinear maps induce superposition operators on both
function spaces. For $\CcFvanK{\UF}{\SY}{k}$, we can prove a much stronger result:
A smooth function mapping $0$ on $0$ induces a superposition operator between these spaces.
\subsubsection{Composition with a multilinear map}
The following definition and lemma are mostly the same as in \refer{susec:Multilineare_Abb},
but here $\SZ$ denotes a locally convex space.
\begin{defi}
	Let $\SX$ be a normed space, $\SY_1,\dotsc,\SY_m$ and $\SZ$ locally convex spaces and
	$b:\SY_1 \times\dotsb\times \SY_m \to \SZ$
	a continuous $m$-linear map.
	For each $i \in \sset{1,\dotsc,m}$, we define the $m$-linear continuous map
	\begin{align*}
		b^{(i)}:&
		\SY_1 \times \dotsb \times
		\SY_{i-1} \times \Lin{\SX}{\SY_i}\times \SY_{i+1}
		\times \dotsb\times \SY_m
		\to \Lin{\SX}{\SZ}\\
		:&(y_1,\dotsc,y_{i-1},T,y_{i+1},\dotsc,y_m)\mapsto
		(h\mapsto b(y_1,\dotsc,y_{i-1},T\eval h, y_{i+1},\dotsc,y_m)).
	\end{align*}
\end{defi}

\begin{lem}\label{lem:Ableitung_Komposition_mit_multilinearer_Abb-lokalkonvex}
	Let	$\SY_1,\dotsc,\SY_m$ and $\SZ$ be locally convex spaces,
	$\UF$ be an open nonempty subset of the normed space $\SX$
	and $k\in\cl{\N}$.
	Further let
	$b:\SY_1 \times\dotsb\times \SY_m \to \SZ$
	be a continuous $m$-linear map and
	$\gamma_1 \in \ConDiff{\UF}{\SY_1}{k},
	\dotsc, \gamma_m \in \ConDiff{\UF}{\SY_m}{k}$.
	Then
	\[
		b\circ(\gamma_1,\dotsc,\gamma_m) \in \ConDiff{\UF}{\SZ}{k}
	\]
	with
	\begin{equation}
		\label{id:Ableitung_Komposition_mit_multilinearen_Abb-lokalkonvex}
		\FAbl{(b\circ(\gamma_1,\ldots,\gamma_m))}
		= \sum_{i=1}^m
			b^{(i)}\circ (\gamma_1, \dotsc,\gamma_{i-1},\FAbl{\gamma_i},
				\gamma_{i+1}, \dotsc,\gamma_m).
	\end{equation}
\end{lem}
\begin{proof}
	To calculate the derivative of $b\circ(\gamma_1,\dotsc,\gamma_m)$,
	we apply the chain rule and get
	\begin{align*}
		\dA{(b\circ(\gamma_1,\dotsc,\gamma_m))}{x}{h}
		&= \sum_{i=1}^m
			b(\gamma_1(x),\dotsc,\gamma_{i-1}(x),\dA{\gamma_i}{x}{h},
				\gamma_{i+1}(x),\dotsc,\gamma_m(x))\\
		&= \sum_{i=1}^m
			b^{(i)} (\gamma_1(x),\dotsc,\gamma_{i-1}(x),\FAbl{\gamma_i}(x),
				\gamma_{i+1}(x),\dotsc,\gamma_m(x))
			\eval h .
	\end{align*}
	This implies \eqref{id:Ableitung_Komposition_mit_multilinearen_Abb-lokalkonvex}.
\end{proof}
Now we can prove the results about the multilinear superposition.
\begin{prop}\label{prop:multilineare_Abb_und_CF-lokalkonvex}\label{prop:lineare_Wirkung_auf_CFk}
	Let $\UF$ be an open nonempty subset of the normed space $\SX$.
	Let $\SY_1,\dotsc,\SY_m$ be locally convex spaces, $k \in \cl{\N}$ and
	$\GewFunk, \cW_1,\dotsc,\cW_m \subseteq \cl{\R}^{\UF}$ nonempty sets
	such that
	\[
		(\forall f \in \GewFunk) (\exists g_{f,1} \in \cW_1,\dotsc, g_{f,m} \in\cW_m)\,
		\abs{f} \leq \abs{g_{f,1}}\dotsm \abs{g_{f, m}}.
	\]
	Further let $\SZ$ be another locally convex space and
	$b:\SY_1 \times\dotsb\times \SY_m \to \SZ$
	a continuous $m$-linear map.
	Then
	\[
		b\circ(\gamma_1,\dotsc,\gamma_m) \in \CF{\UF}{\SZ}{\GewFunk}{k}
	\]
	for all $\gamma_1 \in \CF{\UF}{\SY_1}{\cW_1}{k},
	\dotsc, \gamma_m \in \CF{\UF}{\SY_m}{\cW_m}{k}$.
	The map
	\[
		b_*
		:
		\CF{\UF}{\SY_1}{\cW_1}{k} \times \dotsb \times
		\CF{\UF}{\SY_m}{\cW_m}{k}
		\to\CF{\UF}{\SZ}{\GewFunk}{k}
		:
		(\gamma_1,\dotsc,\gamma_m)\mapsto b\circ(\gamma_1,\dotsc,\gamma_m)
	\]
	is $m$-linear and continuous.
	\index{superposition!with a multilinear map}
\end{prop}
\begin{proof}
	Let $p$ be a continuous seminorm on $\SZ$. Then there exist
	$q_1 \in \normsOn{\SY_1}, \dotsc, q_m \in \normsOn{\SY_m}$ such that
	for all $y_1 \in \SY_1$, \ldots, $y_m \in \SY_m$,
	\[
		\norm{b(y_1,\dotsc,y_m)}_p \leq \norm{y_1}_{q_1}\dotsm \norm{y_m}_{q_m}.
	\]
	Hence there exists an $m$-linear map $\widetilde{b}$ that makes
	\[
		\xymatrix{
		{ \SY_1 \times \dotsb \times \SY_m } \ar[rr]^-{ b } \ar@{->>}[d]|{\morQuot{q_1} \times \dotsb \times \morQuot{q_m} } && { \SZ } \ar@{->>}[d]|{\morQuot{p}}
		\\
		{ \SY_{1, q_1} \times \dotsb \times \SY_{m, q_m} } \ar[rr]_-{ \widetilde{b} } && {\SZ_p}
		}
	\]
	a commutative diagram.
	For $\gamma_1 \in \CF{\UF}{\SY_1}{\cW_1}{k}, \dotsc, \gamma_m \in \CF{\UF}{\SY_m}{\cW_1}{k}$
	we know from \refer{prop:multilineare_Abb_und_CF} that
	\[
		\widetilde{b} \circ (\HomQuot{\gamma_1}{q_1},\dotsc, \HomQuot{\gamma_m}{q_m})
		\in \CF{\UF}{\SZ_p}{\GewFunk}{k}
	\]
	and the map $\widetilde{b}_*$ is continuous.
	Since
	\[
		\widetilde{b}_* \circ ((\morQuot{q_1})_* \times \dotsb \times (\morQuot{q_m})_*)
		= (\morQuot{p})_* \circ b_*
	\]
	and the left hand side is continuous,
	we conclude using \refer{lem:gew_Abb_in_Lokalkovexe_ist_Produkt_von_gew_Abb_in_normierte}
	that $b_*$ is well-defined and continuous since $p$ was arbitrary.
\end{proof}

\begin{cor}\label{cor:multilineare_Abb_und_CFvan_endl-dim}
	\index{superposition!with a multilinear map}
	Let $\SY_1,\dotsc,\SY_m$ be locally convex spaces,
	$\UF$ be an open nonempty subset of the finite-dimensional space $\SX$,
	$k \in \cl{\N}$
	and $\GewFunk, \cW_1,\dotsc,\cW_m \subseteq \cl{\R}^{\UF}$ nonempty such that
	\[
		(\forall f \in \GewFunk) (\exists g_{f,1} \in \cW_1,\dotsc, g_{f,m} \in\cW_m)\,
		\abs{f} \leq \abs{g_{f,1}}\dotsm \abs{g_{f, m}}.
	\]
	Further let $\SZ$ be another locally convex space,
	$b:\SY_1 \times\dotsb\times \SY_m \to \SZ$
	a continuous $m$-linear map, and $j \in \{1,\dotsc,m\}$.
	Then
	\[
		\tag{\ensuremath{\dagger}}\label{Superposition_multilinear_CWvanK}
		b\circ(\gamma_1,\dotsc,\gamma_j, \dotsc,\gamma_m) \in \CFvanK{\UF}{\SZ}{\GewFunk}{k}
	\]
	for all $\gamma_i \in \CF{\UF}{\SY_i}{\cW_i}{k}$ ($i\neq j$)
	and $\gamma_j \in \CFvanK{\UF}{\SY_j}{\cW_j}{k}$.
	The map
	\begin{gather*}
		\CF{\UF}{\SY_1}{\cW_1}{k}
		\times \dotsb \times
		\CFvanK{\UF}{\SY_j}{\cW_j}{k}
		\times \dotsb \times
		\CF{\UF}{\SY_m}{\cW_m}{k}
		\to\CFvanK{\UF}{\SZ}{\GewFunk}{k}
		\\
		(\gamma_1, \dotsc,\gamma_j, \dotsc,\gamma_m)\mapsto b\circ(\gamma_1, \dotsc,\gamma_j, \dotsc, \gamma_m)
	\end{gather*}
	is $m$-linear and continuous.
\end{cor}
\begin{proof}
	Using \refer{prop:multilineare_Abb_und_CF-lokalkonvex} and \refer{lem:CWvan_(endlichdimensional)_abgeschlossen},
	we only have to prove that \eqref{Superposition_multilinear_CWvanK} holds.
	This is done by induction on $k$.
	
	$k=0$:
	Let $p \in \normsOn{\SZ}$. Then there exist
	$q_1 \in \normsOn{\SY_1}, \dotsc, q_m \in \normsOn{\SY_m}$ such that
	\[
		\norm{b(y_1,\dotsc,y_m)}_p \leq \norm{y_1}_{q_1}\dotsm \norm{y_m}_{q_m}
	\]
	for all $y_1 \in \SY_1$, \ldots, $y_m \in \SY_m$.
	So for $f\in \GewFunk$, $x\in\UF$ and
	$\gamma_1\in\CF{\UF}{\SY_1}{\cW_1}{0},\dotsc, \gamma_j \in \CFvanK{\UF}{\SY_j}{\cW_j}{0}, \dotsc,
	\gamma_m\in\CF{\UF}{\SY_m}{\cW_m}{0}$
	we compute
	\begin{multline*}
		\abs{f(x)}\,\norm{b\circ(\gamma_1, \dotsc, \gamma_j, \dotsc, \gamma_m)(x)}_p
		\\
		\leq \prod_{i=1}^m \abs{g_{f,i}(x)}\,\norm{\gamma_i(x)}_{q_i}
		\leq \left( \prod_{i\neq j}\hn{\gamma_i}{q_i, g_{f,i}}{0}\right)
			\abs{g_{f,j}(x)}\,\norm{\gamma_j(x)}_{q_j} .
	\end{multline*}
	With this estimate we easily deduce that $b\circ(\gamma_1, \dotsc, \gamma_j, \dotsc,\gamma_m) \in \CFvanK{\UF}{\SZ}{\cW_j}{0}$.

	$k\to k + 1$: From \refer{prop:topologische_Zerlegung_von_CFkvan-kompakte_Version}
	(together with the induction base) we know that for
	$\gamma_1\in\CF{\UF}{\SY_1}{\cW_1}{k + 1},
	\dotsc,
	\gamma_j \in \CFvanK{\UF}{\SY_j}{\cW_j}{k + 1},
	\dotsc,
	\gamma_m\in\CF{\UF}{\SY_m}{\cW_m}{k + 1}$
	\[
		b\circ(\gamma_1,\dotsc, \gamma_j, \dotsc, \gamma_m) \in \CFvanK{\UF}{\SZ}{\GewFunk}{k + 1}
		\iff
		\FAbl{(b\circ(\gamma_1,\dotsc, \gamma_j, \dotsc, \gamma_m))}
		\in\CFvanK{\UF}{\Lin{\SX}{\SZ}}{\GewFunk}{k}.
	\]
	We know from \eqref{id:Ableitung_Komposition_mit_multilinearen_Abb-lokalkonvex} in
	\refer{lem:Ableitung_Komposition_mit_multilinearer_Abb-lokalkonvex} that
	\begin{align*}
		D(b\circ(\gamma_1,\dotsc,\gamma_j, \dotsc, \gamma_m))
		&= \sum_{\substack{i=1\\ i\neq j}}^m
			b^{(i)}\circ (\gamma_1, \dotsc,\gamma_j, \dotsc,\gamma_{i-1},D\gamma_i,
				\gamma_{i+1}, \dotsc,\gamma_m)
		\\
		&+ b^{(j)}\circ (\gamma_1, \dotsc, \gamma_{j-1},D\gamma_j,
				\gamma_{j+1}, \dotsc,\gamma_m).
	\end{align*}
	Noticing that $\gamma_j \in \CFvanK{\UF}{\SY_j}{\cW_j}{k}$
	and $D\gamma_j \in \CFvanK{\UF}{\Lin{\SX}{\SY_j}}{\cW_j}{k}$,
	we can apply the inductive hypothesis to all $b^{(i)}$ and the
	$\ConDiff{}{}{k}$-maps
	$\gamma_1, \dotsc, \gamma_{m}$ and $D\gamma_1, \dotsc, D\gamma_m$.
	Hence $D(b\circ(\gamma_1,\dotsc,\gamma_j, \dotsc, \gamma_m)) \in \CFvanK{\UF}{\Lin{\SX}{\SZ}}{\GewFunk}{k}$.
\end{proof}
As an application, we prove that the space of weighted functions into a product
is canonicly isomorphic to the product of the weighted function spaces.
\begin{lem}\label{lem:gewichtete,verschwindende_Abb_Produktisomorphie-lokalkonvex}
	Let $\SX$ be a normed space, $\UF \subseteq \SX$ an open nonempty set,
	$(\SY_i)_{i\in I}$ a family of locally convex spaces,
	$k \in \cl{\N}$ and $\GewFunk \subseteq \cl{\R}^\UF$ nonempty.
	Then for each $\gamma \in \CcF{\UF}{\prod_{i\in I} \SY_i}{k}$ and $j \in I$
	\[
		\pi_{j} \circ \gamma \in \CcF{\UF}{\SY_j}{k},
	\]
	and the map
	\[
		\tag{\ensuremath{\dagger}}
		\label{id:gewichtete_Abb_Produktisomorphie-lokalkonvex}
		\CcF{\UF}{\prod_{i\in I} \SY_i}{k} \to \prod_{i\in I}\CcF{\UF}{ \SY_i}{k}
		: \gamma \mapsto (\pi_{i} \circ \gamma)_{i \in I}
	\]
	is an isomorphism of locally convex topological vector spaces.

	The same statement holds for $\CcFvanK{\UF}{\prod_{i\in I} \SY_i}{k}$:
	\[
		\tag{\ensuremath{\dagger\dagger}}
		\label{id:gewichtete_Abb_Produktisomorphie-lokalkonvex,vanishing}
		\CcFvanK{\UF}{\prod_{i\in I} \SY_i}{k} \to \prod_{i\in I}\CcFvanK{\UF}{ \SY_i}{k}
		: \gamma \mapsto (\pi_{i} \circ \gamma)_{i \in I}
	\]
	is an isomorphism of locally convex topological vector spaces.
\end{lem}
\begin{proof}
	We proved in \refer{prop:lineare_Wirkung_auf_CFk}
	that for $\gamma \in \CcF{\UF}{\prod_{i\in I} \SY_i}{k}$
	and $j\in I$,
	$\pi_{j} \circ \gamma \in \CcF{\UF}{\SY_j}{k}$ and the map \eqref{id:gewichtete_Abb_Produktisomorphie-lokalkonvex}
	is linear and continuous.
	Since a function to a product is determined by its components, the map \eqref{id:gewichtete_Abb_Produktisomorphie-lokalkonvex} is also injective.
	What remains to be shown is the surjectivity, and the continuity of the inverse mapping.
	To this end, we notice that for each $j \in I$ and $p \in \normsOn{\SY_j}$, the map
	\[
		P_{j, p} : \prod_{i\in I} \SY_i \to \R : (y_i)_{i \in I} \mapsto \norm{y_j}_p
	\]
	is a continuous seminorm, and the set
	$
		\{P_{j, p} : j \in I, p \in \normsOn{\SY_j}\}
	$
	generates $\normsOn{\prod_{i\in I}\SY_i}$.
	For each $i \in I$, let $\gamma_i \in \CcF{\UF}{\SY_i}{k}$. We define the map
	\[
		\gamma : \UF \to \prod_{i\in I} \SY_i : x \mapsto (\gamma_i(x))_{i \in I}.
	\]
	Then $\gamma$ is a $\ConDiff{}{}{k}$-map, and
	$
		P_{j, p} \circ \gamma = p \circ \gamma_j.
	$
	We see with \refer{prop:Char_Frechet_Diff} that this implies that
	$\HomQuot{\gamma}{P_{j, p}}$ is an $\FC{}{}{k}$-map,
	and for each $f \in \GewFunk$ and $\ell \in \N$ with $\ell \leq k$ we derive the identity
	\[
		\hn{\HomQuot{\gamma}{P_{j, p}}}{P_{j, p}, f}{\ell} = \hn{\HomQuot{\gamma_j}{p}}{p, f}{\ell}.
	\]
	We proved in \refer{lem:lokalkonvexe_Abbildungsraeume_Erzeugendensystem_der_HN_reicht_aus}
	that this identity implies that
	$
		\gamma \in \CcF{\UF}{\prod_{i\in I} \SY_i}{k}.
	$
	Further it also proves that the inverse map of
	\eqref{id:gewichtete_Abb_Produktisomorphie-lokalkonvex} is continuous using that it is linear.

	The assertions about \eqref{id:gewichtete_Abb_Produktisomorphie-lokalkonvex,vanishing} follow
	from \refer{cor:multilineare_Abb_und_CFvan_endl-dim} and the assertions proved above about
	\eqref{id:gewichtete_Abb_Produktisomorphie-lokalkonvex}.
\end{proof}
\subsubsection{Superposition with differentiable functions on weighted decreasing maps}
We show that a smooth functions mapping $0$ on $0$ induces
a superposition operator on $\CcFvanK{\UF}{\SY}{k}$, provided that $1_\UF \in \GewFunk$.
The proof uses that the image of decreasing maps is (almost) compact,
and so the composition with the smooth map can be described in terms
of compositions with bounded maps taking values in normed spaces.
\paragraph{On the image of decreasing maps}
\begin{lem}\label{lem:Bild_einer_verschwindenden_gewichteten_Abb_ist_kompakt}
	Let $\UF$ be an open nonempty subset of the finite-dimensional space $\SX$,
	$\SY$ a locally convex space, $k\in \cl{\N}$, $\GewFunk \subseteq \cl{\R}^\UF$ with $1_\UF \in \GewFunk$,
	and $\gamma \in \CcFvanK{\UF}{\SY}{k}$.
	Then
	\[
		\gamma(\UF) \cup \sset{0}
	\]
	is compact.
\end{lem}
\begin{proof}
	Since $1_\UF \in \GewFunk$, $\gamma \in \CFvanK{\UF}{\SY}{\{1_\UF\}}{0}$. By the definition
	of this space, $\gamma$ extends to a continuous map $\widetilde{\gamma} : \UF \cup \sset{\infty} \to \SY$
	defined on the Alexandroff compactification of $\UF$ by setting
	$\widetilde{\gamma}(\infty) \ndef 0$.
	Hence
	\[
		\widetilde{\gamma}(\UF \cup \{\infty\}) =\gamma(\UF) \cup \{0\}
	\]
	is compact.
\end{proof}
We describe two easy consequences of the last lemma.
\begin{lem}\label{lem:CWvan_(endl-dim/normiert)_Abstand_Rand}
	Let $\UF$ be an open nonempty subset of the finite-dimensional space $\SX$,
	$\VF$ an open nonempty zero neighborhood of the \emph{normed} space $\SY$,
	$\GewFunk \subseteq \cl{\R}^\UF$ with $1_\UF \in \GewFunk$, and $k\in \cl{\N}$.
	Then
	$\CcFvanK{\UF}{\VF}{k} \subseteq \CcFo{\UF}{\VF}{k}$.
\end{lem}
\begin{proof}
	This is an immediate consequence of \refer{lem:Bild_einer_verschwindenden_gewichteten_Abb_ist_kompakt}.
\end{proof}

\begin{lem}\label{lem:CFvan_kompakt_mit_Werten_in_offener_Menge_offen-Normierte_Version}
	Let $\UF$ be an open nonempty subset of the finite-dimensional space $\SX$,
	$\SY$ a normed space, $\VF\subseteq \SY$ an open zero neighborhood,
	$k \in \N$ and $\GewFunk \subseteq \cl{\R}^\UF$ with $1_\UF \in \GewFunk$.
	Then $\CcFvanK{\UF}{\VF}{k}$ is open in $\CcFvanK{\UF}{\SY}{k}$.
\end{lem}
\begin{proof}
	We proved in \refer{lem:CWvan_(endl-dim/normiert)_Abstand_Rand} that
	$\CcFvanK{\UF}{\VF}{k} \subseteq \CcFo{\UF}{\VF}{k}$.
	Hence $\CcFvanK{\UF}{\VF}{k} = \CcFo{\UF}{\VF}{k} \cap \CcFvanK{\UF}{\SY}{k}$
	is open in $\CcFvanK{\UF}{\SY}{k}$.
\end{proof}
\paragraph{Superposition with a bounded map}
As a preparation, we prove an analogous version of \refer{lem:BC0_operiert_stetig_auf_CW} for decreasing functions.
\begin{lem}\label{lem:CFvan_kompakt_unter_Superposition-Normierte_Version}
	Let $\UF$ be an open nonempty subset of the finite-dimensional space $\SX$,
	$\SY$ and $\SZ$ normed spaces, $\VF\subseteq \SY$ open and star-shaped with center~$0$,
	$k, \ell \in \N$ and $\GewFunk \subseteq \cl{\R}^\UF$ with $1_\UF \in \GewFunk$.
	Further let $\phi \in \BC{\VF}{\SZ}{k + \ell + 1}$ with $\phi(0)=0$.
	Then
	\[
		\phi \circ \CcFvanK{\UF}{\VF}{k} \subseteq \CcFvanK{\UF}{\SZ}{k},
	\]
	and
	\[
		\CcFvanK{\UF}{\VF}{k} \to \CcFvanK{\UF}{\SZ}{k} : \gamma \mapsto \phi \circ \gamma
	\]
	is a $\ConDiff{}{}{\ell}$-map.
\end{lem}
\begin{proof}
	We proved in \refer{lem:CWvan_(endl-dim/normiert)_Abstand_Rand} that
	$\CcFvanK{\UF}{\VF}{k} \subseteq \CcFo{\UF}{\VF}{k}$.
	Hence we can apply \refer{prop:BC0_operiert_glatt_auf_CW} to see that
	\[
		\phi \circ \CcFvanK{\UF}{\VF}{k} \subseteq \CcF{\UF}{\SZ}{k}
	\]
	and the map
	\[
		\CcFvanK{\UF}{\VF}{k} \to \CcF{\UF}{\SZ}{k} : \gamma \mapsto \phi \circ \gamma
	\]
	is $\ConDiff{}{}{\ell}$;
	here we used that $\CcFvanK{\UF}{\VF}{k} = \CcFo{\UF}{\VF}{k} \cap \CcFvanK{\UF}{\SY}{k}$.
	Because $\CcFvanK{\UF}{\SY}{k}$ is closed in $\CcF{\UF}{\SY}{k}$
	by \refer{lem:CWvan_(endlichdimensional)_abgeschlossen},
	it only remains to show that for each $\gamma \in \CcFvanK{\UF}{\VF}{k}$, we have
	$
		\phi \circ \gamma \in \CcFvanK{\UF}{\SZ}{k}.
	$
	This is done by induction on $k$:

	$k = 0$: Let $f \in \cW$ and $x\in\UF$. Then
	\begin{multline*}
		\abs{f(x)}\,\norm{\phi(\gamma(x))}
		= \abs{f(x)}\,\norm{\phi(\gamma(x)) - \phi(0)}\\
		= \abs{f(x)}\,\left\norm{\Mint{\FAbl{\phi}(t \gamma(x)) \eval \gamma(x)}{t}\right}
		\leq \OpInf{\FAbl{\phi}} \abs{f(x)}\,\norm{\gamma(x)};
	\end{multline*}
	here we used that the line segment from $0$ to $\gamma(x)$ is contained in $\VF$.
	From this estimate we conclude that $\phi\circ \gamma \in \CcFvanK{\UF}{\SZ}{0}$.

	$k \to k + 1$: By the chain rule
	\[
		\FAbl{ (\phi \circ \gamma)} = (\FAbl{ \phi } \circ \gamma) \MaMu \FAbl{ \gamma } .
	\]
	Now $\FAbl{ \phi} \circ \gamma \in \BC{\UF}{\Lin{\SY}{\SZ}}{k + 1}$ because of
	\refer{lem:BC_ist_unter_Komposition_abgeschlossen}, since
	$\gamma \in \BC{\UF}{\VF}{k + 1}$.
	Further $\FAbl{\gamma} \in \CcFvanK{\UF}{\Lin{\SX}{\SY}}{k}$,
	so we conclude using \refer{cor:multilineare_Abb_und_CFvan_endl-dim} that
	$( \FAbl{ \phi } \circ \gamma) \MaMu \FAbl{ \gamma } \in \CcFvanK{\UF}{\Lin{\SX}{\SZ}}{k}$.
	By \refer{prop:topologische_Zerlegung_von_CFkvan-kompakte_Version},
	the case $k + 1$ follows from the inductive hypothesis.
\end{proof}
We calculate the higher differentials of the superposition map on weighted functions
that is induced by a bounded function,
see \refer{lem:BC0_operiert_stetig_auf_CW} where a more general assertion was proved.
We will need this later to show that $\ConDiff{}{}{k + \ell + 2}$-functions induce a superposition operator
on the spaces $\CcFvanK{\UF}{\VF}{k}$, and that this superposition operator is $\ConDiff{}{}{\ell}$.
\begin{lem}\label{lem:Hoehere_Ableitungen_der_kovarianten_Kompositionsabb}
	Let $\SX$, $\SY$ and $\SZ$ be normed spaces,
	$\UF\subseteq \SX$ and $\VF\subseteq \SY$ open subsets
	such that $\VF$ is star-shaped with center~$0$,
	$k \in \cl{\N}$, $m \in \N^\ast$, $\phi \in \BCzero{\VF}{\SZ}{k + m + 1}$
	and $\cW \subseteq \cl{\R}^{\UF}$ with $1_{\UF} \in \cW$.
	By \refer{lem:BC0_operiert_stetig_auf_CW},
	\begin{equation*}
		\phi_*
		:
		\CcFo{\UF}{\VF}{k}
		\to \CcF{\UF}{\SZ}{k}
		: \gamma \mapsto \phi \circ \gamma
	\end{equation*}
	is defined and $\ConDiff{}{}{m}$. For its $\ell$-th differential, the identity
	\begin{equation*}
		\dA[\ell]{\phi_*}{\gamma}{\gamma_1,\dotsc,\gamma_\ell}
		= \dA[\ell]{\phi}{}{} \circ (\gamma,\gamma_1,\dotsc,\gamma_\ell)
	\end{equation*}
	holds ($\ell \leq m$).
\end{lem}
\begin{proof}
	Let $x\in \UF$. Using the identity
	\[
		\evTwo_x^\SZ \circ \phi_* = \phi \circ \evTwo_x^\SY
	\]
	(with self-explanatory notation for point evaluations),
	we calculate
	\begin{multline*}
		(\evTwo_x^\SZ \circ \dA[\ell]{\phi_*)}{\gamma}{\gamma_1,\dotsc,\gamma_\ell}
		= \dA[\ell]{(\evTwo_x^\SZ \circ \phi_*)}{\gamma}{\gamma_1,\dotsc,\gamma_\ell}
		\\
		= \dA[\ell]{(\phi \circ \evTwo_x^\SY)}{\gamma}{\gamma_1,\dotsc,\gamma_\ell}
		= \bigl(\dA[\ell]{\phi}{}{}
		\circ (\evTwo_x^\SY)^{\ell + 1} \bigr) (\gamma,\gamma_1,\dotsc,\gamma_\ell)
		\\
		= \evTwo_x^\SZ \bigl( \dA[\ell]{\phi}{}{} \circ (\gamma,\gamma_1,\dotsc,\gamma_\ell)\bigr);
	\end{multline*}
	here we used \refer{lem:Hoehere_Ableitungen_kovariante_komposition_linearer_Abb}
	and \refer{lem:Hoehere_Ableitungen_kontravariante_komposition_linearer_Abb}.
\end{proof}
\paragraph{The main result}
Before we can prove the main result, we need the following facts
concerning compact and star-shaped sets in topological vector spaces.
\begin{lem}\label{lem:Kompakta_und_Stern-shapedness_in_normierten_Raeumen}
	Let $\SZ$ be a locally convex space and $K \subseteq \SZ$ a compact set.
	\begin{enumerate}
		\item\label{enum1:Kom-SSness-normed-1}
		The set $[0,1] \cdot K$ is compact and star-shaped with center~$0$.

		\item\label{enum1:Kom-SSness-normed-2}
		Let $K$ be star-shaped and $\VF$ an open neighborhood of $K$.
		Then there exists an open star-shaped set $\WF$ such that
		$K \subseteq \WF \subseteq \VF$.
	\end{enumerate}
\end{lem}
\begin{proof}
	\refer{enum1:Kom-SSness-normed-1}
	$[0,1] \cdot K$ is compact since it is the image of a compact set under a continuous map.

	\refer{enum1:Kom-SSness-normed-2}
	The set $K \times \sset{0}$ is compact, hence using the continuity of the addition
	and the Wallace lemma, we find an open $0$-neighborhood $\UF$ such that
	$K + \UF \sub \VF$. We may assume w.l.o.g. that $\UF$ is absolutely convex.
	Then $K + \UF$ is open, star-shaped and contained in $\VF$.
\end{proof}
\begin{prop}\label{prop:Superposition_glatter_Abb_auf_weig-van_Abb}
	Let $\UF$ be an open nonempty subset of the finite-dimensional space $\SX$,
	$\SY$ and $\SZ$ locally convex spaces, $\VF\subseteq \SY$ open and star-shaped with center~$0$,
	$k, m \in \N$ and $\GewFunk \subseteq \cl{\R}^\UF$ with $1_\UF \in \GewFunk$.
	Let $\phi \in \ConDiff{\VF}{\SZ}{k + m + 2}$ with $\phi(0)=0$.
	Then for $\gamma \in \CcFvanK{\UF}{\VF}{k}$,
	\[
		\phi \circ \gamma \in \CcFvanK{\UF}{\SZ}{k}
	\]
	holds, and the map
	\[
		\phi_* : \CcFvanK{\UF}{\VF}{k} \to \CcFvanK{\UF}{\SZ}{k} : \gamma \mapsto \phi \circ \gamma
	\]
	is $\ConDiff{}{}{m}$ with
	\begin{equation*}
		\dA[\ell]{\phi_*}{\gamma}{\gamma_1,\dotsc,\gamma_\ell}
		= \dA[\ell]{\phi}{}{} \circ (\gamma,\gamma_1,\dotsc,\gamma_\ell)
	\end{equation*}
	for all $\ell \leq m$.
	\index{superposition!with a differentiable map}
\end{prop}
\begin{proof}
	Let $\widetilde{\gamma} \in \CcFvanK{\UF}{\VF}{k}$. By \refer{lem:Bild_einer_verschwindenden_gewichteten_Abb_ist_kompakt}
	and \refer{lem:Kompakta_und_Stern-shapedness_in_normierten_Raeumen},
	the set
	\[
		K \ndef [0,1] \cdot (\widetilde{\gamma}(\UF) \cup \{0\})
	\]
	is compact and star-shaped with center $0$.
	Hence by \refer{lem:stetig_diffbar_impliziert_lokal_Lipschitz_und_beschraenkt},
	for each $p \in \normsOn{\SZ}$ there exists a $q \in \normsOn{\SY}$
	and an open set $\WF \supseteq K$ w.r.t. $q$
	such that $\FakLC{q}{p}{\phi} \in \BC{\WF_q}{\SZ_p}{k + m + 1}$.
	In view of \refer{lem:Kompakta_und_Stern-shapedness_in_normierten_Raeumen},
	we may assume that $\WF$ (and hence $\WF_q$) is star-shaped with center $0$.
	We know from \refer{lem:CFvan_kompakt_mit_Werten_in_offener_Menge_offen-Normierte_Version}
	that $\CcFvanK{\UF}{\WF_q}{k}$ is a neighborhood of $\HomQuot{\widetilde{\gamma}}{q}$
	in $\CcFvanK{\UF}{\SY_q}{k}$.
	In \refer{lem:CFvan_kompakt_unter_Superposition-Normierte_Version} we stated that
	\[
		\FakLC{q}{p}{\phi}_* : \CcFvanK{\UF}{\WF_q}{k} \to \CcFvanK{\UF}{\SZ_{p}}{k} : \gamma \mapsto \FakLC{q}{p}{\phi} \circ \gamma
	\]
	is $\ConDiff{}{}{m}$.
	The diagram
	\[
		\xymatrix{
			{\CcFvanK{\UF}{\WF}{k} }
			\ar[rr]^{\morQuot{q *}}
			\ar[rd]|{(\HomQuot{\phi}{p})_* }
			&&
			{\CcFvanK{\UF}{\WF_q}{k}}
			\ar[ld]|{\FakLC{q}{p}{\phi}_*}
			\\
			&
			{\CcFvanK{\UF}{\SZ_{p}}{k}}
			&
		}
	\]
	is commutative.
	This implies that $(\HomQuot{\phi}{p})_*$ is $\ConDiff{}{}{m}$ on $\CcFvanK{\UF}{\WF}{k}$
	since it is the composition of $\FakLC{q}{p}{\phi}_*$ and the smooth map $\morQuot{q *}$ (see \refer{cor:multilineare_Abb_und_CFvan_endl-dim}).
	Using \refer{lem:Hoehere_Ableitungen_kontravariante_komposition_linearer_Abb}
	and \refer{lem:Hoehere_Ableitungen_der_kovarianten_Kompositionsabb},
	we can calculate its higher derivatives:
	\begin{multline*}
		\dA[\ell]{\rest{(\HomQuot{\phi}{p})_*}{ \CcFvanK{\UF}{\WF}{k} }}{\gamma}{\gamma_1,\dotsc,\gamma_\ell}
		\\
		= \dA[\ell]{\rest{(\FakLC{q}{p}{\phi} \circ \morQuot{q})_*}{ \CcFvanK{\UF}{\WF}{k} }}{\gamma}{\gamma_1,\dotsc,\gamma_\ell}
		= \dA[\ell]{ \FakLC{q}{p}{\phi}_*} {\HomQuot{\gamma}{q}}{\HomQuot{\gamma_1}{q},\dotsc,\HomQuot{\gamma_\ell}{q}}
		\\
		= \dA[\ell]{ \FakLC{q}{p}{\phi} }{}{} \circ (\HomQuot{\gamma}{q}, \HomQuot{\gamma_1}{q},\dotsc,\HomQuot{\gamma_\ell}{q})
		= \dA[\ell]{ (\FakLC{q}{p}{\phi} \circ \morQuot{q})}{}{} \circ (\gamma,\gamma_1,\dotsc,\gamma_\ell)
		\\
		= \dA[\ell]{ (\HomQuot{\phi}{p})}{}{} \circ (\gamma,\gamma_1,\dotsc,\gamma_\ell)
		= \HomQuot{\dA[\ell]{ \phi }{}{}}{p} \circ (\gamma,\gamma_1,\dotsc,\gamma_\ell)
	\end{multline*}
	for $\ell \in \N$ with $\ell \leq m$.
	
	Since $\widetilde{\gamma}$ and $p$ were arbitrary, we conclude that the map
	\[
		\CcFvanK{\UF}{\VF}{k} \to \prod_{p \in \normsOn{\SZ}} \CcFvanK{\UF}{\SZ_p}{k}
		: \gamma \mapsto (\HomQuot{\phi \circ \gamma}{p})_{p \in \normsOn{\SZ}}
	\]
	is $\ConDiff{}{}{m}$. Since its image and all directional derivatives are contained in $\CcFvanK{\UF}{\SZ}{k}$
	(in the sense of \refer{lem:gew_Abb_in_Lokalkovexe_ist_Produkt_von_gew_Abb_in_normierte}),
	we conclude that it is $\ConDiff{}{}{m}$ as a map to $\CcFvanK{\UF}{\SZ}{k}$.
\end{proof}
\chapter{Lie groups of weighted diffeomorphisms}
In this chapter, we prove that for each Banach space $\SX$ appropriate
subgroups of the diffeomorphism group $\Diff{\SX}{}{}$ can be turned
into Lie groups that are modelled on some weighted function space described earlier.
Further, we show that these Lie groups are regular.
Here
\[
	\glstext{Diffeomorphismen_Braum} \ndef
	\{\phi \in \FC{\SX}{\SX}{\infty}
		: \phi \text{ is bijective and }
			\phi^{-1} \in \FC{\SX}{\SX}{\infty}
	\} ;
\]
\index{diffeomorphisms}
the chain rule ensures that
$\Diff{\SX}{}{}$ is actually a group with the composition
and inversion of maps as the group operations.
\section{Weighted diffeomorphisms and endomorphisms}
In this section, we define and examine sets of \emph{weighted endomorphisms} $\EndW$
and \emph{weighted diffeomorphisms} $\DiffW$.
We show that if $1_\SX \in \GewFunk$, then $\EndW$ is a smooth monoid and $\DiffW$ is its group of units
that can be turned into a Lie group.
Further, we discuss certain subsets of these, the \emph{decreasing weighted diffeomorphisms}
respective \emph{endomorphisms}.

For nonempty $\cW \subseteq \cl{\R}^\SX$, we define
\[
	\glstext{gew_Diffeomorphismen_Braum} \ndef
	\{ \phi \in \Diff{\SX}{}{} :
		\phi - \id{\SX} ,\,\phi^{-1} - \id{\SX}
		\in \CF{\SX}{\SX}{\cW}{\infty}
	\}
\]
\index{weighted diffeomorphisms}%
\index{diffeomorphisms!weighted|see{weighted diffeomorphisms}}%
and
\[
	\glstext{gew_Endomorphismen_Braum} \ndef
	\{	\gamma + \id{\SX} :
		\gamma \in \CF{\SX}{\SX}{\cW}{\infty}
	\}.
\]
The set $\Endos{\SX}{\cW}$ can be turned into a smooth manifold
using the differentiable structure generated by the bijective map
\begin{equation}\label{def_der_Karte}
	\glstext{chart_gewEndos}
	: \CF{\SX}{\SX}{\cW}{\infty} \to \Endos{\SX}{\cW}
	: \gamma \mapsto \gamma + \id{\SX}.
\end{equation}
We clarify the relation between $\Endos{\SX}{\cW}$ and $\Diff{\SX}{}{\cW}$.
The following is obvious from the definition:
\begin{lem}\label{lem:char_DiffW}
	Let $\cW \subseteq \cl{\R}^\SX$ and	$\phi \in \Diff{\SX}{}{}$. Then
	\[
		\phi \in \Diff{\SX}{}{\cW}
		\iff
		\phi, \phi^{-1} \in \Endos{\SX}{\cW} .
	\]
\end{lem}
Furthermore, we have
\begin{lem}\label{lem:EndW_und_DiffW}
	Let $\cW \subseteq \cl{\R}^\SX$ such that $\Endos{\SX}{\cW}$ is a monoid
	with respect to the composition of maps. Then the group of units is given by
	\[
		\Endos{\SX}{\cW}^\times = \Diff{\SX}{}{\cW};
	\]
	in particular $\Diff{\SX}{}{\cW}$ is a subgroup of $\Diff{\SX}{}{}$.
\end{lem}
\begin{proof}
	Obviously
	\[
		\phi \in \Endos{\SX}{\cW}^\times
		\iff
		\text{$\phi$ is bijective and }
		\phi, \phi^{-1} \in \Endos{\SX}{\cW}.
	\]
	Since $\Endos{\SX}{\cW}$ consists of smooth maps, the assertion follows
	from \refer{lem:char_DiffW}.
\end{proof}
In the rest of this section, we prove that $\Endos{\SX}{\cW}$ is a smooth monoid if
$1_\SX \in \cW$;
thus $\Diff{\SX}{}{\cW}$ is a group by \refer{lem:EndW_und_DiffW}.
Further, we define the set of \emph{weighted decreasing endomorphisms}
and show that it is a closed submonoid of $\Endos{\SX}{\cW}$.
The main part is to show that the monoid multiplication
\[
	\circ : \Endos{\SX}{\cW}\times\Endos{\SX}{\cW}
		\to \Endos{\SX}{\cW}
\]
is defined and smooth, so we elaborate on this.
\subsection{Composition of weigthed endomorphisms in charts}
We study how the composition looks like
with respect to the global chart $\kF^{-1}$ (from \eqref{def_der_Karte}).
For $\eta ,\,\gamma \in \CcF{\SX}{\SX}{\infty}$,
\begin{equation}\label{id:Komposition_in_Koordinaten}
	\kF(\gamma)\circ\kF(\eta)
	= (\gamma + \id{\SX})\circ (\eta + \id{\SX})
	= \gamma\circ(\eta + \id{\SX}) + \eta + \id{\SX}.
\end{equation}
Obviously $\kF(\gamma)\circ\kF(\eta) \in \EndW$ if and only if
$\gamma\circ(\eta + \id{\SX}) \in \CcF{\SX}{\SX}{\infty}$;
and the smoothness of $\circ$ is equivalent to that of
\begin{equation*}
	\CcF{\SX}{\SX}{\infty}
		\times \CcF{\SX}{\SX}{\infty}
		\to \CcF{\SX}{\SX}{\infty}
	: (\gamma,\eta) \mapsto \gamma\circ(\eta + \id{\SX}) .
\end{equation*}
\subsubsection{Important maps}
For technical reasons we look at more general maps of the form
\begin{equation}\label{Kompo_mit_+id}
	\compIdRaw
	:  \SY^{\WF} \times \VF^{\UF}
		\to \SY^{\UF}
	:(\gamma,\eta) \mapsto \gamma\circ(\eta + \id{\UF});
\end{equation}
here $\UF, \VF, \WF \subseteq \SX$ are open nonempty subsets with
$\VF + \UF \subseteq \WF$ and $\SY$ is a normed space.
These maps play an important role in further discussions.
\paragraph{Continuity properties}
We discuss when the restriction of $\compIdRaw$ to weighted function spaces
has values in a weighted function space and is continuous.
We start with the following lemma whose assertion is used as the base case for \refer{lem:comp_CF(k+1)xCF(k)_CF(k)}.%
\begin{lem}
\label{lem:Hinreichendes_fuer_Kompo_in_CF0}
\label{lem:kompo_1x0->0_stetig}
	Let $\SX$ and $\SY$ be normed spaces, $\UF, \VF, \WF \subseteq \SX$ open nonempty subsets
	such that $\VF + \UF \subseteq \WF$ and $\VF$ is balanced,
	and $\GewFunk \subseteq \cl{\R}^{\WF}$.
	\begin{enumerate}
	\item\label{enum1:Hinreichendes_fuer_Kompo_in_CF0}
		For $\gamma \in \FC{\WF}{\SY}{1}$,
		$\eta : \UF \to \VF$, $f \in \GewFunk$ and $x\in\UF$,
		the estimate
		\begin{equation}\label{est:Funktionswerte_Gewicht_K-Kompo}
			\abs{f(x)}\,\norm{\gamma \circ (\eta + \id{\SX})(x)}
			\leq \abs{f(x)}\, (\hn{\gamma}{1_{\sset{x} + \Disk\eta(\UF)}}{1}\,\norm{\eta(x)} + \norm{\gamma(x)})
		\end{equation}
		holds. In particular, if $\gamma \in \CcF{\WF}{\SY}{0} \cap \BC{\WF}{\SY}{1}$
		and $\eta \in \CcF{\UF}{\VF}{0}$, then
		\[
			\compIdRaw(\gamma , \eta)
			= \gamma \circ (\eta + \id{\SX})
			\in \CcF{\UF}{\SY}{0} .
		\]
	
	\item\label{enum1:kompo_1x0->0_stetig}
		Let $\gamma, \gamma_0 \in \CcF{\WF}{\SY}{0} \cap \BC{\WF}{\SY}{1}$
		and $\eta, \eta_0 \in \CcF{\UF}{\VF}{0}$ such that
		\[
			\{t \eta(x) + (1-t) \eta_0(x) : t \in [0,1], x\in\UF\} \subseteq \VF .
		\]
		Then for each $f \in \GewFunk$ the estimate
		\begin{equation}\label{est:f,0-Norm_Differenz_Kompo}
			\begin{multlined}[0.8\columnwidth]
				\hn{\compIdRaw(\gamma , \eta) -
					\compIdRaw(\gamma_0 , \eta_0)}{f}{0}
				\leq
					\hn{\gamma}{1_\WF}{1} \hn{\eta - \eta_0}{f}{0}\\
					+ \hn{\gamma - \gamma_0}{1_\WF}{1} \hn{\eta_0}{f}{0}
					+ \hn{\gamma - \gamma_0}{f}{0}
				\end{multlined}
		\end{equation}
		holds. In particular, if $1_\WF \in \cW$ then the map
		\[
			\compIdConCW{\SY}{0}
			:\CcF{\WF}{\SY}{1} \times \CcFo{\UF}{\VF}{0}
			\to \CcF{\UF}{\SY}{0}
			: (\gamma, \eta) \mapsto \compIdRaw(\gamma, \eta)
		\]
		is continuous.
	\end{enumerate}
\end{lem}
\begin{proof}
	\refer{enum1:Hinreichendes_fuer_Kompo_in_CF0}
	For $x \in \UF$ we derive using the triangle inequality and the mean value theorem
	\begin{align*}
		\abs{f(x)}\,\norm{\compIdRaw(\gamma , \eta)(x)}
		&= \abs{f(x)}\,\norm{\gamma(\eta(x) + x)}\\
		&\leq \abs{f(x)}\,\norm{\gamma(\eta(x) + x) - \gamma(x)}
			+ \abs{f(x)}\,\norm{\gamma(x)}\\
		&= \abs{f(x)}\,\left\norm{\Mint{\FAbl{\gamma}(x + t \eta(x))\eval \eta(x) }{t}\right}
			+ \abs{f(x)}\,\norm{\gamma(x)}\\
		&\leq \abs{f(x)}\,\OpInf{\rest{\FAbl{\gamma}}{\{x\} + \Disk\eta(\UF)}}\norm{\eta(x)} + \abs{f(x)}\,\norm{\gamma(x)}
	\end{align*}
	and from this we easily conclude the assertion. We could apply the mean value theorem
	because the line segment $\{x + t \eta(x) : t \in [0,1]\}$ is contained in $\UF + \VF$
	since $\VF$ is balanced.
	
	\refer{enum1:kompo_1x0->0_stetig}
	For $x \in \UF$ we have
	\begin{gather*}
		\abs{f(x)}\, \norm{\compIdConCW{\SY}{0}(\gamma , \eta)(x) -
		 \compIdConCW{\SY}{0}(\gamma_0 , \eta_0)(x)}
		= \abs{f(x)}\, \norm{\gamma(\eta(x) + x)
			- \gamma_0(\eta_0(x) + x)}.\\
		\intertext{
		We add $0 = \gamma(\eta_0(x) + x) - \gamma(\eta_0(x) + x)$
		and apply the triangle inequality:
		}
		= \abs{f(x)}\, \norm{\gamma(\eta(x) + x)
			- \gamma(\eta_0(x) + x)
			+ \gamma(\eta_0(x) + x)
			- \gamma_0(\eta_0(x) + x)}\\
		\leq \abs{f(x)}\, \norm{\gamma(\eta(x) + x) - \gamma(\eta_0(x) + x)}
			+ \abs{f(x)}\,\norm{(\gamma - \gamma_0)(\eta_0(x) + x)}.
	\end{gather*}
	We discuss the summands separately. For the first summand, we can apply the mean value theorem (\refer{prop:MWS_FC1_Abb}) because we assumed that
	the line segment $\{t \eta(x) + (1-t) \eta_0(x) : t \in [0,1]\}$
	is contained in $\VF$, and get
	\begin{align*}
		&\abs{f(x)}\, \norm{\gamma(\eta(x) + x) - \gamma(\eta_0(x) + x)}\\
		=& \abs{f(x)}\, \left\norm{\Rint{0}{1}{D\gamma(t\eta(x) + (1-t)\eta_0(x) + x)\eval (\eta(x) - \eta_0(x))}{t}\right}\\
		\leq& \abs{f(x)}\, \hn{\gamma}{1_{\WF}}{1} \norm{\eta(x) - \eta_0(x)} .
	\end{align*}
	By applying the mean value theorem, which is possible because $\VF$ is balanced,
	the second summand becomes:
	\begin{align*}
		&\abs{f(x)}\,\norm{(\gamma - \gamma_0)(\eta_0(x) + x)}\\
		=& \abs{f(x)}\, \norm{(\gamma - \gamma_0)(\eta_0(x) + x)
			- (\gamma - \gamma_0)(x) + (\gamma - \gamma_0)(x)}
		\\
		\leq&
		\abs{f(x)}\, \left(\left\norm{\Rint{0}{1}{D(\gamma - \gamma_0)(t \eta_0(x) + x)\eval \eta_0(x)}{t} \right}
			+ \norm{(\gamma - \gamma_0)(x)}\right)
		\\
		\leq& \abs{f(x)}\, \bigl(\hn{\gamma - \gamma_0}{1_{\WF}}{1} \norm{\eta_0(x)}
			+ \norm{(\gamma - \gamma_0)(x)}\bigr).
	\end{align*}
	Combining these two estimates gives \eqref{est:f,0-Norm_Differenz_Kompo}.

	The continuity of $\compIdConCW{\SY}{0}$ follows from this estimate:
	For each $\eta \in \CcFo{\UF}{\VF}{0}$, there exists an $r > 0$ such that
	\[
		\eta(\UF) + \Ball{0}{r} \subseteq \VF,
	\]
	and since $1_\WF \in \cW$,
	\[
		F_\eta \ndef \{\widetilde{\eta} \in \CcF{\UF}{\SX}{0} : \hn{\eta - \widetilde{\eta}}{1_\WF}{0} < r\}
	\]
	is a neighborhood of $\eta$ in $\CcFo{\UF}{\VF}{0}$.
	The \refer{est:f,0-Norm_Differenz_Kompo}
	ensures that $\compIdConCW{\SY}{0}$ is continuous on $\CcF{\WF}{\SY}{1}\times F_\eta$.
\end{proof}
\begin{lem}\label{lem:comp_CF(k+1)xCF(k)_CF(k)}
	Let $\SX$ and $\SY$ be normed spaces, $\UF, \VF, \WF \subseteq \SX$ open nonempty subsets such that
	$\VF + \UF \subseteq \WF$ and $\VF$ is balanced,
	$k \in \N$ and $\GewFunk \subseteq \cl{\R}^\WF$ with $1_\WF \in \cW$.
	Then
	\[
		\compIdRaw(
			\CcF{\WF}{\SY}{k+1} \times
			\CcF{\UF}{\VF}{k}
		) 
		\subseteq \CcF{\UF}{\SY}{k} ,
	\]
	and the map
	\[
		\compIdConCW{\SY}{k}
		:
		\CcF{\WF}{\SY}{k+1} \times \CcFo{\UF}{\VF}{k}
		\to \CcF{\UF}{\SY}{k}
		:
		(\gamma, \eta) \mapsto \compIdRaw(\gamma, \eta)
	\]
	which arises by restricting $\compIdRaw$ is continuous.
\end{lem}
\begin{proof}
	The proof is by induction. The case $k=0$ was treated in
	\refer{lem:Hinreichendes_fuer_Kompo_in_CF0}.
	
	$k\to k + 1$:
	We use \refer{prop:topologische_Zerlegung_von_CFk}
	(and \refer{lem:Hinreichendes_fuer_Kompo_in_CF0}) to see that
	\[
		\compIdRaw(\CcF{\WF}{\SY}{k+2} \times
			\CcF{\UF}{\VF}{k+1})
		\subseteq \CcF{\UF}{\SY}{k+1}
	\]
	is equivalent to
	\[
		(\FAbl{} \circ \compIdRaw) (\CcF{\WF}{\SY}{k+2}
			\times\CcF{\UF}{\VF}{k+1})
		\subseteq \CcF{\UF}{ \Lin{\SX}{\SY} }{k};
	\]
	and that the continuity of $\compIdConCW{\SY}{k + 1}$ is equivalent to that of
	$\FAbl{} \circ \compIdConCW{\SY}{k + 1}$.
	\\
	Applying the chain rule to $\compIdRaw$ yields that for
	$\gamma \in \CcF{\WF}{\SY}{k+2}$ and
	$\eta \in \CcF{\UF}{\VF}{k+1}$
	\begin{equation*}\label{id:Ableitung_nach_g}
		\tag{\ensuremath{\ast}}
		(\FAbl{}\circ \compIdRaw)(\gamma,\eta)
		=\compIdConCW{\Lin{\SX}{\SY}}{ k}(\FAbl{\gamma}, \eta)\MaMu(\FAbl{\eta} + \idco)
	\end{equation*}
	holds, where $\MaMu$ denotes the composition of linear maps (see \refer{cor:Komposition_linearer_Abb_und_CF})
	and $\idco$ denotes the constant map $x\mapsto\id{\SX}$.
	Since
	$\FAbl{\gamma} \in \CcF{\WF}{\Lin{\SX}{\SY}}{k+1}$,
	we derive from the induction hypothesis that
	\[
		\compIdConCW{\Lin{\SX}{\SY}}{k}(\FAbl{\gamma}, \eta)
		\in \CcF{\UF}{\Lin{\SX}{\SY}}{k}.
	\]
	Hence we conclude from \refer{cor:Komposition_linearer_Abb_und_CF}
	and $\FAbl{\eta} + \idco \in\BC{\UF}{\Lin{\SX}{\SX}}{k}$
	that
	\[
		(\FAbl{} \circ \compIdRaw)(\gamma,\eta) \in
		\CcF{\UF}{\Lin{\SX}{\SY}}{k}.
	\]
	The continuity of $\FAbl{} \circ \compIdConCW{\SY}{k + 1}$ follows easily from
	\eqref{id:Ableitung_nach_g}: We use the inductive hypothesis to conclude
	that $\compIdConCW{\Lin{\SX}{\SY}}{ k}$ is continuous.
	Since $\FAbl{}$ and
	\[
		\MaMu :
		\CcF{\UF}{\Lin{\SX}{\SY}}{k}
			\times \BC{\UF}{\Lin{\SX}{\SX}}{k}
		\to
		\CcF{\UF}{\Lin{\SX}{\SY}}{k}
	\]
	are smooth (see \refer{prop:Ableitung_ist_stetig}
	and \refer{cor:Komposition_linearer_Abb_und_CF}) as well as the translation
	with $\idco$ in $\BC{\UF}{\Lin{\SX}{\SX}}{k}$,
	the continuity of $\compIdConCW{\SY}{k + 1}$ is proved.
\end{proof}
\subparagraph{Restriction to decreasing functions}
Finally, we study the restriction of $\compIdConCW{\SY}{k}$ to decreasing functions.
\begin{lem}\label{lem:EndWwan_und_Kompostion}
	Let $\SX$ and $\SY$ be normed spaces, $\UF, \VF, \WF \subseteq \SX$ open nonempty subsets such that
	$\VF + \UF \subseteq \WF$ and $\VF$ is balanced,
	$k \in \N$ and $\GewFunk \subseteq \cl{\R}^\SX$ with $1_\SX \in \cW$.
	Then
	\[
		\compIdConCW{\SY}{k}(
			\CcFvan{\WF}{\SY}{k+1} \times
			\CcF{\UF}{\VF}{k}
		) 
		\subseteq \CcFvan{\UF}{\SY}{k} .
	\]
\end{lem}
\begin{proof}
	The proof is by induction on $k$:

	$k = 0$: We use \refer{est:Funktionswerte_Gewicht_K-Kompo}
	in \refer{lem:Hinreichendes_fuer_Kompo_in_CF0}:
	\\
	Let $f \in \cW$, $\gamma \in \CcFvan{\WF}{\SY}{1}$ and $\eta \in \CcF{\UF}{\VF}{0}$.
	Then for every $\eps > 0$ there exists $r > 0$ such that
	\[
		\hn{\rest{\gamma}{\WF\setminus \Ball{0}{r}}}{f}{0} < \frac{\eps}{2}
	\]
	and (as $1_\SX \in \cW$)
	\[
		\hn{\rest{\gamma}{\WF\setminus \Ball{0}{r}}}{1_\WF}{1} < \frac{\eps}{2 (\hn{\eta}{f}{0} + 1)}.
	\]
	Since $1_\SX \in \cW$, we have $K \ndef \hn{\eta}{1_\UF}{0} < \infty$.
	Let $R \in \R$ such that $R > r + K$. Then for each $x\in \UF\setminus\clBall{0}{R}$, we have
	\[
		x + \Disk\eta(x) \subseteq \WF\setminus\clBall{0}{r},
	\]
	so we conclude from \refer{est:Funktionswerte_Gewicht_K-Kompo} that
	\[
		\abs{f(x)}\,\norm{\compIdConCW{\SY}{k}(\gamma,\eta)(x)}
		\leq \hn{\gamma}{1_{\{x\} + \Disk\eta(\UF)}}{1}\,\hn{\eta}{f}{0} + \abs{f(x)}\, \norm{\gamma(x)}
		<  \frac{\eps}{2 (\hn{\eta}{f}{0} + 1)} \hn{\eta}{f}{0} + \frac{\eps}{2}
		.
	\]
	Thus $\compIdConCW{\SY}{k}(\gamma,\eta) \in \CcFvan{\UF}{\SY}{0}$.

	$k \to k + 1$:
	We calculate using the chain rule that
	\begin{equation*}
		(\FAbl{} \circ \compIdConCW{\SY}{ k + 1})(\gamma,\eta)
		=\compIdConCW{\Lin{\SX}{\SY}}{ k}(D\gamma, \eta)\MaMu(\FAbl{\eta} + \idco) .
	\end{equation*}
	Since $\FAbl{\gamma} \in \CcFvan{\WF}{\Lin{\SX}{\SY} }{k+1}$
	(see \refer{cor:topologische_Zerlegung_von_CFkvan}),
	\[
		\compIdConCW{\Lin{\SX}{\SY}}{ k}(\FAbl{\gamma}, \eta) \in \CcFvan{\UF}{\Lin{\SX}{\SY} }{k}
	\]
	by the inductive hypothesis.
	Further, $\FAbl{\eta} + \idco \in \BC{\UF}{\Lin{\SX}{\SX}}{k}$, so we conclude
	with \refer{cor:multilineare_Abb_und_CFvan} that
	\[
		(\FAbl{} \circ \compIdConCW{\SY}{ k + 1})(\gamma,\eta) \in \CcFvan{\UF}{\Lin{\SX}{\SY} }{k}.
	\]
	From this (and the base case $k=0$) we see with \refer{cor:topologische_Zerlegung_von_CFkvan} that
	\[
		\compIdConCW{\SY}{ k + 1}(\gamma,\eta) \in \CcFvan{\UF}{\SY }{k + 1},
	\]
	so the proof is complete.
\end{proof}
\paragraph{Differentiability properties}
We discuss whether restrictions of $\compIdConCW{\SY}{k}$ to certain
weighted function spaces are differentiable.
Before we do this, we give the following definitions.
\begin{defi}
	Let $\SX$ and $\SY$ be normed spaces, $\UF, \VF, \WF \subseteq \SX$ open nonempty subsets
	such that $\VF + \UF \subseteq \WF$ and $\VF$ is balanced,
	$\GewFunk \subseteq \cl{\R}^\WF$ with $1_\WF \in \cW$
	and $k, \ell \in \cl{\N}$.
	Then the map
	\[
		\glstext{composition_map_id}
		:\CcF{\WF}{\SY}{k + \ell + 1} \times
			\CcFo{\UF}{\VF}{k}
		\to \CcF{\UF}{\SY}{k}
		: (\gamma,\eta) \mapsto \gamma\circ(\eta + \id{\UF})
	\]
	is defined by \refer{lem:comp_CF(k+1)xCF(k)_CF(k)}. Additionally,
	we set $\compIdSmoothLcW{\SY}{k} \ndef \compIdDiffKLcW{\SY}{ k}{ \infty}$
	and $\compIdcW{\SY} \ndef \compIdDiffKLcW{\SY}{ \infty}{ \infty}$.
\end{defi}
We provide a nice identity for the differential quotient of $\compIdConCW{\SY}{k}$.
\begin{lem}\label{lem:Integraldarstellung_des_DiffQuotient_Kompo}
	Let $\SX$ and $\SY$ be normed spaces, $\UF, \VF, \WF \subseteq \SX$ open nonempty subsets such that
	$\VF + \UF \subseteq \WF$ and $\VF$ is balanced,
	$\gamma, \gamma_1 \in \FC{\WF}{\SY}{1}$,
	$\eta \in \VF^\UF$, $\eta_1 \in \SX^\UF$ and $t\in\K^\ast$.
	Further, suppose that
	\[
		\set{\eta + s t \eta_1}{ s \in [0,1]} \sub \VF^\UF.
	\]
	Then for each $x \in \UF$,
	\begin{equation*}
		\evTwo_{x}
		\left(
			\tfrac{
				\compIdRaw(\gamma + t\gamma_1, \eta + t\eta_1 ) -
					\compIdRaw(\gamma ,\eta )
				}%
			{t}
		\right)
		\\=
		\Rint{0}{1}
		{
			\evTwo_{x}
			\bigl(
			\compIdRaw
			(\FAbl{(\gamma + s t \gamma_1)}, \eta + s t \eta_1) \eval \eta_1
			+
			\compIdRaw(\gamma_1, \eta + s t \eta_1)
			\bigr)
		}{s}.
	\end{equation*}
\end{lem}
\begin{proof}
	We first prove that the relevant weak integral exists.
	To this end, we take a closer look at the integrand.
	\begin{multline*}
		\evTwo_{x}
		\bigl(
			\compIdRaw
			(\FAbl{(\gamma + s t \gamma_1)}, \eta + s t \eta_1) \eval \eta_1
			+
			\compIdRaw(\gamma_1, \eta + s t \eta_1)
		\bigr)
		\\= \FAbl{\gamma}(\eta(x) + s t \eta_1(x) + x) \eval \eta_1(x)
				+ s t \FAbl{\gamma_1}(\eta(x) + s t \eta_1(x) + x) \eval \eta_1(x)
				+ \gamma_1( \eta(x) +  s t \eta_1(x) + x).
	\end{multline*}
	Since $\set{\eta(x) + s t \eta_1(x)}{ s \in [0,1]} \subseteq \VF$,
	we apply the mean value theorem to $\gamma$
	\[
		\Rint{0}{1}{
		\FAbl{\gamma}(\eta(x) + s t \eta_1(x) + x)\eval \eta_1(x)
		}{s}
		= \frac{
		\gamma( \eta(x) +  t \eta_1(x) + x) - \gamma( \eta(x) + x)
		}{t}
	\]
	and to the function $I \to \SY : s \mapsto s \gamma_1( \eta(x) +  s t \eta_1(x) + x)$,
	where $I\supseteq [0, 1]$ is an open interval,
	\begin{multline*}
		\Rint{0}{1}
		{  \bigl(s t \FAbl{\gamma_1}(\eta(x) + s t \eta_1(x) + x)
			\eval \eta_1(x)
			+ \gamma_1( \eta(x) +  s t \eta_1(x) + x)
			\bigr)
		}{s}
		\\= \gamma_1( \eta(x) +  t \eta_1(x) + x);
	\end{multline*}
	the latter identity follows from the fact that
	\[
		\frac{d}{ds} s \gamma_1( \eta(x) +  s t \eta_1(x) + x)
		= s t \FAbl{\gamma_1}(\eta(x) + s t \eta_1(x) + x)
			\eval \eta_1(x)
			+ \gamma_1( \eta(x) +  s t \eta_1(x) + x).
	\]
	So the integral exists and has the value
	\begin{equation*}
		\tfrac{
		\gamma( \eta(x) +  t\eta_1(x) + x) - \gamma( \eta (x) + x)
		}{t}
		+ \gamma_1 ( \eta(x) +  t\eta_1(x) + x)
		\\
		= \frac{
		\compIdRaw(\gamma + t\gamma_1, \eta + t\eta_1 )(x) -
			\compIdRaw(\gamma ,\eta )(x)
		}{t},
	\end{equation*}
	and that shows the assertion.
\end{proof}

\begin{prop}\label{prop:Kompo_Koord_glatt}
	Let $\SX$ and $\SY$ be normed spaces, $\UF, \VF, \WF \subseteq \SX$ open nonempty subsets such that
	$\VF + \UF \subseteq \WF$ and $\VF$ is balanced,
	$\GewFunk \subseteq \cl{\R}^\WF$ with $1_\WF \in \cW$
	and $k, \ell \in \cl{\N}$.
	Then
	$\compIdDiffKLcW{\SY}{ k}{ \ell}$
	is a $\ConDiff{}{}{\ell}$-map. If $\ell > 0$, then it has the directional derivative
	\begin{equation}\label{id:Ableitung_Kompo}
		\dA{ \compIdDiffKLcW{\SY}{ k}{ \ell} }{\gamma, \eta}{\gamma_1, \eta_1}
		= \compIdDiffKLcW{\Lin{\SX}{\SY}}{ k}{\ell - 1}(\FAbl{\gamma}, \eta) \eval \eta_1
		+ \compIdDiffKLcW{\SY}{ k}{ \ell}(\gamma_1, \eta).
	\end{equation}
	In particular, $\compIdcW{\SY}$ and $\compIdSmoothLcW{\SY}{k}$ are smooth.
\end{prop}
\begin{proof}
	We first prove the assertion for $k, \ell < \infty$.
	We proved the assertion for $\ell = 0$ in \refer{lem:comp_CF(k+1)xCF(k)_CF(k)}.
	For $\ell > 0$, the proof is by induction on $\ell$.
	\\
	$\ell = 1$:
	From \refer{lem:Integraldarstellung_des_DiffQuotient_Kompo} and
	\refer{lem:Kriterium_Integrierbarkeit_in_CW}
	we conclude that
	for $\gamma, \gamma_1 \in \CcF{\WF}{\SY}{k + \ell + 1}$,
	$\eta \in \CcFo{\UF}{\VF}{k}$, $\eta_1 \in \CcF{\UF}{\SX}{k}$ and
	for all $t \in \R^\ast$ in a suitable neighborhood of $0$ the identity
	\begin{multline*}
		\frac{
		\compIdDiffKLcW{\SY}{ k}{ \ell}(\gamma + t\gamma_1, \eta + t\eta_1 ) -
			\compIdDiffKLcW{\SY}{ k}{ \ell}(\gamma ,\eta )
		}{t}
		= \Rint{0}{1}{
		\compIdDiffKLcW{\Lin{\SX}{\SY}}{ k}{ \ell - 1}
			(D(\gamma + s t \gamma_1), \eta + s t \eta_1) \eval \eta_1
		}{s}\\
		+
		\Rint{0}{1}{
			\compIdDiffKLcW{\SY}{ k}{ \ell}(\gamma_1, \eta + s t \eta_1)
		}{s} 
	\end{multline*}
	holds. The theorem about parameter dependent integrals (\refer{prop:Stetigkeit_parameterab_Int})
	yields the assertions if we let $t \to 0$ in the above expression.
	
	$\ell - 1 \to \ell$:
	This follows easily from \eqref{id:Ableitung_Kompo}: Since
	$D$ and $\eval$ are smooth (see \refer{prop:Ableitung_ist_stetig} and \refer{cor:Auswertung_linearer_abb_und_CF})
	and $\compIdDiffKLcW{\Lin{\SX}{\SY}}{ k}{ \ell - 1}$ resp. $\compIdDiffKLcW{\SY}{ k}{ \ell}$
	are $\ConDiff{}{}{\ell - 1}$ by the inductive hypothesis,
	$\dA{ \compIdDiffKLcW{\SY}{ k}{ \ell} }{}{}$ is $\ConDiff{}{}{\ell - 1}$
	and hence $\compIdDiffKLcW{\SY}{ k}{ \ell}$ is $\ConDiff{}{}{\ell}$.

	Now we treat the case $\ell = \infty$.
	If $k<\infty$, then it follows from the things proved above that
	$\compIdDiffKLcW{\SY}{ k}{ \infty} $ is $\ConDiff{}{}{\ell}$ for each $\ell \in \N$
	since the inclusion maps
	\[
		\CcF{\WF}{\SY}{\infty} \to \CcF{\WF}{\SY}{k + \ell + 1}
	\]
	are smooth.
	Now let $k=\infty$.
	\BeweisschrittCkCinftySmooth%
	{\CcF{\WF}{\SY}{\infty} \times
			\CcF{\UF}{\VF}{\infty}}
	{\compIdcW{\SY}}
	{\CcF{\UF}{\SY}{\infty}}
	{}
	{}
	{\CcF{\WF}{\SY}{\infty} \times
			\CcF{\UF}{\VF}{n}}
	{\compIdSmoothLcW{\SY}{n}}
	{\CcF{\UF}{\SY}{n}}
	{n}
\end{proof}
\subparagraph{Restriction to decreasing functions}
We examine the restriction of $\compIdSmoothLcW{\SY}{k}$ to decreasing functions.
We show that it takes values in the decreasing functions and is also smooth.
\begin{cor}\label{cor:Kcomp_g_van_ist_glatt_(und_definiert)}
	Let $\SX$ and $\SY$ be normed spaces, $\UF, \VF, \WF \subseteq \SX$ open nonempty subsets such that
	$\VF + \UF \subseteq \WF$ and $\VF$ is balanced,
	$\GewFunk \subseteq \cl{\R}^\WF$ with $1_\WF \in \cW$
	and $k \in \cl{\N}$.
	Then
	\[
		\compIdSmoothLcW{\SY}{k}
		(\CcFvan{\WF}{\SY}{\infty} \times \CcFvan{\UF}{\VF}{k})
		\subseteq
		\CcFvan{\UF}{\SY}{k},
	\]
	and the restriction
	$\rest{\compIdSmoothLcW{\SY}{k} }{\CcFvan{\WF}{\SY}{\infty} \times \CcFvan{\UF}{\VF}{k}}^{\CcFvan{\UF}{\SY}{k}}$
	is smooth.
\end{cor}
\begin{proof}
	We deduce this from \refer{lem:EndWwan_und_Kompostion},
	the smoothness of the unrestricted map (\refer{prop:Kompo_Koord_glatt})
	and \refer{prop:Differenzierbarkeit_Abb_in_projektiven_Limes}
	that can be used because $\CcFvan{\UF}{\SY}{k}$ is closed
	by \refer{lem:CWvan_closed_CW}.
\end{proof}
\subsection{Smooth monoids of weighted endomorphisms}
We are able to prove that $\EndW$ and the set $\EndWvan$ -- which is defined below --
are smooth monoids, provided that $1_\SX \in \GewFunk$.
\begin{cor}\label{cor:EndW_glattes_Monoid}\label{cor:EndWvan_Untermonoid_von_EndW}
	For $\GewFunk \subseteq \cl{\R}^\SX$ with $1_\SX \in \GewFunk$,
	$\EndW$ is a smooth monoid with the group of units
	\[
		\EndW^{\times} = \DiffW.
	\]
	Further, the set
	\begin{equation}\label{def_EndWvan}
		\glstext{gew_abfallende_Endomorphismen_Braum} \ndef \{\gamma + \id{\SX} : \gamma \in \CcFvan{\SX}{\SX}{\infty} \}
	\end{equation}
	is a closed submonoid of $\EndW$ that is a smooth monoid.
\end{cor}
\begin{proof}
	We first show that $\EndW$ is a monoid. Since $\id{\SX} \in \EndW$
	is obviously satisfied, it remains to show that it is closed under composition.
	Since every element of $\EndW$ can uniquely be written as $\phi + \id{\SX}$
	with $\phi \in \CcF{\SX}{\SX}{\infty}$, we have to show
	that for arbitrary $\gamma, \eta \in \CcF{\SX}{\SX}{\infty}$
	the relation
	\[
		\kF(\gamma)\circ\kF(\eta) - \id{\SX}
		\in \CcF{\SX}{\SX}{\infty}
	\]
	holds. But we know from \refer{id:Komposition_in_Koordinaten} that
	\[
		\kF(\gamma)\circ\kF(\eta) - \id{\SX}
		= \compIdcW{\SX}(\gamma,\eta) + \eta,
	\]
	which is in $\CcF{\SX}{\SX}{\infty}$ by \refer{prop:Kompo_Koord_glatt},
	hence $\EndW$ is a monoid.
	Further, from this identity we easily conclude the smoothness of the composition
	from the one of
	$\compIdcW{\SX}$, which was also proved in \refer{prop:Kompo_Koord_glatt}.

	$\EndWvan$ is a closed subset of $\EndW$ since $\kF$ is a homeomorphism and
	by \refer{lem:CWvan_closed_CW},
	$\CcFvan{\SX}{\SX}{\infty}$ is a closed vector subspace of $\CcF{\SX}{\SX}{\infty}$.
	We know from \refer{cor:Kcomp_g_van_ist_glatt_(und_definiert)} and the fact that $\CcFvan{\SX}{\SX}{\infty}$
	is a vector space that for $\gamma, \eta \in \CcFvan{\SX}{\SX}{\infty}$
	\[
		\kF(\gamma)\circ\kF(\eta) - \id{\SX}
		= \compIdcW{\SX}(\gamma,\eta) + \eta
		\in \CcFvan{\SX}{\SX}{\infty}
		.
	\]
	Further, we proved there that the restriction of $\compIdcW{\SX}(\gamma,\eta)$
	to decreasing maps is smooth, hence $\EndWvan$ is a smooth submonoid of $\EndW$.
		
	The relation $\EndW^{\times} = \DiffW$ was proved in \refer{lem:EndW_und_DiffW}.
\end{proof}
\section{Lie group structures on weighted diffeomorphisms}
In this section, we first prove that $\DiffW$
-- which was already shown to be a group in \refer{cor:EndW_glattes_Monoid} --
is in fact a Lie group.
Also we define and discuss the set $\DiffWvan$ of \emph{decreasing weighted diffeomorphisms}.
We show that it is a normal subgroup of $\DiffW$ that can be turned into a Lie group.
Finally, we explain when diffeomorphisms that are weighted endomorphisms are weighted diffeomorphisms.
\subsection[The Lie group structure of \texorpdfstring{$\DiffW$}{DiffW}]%
{The Lie group structure of \texorpdfstring{$\boldsymbol{\DiffW}$}{DiffW}}
\label{susec:DiffW_Inversion}
We show that $\DiffW$ is an open subset of $\EndW$
and the group inversion is smooth, whence $\DiffW$ is a Lie group.
In order to do this, we have to examine the inversion map on $\Diff{\SX}{}{} \cap \EndW$.
We begin with more general considerations.
\begin{defi}\label{defi:InversionsAbb-Koor_Domains}
	Let $\SX$ be a normed space and $U, V \sub \SX$ open nonempty subsets.
	We define
	\[
		\InvertDomRanAlg{U}{V}
		\ndef
		\{\phi \in  \SX^U:
			\phi + \id{U} \text{ injective}, V \sub (\phi + \id{U})(U) \}
	\]
	and
	\begin{equation}\label{id:Definition_der_verallgemeinerten_Koor-Inversion}
		\InvIdAlg{V}
		:
		\InvertDomRanAlg{U}{V}
		\to
		\SX^V
		:
		\phi
		\mapsto
		\rest{(\phi + \id{U})^{-1}}{V} - \id{V} .
	\end{equation}
	Further, for nonempty $\GewFunk \subseteq \cl{\R}^U$ we set $\glstext{domain_inversion_map_id} \ndef \InvertDomRanAlg{U}{V} \cap \CcF{U}{\SX}{\infty}$
	and $\glstext{inversion_map_id} \ndef \rest{\InvIdAlg{V}}{\InvertWinfDomRan{U}{V}}$.
\end{defi}

\begin{lem}
	Let $\SX$ be a normed space, $U, V \sub \SX$ open nonempty subsets
	and $\phi\in \InvertDomRanAlg{U}{V}$. Then
	\begin{equation}\label{id:I(phi)_und_phi}
		(\InvIdAlg{V}(\phi) + \id{V}) \circ \rest{(\phi + \id{U})}{(\phi + \id{U})^{-1}(V)}
		= \id{(\phi + \id{U})^{-1}(V)}
	\end{equation}
	\begin{equation}\label{id:I(phi)_und_phi2}
		(\phi + \id{U})\circ(\InvIdAlg{V}(\phi) + \id{V})
		= \id{V},
	\end{equation}
	and the identities
	\begin{align}
		\label{id:inversion_psi_gleichungen1}%
		\InvIdAlg{V}(\phi) \circ \rest{(\phi + \id{U})}{(\phi + \id{U})^{-1}(V)}
		&= - \rest{\phi}{(\phi + \id{U})^{-1}(V)}\\
		\label{id:inversion_psi_gleichungen2}%
		\phi \circ (\InvIdAlg{V}(\phi) + \id{V}) &= -\InvIdAlg{V}(\phi)
	\end{align}
	hold.
\end{lem}
\begin{proof}
	This is obvious.
\end{proof}
\subsubsection{On the range of the inversion map}
We first discuss whether the range of $\InvIdcW{V}$ consists of weighted functions,
under certain assumptions on $U$ and $V$.
\begin{lem}\label{lem:Inverse_in_BC0}
	Let $\SX$ be a normed space, $U, V \sub \SX$ open nonempty subsets
	and $\phi \in \InvertDomRanAlg{U}{V}$.
	Then $\hn{\InvIdAlg{V}(\phi)}{1_V}{0} \leq \hn{\phi}{1_U}{0}$.
\end{lem}
\begin{proof}
	This is an immediate consequence of \refer{id:inversion_psi_gleichungen2}.
\end{proof}

We provide a formula for $\FAbl{\InvIdcW{V}(\phi)}$.
\begin{lem}\label{lem:Differential_der_Inversen}
	Let $\SX$ be a Banach space, $U, V \sub \SX$ open nonempty subsets,
	$\GewFunk \subseteq \cl{\R}^{U}$ with $1_{U} \in \GewFunk$
	and $\phi \in \InvertWinfDomRan{U}{V}$.
	\begin{enumerate}
		\item\label{enum1:Differential_der_Inversen_an_einer_Stelle}
		Let $x \in (\phi + \id{U})^{-1}(V)$ such that $\Opnorm{\FAbl{\phi}(x)} < 1$. Then
		\begin{equation*}
			\FAbl{(\InvIdcW{V}(\phi))}((\phi + \id{U})(x))
			= \FAbl{\phi}(x) \MaMu \QuasiInv_{\Lin{\SX}{\SX}}(- \FAbl{\phi}(x))
				- \FAbl{\phi}(x).
		\end{equation*}

		\item\label{enum1:Differential_der_Inversen_glob}
		Suppose that $\hn{\phi}{1_U}{1} < 1$. Then
		\begin{equation}\label{id:Differential_der_inversen_Abb}
				\FAbl{\, \InvIdcW{V}(\phi)}
				= (\FAbl{\phi} \MaMu \QuasiInv(- \FAbl{\phi}) - \FAbl{\phi})
						\circ(\InvIdcW{V}(\phi) + \id{V}).
		\end{equation}
	\end{enumerate}
	Here $\QuasiInv_{\Lin{\SX}{\SX}}$ and $\QuasiInv \ndef \QuasiInv_{ \CcF{ U }{ \Lin{\SX}{\SX} }{\infty} }$
	denote the quasi-inversion (which is discussed in \refer{app:Quasi-Inversion}).
\end{lem}
\begin{proof}
	\refer{enum1:Differential_der_Inversen_an_einer_Stelle}
	From \refer{id:inversion_psi_gleichungen1} and the chain rule, we get
	\begin{equation*}
		\FAbl{\,\InvIdcW{V}(\phi)}((\phi + \id{U})(x)) \MaMu (\FAbl{\phi}(x) + \id{\SX})
		= - \FAbl{\phi}(x).
	\end{equation*}
	Since $\Opnorm{\FAbl{\phi}(x)} < 1$,
	the linear map $\FAbl{\phi}(x) + \id{\SX}$ is bijective with
	\[
		(\FAbl{\phi}(x) + \id{\SX})^{-1}
		= \sum_{k = 0}^{\infty} (- \FAbl{\phi}(x))^{k}
		= - \QuasiInv_{\Lin{\SX}{\SX}}(- \FAbl{\phi}(x)) + \id{\SX};
	\]
	(c.f. \refer{lem:Einheitskugel_quasiinvertierbar}).
	Using these two identities, we easily derive the one desired.

	\refer{enum1:Differential_der_Inversen_glob}
	Since $\hn{\phi}{1_U}{1} < 1$, we see with \refer{lem:CkF(,A)_topologische_Algebra}
	that $- \FAbl{\phi}$ is quasi-invertible in $\CinfF{U }{\Lin{\SX}{\SX}}$ with
	\[
		\QuasiInv(- \FAbl{\phi})
		= \QuasiInv_{\Lin{\SX}{\SX}} \circ (- \FAbl{\phi}).
	\]
	Hence we get with \refer{enum1:Differential_der_Inversen_an_einer_Stelle} that
	\begin{equation*}
		\FAbl{(\InvIdcW{V}(\phi))} \circ (\phi + \id{U})
		= \FAbl{\phi} \MaMu \QuasiInv(- \FAbl{\phi})
			- \FAbl{\phi}
	\end{equation*}
	on $(\phi + \id{U})^{-1}(V)$.
	Composing both sides of this identity with $\InvIdcW{V}(\phi) + \id{V}$
	on the right (see \refer{id:I(phi)_und_phi2})
	gives \refer{id:Differential_der_inversen_Abb}.
\end{proof}
Next, we discuss whether $\InvIdcW{V}(\phi) \in \CcF{V}{\SX}{\infty}$.
\begin{prop}\label{prop:Inverse_von_Abb._mit_Opnorm_kleiner_1_auf_X-B(0,r)_ist_in_CW(X-B(0,R))}
	Let $\SX$ be a Banach space, $U, V \sub \SX$ open nonempty subsets
	such that there exists $r > 0$ with $V + \Ball{0}{r} \sub U$.
	Further, let $\GewFunk \subseteq \cl{\R}^{U}$ with $1_{U} \in \GewFunk$
	and $\phi \in \InvertWinfDomRan{U}{V}$ such that
	$
		\hn{\phi }{1_{U}}{1} < 1
	$
	and
	$
		\hn{\phi}{1_U}{0} < r .
	$
	Then
	$
		\InvIdcW{V}(\phi) \in \CcF{V}{\SX}{\infty}.
	$
	In particular, for all $f\in\GewFunk$ and $x \in V$, we have the estimate
	\begin{equation}\label{est:Abschaetzung_gewichteter_FWert_der_K-Inversion}
		\abs{f(x)}\,\norm{\InvIdcW{V}(\phi)(x)}
		\leq
		\frac{\abs{f(x)}\,\norm{\phi(x)}}{1 - \hn{\phi }{1_{U}}{1}} .
	\end{equation}
\end{prop}
\begin{proof}
	By the inverse function theorem, $\InvIdcW{V}(\phi)$ is smooth.
	We prove by induction that
	$\InvIdcW{V}(\phi) \in \CcF{V}{\SX}{k}$
	for all $k \in \N$.
	
	$k = 0$:
	We compute for $f \in \GewFunk$ and $x \in V$
	using \refer{id:inversion_psi_gleichungen2} and \eqref{est:Funktionswerte_Gewicht_K-Kompo} that
	\begin{align*}
		\abs{f(x)}\,\norm{\InvIdcW{V}(\phi)(x)}
		= \abs{f(x)}\,\norm{\phi(\InvIdcW{V}(\phi)(x) + x)}
		&\leq \abs{f(x)}( \hn{\FAbl{\phi}}{1_{U}}{0} \norm{\InvIdcW{V}(\phi)(x)} + \norm{\phi(x)}) ;
	\end{align*}
	here we used that $\hn{\InvIdcW{V}(\phi)}{1_V}{0} < r$ by \refer{lem:Inverse_in_BC0}.
	From this we can derive \eqref{est:Abschaetzung_gewichteter_FWert_der_K-Inversion} since
	$\hn{\FAbl{\phi} }{1_{U}}{0} = \hn{\phi }{1_{U}}{1} < 1$,
	and we see that $\InvIdcW{V}(\phi) \in \CcF{V}{\SX}{0}$.

	$k \to k + 1$:
	Using \refer{prop:topologische_Zerlegung_von_CFk} (and
	the induction base), we see that
	\[
		\InvIdcW{V}(\phi) \in \CcF{V}{\SX}{k + 1}
		\iff
		\FAbl{ \,\InvIdcW{V}(\phi)} \in \CcF{V}{\Lin{\SX}{\SX}}{k};
	\]
	the second condition shall be verified now.
	Remember that we already provided an identity for $\FAbl{ \,\InvIdcW{V}(\phi)}$
	in \eqref{id:Differential_der_inversen_Abb}.
	We use \refer{lem:CkF(,A)_topologische_Algebra} and
	\refer{cor:Komposition_linearer_Abb_und_CF} to see that
	\[
		\FAbl{\phi} \MaMu \QuasiInv(- \FAbl{\phi})
			\in\CinfF{U}{\Lin{\SX}{\SX}}.
	\]
	Since we know from the induction hypothesis that
	$\InvIdcW{V}(\phi) \in \CcF{V }{\SX}{k}$,
	we derive from \refer{id:Differential_der_inversen_Abb}
	and \refer{prop:Kompo_Koord_glatt}
	(applied on $\CcF{U}{\SX}{\infty} \times \CcF{V}{\Ball{0}{r} }{k}$)
	that
	\begin{equation*}
		\FAbl{\, \InvIdcW{V}(\phi)} =
		\compIdSmoothLcW{\Lin{\SX}{\SX}}{ k}
		(\FAbl{\phi} \MaMu \QuasiInv(- \FAbl{\phi}) - \FAbl{\phi} ,\, \InvIdcW{V}(\phi))
		\in \CcF{V}{\Lin{\SX}{\SX}}{k},
	\end{equation*}
	which finishes the proof.
\end{proof}
%
\subsubsection{On the domain and the smoothness of the inversion map}
We investigate the smoothness of $\InvIdcW{V}$.
Later, we dicuss when $\InvertWinfDomRan{U}{V}$ is an open $0$-neighborhood.
Finally, we conclude that the inversion on $\DiffW$ is smooth.
\paragraph{Smoothness of the inversion map}
Here, we assume that $\InvertWinfDomRan{U}{V}$ contains a suitable open set.
\begin{prop}\label{prop:Stetigkeit_der_Inversion_und_Offenheit}\label{prop:KoorInversion_stetig_glatt}
	Let $\SX$ be a Banach space, $U, V \sub \SX$ open nonempty subsets
	such that $V + \Ball{0}{r} \sub U$ for some $r > 0$.
	Further, let $\GewFunk \subseteq \cl{\R}^{U}$ with $1_{U} \in \GewFunk$.
	Let $G_\GewFunk \sub \InvertWinfDomRan{U}{V}$ be an open nonempty set such that
	for each $\phi \in G_\GewFunk$, $\hn{\phi }{1_{U}}{0} < r$ and $\hn{\phi }{1_{U}}{1} < 1$.
	\begin{enumerate}
		\item\label{enum1:Differenz_KoorInversion}
		Then for $\phi, \psi \in G_\GewFunk$, the following identity holds:
		\begin{equation}\label{id:Differenz_KoorInversion}\tag{\ensuremath{\dagger}}
			\InvIdcW{V}(\psi) - \InvIdcW{V}(\phi) 
			=
			( \idco + T_{\psi, \phi})^{-1} \eval \compIdcW{\SX}(\phi - \psi,\InvIdcW{V}(\phi)),
		\end{equation}
		where the inversion is the pointwise inversion in $\Lin{\SX}{\SX}$, and
		\[
			T_{\psi, \phi} \ndef \Mint{ \compIdcW{\Lin{\SX}{\SX}}(\FAbl{\psi}, t \InvIdcW{V}(\phi) +  (1-t) \InvIdcW{V}(\psi)) }{t}
			\in \CcF{V}{\Lin{\SX}{\SX}}{\infty} .
		\]
		In particular, for $f \in \GewFunk$ we have the estimate
		\begin{equation}\label{est:f0-norm_Diff_KoorInv}
			\hn{\InvIdcW{V}(\psi) - \InvIdcW{V}(\phi)}{f}{0}
			\leq
			\tfrac{1 }{ 1 - \hn{\psi}{1_U}{1} }
			\left(
				\hn{\phi - \psi}{1_U}{1} \tfrac{\hn{\phi}{f}{0} }{ 1 - \hn{\phi}{1_U}{1}}
				+   \hn{\phi - \psi}{f}{0}
			\right).
		\end{equation}
		
		\item\label{enum1:KoorInversion_stetig}
		$
			G_\GewFunk \to \CcF{V}{\SX}{\infty}
			:\phi \mapsto \InvIdcW{V}(\phi)
		$
		is continuous.
		
		\item\label{enum1:KoorInversion_diffbar}
		$
			G_\GewFunk \to \CcF{V}{\SX}{\infty}
			:\phi \mapsto \InvIdcW{V}(\phi)
		$
		is smooth with
		\begin{equation}\label{id:Ableitung_Inversion}
			\dA{\InvIdcW{V}}{\phi}{\phi_1}
			= - \compIdcW{\SX}(\QuasiInv(\FAbl{\phi}) \eval \phi_1 + \phi_1,  \InvIdcW{V}(\phi) ).
		\end{equation}
	\end{enumerate}
\end{prop}
\begin{proof}
	By \refer{prop:Inverse_von_Abb._mit_Opnorm_kleiner_1_auf_X-B(0,r)_ist_in_CW(X-B(0,R))},
	$\InvIdcW{V}(G_\GewFunk) \sub \CcF{V}{\SX}{\infty}$, which we will use implicitly.
	
	\refer{enum1:Differenz_KoorInversion}
	Let $\phi, \psi \in G_\GewFunk$. We compute for $x \in V$ with
	\refer{id:inversion_psi_gleichungen2}, the mean value theorem (using that $\Ball{0}{r}$ is convex) and by adding
	$0 = \psi( \InvIdcW{V}(\phi)(x) + x) - \psi( \InvIdcW{V}(\phi)(x) + x)$
	that
	\begin{align*}
		&\InvIdcW{V}(\psi)(x) - \InvIdcW{V}(\phi)(x)\\
		= & \psi( \InvIdcW{V}(\phi)(x) + x) - \psi( \InvIdcW{V}(\psi)(x) + x)
			+ \phi( \InvIdcW{V}(\phi)(x) + x) - \psi( \InvIdcW{V}(\phi)(x) + x)\\
		= & \Mint%
			{\FAbl{\psi}( t \InvIdcW{V}(\phi)(x) +  (1-t) \InvIdcW{V}(\psi)(x) + x) \eval(\InvIdcW{V}(\phi)(x) - \InvIdcW{V}(\psi)(x))
			}{t}
		\\
		& \qquad + \compIdcW{\SX}(\phi - \psi,\InvIdcW{V}(\phi))(x) ;
	\end{align*}
	note that the identity
	$\compIdcW{\SX}(\phi - \psi,\InvIdcW{V}(\phi)) = (\phi - \psi)\circ ( \InvIdcW{V}(\phi) + \id{V})$
	holds because of \refer{lem:Inverse_in_BC0} and the definition of $r$.
	From this identity we can derive \eqref{id:Differenz_KoorInversion};
	note that the integral defining $T_{\psi, \phi}$ exists because $\CcF{V}{\Lin{\SX}{\SX} }{\infty}$ is complete.
	Further
	\[
		\norm{T_{\psi, \phi}(x)} \leq \hn{\psi}{1_U}{1} < 1
	\]
	for all $x \in V$,
	hence each $\id{\SX} + T_{\psi, \phi}(x)$ is invertible.
	With the Neumann series, we have
	\[
		\hn{(\idco + T_{\psi, \phi})^{-1}}{1_V}{0}
		\leq
		\frac{1}{ 1 - \hn{\psi}{1_U}{1} }.
	\]
	So we see using \eqref{est:Funktionswerte_Gewicht_K-Kompo} and \eqref{est:Abschaetzung_gewichteter_FWert_der_K-Inversion}
	that \refer{est:f0-norm_Diff_KoorInv} holds.
	
	\refer{enum1:KoorInversion_stetig}
	By \refer{cor:Topologie_von_CinfF}, $\InvIdcW{V}$ is continuous iff the corresponding maps
	\[
		I_{\ell}: G_\GewFunk \to \CcF{V}{\SX}{\ell}
	\]
	are so for each $\ell \in \N$.
	We shall verify this condition by induction on $\ell$.
	
	$\ell = 0$:
	We use \eqref{est:f0-norm_Diff_KoorInv} to see that $I_{0}$ is continuous in $\phi$.
	
	$\ell \to \ell + 1$:
	Because of \refer{prop:topologische_Zerlegung_von_CFk} (and	the induction base)
	$I_{\ell + 1}$ is continuous iff
	$\FAbl{}\circ I_{\ell + 1}: G_\GewFunk\to\CcF{V}{\Lin{\SX}{\SX}}{\ell}$
	is so.
	Using \refer{id:Differential_der_inversen_Abb}, we see that for $\phi\in G_\GewFunk$
	\[
		(\FAbl{}\circ I_{\ell + 1})(\phi)
		= \compIdSmoothLcW{\Lin{\SX}{\SX} }{\ell}
		(\FAbl{\phi} \MaMu \QuasiInv(- \FAbl{\phi}) - \FAbl{\phi}, \, I_\ell(\phi))
	\]
	holds, where $\QuasiInv \ndef \QuasiInv_{ \CcF{ U}{ \Lin{\SX}{\SX} }{\infty} }$.
	Since $\compIdSmoothLcW{\Lin{\SX}{\SX} }{\ell}$, $\FAbl{}$, $\MaMu$,
	$\QuasiInv$ and $I_{\ell}$ are continuous (see
	\refer{prop:Kompo_Koord_glatt},
	\refer{prop:Ableitung_ist_stetig},
	\refer{cor:Komposition_linearer_Abb_und_CF},
	\refer{lem:CkF(,A)_topologische_Algebra}
	and the inductive hypothesis, respectively),
	we conclude that $\FAbl{}\circ I_{\ell + 1}$ is continuous.
	
	\refer{enum1:KoorInversion_diffbar}
	We prove by induction that $\InvIdcW{V}$ is a $\ConDiff{}{}{k}$ map for all
	$k \in \N$.
	
	$k = 1:$
	Let $\phi \in G_\GewFunk $, $\phi_1 \in \CcF{U}{\SX}{\infty}$
	and $t \in \K^\ast$ such that $\phi + t \phi_1 \in  G_\GewFunk$.
	We use \eqref{id:Differenz_KoorInversion} to see that
	\begin{equation*}
		\frac{\InvIdcW{V}(\phi + t \phi_1) - \InvIdcW{V}(\phi) }{t}
		= (\idco + T_{\phi + t \phi_1, \phi})^{-1}
			\eval \compIdcW{\SX}( -\phi_1,\InvIdcW{V}(\phi)).
	\end{equation*}
	Using \refer{prop:Stetigkeit_parameterab_Int}, we see that
	\[
		\lim_{t\to 0} T_{\phi + t \phi_1, \phi}
		= \compIdcW{\Lin{\SX}{\SX}}(\FAbl{\phi}, \InvIdcW{V}(\phi) ).
	\]
	Further, by \refer{lem:Adversion_und_Inversion}
	\[
		(\idco + T_{\phi + t \phi_1, \phi})^{-1} - \idco + \idco
		= \QuasiInv(T_{\phi + t \phi_1, \phi} ) + \idco
	\]
	where $\QuasiInv \ndef \QuasiInv_{ \CcF{ V}{ \Lin{\SX}{\SX} }{\infty} }$.
	Using that $\QuasiInv$ and $\eval$ are continuous
	by \refer{lem:CkF(,A)_topologische_Algebra} and \refer{prop:multilineare_Abb_und_CF},
	we therefore get
	\begin{align*}
		\lim_{t\to 0} \frac{\InvIdcW{V}(\phi + t \phi_1) - \InvIdcW{V}(\phi) }{t}
		&= (\QuasiInv(\compIdcW{\Lin{\SX}{\SX}}(\FAbl{\phi}, \InvIdcW{V}(\phi) ) ) + \idco)
			\eval \compIdcW{\SX}( -\phi_1,\InvIdcW{V}(\phi))
		\\
		&= - (\compIdcW{\Lin{\SX}{\SX}}(\QuasiInv(\FAbl{\phi}), \InvIdcW{V}(\phi) ) ) + \idco)
			\eval \compIdcW{\SX}( \phi_1, \InvIdcW{V}(\phi))
		\\
		&= - \compIdcW{\SX}(\QuasiInv(\FAbl{\phi}) \eval \phi_1 + \phi_1, \InvIdcW{V}(\phi));
	\end{align*}
	here we used that $\QuasiInv(\Phi) = \QuasiInv_{\Lin{\SX}{\SX}} \circ \Phi$.
	Since we proved in \refer{enum1:KoorInversion_stetig} that $\InvIdcW{V}$ is continuous,
	we see from this that $\InvIdcW{V}$ is $\ConDiff{}{}{1}$ and \eqref{id:Ableitung_Inversion} holds.
	
	$k \to k + 1$:
	Since $\InvIdcW{V}$ is $\ConDiff{}{}{k}$, we conclude from \eqref{id:Ableitung_Inversion}
	and the fact that $\FAbl{}$, $\eval$, $\compIdcW{\SX}$ and $\QuasiInv$
	are smooth (see \refer{prop:Ableitung_ist_stetig},
	\refer{cor:Auswertung_linearer_abb_und_CF}
	(together with \refer{beisp:Helge_diffbare_Abb}),
	\refer{prop:Kompo_Koord_glatt}
	and \refer{lem:CkF(,A)_topologische_Algebra}, respectively)
	that $\dA{\InvIdcW{V}}{}{}$ is $\ConDiff{}{}{k}$.
	Hence $\InvIdcW{V}$ is $\ConDiff{}{}{k + 1}$ by definition.
\end{proof}
\paragraph{$\boldsymbol{\InvertWinfDomRan{U}{V}}$ contains open sets}
We show that $\InvertWinfDomRan{U}{V}$ is a neighborhood of $0$
if $V \sub U$ and $\dist{V}{\SX\setminus U} > 0$.
To this end, we need the following two technical lemmas.
\begin{lem}\label{lem:Lipschitz_gestoerter_OP-konvexer_Defb-injektiv}\label{lem:Lipschitz_gestoerter_OP-Defb_alles-surjektiv}
	Let $\SX$ be a Banach space, $U$ a convex open nonempty subset
	and $\phi \in \FC{U}{\SX}{1}$ such that $\hn{\phi}{1_U}{1} < 1$.
	Then the map $\id{U} + \phi$ is injective.
\end{lem}
\begin{proof}
	Let $x, y \in U$. Then
	\[
		(\id{U} + \phi)(y) - (\id{U} + \phi)(x)
		= y - x + \Mint{\FAbl{\phi}(t y + (1 - t) x) (y - x)}{t}.
	\]
	Since $\hn{\phi}{1_U}{1} < 1$, the norm of the integral is smaller then $\norm{y - x}$.
	We deduce with the triangle inequality that for $x \neq y$, $(\id{U} + \phi)(y) \neq (\id{U} + \phi)(x)$.
\end{proof}

\begin{lem}\label{lem:On_the_Image_of_Lipschitz_gestoerter_OP}
	Let $\SX$ be a Banach space, $U\sub \SX$ an open nonempty subset
	and $r > 0$.
	Let $\phi \in \BC{U}{\SX}{0}$ with
	$
		\hn{\phi}{1_U}{0} < r .
	$
	Further, let $y \in (\id{U} + \phi)(U)$ such that $\Ball{y}{2 r} \sub U$.
	Then for any $\psi \in \FC{U}{\SX}{1}$ with
	$
		\hn{\psi }{1_{U}}{1} < 1
	$
	and
	$
		\hn{\psi - \phi}{1_U}{0} < r (1 - \hn{\psi }{1_{U}}{1}),
	$
	$y \in (\id{U} + \psi)(U)$.
\end{lem}
\begin{proof}
	There exists $x \in U$ with $x + \phi(x) = y$. Then
	$\norm{y - x} = \norm{\phi(x)} < r$, and hence $\Ball{x}{r} \sub U$ by the triangle inequality.
	Further, we derive from the Lipschitz inverse function theorem (\refer{cor:Lipschitz_inverse-function_theorem}) that
	$\Ball{(\id{U} + \psi)(x)}{r(1 - \hn{\psi }{1_{U}}{1})}$ is contained in the image of $\id{U} +\psi$,
	and since
	\[
		\norm{y - (x + \psi(x))}
		= \norm{\phi(x) - \psi(x)}
		< r (1 - \hn{\psi }{1_{U}}{1}),
	\]
	$y$ is contained in the image of $\id{U} + \psi$.
\end{proof}

\begin{lem}\label{lem:Offenheit_H_UV}
	Let $\SX$ be a Banach space, $U, V \sub \SX$ open nonempty subsets
	such that $U$ is convex and there exists $r > 0$ with $V + \Ball{0}{r} \sub U$.
	Further, let $\GewFunk \subseteq \cl{\R}^{U}$ with $1_{U} \in \GewFunk$
	and $\phi \in \InvertWinfDomRan{U}{V}$ such that
	$
		\hn{\phi }{1_{U}}{1} < 1
	$
	and
	$
		\hn{\phi}{1_U}{0} < \frac{r}{2}  .
	$
	Then for any $\eps > 0$ such that $\hn{\phi }{1_{U}}{1} + \eps < 1$,
	\[
		\left\set{\psi \in \CcF{U}{\SX}{\infty} }{%
			\hn{\psi - \phi}{1_U}{1} < \eps
			\text{ and }
			\hn{\psi - \phi}{1_U}{0} < \frac{r}{2} (1 - \eps - \hn{\phi }{1_{U}}{1})
		\right}
	\]
	is a  neighborhood of $\phi$  that is contained in $\InvertWinfDomRan{U}{V}$
	and whose image under $\InvIdcW{V}$ is contained in $\CcF{V}{\SX}{\infty}$.
\end{lem}
\begin{proof}
	Let $\psi$ be an element of the neighborhood. Then $\hn{\psi}{1_U}{1} < 1$, hence we can apply
	\refer{lem:Lipschitz_gestoerter_OP-konvexer_Defb-injektiv}
	to see that $\id{U} + \psi$ is injective.
	Further, since
	\[
		\hn{\psi - \phi}{1_U}{0}
		< \frac{r}{2} (1 - \eps - \hn{\phi }{1_{U}}{1})
		< \frac{r}{2} (1 - \hn{\psi }{1_{U}}{1}),
	\]
	we see with \refer{lem:On_the_Image_of_Lipschitz_gestoerter_OP} that
	$V \sub (\id{U} + \psi)(U)$; hence $\psi \in \InvertWinfDomRan{U}{V}$.
	Finally, we can apply \refer{prop:Inverse_von_Abb._mit_Opnorm_kleiner_1_auf_X-B(0,r)_ist_in_CW(X-B(0,R))}
	since
	\[
		\hn{\psi }{1_U}{0} \leq \hn{\phi }{1_U}{0} + \hn{\psi - \phi }{1_U}{0} < r
	\]
	and see that $\InvIdcW{V}(\psi) \in \CcF{V}{\SX}{\infty}$.
\end{proof}
\subsubsection{The Lie group $\boldsymbol{\DiffW}$}
We put it all together and see that $\DiffW$ is a Lie group.
\begin{satz}\label{satz:DiffW_Lie-Gruppe}\label{satz:DiffW_ist_offen}
	Let $\SX$ be a Banach space and $\GewFunk \subseteq \cl{\R}^\SX$
	with $1_{\SX} \in \GewFunk$.
	Then $\DiffW$ is an open subset of $\EndW$,
	and a Lie group when endowed with the canonical differential structure.
	\index{diffeomorphisms!groups of}%
\end{satz}
\begin{proof}
	We established in \refer{cor:EndW_glattes_Monoid} that
	$\EndW$ is a smooth monoid with the unit group $\DiffW$.
	By \refer{lem:topologische_Monoide}, $\DiffW$ is open in $\EndW$
	if there exists an open neighborhood of $\id{\SX}$ in $\EndW$ that is contained in $\DiffW$.
	Moreover, the inversion is smooth if it is so on this neighborhood
	(The proof for the continuity is in \refer{lem:topologische_Monoide},
	the smoothness can be derived from \refer{lem:Erzeugung_von_Liegruppen_aus_lokalen}).
	To this end, we set
	\[
		U_\GewFunk \ndef
		\set{\phi \in \CcF{\SX}{\SX}{\infty} }{\hn{\phi}{1_\SX}{1} < 1}.
	\]
	Then for $\phi \in U_\GewFunk$, we have that $\kF(\phi) \in \DiffW$ since
	we can apply \refer{lem:Offenheit_H_UV} ($0 \in \InvertWinfDomRan{\SX}{\SX}$)
	and see that
	$\phi \in \InvertWinfDomRan{\SX}{\SX}$
	with
	$(\phi + \id{\SX})^{-1} - \id{\SX} = \InvIdcW{X}(\phi) \in \CcF{\SX}{\SX}{\infty}$;
	enabling us to use \refer{lem:char_DiffW}.
	Further, we know from \refer{prop:KoorInversion_stetig_glatt} that $\InvIdcW{\SX}$ is smooth on $U_\GewFunk$.
	This implies that the inversion map is smooth on $\kF(U_\GewFunk)$,
	see the commutative diagram
	\[
		\xymatrix{
		{\kF(U_\GewFunk)} \ar[rr]^{{}^{-1}} & & {\DiffW} \ar@{>->}[d]_{\kF^{-1}}
		\\
		{U_\GewFunk} \ar[rr]^{\InvIdcW{\SX}} \ar@{>->>}[u]^{\kF} & & {\CcF{\SX}{\SX}{\infty}} .
		}
	\]
	This finishes the proof.
\end{proof}
\subsection{On decreasing weighted diffeomorphisms and dense subgroups}
We define the set $\DiffWvan$ of \emph{decreasing weighted diffeomorphisms}
and show that it is a closed normal subgroup of $\DiffW$
which can be turned into a Lie group.
Further, we give sufficient conditions on $\GewFunk$ ensuring that the
group $\Diffc$ of compactly supported diffeomorphisms is dense in $\DiffWvan$.
\paragraph{Inversion on weighted diffeomorphisms}
First, we have to discuss the inversion map restricted to weighted functions.
\begin{lem}\label{lem:HUVsein-Mengen_verkleinern}
	Let $\SX$ be a Banach space, $U, V \sub \SX$ open nonempty subsets,
	$\phi \in \InvertWinfDomRan{U}{V} \cap \BC{U}{\SX}{0}$,
	$r > 0$ with $r > \hn{\phi}{1_U}{0}$ and
	$\widetilde{U} \sub U$ and $\widetilde{V} \sub V$ open nonempty subsets
	such that $\widetilde{V} + \Ball{0}{r} \sub \widetilde{U}$.
	\begin{enumerate}
		\item\label{enum1:HUVsein-Mengen_verkleinern-a}
		Then $\phi \in \InvertWinfDomRan{\widetilde{U}}{\widetilde{V}}$.

		\item\label{enum1:HUVsein-Mengen_verkleinern-b}
		In particular, if $R > s > 0$, $U = V = \SX$
		and $\hn{\phi}{1_\SX}{0} < R - s$, then
		$\phi \in \InvertWinfDomRan{ \SX\setminus\clBall{0}{s} }{ \SX\setminus\clBall{0}{R} }$.
	\end{enumerate}
\end{lem}
\begin{proof}
	\refer{enum1:HUVsein-Mengen_verkleinern-a}
	Obviously $\phi + \id{U}$ is injective on $\widetilde{U}$,
	so we just need to show that $\widetilde{V} \sub (\phi + \id{U})(\widetilde{U})$.
	To this end, let $y \in \widetilde{V}$.
	Since $\phi \in \InvertWinfDomRan{U}{V}$ and $\widetilde{V} \sub V$,
	there exists $x \in U$ with $\phi(x) + x = y$.
	This implies that $\norm{y - x} \leq \hn{\phi}{1_U}{0} < r$,
	and hence
	\[
		x = y + x - y \in \widetilde{V} + \Ball{0}{r} \sub \widetilde{U}.
	\]

	\refer{enum1:HUVsein-Mengen_verkleinern-b}
	This is an easy application of \refer{enum1:HUVsein-Mengen_verkleinern-a}
	since by the triangle inequality
	$\SX\setminus\clBall{0}{R} + \Ball{0}{R - s} \sub \SX\setminus\clBall{0}{s}$.
\end{proof}
\begin{lem}\label{lem:Inverse_von_CWvan_ist_(fast)_in_CWvan}
	Let $\SX$ be a Banach space, $\GewFunk \subseteq \cl{\R}^{\SX}$ with
	$1_{\SX} \in \GewFunk$ and $\phi \in \InvertWinfDomRan{\SX}{\SX} \cap \CcFvan{\SX}{\SX}{\infty}$.
	Then there exists an $R > 0$ such that
	\[
		\rest{\InvIdcW{X}(\phi)}{ \SX\setminus\clBall{0}{R} } \in \CcFvan{\SX\setminus\clBall{0}{R}}{\SX}{\infty}.
	\]
\end{lem}
\begin{proof}
	Since $\phi \in \CcFvan{\SX}{\SX}{\infty}$, there exists an $r >0$ such that
	$
		\sup_{x \in \SX\setminus\clBall{0}{r}} \Opnorm{\FAbl{\phi}(x)}
		< 1 .
	$
	We choose $R > 0$ such that $\hn{\phi }{1_\SX}{0} + r < R$.
	We see with \refer{lem:HUVsein-Mengen_verkleinern}
	that $\phi \in \InvertWinfDomRan{\SX\setminus\clBall{0}{r} }{\SX\setminus\clBall{0}{R} }$,
	and this allows the application of
	\refer{prop:Inverse_von_Abb._mit_Opnorm_kleiner_1_auf_X-B(0,r)_ist_in_CW(X-B(0,R))}
	to see that
	$\InvIdcW{\SX\setminus\clBall{0}{R} }(\phi) \in \CcF{\SX\setminus\clBall{0}{R}}{\SX}{\infty}$.
	Further, by \refer{id:inversion_psi_gleichungen2}
	\[
		\InvIdcW{\SX\setminus\clBall{0}{R}}(\phi)
		= -\phi \circ (\InvIdcW{\SX\setminus\clBall{0}{R}}(\phi) + \id{\SX\setminus\clBall{0}{R}})
		= \compIdcW{\SX}(-\phi, \InvIdcW{\SX\setminus\clBall{0}{R}}(\phi)),
	\]
	hence an application of \refer{lem:EndWwan_und_Kompostion} finishes the proof.
\end{proof}
\paragraph{A normal Lie subgroup}
To derive the desired result, we need the following technical lemma.
\begin{lem}\label{lem:Differenz_DiffW-DiffW_nach_DiffWvan}
	Let $\SX$ be a Banach space and $\GewFunk \subseteq \cl{\R}^{\SX}$ with
	$1_{\SX} \in \GewFunk$. Further, let $\phi \in \EndWvan$ and $\psi \in \DiffW$.
	Then $\psi - \psi \circ \phi \in \CcFvan{\SX}{\SX}{\infty}$.
\end{lem}
\begin{proof}
	We calculate using \refer{lem:Kriterium_Integrierbarkeit_in_CW} and the mean value theorem that
	\[
		\psi - \psi \circ \phi
		= \Rint{0}{1}{ \FAbl{\psi}(\id{\SX} + t (\phi - \id{\SX})) \eval (\phi - \id{\SX}) }{t}.
	\]
	Since $\FAbl{\psi} \in \BC{\SX}{\Lin{\SX}{\SX}}{\infty}$,
	we conclude with \refer{prop:Kompo_Koord_glatt} that
	$\FAbl{\psi}(\id{\SX} + t (\phi - \id{\SX})) \in \BC{\SX}{\Lin{\SX}{\SX}}{\infty}$.
	Since $\phi - \id{\SX} \in \CcFvan{\SX}{\SX}{\infty}$,
	the assertion follows from \refer{cor:multilineare_Abb_und_CFvan}
	and the fact that $\CcFvan{\SX}{\SX}{\infty}$ is closed in $\CcF{\SX}{\SX}{\infty}$.
\end{proof}

\begin{prop}\label{prop:DiffWvan_UnterLiegruppen_von_DiffW}
	Let $\SX$ be a Banach space and $\GewFunk \subseteq \cl{\R}^{\SX}$ with
	$1_{\SX} \in \GewFunk$. The set
	\[
		\glstext{gew_abfallende_Diffeomorphismen_Braum} \ndef
		\DiffW \cap \EndWvan
		= \set{\phi \in \DiffW}{\phi - \id{\SX} \in \CcFvan{\SX}{\SX}{\infty}}
	\]
	\index{weighted diffeomorphisms!decreasing}%
	\index{diffeomorphisms!groups of}%
	is a closed normal Lie subgroup of $\DiffW$.
\end{prop}
\begin{proof}
	In \refer{cor:EndWvan_Untermonoid_von_EndW} it was proved that $\EndWvan$ is a
	smooth submonoid of $\EndW$ and a closed subset.
	Since $\DiffW$ is open in $\EndW$, we conclude that
	$\DiffWvan$ is a smooth submonoid of $\DiffW$ that is closed.
	Further, it is a direct consequence of
	\refer{lem:Inverse_von_CWvan_ist_(fast)_in_CWvan}
	that the inverse function of an element of $\DiffWvan$ is in $\DiffWvan$,
	whence using \refer{lem:SUbmanifols_sind_initial} we see that the latter is a closed Lie subgroup of $\DiffW$.

	It remains to show that $\DiffWvan$ is normal.
	To this end, let $\phi \in \DiffWvan$ and $\psi \in \DiffW$.
	Then
	\[
		\psi \circ \phi \circ \psi^{-1} - \id{\SX}
		= \psi \circ \phi \circ \psi^{-1} - \psi \circ \phi^{-1} \circ \phi \circ \psi^{-1}
		= (\psi - \psi \circ \phi^{-1}) \circ \phi \circ \psi^{-1},
	\]
	so we derive the assertion from \refer{lem:Differenz_DiffW-DiffW_nach_DiffWvan}
	and \refer{lem:EndWwan_und_Kompostion}.
\end{proof}
\paragraph{On the density of compactly supported diffeomorphisms}
As promised, we give a sufficient criterium on $\GewFunk$ that makes $\Diffc$
a dense subgroup of $\DiffW$.
\begin{lem}\label{lem:bedingungen_an_Gewichte_fuer_Dichte_Testfunktionen}
	Let $\SX$ and $\SY$ be finite-dimensional normed spaces
	and $\UF \subseteq \SX$ an open nonempty set.
	Further, let $\GewFunk \subseteq \R^\UF$ a set of weights such that
	\begin{equation}\label{bedingungen_an_Gewichte_fuer_Dichte_Testfunktionen}
		\begin{aligned}
			&\bullet \GewFunk \subseteq \ConDiff{\UF}{[0,\infty[}{\infty}\\
			&\bullet (\forall x \in \UF) (\exists f \in \GewFunk)\, f(x) > 0\\
			&\bullet 
			\begin{aligned}(\forall f_1, \dotsc,f_n \in \GewFunk)(\forall k_1,\dotsc,k_n \in \N) (\exists f \in \GewFunk, C > 0)
			\\(\forall x \in \UF) \Opnorm{\FAbl[k_1]{f_1}(x)} \dotsb \Opnorm{\FAbl[k_n]{f_n}(x)} \leq C f(x).
			\end{aligned}
		\end{aligned}
	\end{equation}
	Then $\CF{\UF}{\SY}{c}{\infty}$ is dense in $\CcFvan{\UF}{\SY}{r}$.
\end{lem}
\begin{proof}
	A proof can be found in \cite[\S \RN{5}, 19 b)]{MR0442658}.
\end{proof}

\begin{lem}
	Let $\SX$ be a finite-dimensional normed space,
	$\GewFunk \subseteq \R^\SX$ such that $1_\SX \in \GewFunk$
	and \eqref{bedingungen_an_Gewichte_fuer_Dichte_Testfunktionen}
	is satisfied (where $\UF = \SX$).
	Then the set of compactly supported diffeomorphisms $\Diffc$ is dense in $\DiffWvan$.
	\index{compactly supported diffeomorphisms!density in $\DiffWvan$}%
\end{lem}
\begin{proof}
	The set
	$M_\GewFunk^\circ \ndef \kF^{-1}(\DiffW) \cap \CcFvan{\SX}{\SX}{\infty} = \kF^{-1}(\DiffWvan)$
	is open in $\CcFvan{\SX}{\SX}{\infty}$,
	and hence $M_c \ndef \DCcInf{\SX}{\SX} \cap M_\GewFunk^\circ$
	is dense in $M_\GewFunk^\circ$ by \refer{lem:bedingungen_an_Gewichte_fuer_Dichte_Testfunktionen}.
	But $M_c = \kF^{-1}(\Diffc)$, from which the assertion follows.
\end{proof}

\subsection{On diffeomorphisms that are weighted endomorphisms}
It is obvious that the relation
\[
	\Diff{\SX}{}{\cW} \subseteq \Endos{\SX}{\cW} \cap \Diff{\SX}{}{}
\]
holds.
We give a sufficient criterion on $\cW$ that ensures that these two sets are identical,
provided that $\SX$ is finite-dimensional.
Further we show that
$\Diff{\R}{}{ \{1_{\R}\} } \neq \Endos{\R}{\{1_{\R}\}} \cap \Diff{\SX}{}{}$.
\begin{prop}
	Let $\SX$ be a \emph{finite-dimensional} Banach space
	and $\cW \subseteq \cl{\R}^\SX$
	with $1_\SX \in \cW$. If there exists $\widehat{f} \in \cW$ such that
	\begin{equation}\label{bed:Bedingung_Gleichheit_DiffW_Diff_cap_EndW}
		(\forall R > 0) (\exists r > 0)\,
		\norm{x} \geq r \implies \abs{\widehat{f}(x)} \geq R
	\end{equation}
	and if each function in $\cW$ is bounded on bounded sets, then
	\[
		\Diff{\SX}{}{\cW} =
		\Endos{\SX}{\cW} \cap \Diff{\SX}{}{}.
	\]%
	\index{weighted diffeomorphisms!easier description}%
\end{prop}
\begin{proof}
	It remains to show that
	\[
		\Endos{\SX}{\cW} \cap \Diff{\SX}{}{} \subseteq \Diff{\SX}{}{\cW}.
	\]
	So let $\psi$ be in $\Endos{\SX}{\cW} \cap \Diff{\SX}{}{}$.
	Then $\phi \ndef \psi - \id{\SX} \in \InvertWinfDomRan{\SX}{\SX}$,
	and the equivalences
	\begin{multline*}
		\psi \in \Diff{\SX}{}{\cW}
		\iff
		\psi^{-1} \in \Endos{\SX}{\cW}
		\\
		\iff
		\psi^{-1} - \id{\SX} \in \CF{\SX}{\SX}{\cW}{\infty}
		\iff
		\InvIdcW{\SX}(\phi) \in \CF{\SX}{\SX}{\cW}{\infty}
	\end{multline*}
	hold
	(see \refer{lem:char_DiffW} and the definition of $\InvIdcW{X}$
	in \eqref{id:Definition_der_verallgemeinerten_Koor-Inversion}).
	The last statement clearly holds iff
	\[
		(\exists R, r > 0)\,
		\rest{\InvIdcW{\SX}(\phi) }{\SX\setminus\clBall{0}{R}} \in \CF{\SX\setminus\clBall{0}{R}}{\SX}{\cW}{\infty}
		\text{ and }
		\rest{\InvIdcW{\SX}(\phi)}{\Ball{0}{R + r} } \in \CF{\Ball{0}{R + r}}{\SX}{\cW}{\infty},
	\]
	and this shall be proved now.
	Obviously $\rest{\InvIdcW{\SX}(\phi)}{\Ball{0}{R}} \in \CF{\Ball{0}{R}}{\SX}{\cW}{\infty}$ for each $R > 0$
	because each $f \in \cW$ is bounded on bounded sets,
	the maps $\FAbl[\ell]{\InvIdcW{\SX}(\phi)}$ are continuous and each
	 bounded subset of $\SX$ is relatively compact (as $\SX$ is finite-dimensional).
	It remains to show that there exists $R  >0$ such that
	$
		\rest{\InvIdcW{\SX}(\phi)}{\SX\setminus\clBall{0}{R} } \in \CF{\SX\setminus\clBall{0}{R}}{\SX}{\cW}{\infty}.
	$
	We set $K_{\phi} \ndef \hn{\phi}{\widehat{f} }{1 } < \infty$
	and conclude from \eqref{bed:Bedingung_Gleichheit_DiffW_Diff_cap_EndW}
	that there exists an $r_{\phi}$ with
	\[
		\norm{x} \geq r_{\phi}
		\implies
		\abs{\widehat{f}(x)} \geq K_{\phi} + 1 .
	\]
	Since
	$
		\abs{\widehat{f}(x)}\, \Opnorm{\FAbl{\phi}(x)} \leq K_{\phi}
	$
	for each $x \in \SX$, we see that
	\[
		\hn{ \rest{\phi}{\SX\setminus\clBall{0}{r_\phi} } }{1_{\SX}}{1}
		\leq \frac{K_{\phi}}{K_{\phi} + 1}
		< 1.
	\]
	We choose $R_\phi > 0$ such that $R_\phi  > r_\phi + \hn{\phi}{1_\SX}{0}$.
	We see with \refer{lem:HUVsein-Mengen_verkleinern} that
	$\phi \in \InvertWinfDomRan{\SX\setminus \clBall{0}{r_\phi}}{\SX\setminus \clBall{0}{R_\phi}}$,
	so we can apply \refer{prop:Inverse_von_Abb._mit_Opnorm_kleiner_1_auf_X-B(0,r)_ist_in_CW(X-B(0,R))}
	to see that
	\[
		\rest{\InvIdcW{\SX}(\phi)}{\SX\setminus\clBall{0}{R_\phi} }
		= \InvIdcW{\SX\setminus\clBall{0}{R_\phi} }(\phi)
		\in \CF{\SX\setminus\clBall{0}{R_\phi} }{\SX}{\cW}{\infty},
	\]
	and this finishes the proof.
\end{proof}
We give an affirmative example.
\begin{beisp}
	The space $\Diff{\R^n}{}{\mathcal{S}}$ satisfies
	\refer{bed:Bedingung_Gleichheit_DiffW_Diff_cap_EndW}.
	We just have to set $\widehat{f}(x_{1},\dotsc, x_{n}) = x_{1}^2 + \dotsb + x_{n}^2$
	which clearly is a polynomial function on $\R^n$.
\end{beisp}
As announced, we give a counterexample. As preparation, we prove the following lemma.
\begin{lem}\label{lem:Kriterium_id+func_Diffeo_R}
	Let $\gamma \in \ConDiff{\R}{\R}{\infty}$ be a \emph{bounded} map that satisfies
	\begin{equation*}\label{bed:Abschaetzung_Ableitung}\tag{\ensuremath{\ast}}
		(\forall x\in \R) \: \gamma'(x) > -1.
	\end{equation*}
	Then $\gamma + \id{\R} \in \Diff{\R}{}{}$.
\end{lem}
\begin{proof}
	We conclude from \eqref{bed:Abschaetzung_Ableitung} that
	$
		(\gamma(x) + \id{\R})'(x) > 0
	$
	for all $x \in \R$, so $\gamma + \id{\R}$ is strictly monotone and hence injective.
	Since $\gamma$ is bounded, $\gamma + \id{\R}$ is unbounded above and below
	and hence surjective
	(by the intermediate value theorem).
\end{proof}

\begin{beisp}
	We give an example of a map $\gamma \in \BC{\R}{\R}{\infty}$
	with the property that $\gamma + \id{\R} \in \Diff{\R}{}{}$,
	but $(\gamma + \id{\R})^{-1} - \id{\R} \not\in \BC{\R}{\R}{\infty}$.
	To this end, let $\phi$ be an antiderivative of the function
	$x \mapsto \frac{2}{\pi} \arctan(x)$ with $\phi(0) = 0$.
	Then $\sin \circ\, \phi$ and $\cos \circ\, \phi$ are in $\BC{\R}{\R}{\infty}$
	by a simple induction
	since $\cos, \sin, \arctan \in \BC{\R}{\R}{\infty}$,
	\[
		\tag{\ensuremath{\ast}}
		\label{Ableitung_Beispielfunktion}
		(\sin \circ\, \phi)'(x) = \frac{2}{\pi} \arctan(x) (\cos \circ\, \phi)(x),
	\]
	and an analogous formula holds for $(\cos \circ\, \phi)'$.
	We see with \eqref{Ableitung_Beispielfunktion} that $(\sin \circ\, \phi)'(x) > -1$ for all $x \in \R$,
	so $\sin \circ\, \phi + \id{\R} \in \Diff{\R}{}{}$
	by \refer{lem:Kriterium_id+func_Diffeo_R}.
	But since
	\[
		((\sin \circ\, \phi + \id{\R})^{-1} - \id{\R})'(x)
		= \frac{1}{(\sin \circ\, \phi)'((\sin \circ\, \phi + \id{\R})^{-1}(x)) + 1} - 1
	\]
	and there exists a sequence $(y_{n})_{n\in\N}$ in $\R$ with
	\[
		\lim_{n\to \infty} \frac{2}{\pi} \arctan(y_{n}) (\cos \circ\, \phi)(y_{n}) = -1,
	\]
	$((\sin \circ\, \phi + \id{\R})^{-1} - \id{\R})'$ clearly is not bounded.
\end{beisp}

\section{Regularity}
We prove that the Lie groups $\DiffW$ and $\DiffWvan$ are regular.
For the definition of regularity, see \refer{susec:LieGroups-Regularity}.
\subsection[The regularity differential equation of \texorpdfstring{$\DiffW$}{DiffW}]%
{The regularity differential equation of \texorpdfstring{$\boldsymbol{\DiffW}$}{DiffW}}
\newcommand{\mW}{\ensuremath{m_\cW}}
\newcommand{\TmW}{\ensuremath{\Tang{m_\cW}}}
We examine the general (right) regularity differential equation
(which is stated in \refer{ivp:diffeq_Rechts-regularitaet_allgemein})
and turn it into a differential equation on $\CcF{\SX}{\SX}{\infty}$.
To this end, we first describe the group multiplication of the tangent group
$\Tang{\DiffW}$ and the right action of $\DiffW$ on $\Tang{\DiffW}$ with respect to
the chart $\Tang{\kF^{-1}}$.
\begin{lem}[Tangent group of $\DiffW$]\label{lem:Tangentengruppenop_DiffW_in_Karten}
	Let $\SX$ be a Banach space and $\GewFunk \subseteq \cl{\R}^\SX$ with $1_\SX \in \GewFunk$.
	In the following, we denote the multiplication on $\DiffW$ with respect to
	the chart $\kF^{-1}$ by $\mW$.
	Note that the tangent group $\Tang{\DiffW}$ is canonicly isomorphic to
	$\CcF{\SX}{\SX}{\infty} \rtimes \DiffW$.
	\begin{enumerate}
	\item\label{enum1:Tangentengruppenop_in_Karten}
	The group multiplication $\TmW$ on $\Tang{\DiffW}$
	(with respect to $\Tang{\kF^{-1}}$) is given by
	\[
		\TmW\bigl((\gamma, \gamma_1),(\eta, \eta_1)\bigr)
		=
		\bigl(\mW(\gamma, \eta), \FAbl{\gamma} \circ (\eta + \id{\SX}) \eval \eta_1 + \gamma_1 \circ (\eta + \id{\SX}) + \eta_1\bigr) .
	\]
	\item\label{enum1:Rechtswirkung_in_Karten}
	Let $\phi \in \DiffW$.
	Then the right action $\Tang{\rho_\phi}$ of $\phi$ on $\Tang{\DiffW}$
	with respect to $\Tang{\kF^{-1}}$ is given by
	\[
		\Tang{(\kF^{-1} \circ \rho_\phi \circ \kF)}(\gamma, \gamma_1)
		= \bigl(\mW(\gamma, \kF^{-1}(\phi)),
			\gamma_1 \circ \phi\bigr).
	\]
	\end{enumerate}
\end{lem}
\begin{proof}
	\refer{enum1:Tangentengruppenop_in_Karten}
	We have
	\[
		\mW(\gamma, \eta) = \gamma \circ (\eta + \id{\SX}) + \eta
	\]
	and the commutative diagram
	\[
		\xymatrix{
		{\DiffW \times \DiffW} \ar[rr]^-{\circ} && {\DiffW}
		\\
		{\kF^{-1}(\DiffW) \times \kF^{-1}(\DiffW)} \ar[rr]^-{\mW} \ar@{>->>}[u]|{\kF \times \kF} && {\kF^{-1}(\DiffW)} \ar@{>->>}[u]|{\kF}
		}.
	\]
	The group multiplication on the tangent group is given by applying
	the tangent functor $\Tang{}$ to the group multiplication on $\DiffW$,
	and therefore we obtain the group multiplication on $\Tang{\DiffW}$ in charts
	by applying $\Tang{}$ to $\mW$ (up to a permutation). Since
	\[
		\Tang{\mW}(\gamma, \eta; \gamma_1, \eta_1)
		= \bigl(\mW(\gamma, \eta), \FAbl{\gamma} \circ (\eta + \id{\SX}) \eval \eta_1 + \gamma_1 \circ (\eta + \id{\SX}) + \eta_1\bigr)
	\]
	by \eqref{id:Ableitung_Kompo}, the asserted identity holds.
	
	\refer{enum1:Rechtswirkung_in_Karten}
	Obviously
	$(\kF^{-1} \circ \rho_\phi \circ \kF)(\cdot) = \mW(\cdot,\kF^{-1}(\phi))$,
	so we derive the assertion if we apply
	the identity proved in \refer{enum1:Tangentengruppenop_in_Karten}
	with $\eta = \kF^{-1}(\phi)$ and $\eta_1 = 0$.
\end{proof}
We aim to turn \eqref{ivp:diffeq_Rechts-regularitaet_allgemein}
into an ODE on a vector space. Before we can do this, a definition is useful:
\newcommand{\TkF}{\ensuremath{\Tang{\kF}}}
\begin{defi}\label{def:AWP_Regularitaet_Koordinaten}
	Let $\SX$ be a normed space, $\GewFunk \subseteq \cl{\R}^\SX$ with $1_\SX \in \GewFunk$,
	$k \in \cl{\N}$ and $\cF$ be a subset of $\GewFunk$ with $1_\SX \in \cF$.
	By \refer{prop:Kompo_Koord_glatt}, the map
	\begin{align*}
		F_{\cF,k}
		&:
		[0,1] \times \CF{\SX}{\SX}{\cF}{k} \times
			\ConDiff{[0,1]}{ \CcF{\SX}{\SX}{\infty} }{\infty}
			\to \CF{\SX}{\SX}{\cF}{k}
		\\
		&: (t, \gamma, p) \mapsto p(t) \circ (\gamma + \id{\SX})
	\end{align*}
	is well-defined and smooth (since the evaluation of curves is smooth by \refer{lem:Auswertung_von_Kurven_glatt}).
	For each parameter curve
	$p \in \ConDiff{[0,1]}{ \CcF{\SX}{\SX}{\infty} }{\infty}$,
	we consider the initial value problem
	\begin{align}\label{ivp:DGL_Regularitaet_Koordinaten}
		\begin{split}
			\Gamma'(t) &= F_{\cF,k}(t, \Gamma(t), p)\\
			\Gamma(0) &= 0,
		\end{split}
	\end{align}
	where $t \in [0,1]$.
\end{defi}

\begin{lem}\label{lem:Aequivalenz_Regularitaet_DiffW}
	Let $\SX$ be a Banach space and $\GewFunk \subseteq \cl{\R}^\SX$ with $1_\SX \in \GewFunk$.
	\begin{enumerate}
		\item 
		For $\gamma \in 
		\ConDiff{[0,1]}{ \Tang[\id{\SX}]{\DiffW} }{\infty}$,
		the initial value problem
		\begin{align*}
			\begin{split}
				\eta'(t) &= \gamma(t) \cdot \eta(t)\\
				\eta(0) &= \id{\SX}
			\end{split}
		\end{align*}
		has a smooth solution
		\[
			\rEvol[\DiffW]{\gamma} : [0,1] \to \DiffW
		\]
		iff the initial value problem
		\eqref{ivp:DGL_Regularitaet_Koordinaten}
		(in \refer{def:AWP_Regularitaet_Koordinaten})
		with $\cF = \cW$, $k = \infty$ and $p = \dA{\kF^{-1}}{}{} \circ \gamma$
		has a smooth solution
		\[
			\Gamma_p : [0,1] \to \kF^{-1}(\DiffW).
		\]
		In this case,
		\[
			\rEvol[\DiffW]{\gamma} = \kF \circ \Gamma_{p} .
		\]

		\item
		Let
		$\Omega \subseteq \ConDiff{[0,1]}{  \Tang[\id{\SX}]{\DiffW} }{\infty}$
		be an open set such that for each $\gamma \in \Omega$ there exists
		a right evolution $\rEvol[\DiffW]{\gamma}$.
		Then $\rest{\revol[\DiffW]{} }{\Omega}$ is smooth iff the map
		\[
			(\dA{\kF^{-1}}{}{} \circ \Omega) \to \CcF{\SX}{\SX}{\infty}
			: p \mapsto \Gamma_p(1)
		\]
		is so. As above, $\Gamma_p$ denotes a solution to \eqref{ivp:DGL_Regularitaet_Koordinaten}
		with respect to $p$.
	\end{enumerate}
\end{lem}
\begin{proof}
	This is an easy computation involving the previous results.
\end{proof}
\subsubsection{Solving the differential equation}
We show that the regularity differential equation for $\DiffW$ is solvable.
In order to do this, we use that $\CcF{\SX}{\SX}{\infty}$ is a
projective limit of Banach spaces, see \refer{prop:CW_projektiver_LB-Raum}.
We solve the differential equation on each step of the projective limit,
see that these solutions are compatible with the bonding morphisms
of the projective limit and thus obtain a solution on the limit.
Before we do this, we state the following obvious lemma.
\begin{lem}\label{lem:Vertraeglichkeit_Loesungen_mit_Limes}
	Let $\SX$ be a Banach space and $\GewFunk \subseteq \cl{\R}^\SX$ with $1_\SX \in \GewFunk$.
	Further, let $\cF \sub \GewFunk$ with
	$1_\SX \in \cF$ and $k \in \cl{\N}$,
	$p \in \ConDiff{[0,1]}{ \CcF{\SX}{\SX}{\infty} }{\infty}$
	and $\Gamma : I \to \CF{\SX}{\SX}{\cF}{k}$
	a solution to \eqref{ivp:DGL_Regularitaet_Koordinaten}
	corresponding to $p$.
	Then $\Gamma$ solves \eqref{ivp:DGL_Regularitaet_Koordinaten}
	also for all subsets $\mathcal{G} \sub \cF$ containing $1_\SX$ and $\ell \in \cl{\N}$ with $\ell \leq k$.
\end{lem}
\begin{proof}
This is an easy calculation since the
inclusion map $\CF{\SX}{\SX}{\cF}{k} \to \CF{\SX}{\SX}{\mathcal{G}}{\ell}$
is continuous linear.
\end{proof}
\paragraph{Solving the differential equation on the steps}
First, we solve \eqref{ivp:DGL_Regularitaet_Koordinaten}
on function spaces that are Banach spaces.
To this end, we need tools from the theory of ordinary
differential equations on Banach spaces. The required facts are described
in \refer{sec:Fakten_ueber_DGLn}.
The hard part will be to show that the solutions are defined on the whole interval $[0,1]$.%
\subparagraph{The solution on $\boldsymbol{ \CF{\SX}{\SX}{\cF}{0} }$}
We start with the function space $\CF{\SX}{\SX}{\cF}{0}$,
where $\cF \sub \GewFunk$ is finite and contains $1_\SX$.
Then the initial value problem \eqref{ivp:DGL_Regularitaet_Koordinaten} satisfies
a global Lipschitz condition and hence is globally solvable.
\begin{lem}\label{lem:rechte_Seite_RegDGL_DiffW_global_Lipschitz}
	Let $\SX$ be a normed space, $\GewFunk \subseteq \cl{\R}^\SX$ with $1_\SX \in \GewFunk$,
	$\cF \sub \GewFunk$ with $1_\SX \in \cF$
	and $p \in \ConDiff{[0,1]}{ \CcF{\SX}{\SX}{\infty} }{\infty}$.
	Then there exists $K > 0$ such that for each $f\in\cF$,
	all $t \in [0,1]$ and $\gamma, \gamma_0 \in \CF{\SX}{\SX}{\cF}{0}$
	\[
		\hn{F_{\cF,0}(t,\gamma, p) - F_{\cF,0}(t,\gamma_0, p)}{f}{0}
		\leq
		K \cdot \hn{\gamma - \gamma_0}{f}{0}.
	\]
\end{lem}
\begin{proof}
	We have
	\[
		F_{\cF,0}(t,\gamma, p) - F_{\cF,0}(t,\gamma_0, p)
		= \compIdSmoothLWeights{\SX}{0}{\cF}(p(t), \gamma) - \compIdSmoothLWeights{\SX}{0}{\cF}(p(t), \gamma_0),
	\]
	and deduce from \refer{est:f,0-Norm_Differenz_Kompo}
	in \refer{lem:Hinreichendes_fuer_Kompo_in_CF0} that
	\[
		\hn{F_{\cF,0}(t,\gamma, p) - F_{\cF,0}(t,\gamma_0, p)}{f}{0}
		\leq \hn{p(t)}{1_\SX}{1} \hn{\gamma - \gamma_0}{f}{0}.
	\]
	Since $p([0,1])$ is a compact (and therefore bounded) subset of
	$\CcF{\SX}{\SX}{\infty}$,
	\[
		K \ndef \sup_{t\in[0,1]} \hn{p(t)}{1_\SX}{1}
	\]
	is finite. This proves the assertion.
\end{proof}

\begin{lem}\label{lem:loesung_regularitaetsDGL_k=0_F_endl}
	Let $\SX$ be a Banach space, $\cF, \GewFunk \subseteq \cl{\R}^\SX$ with $1_\SX \in \cF \subseteq \GewFunk$
	and $\abs{\cF} < \infty$,
	$p \in \ConDiff{[0,1]}{ \CcF{\SX}{\SX}{\infty} }{\infty}$ and $k = 0$.
	Then the initial value problem \eqref{ivp:DGL_Regularitaet_Koordinaten}
	corresponding to $p$
	has a unique solution which is defined on the whole interval $[0,1]$.
\end{lem}
\begin{proof}
	We deduce from \refer{lem:rechte_Seite_RegDGL_DiffW_global_Lipschitz}
	that we can find a norm on $\CF{\SX}{\SX}{\cF}{0}$
	such that $F_{\cF,0}(\cdot, \cdot, p)$ satisfies a global Lipschitz condition
	with respect to the second argument.
	Since $\CF{\SX}{\SX}{\cF}{0}$ is a Banach space,
	there exists a unique solution
	\[
		\Gamma : [0,1] \to \CF{\SX}{\SX}{\cF}{0}
	\]
	of \eqref{ivp:DGL_Regularitaet_Koordinaten}
	which is defined on the whole interval $[0,1]$; see
	\cite[\S 10.6.1]{MR0120319} or \refer{satz:globale_Loesbarkeit_von_AWP_lineare_Gebundenheit}
	and \refer{lem:VF-global_Lipschitz_impliziert_linear_beschrankt}.
\end{proof}
\subparagraph{Solutions in spaces of differentiable functions}
On the spaces $\CF{\SX}{\SX}{\cF}{k}$ with $k \geq 1$,
it is harder to show that the maximal solution is defined on the whole of $[0,1]$.
To show this, we first verify that the differential curve $\FAbl{}\circ\gamma$
of a solution $\gamma : I \to \CF{\SX}{\SX}{\cF}{k}$ to \eqref{ivp:DGL_Regularitaet_Koordinaten}
is itself a solution to a linear ODE.
We start with the following definition.
\begin{defi}
	Let $\SX$ be a Banach space and $\GewFunk \subseteq \cl{\R}^\SX$ with $1_\SX \in \GewFunk$.
	Further, let $\cF$ be a subset of $\GewFunk$ with $1_\SX \in \cF$,
	$k \in \cl{\N}$ and
	$\Gamma : [0,1] \to \CF{\SX}{\SX}{\cF}{k}$
	and
	$P : [0,1] \to \CcF{\SX}{ \Lin{\SX}{\SX} }{\infty}$
	be continuous curves.
	We define the continuous map
	\begin{align*}
		G_{\cF,k}^{\Gamma, P}
		&:
		[0,1] \times
		\CF{\SX}{\Lin{\SX}{\SX}}{\cF}{k}
		\to \CF{\SX}{\Lin{\SX}{\SX}}{\cF}{k}
		\\
		&: (t, \gamma) \mapsto
		\bigl(P(t) \circ (\Gamma(t) + \id{\SX})\bigr) \MaMu (\gamma +\idco)
	\end{align*}
	and consider the initial value problem
	\begin{align}\label{ivp:Hilfs_DGL_fuer_Regularitaet}
		\begin{split}
			\Phi'(t) &= G_{\cF,k}^{\Gamma, P}(t, \Phi(t))\\
			\Phi(0) &= 0.
		\end{split}
	\end{align}
\end{defi}
\begin{lem}\label{lem:Differential_C(K+1)-Kurve_loest_lineare_DGL}
	Let $\SX$ be a Banach space and $\GewFunk \subseteq \cl{\R}^\SX$ with $1_\SX \in \GewFunk$.
	Further, let $\cF$ be a finite subset of $\GewFunk$ with $1_\SX \in \cF$, $k \in \N$ and
	$p \in \ConDiff{[0,1]}{ \CcF{\SX}{\SX}{\infty} }{\infty}$.
	If
	\[
		\Gamma_k : [0,1] \to \CF{\SX}{\SX}{\cF}{k}
		\quad\text{ and }\quad
		\Gamma_{k + 1} : I \subseteq [0,1] \to \CF{\SX}{\SX}{\cF}{k + 1}
	\]
	are solutions to \eqref{ivp:DGL_Regularitaet_Koordinaten} corresponding to $p$,
	then the curve $\FAbl{} \circ \Gamma_{k+1} : I \to \CF{\SX}{\Lin{\SX}{\SX} }{\cF}{k}$ is
	a solution to the \refer{ivp:Hilfs_DGL_fuer_Regularitaet}
	with $\Gamma = \Gamma_{k}$ and $P = \FAbl{} \circ p$.
\end{lem}
\begin{proof}
	We have
	\[
		(\FAbl{} \circ \Gamma_{k+1})' = \FAbl{} \circ \Gamma_{k+1}'
	\]
	and therefore for $t \in I$
	\begin{align*}
		(\FAbl{} \circ \Gamma_{k+1})'(t)
		&= \FAbl{ \,F_{\cF, k+1 }(t, \Gamma_{k+1}(t), p) }\\
		&= \bigl( \FAbl{ p(t) } \circ (\Gamma_{k+1}(t) + \id{\SX}) \bigr)
			\MaMu (\FAbl{ \Gamma_{k+1}(t) } + \idco).\\
		&= \bigl( (\FAbl{} \circ p)(t) \circ (\Gamma_{k+1}(t) + \id{\SX}) \bigr)
			\MaMu \bigl( (\FAbl{} \circ \Gamma_{k+1})(t) + \idco \bigr)\\
		&= G_{\cF, k}^{\Gamma_{k}, \FAbl{} \circ p}(t, (\FAbl{} \circ \Gamma_{k+1})(t)),
	\end{align*}
	where we used that $\rest{\Gamma_{k}}{I} = \Gamma_{k+1}$ by \refer{lem:Vertraeglichkeit_Loesungen_mit_Limes}
	since $\CF{\SX}{\SX}{\cF}{k}$ is a Banach space.
	Obviously $(\FAbl{} \circ \Gamma_{k+1})(0) = 0$, so the assertion is proved.
\end{proof}
Now we use the embedding from \refer{prop:topologische_Zerlegung_von_CFk}
to show that the maximal solution to \eqref{ivp:DGL_Regularitaet_Koordinaten} is defined on $[0,1]$.
\begin{lem}\label{lem:loesung_regularitaetsDGL_k_F<INF}
	Let $\SX$ be a Banach space, $\GewFunk \subseteq \cl{\R}^\SX$ with $1_\SX \in \GewFunk$,
	$\cF \subseteq \GewFunk $ finite with $1_\SX \in \cF$,
	$p \in \ConDiff{[0,1]}{ \CcF{\SX}{\SX}{\infty} }{\infty}$ and $k \in \N$.
	Then the initial value problem \eqref{ivp:DGL_Regularitaet_Koordinaten}
	corresponding to $p$ has a unique solution
	which is defined on the whole interval $[0,1]$.
\end{lem}
\begin{proof}
	This is proved by induction on $k$. The case $k = 0$ was treated
	in \refer{lem:loesung_regularitaetsDGL_k=0_F_endl}.
	
	$k \to k + 1$:
	We denote the solutions for $k$ and $0$ with $\Gamma_{k}$ and $\Gamma_{0}$, respectively.
	Since the function $F_{\cF, k + 1}$ is smooth
	and $\CF{\SX}{\SX}{\cF}{k + 1}$ is a Banach space, there exists a unique
	\emph{maximal} solution
	$\Gamma_{k + 1} : I \to \CF{\SX}{\SX}{\cF}{k + 1}$
	to \eqref{ivp:DGL_Regularitaet_Koordinaten} (see \refer{prop:Existenz_maximale_Lsg_ODEs}).
	Using \refer{lem:Differential_C(K+1)-Kurve_loest_lineare_DGL}, we conclude that
	$\FAbl{} \circ \Gamma_{k+1}$ is a solution to \eqref{ivp:Hilfs_DGL_fuer_Regularitaet},
	where $\Gamma = \Gamma_{k}$ and $P = \FAbl{} \circ p$;
	here we used that by the induction hypothesis, $\Gamma_{k}$ is defined on $[0,1]$.
	Since the latter ODE is linear, there exists a unique solution
	\[
		S : [0,1] \to \CF{\SX}{\Lin{\SX}{\SX}}{\cF}{k}
	\]
	that is defined on the whole interval $[0,1]$
	(see \cite[\S 10.6.3]{MR0120319} or \refer{satz:globale_Loesbarkeit_von_AWP_lineare_Gebundenheit}).
	Let
	\[
		\iota : \CF{\SX}{\SX}{\cF}{k + 1} \to
			\CF{\SX}{\SX}{\cF}{0}
			\times \CF{\SX}{\Lin{\SX}{\SX}}{\cF}{k}
	\]
	be the embedding from \refer{prop:topologische_Zerlegung_von_CFk}.
	By \refer{lem:Vertraeglichkeit_Loesungen_mit_Limes},
	$\Gamma_{k+1}$ is a solution to \eqref{ivp:DGL_Regularitaet_Koordinaten}
	for the right hand side $F_{\cF, 0}$, so $\Gamma_{k+1} = \rest{\Gamma_0}{I}$
	since solutions to initial value problems in Banach spaces are unique.
	Hence
	\[
		\Gamma_{k + 1}(I) \subseteq
		\iota^{-1}\bigl(\Gamma_{0}([0,1]) \times S([0,1])\bigr).
	\]
	Further, $\Gamma_{0}([0,1]) \times S([0,1])$ is compact
	and the image of $\iota$ is a closed subset of
	$\CF{\SX}{\SX}{\cF}{0}
	\times \CF{\SX}{\Lin{\SX}{\SX}}{\cF}{k}$
	(by \refer{prop:Abgeschlossenheit_des_Bildes_unter_der_Einbettung}).
	Hence, because $\iota^{-1}$ is a homeomorphism, the image of $\Gamma_{k + 1}$
	is contained in a compact set. Since $\Gamma_{k + 1}$ is maximal,
	this implies that $\Gamma_{k + 1}$ must be defined on the whole of $[0,1]$;
	see \refer{satz:globale_Loesbarkeit_von_AWP_kompaktes_Bild}.
\end{proof}
\paragraph{Smooth dependence on the parameter and taking the solution to the limit}
We use the constructed solutions on $\CF{\SX}{\SX}{\cF}{k}$
and show that there exists a solution to \eqref{ivp:DGL_Regularitaet_Koordinaten}
on $\CcF{\SX}{\SX}{\infty}$, depending smoothly on the parameter curve.
\begin{prop}\label{prop:Loesungen_der_RegDGL}
	Let $\SX$ be a Banach space and $\cW \subseteq \cl{\R}^\SX$ with $1_\SX \in \cW $.
	For each
	$p \in \ConDiff{[0,1]}{ \CF{\SX}{\SX}{\cW}{\infty} }{\infty}$
	there exists a solution $\Gamma_{p}$ to \eqref{ivp:DGL_Regularitaet_Koordinaten}
	defined on $[0,1]$ which corresponds to $p$, $\cW$ and $\infty$.
	The map
	\begin{equation}\label{Parameterloesung_RegDGL}\tag{\ensuremath{\dagger}}
		[0,1] \times \ConDiff{[0,1]}{ \CF{\SX}{\SX}{\cW}{\infty} }{\infty}
		\to
		\CF{\SX}{\SX}{\cW}{\infty}
		:
		(t , p) \mapsto \Gamma_{p}(t)
	\end{equation}
	is smooth.
\end{prop}
\begin{proof}
	For $p \in \ConDiff{[0,1]}{ \CcF{\SX}{\SX}{\infty} }{\infty}$,
	we denote the solution $[0, 1] \to \CF{\SX}{\SX}{\sset{1_\SX}}{0}$ to \eqref{ivp:DGL_Regularitaet_Koordinaten}
	corresponding to $p$, $0$ and $\sset{1_\SX}$
	-- which exists by \refer{lem:loesung_regularitaetsDGL_k_F<INF} -- with $\Gamma_p$.
	By \refer{lem:Vertraeglichkeit_Loesungen_mit_Limes},
	a solution $\Gamma : [0,1] \to \CF{\SX}{\SX}{\cF}{k}$ to \eqref{ivp:DGL_Regularitaet_Koordinaten}
	corresponding to $p$, a finite set $\cF \sub \cW$ containing $1_\SX$ and $k\in\N$
	-- which exists by \refer{lem:loesung_regularitaetsDGL_k_F<INF} --
	also solves \eqref{ivp:DGL_Regularitaet_Koordinaten} for $p$, $0$ and $\sset{1_\SX}$.
	Hence, by the uniqueness of solutions to initial value problems for Banach spaces,
	$\Gamma_p = \Gamma$.
	Since $\cF$ and $k$ were arbitrary,
	the image of $\Gamma_p$ is contained in $\CF{\SX}{\SX}{\cW}{\infty}$,
	and we easily calculate that $\Gamma_p$
	is a solution to \eqref{ivp:DGL_Regularitaet_Koordinaten} corresponding to $p$, $\cW$ and $\infty$.

	It remains to show that the map \eqref{Parameterloesung_RegDGL} is smooth.
	The space $\CcF{\SX}{\SX}{\infty}$ is the projective limit of
	\[
		\set{\CF{\SX}{\SX}{\cF}{k} }{ k \in \N, \cF \subseteq \cW, \abs{\cF} < \infty, 1_\SX \in \cF }
	\]
	by \refer{prop:CW_projektiver_LB-Raum}.
	Hence using the universal property of the projective limit (see \refer{prop:Differenzierbarkeit_Abb_in_projektiven_Limes}),
	we just have to show that the map
	\[
		[0,1] \times \ConDiff{[0,1]}{ \CF{\SX}{\SX}{\cW}{\infty} }{\infty}
		\to
		\CF{\SX}{\SX}{\cF}{k}
		:
		(t , p) \mapsto \Gamma_{p}(t)
	\]
	with a finite set $\cF \sub \cW$ containing $1_\SX$ and $k\in\N$ is smooth.
	We deduce this from \refer{cor:Ck_Abhaengigkeit_LsgDGL_von_Parameter}
	since the map
	$\ConDiff{[0,1]}{ \CcF{\SX}{\SX}{\infty} }{\infty} \to \CF{\SX}{\SX}{\cF}{k}
	: p \mapsto 0$
	is smooth. Here, we used implicitely that the inclusion map
	$\CF{\SX}{\SX}{\cW}{\infty} \to \CF{\SX}{\SX}{\cF}{k}$ is smooth.
\end{proof}
%
\subsection{Conclusion and calculation of one-parameter groups}
We are ready to prove the regularity of $\DiffW$ and $\DiffWvan$.
After that, we calculate their one-parameter groups and show that these
induce flows on certain weighted vector fields.
\begin{satz}\label{satz:Gewichtete_Diffeos_regulaere_Liegruppen}
	Let $\SX$ be a Banach space and $\GewFunk \subseteq \cl{\R}^\SX$ with $1_\SX \in \GewFunk$.
	Then the Lie group $\DiffW$ is regular.
	\index{regularity!of $\DiffW$}
\end{satz}
\begin{proof}
	We proved in \refer{prop:Loesungen_der_RegDGL} that for each
	smooth curve $p : [0,1] \to \CcF{\SX}{\SX}{\infty}$ the initial value problem
	\eqref{ivp:DGL_Regularitaet_Koordinaten} has a solution
	$\Gamma_{p} : [0,1] \to \CcF{\SX}{\SX}{\infty}$ and that the map
	\[
		\Gamma :
		[0,1] \times \ConDiff{[0,1]}{ \CcF{\SX}{\SX}{\infty} }{\infty}
		\to
		\CcF{\SX}{\SX}{\infty}
		: (t,p) \mapsto \Gamma_{p}(t)
	\]
	is smooth. Obviously, $\Gamma$ maps $[0,1] \times \{0\}$ to $0$.
	Since $\kF^{-1}(\DiffW)$ is an open neighborhood of $0$ in $\CcF{\SX}{\SX}{\infty}$
	(see \refer{satz:DiffW_ist_offen}) and $\Gamma$ is continuous,
	a compactness argument gives a neighborhood $U$ of $0$ such that
	\[
		\Gamma([0,1] \times U) \subseteq \kF^{-1}(\DiffW).
	\]
	We recorded in \refer{lem:Aequivalenz_Regularitaet_DiffW} that this is equivalent
	to the existence of an open neighborhood $V$ of $0 \in \ConDiff{[0,1]}{ \CcF{\SX}{\SX}{\infty} }{\infty}$
	such that for each $\gamma \in V$, there exists a right evolution $\rEvol[\DiffW]{\gamma}$
	and that $\rest{\revol[\DiffW]{}}{V}$ is smooth.
	But we know from \refer{lem:Lie-Gruppe_lokal_reg-->global_reg} that this entails the regularity of $\DiffW$.
\end{proof}
\begin{cor}
	Let $\SX$ be a Banach space and $\GewFunk \sub \cl{\R}^\SX$ with $1_\SX \in \GewFunk$.
	Then $\DiffWvan$ is a regular Lie group.
	\index{regularity!of $\DiffWvan$}%
\end{cor}
\begin{proof}
	Let $\gamma \in \ConDiff{[0,1]}{ \Tang[\id{\SX}]{\DiffWvan} }{\infty}$.
	Since $\Tang[\id{\SX}]{\DiffWvan} \sub \Tang[\id{\SX}]{\DiffW}$
	and $\DiffW$ is regular by \refer{satz:Gewichtete_Diffeos_regulaere_Liegruppen},
	there exists a right evolution $\rEvol{\gamma} : [0,1] \to \DiffW$.
	We proved in \refer{lem:Aequivalenz_Regularitaet_DiffW} that the curve
	$\Gamma \ndef \kF \circ \rEvol{\gamma}$ is a solution to the initial value problem
	\eqref{ivp:DGL_Regularitaet_Koordinaten},
	where $\cF = \cW$, $k = \infty$ and $p = \dA{\kF^{-1}}{}{} \circ \gamma$.
	So for $t \in [0,1]$,
	\[
		\Gamma(t) = \Rint{0}{t}{\Gamma'(s)}{s}
		= \Rint{0}{t}{p(s) \circ (\Gamma(s) + \id{\SX})}{s}.
	\]
	Hence we see with \refer{lem:EndWwan_und_Kompostion} and the fact that
	$\CcFvan{\SX}{\SX}{\infty}$ is closed in $\CcF{\SX}{\SX}{\infty}$
	by \refer{lem:CWvan_closed_CW} that
	$\rEvol{\gamma}$ takes its values in $\DiffW \cap \EndWvan = \DiffWvan$.
	From this and the smoothness of $\revol[\DiffW]{}$
	we easily conclude that $\revol[\DiffWvan]{}$ is smooth,
	and this finishes the proof.
\end{proof}
\paragraph{On the one-parameter groups}
We calculate the one-parameter groups of $\DiffW$ (and hence for $\DiffWvan$).
As excepted, these arise as flows of vector fields.
\begin{lem}
	Let $\SX$ be a Banach space and $\GewFunk \subseteq \cl{\R}^\SX$ with $1_\SX \in \GewFunk$.
	Then for $\gamma \in \CcF{\SX}{\SX}{\infty}$,
	the associated flow of the one-parameter subgroup of $\DiffW$
	with the right logarithmic derivative $\Tang[0]{\kF}(\gamma)$
	is the flow of $\gamma$ (as a vector field).
\end{lem}
\begin{proof}
	We proved in \refer{satz:Gewichtete_Diffeos_regulaere_Liegruppen}
	that $\DiffW$ is regular, hence the one-parameter subgroup $\mathcal{P}$ of $\DiffW$
	with $\RLA{\mathcal{P}}(t) =  \Tang[0]{\kF}(\gamma)$ for all $t \in \R$ exists.
	We have to show that for any $x \in \SX$, the curve $\R \to \SX : t \mapsto \mathcal{P}(t)(x)$
	is the solution to the ODE
	\begin{align*}
		\begin{split}
			f'(t) &= \gamma(f(t))\\
			f(0) &= x.
		\end{split}
	\end{align*}
	Obviously, $\mathcal{P}(0)(x) = \id{\SX}(x) = x$.
	Further, $\mathcal{P}(t)(x) = (\evTwo_x \circ \kF \circ \kF^{-1} \circ \mathcal{P})(t)$.
	It is an easy computation to see that $\evTwo_x \circ \kF$ is $\ConDiff{}{}{1}$ with
	\[
		\dA{(\evTwo_x \circ \kF)}{\gamma}{\gamma_1} = \evTwo_x(\gamma_1).
	\]
	By our assumptions, for $t \in \R$
	\[
		\mathcal{P}'(t) = \Tang[0]{\kF}(\gamma) \cdot \mathcal{P}(t)
		= \Tang{\rho_{\mathcal{P}(t)}} (\Tang[0]{\kF}(\gamma))
		= \Tang{(\rho_{\mathcal{P}(t)} \circ \kF) }(0, \gamma).
	\]
	So by using the last two identities and \refer{lem:Tangentengruppenop_DiffW_in_Karten}, we get
	\begin{multline*}
		(\evTwo_x \circ \mathcal{P})'(t)
		= (\dA{ (\evTwo_x \circ \kF) }{}{} \circ \Tang{\kF^{-1}}) (\mathcal{P}'(t))
		\\
		= \dA{ (\evTwo_x \circ \kF) }{\kF^{-1}(\mathcal{P}(t))}{\gamma \circ \mathcal{P}(t)}
		= \gamma(\mathcal{P}(t)(x)).
	\end{multline*}
	This proves that the curve $\R \to \SX : t \mapsto \mathcal{P}(t)(x)$ is the integral curve
	of $\gamma$ to the initial value $x$.
\end{proof}

\chapter{Integration of certain Lie algebras of vector fields}
The aim of this chapter is the integration of Lie algebras
that arise as the semidirect product of a weighted function space
$\CcF{\SX}{\SX}{\infty}$ and $\LieAlg{G}$, where $G$
is a subgroup of $\Diff{\SX}{}{}$
which is a Lie group with respect to composition and inversion of functions.

The canonical candidate for this purpose is the semidirect product
of $\DiffW$ and $G$ -- if it can be constructed.
Hence we need criteria when
\[
	\G \times \DiffW \to \Diff{\SX}{}{}
	: (T, \phi) \mapsto T \circ \phi \circ T^{-1}
\]
takes its image in $\DiffW$ and is smooth.

\section[On the smoothness of the conjugation action on \texorpdfstring{$\DiffWz$}{}]%
{On the smoothness of the conjugation action on \texorpdfstring{$\boldsymbol{\DiffWz}$}{}}
We slightly generalize our approach by allowing arbitrary Lie groups to
act on $\DiffW$. We need the following notation.
\begin{defi}
	Let $\G$ be a group and $\omega : \G \times M \to M$
	an action of $\G$ on the set $M$.
	\begin{enumerate}
		\item
		For $g \in \G$, we denote the partial map $\omega(g, \cdot) : M \to M$ by $\omega_g$.

		\item
		Assume that $\G$ is a locally convex Lie group with the identity element $e$, $M$ a smooth manifold
		and $\omega$ is smooth. We define the linear map
		\[
			\glstext{Ableitung_Gruppenwirkung} : \LieAlg{\G} \to \VecFields{M}
		\]
		by
		\[
			\AblAct{\omega}(x)(m) = - \Tang[e]{\omega(\cdot, m)} (x).
		\]
		Note that $\AblAct{\omega}$ takes its values in the smooth vector fields
		because $\omega$ is smooth.
	\end{enumerate}
\end{defi}
Now we can state a first criterion for smoothness of
the conjugation action -- however only on the identity component $\glstext{gew_Diffeomorphismen_Braum_ZKo}$ of $\DiffW$.
\begin{lem}
\label{lem:Glatte_Wirkung_Lie-Gruppe_auf_DiffWz}
	Let $\SX$ be a Banach space, $\GewFunk \subseteq \cl{\R}^\SX$
	with $1_{\SX} \in \GewFunk$, $\G$ a Lie group and
	$\omega : \G \times \SX \to \SX$ a smooth action.
	We define the map
	\[
		\alpha :
		\G \times \DiffW \to \Diff{\SX}{}{}
		: (T, \phi) \mapsto \omega_T \circ \phi \circ \omega_{T^{-1}}.
	\]
	Assume that there exists an open set $\Omega \in \neigh[\G]{\one}$ such that the maps
	\begin{equation}
		\label{Komposition-gewAbbinfty-Liegruppe}
		\CcF{\SX}{ \SX }{\infty} \times \Omega \to \CcF{\SX}{ \SX }{\infty}
			:
			(\gamma, T) \mapsto \gamma \circ \omega_T
	\end{equation}
	and
	\begin{equation}\label{Auswertung-gewAbbinfty-Ableitung_Liegruppe}
		\CcF{\SX}{ \SX }{\infty} \times \Omega \to \CcF{\SX}{ \SX }{\infty}
			:
			(\gamma, T) \mapsto \FAbl{\omega_T} \eval \gamma
	\end{equation}
	are well-defined and smooth.
	\begin{enumerate}
		\item
		\label{enum1:Glatte_Wirkung_Lie-Gruppe_auf_DiffWz--lokal}
		Then for each open identity neighborhood $\UF_\GewFunk \subseteq \DiffW$
		such that $\Strecke{\phi}{\id{\SX}} \ndef \set{t \phi + (1 - t) \id{\SX}}{t \in [0,1]} \subseteq \DiffW$
		for each $\phi \in \UF_\GewFunk$, the map
		\[
			\tag{\ensuremath{\dagger}}
			\label{lokale_Gruppenwirkung_semidirektes_Produkt}
			(\Omega \cap \Omega^{-1}) \times \UF_\GewFunk \to \EndW
			: (T, \phi) \mapsto \alpha(T, \phi)
		\]
		is well-defined and smooth.
		
		\item
		\label{enum1:Glatte_Wirkung_Lie-Gruppe_auf_DiffWz--global}
		Suppose that $\Omega = \G$. Then the map
		\[
			\tag{\ensuremath{\dagger\dagger}}
			\label{globale_Gruppenwirkung_semidirektes_Produkt-Zshgskomp}
			\G \times \DiffWz \to \DiffWz
			: (T, \phi) \mapsto \alpha(T, \phi)
		\]
		is well-defined and smooth.
	\end{enumerate}
	\index{semidirect product!of $\DiffWz$ and a Lie group acting on $\SX$}
\end{lem}
\begin{proof}
	\refer{enum1:Glatte_Wirkung_Lie-Gruppe_auf_DiffWz--lokal}
	Using \refer{prop:Kompo_Koord_glatt}, \refer{satz:DiffW_Lie-Gruppe}
	and the smoothness of \eqref{Komposition-gewAbbinfty-Liegruppe}
	and \eqref{Auswertung-gewAbbinfty-Ableitung_Liegruppe},
	for each $t \in [0,1]$, $T \in \Omega \cap \Omega^{-1}$ and $\phi \in \UF_\GewFunk$ we see that
	\[
		\psi_{t, T, \phi}\ndef
		(\FAbl{\omega_T} \eval ((\phi - \id{\SX})\circ (t \phi + (1 - t) \id{\SX})^{-1})) \circ (t \phi + (1 - t) \id{\SX}) \circ \omega_T^{-1}
		\in \CcF{\SX}{\SX}{\infty},
	\]
	and $\psi_{t, T, \phi}$ is a smooth map.
	Further, using that $t \phi + (1 - t) \id{\SX}$ is a diffeomorphism for each $t \in [0,1]$, we calculate
	\begin{align*}
		&(\omega_T \circ \phi \circ \omega_{T^{-1}})(x) - x
		\\
		=& (\omega_T \circ \phi \circ \omega_T^{-1})(x) - (\omega_T \circ \omega_T^{-1})(x)
		\\
		=& \Rint{0}{1}{ \FAbl{\omega_T}\circ(t \phi + (1 - t) \id{\SX})(\omega_T^{-1}(x)) \eval (\phi - \id{\SX})(\omega_T^{-1}(x))}{t}
		\\
		=& \Rint{0}{1}{ (\FAbl{\omega_T} \eval ((\phi - \id{\SX})\circ (t \phi + (1 - t) \id{\SX})^{-1})) \circ (t \phi + (1 - t) \id{\SX})(\omega_T^{-1}(x))}{t}.
	\end{align*}
	Hence
	$\omega_T \circ \phi \circ \omega_{T^{-1}} - \id{\SX}
	= \Rint{0}{1}{ \psi_{t, T, \phi} }{t}\in \CcF{\SX}{\SX}{\infty}$
	by \refer{prop:Stetigkeit_parameterab_Int},
	using that we proved in \refer{cor:CkF_komplett} that $\CcF{\SX}{\SX}{\infty}$ is complete.

	Since $\psi_{t, T, \phi}$ is smooth as a function of $t$, $T$ and $\phi$,
	we can use \refer{prop:Glattheit_parameterab_Int}
	to see that \eqref{lokale_Gruppenwirkung_semidirektes_Produkt}
	is defined and smooth.

	\refer{enum1:Glatte_Wirkung_Lie-Gruppe_auf_DiffWz--global}
	Since $\DiffW$ is locally convex,
	we find a symmetric open $\UF_\GewFunk \in \neigh{\id{\SX}}$
	such that $\Strecke{\UF_\GewFunk}{\id{\SX}} \subseteq \DiffW$.
	Using the symmetry of $\UF_\GewFunk$ and the results of
	\refer{enum1:Glatte_Wirkung_Lie-Gruppe_auf_DiffWz--lokal},
	we see that
	$\alpha(\G \times \UF_\GewFunk) \subseteq \DiffWz$.
	Since $\UF_\GewFunk$ generates $\DiffWz$,
	we can apply \refer{lem:Kriterium_fuer_Wirkung_auf_Untergruppe}
	to conclude that $\alpha(\G \times \DiffWz) \subseteq \DiffWz$.
	Further \eqref{globale_Gruppenwirkung_semidirektes_Produkt-Zshgskomp}
	is smooth by \refer{enum1:Glatte_Wirkung_Lie-Gruppe_auf_DiffWz--lokal}
	and \refer{lem:Kriterium-fuer-glatte-Gruppenwirkung}.
\end{proof}
So all we need are criteria for the smoothness of the maps \eqref{Komposition-gewAbbinfty-Liegruppe}
and \eqref{Auswertung-gewAbbinfty-Ableitung_Liegruppe}.
This will be the topic of the next two subsections.
Before we proceed, the following definition is useful.
\begin{defi}\label{def:alle_Gewichte,die_gleichen_Raum_erzeugen}
	Let $\SX$ be a normed space, $\UF \subseteq \SX$ an open nonempty subset,
	and $\GewFunk \subseteq \cl{\R}^\UF$ a nonempty set of weights.
	We define the \emph{maximal extension} $\glstext{max_extend_weights} \subseteq \cl{\R}^\UF$ of $\GewFunk$ as the set of
	functions $f$ for which $\hn{\cdot}{f}{0}$ is a continuous seminorm
	on $\CcF{\UF}{\SY}{0}$, for each normed space $\SY$.
	Obviously $\GewFunk \subseteq \ExtWeights{\GewFunk}$
	and by \refer{lem:Normbeziehung_zwischen_Ableitungen_und_Ableitungen_der_Ableitung},
	$\hn{\cdot}{f}{\ell}$ is a continuous seminorm on $\CcF{\UF}{\SY}{k}$,
	provided that $f \in \ExtWeights{\GewFunk}$ and $\ell \leq k$.
	\index{weights!maximal extension}
\end{defi}
\subsection{Bilinear action on weighted functions}
\newcommand{\mur}{multiplier}
We first elaborate on the map \eqref{Auswertung-gewAbbinfty-Ableitung_Liegruppe}.
To this end, we define a class of functions, the \emph{\mur{}s}.
These have the property that for a \mur{} $M$, a weighted function $\gamma$
and a continuous bilinear map $b$, the map $b \circ (M, \gamma)$ is a weighted function.
Finally, we provide a criterion ensuring that a topology on a set of \mur{}s
makes the map $(M, \gamma) \mapsto b \circ (M, \gamma)$ continuous.
\subsubsection{Multipliers}
\begin{defi}
	Let $\SX$ be a normed space, $\UF \subseteq \SX$ an open nonempty set
	and $\GewFunk \subseteq \cl{\R}^\UF$ a nonempty set of weights.
	\begin{enumerate}
		\item
		A function $g : \UF \to \R$ is called a \emph{multiplicative weight (for \GewFunk)} if
		\index{weights!multiplicative}
		\[
			(\forall f \in \GewFunk)\, f \cdot g \in \ExtWeights{\GewFunk}.
		\]

		\item
		Let $\SY$ be another normed space and $k \in \cl{\N}$.
		A $\ConDiff{}{}{k}$-map $M : \UF \to \SZ$ is called a \emph{$k$-\mur{} (for \GewFunk)}
		if $\Opnorm{\FAbl[\ell]{M}}$ is a multiplicative weight
		for all $\ell \in \N$ with $\ell \leq k$.
		An $\infty$-\mur{} is also called a \emph{\mur{}}.
	\end{enumerate}
\end{defi}

\begin{lem}
\label{lem:Charakterisierung_Multiplier_via_Differential}
	Let $\SX$ and $\SY$ be normed spaces, $\UF \subseteq \SX$ an open nonempty set,
	$\GewFunk \subseteq \cl{\R}^\UF$ a nonempty set of weights and $k \in \cl{\N}$.
	\begin{enumerate}
		\item\label{enum1:k-multiplier_VS}
		The set of $k$-multipliers from $\UF$ to $\SY$ is a vector space.
		
		\item\label{enum1:Charakterisierung_Multiplier_via_Differential}
		A map $M : \UF \to \SY$ is a $(k+1)$-multiplier iff $M$ is a $0$-\mur{}
		and $\FAbl{M} : \UF \to \Lin{\SX}{\SY}$ is a $k$-\mur{}.
	\end{enumerate}
\end{lem}
\begin{proof}
	\refer{enum1:k-multiplier_VS}
	This is obvious from the definition.
	
	\refer{enum1:Charakterisierung_Multiplier_via_Differential}
	This follows from the identity $\Opnorm{\FAbl[\ell]{(\FAbl{M})}} = \Opnorm{\FAbl[\ell + 1]{M}}$,
	see \refer{lem:Normbeziehung_zwischen_Ableitungen_und_Ableitungen_der_Ableitung}.
\end{proof}

\begin{lem}
\label{lem:Multiplier_operieren_stetig_auf_gewichteten_Abb}
	Let $\SX$, $\SY_{1}$, $\SY_{2}$ and $\SZ$ be normed spaces, $\UF \subseteq \SX$ an open nonempty set,
	$\GewFunk \subseteq \cl{\R}^\UF$ a nonempty set of weights and $k \in \cl{\N}$.
	Further, let $b : \SY_{1} \times \SY_{2} \to \SZ$ be a continuous bilinear map,
	$M : \UF \to \SY_{1}$ a $k$-\mur{} and $\gamma \in \CcF{\UF}{\SY_{2}}{k}$.
	Then
	\[
		b \circ (M, \gamma) \in \CcF{\UF}{\SZ}{k}.
	\]
	Moreover, the map
	\[
		\tag{\ensuremath{\dagger}}\label{lineare_Wirkung_fester-Multiplier_gewichtete_Abb}
		\CcF{\UF}{\SY_{2}}{k} \to \CcF{\UF}{\SZ}{k}
		:
		\gamma \mapsto b \circ (M, \gamma)
	\]
	is continuous linear and hence smooth.
\end{lem}
\begin{proof}
	For $k < \infty$ the proof is by induction on $k$:
	\\
	$k = 0$:
	We calculate for $x\in\UF$ and $f\in\GewFunk$:
	\[
		\abs{f(x)}\,\norm{(b \circ (M, \gamma))(x)}
		\leq
		\Opnorm{b} \, \abs{f(x)}\, \norm{M(x)}\, \norm{\gamma(x)}
		\leq
		\Opnorm{b} \, \hn{\gamma}{\abs{f}\cdot \norm{M}}{0},
	\]
	and since $\norm{M}$ is a multiplicative weight,
	the right hand side is finite.
	Hence
	\[
		\hn{b \circ (M, \gamma)}{f}{0}
		\leq \Opnorm{b} \, \hn{\gamma}{\abs{f}\cdot \norm{M}}{0},
	\]
	entailing that $b \circ (M, \gamma) \in \CcF{\UF}{\SZ}{0}$
	and the linear map \eqref{lineare_Wirkung_fester-Multiplier_gewichtete_Abb} is continuous.
	\\
	$k \to k + 1$:
	By \refer{prop:topologische_Zerlegung_von_CFk}, we need to prove that
	$\FAbl{(b \circ (M, \gamma))} \in \CcF{\UF}{\Lin{\SX}{\SZ}}{k}$
	and that the map
	\[
		\CcF{\UF}{\SY_{2}}{k + 1} \to \CcF{\UF}{\Lin{\SX}{\SZ}}{k}
		:
		\gamma \mapsto \FAbl{(b \circ (M, \gamma))}
	\]
	is continuous.
	Using \refer{lem:Ableitung_Kompo_mit_multilinearer_Abb} we get
	\[
		\FAbl{(b \circ (M, \gamma))}
		= b^{(1)} \circ (\FAbl{M}, \gamma) +  b^{(2)} \circ (M, \FAbl{\gamma});
	\]
	for the definition of the maps $b^{(i)}$ see \refer{susec:Multilineare_Abb}.
	So by applying the inductive hypothesis to the maps
	$b^{(1)} \circ (\FAbl{M}, \gamma)$ and $b^{(2)} \circ (M, \FAbl{\gamma})$
	(by \refer{lem:Charakterisierung_Multiplier_via_Differential}, $\FAbl{M}$ is a $k$-\mur{}),
	we see that $\FAbl{(b \circ (M, \gamma))}$ is in $\CcF{\UF}{\Lin{\SX}{\SZ}}{k}$
	and the map \eqref{lineare_Wirkung_fester-Multiplier_gewichtete_Abb} is continuous.
	\\
	$k = \infty$:
	\BeweisschrittCkCinfty%
	{ \CcF{\UF}{\SY_{2}}{\infty}  }
	{ b(M, \cdot)_{*, \infty} }
	{ \CcF{\UF}{\SZ}{\infty} }
	{}
	{}
	{ \CcF{\UF}{\SY_{2}}{n} }
	{ b(M, \cdot)_* }
	{ \CcF{\UF}{\SZ}{n} }
	{n}
\end{proof}


\paragraph{Topologies on spaces of \mur{}s}

\begin{lem}\label{lem:Kriterium_simultane_Stetigkeit_bilineare_Operation_Multiplier-gewAbb}
	Let $\SX$, $\SY_1$, $\SY_2$ and $\SZ$ be normed spaces, $\UF \subseteq \SX$ an open nonempty set,
	$\GewFunk \subseteq \cl{\R}^\UF$ a nonempty set of weights, $k \in \cl{\N}$
	and $b : \SY_1 \times \SY_2 \to \SZ$ a continuous bilinear map.
	Further, let $\mathcal{T}$ be a topological space and $(M_T)_{T \in \mathcal{T}}$
	a family of $k$-\mur{}s such that
	\begin{equation}\label{bed:Bedingung_Stetigkeit_bilineare_OperationMultiplier}
		\begin{multlined}[0.8\columnwidth]
			(\forall f \in \GewFunk, T \in \mathcal{T}, \ell \in \N : \ell \leq k)
			(\exists g \in \ExtWeights{\GewFunk})
			\\
			(\forall \eps > 0) (\exists \Omega \in \neigh[\mathcal{T}]{T})
			\,
			\forall S \in \Omega : 
			\abs{f}\,\norm{\FAbl[\ell]{(M_T - M_S)}} \leq \eps \,\abs{g}.
		\end{multlined}
	\end{equation}
	Then the map
	\[
		\tag{\ensuremath{\dagger}}\label{bilineare_Wirkung_Multiplier_gewichtete_Abb}
		\mathcal{T} \times \CcF{\UF}{\SY_{2}}{k} \to \CcF{\UF}{\SZ}{k}
		:
		(T, \gamma) \mapsto b \circ (M_T, \gamma)
	\]
	which is defined by \refer{lem:Multiplier_operieren_stetig_auf_gewichteten_Abb}
	is continuous.
\end{lem}
\begin{proof}
	For $k < \infty$. the proof is by induction on $k$.
	\\
	$k = 0$:
	For $S, T \in \mathcal{T}$ and $\gamma, \eta \in \CcF{\UF}{\SY_{2}}{0}$,
	we have
	\[
		b \circ (M_S, \eta) - b \circ (M_T,\gamma)
		= b \circ (M_S, \eta - \gamma) + b \circ (M_S - M_T, \gamma).
	\]
	We treat each summand separately.
	To this end, let $f \in \GewFunk$ and $x \in \UF$.
	Then we calculate for first summand
	\[
		\abs{f(x)}\,\norm{b(M_S(x), (\gamma - \eta)(x)}
		\leq \Opnorm{b}\abs{f(x)}\,\norm{M_S(x)}\,\norm{(\gamma - \eta)(x)}.
	\]
	For the second summand we get
	\[
		\abs{f(x)}\, \norm{b \circ (M_S - M_T, \gamma)(x)}
		\leq \Opnorm{b} \abs{f(x)}\, \norm{(M_S - M_T)(x)}\, \norm{\gamma(x)}.
	\]
	Let $g \in \ExtWeights{\GewFunk}$ as in \refer{bed:Bedingung_Stetigkeit_bilineare_OperationMultiplier}.
	Given $\eps > 0$, let $\Omega \in \neigh[\mathcal{T}]{T}$ be as in
	\refer{bed:Bedingung_Stetigkeit_bilineare_OperationMultiplier}.
	For $S \in \Omega$, we derive from the estimates above that
	\[
		\abs{f(x)}\,\norm{(b \circ (M_S, \eta) - b \circ (M_T,\gamma))(x)}
		\leq
		\Opnorm{b} (\hn{\gamma - \eta}{f \cdot \norm{M_S} }{0} + \eps \hn{\gamma}{g}{0}).
	\]
	As the right hand side can be made arbitrarily small,
	we see that \eqref{bilineare_Wirkung_Multiplier_gewichtete_Abb}
	is continuous.

	$k \to k + 1$:
	Using \refer{prop:topologische_Zerlegung_von_CFk}, we just need to prove that
	for $\gamma \in \CcF{\UF}{\SY_{2}}{k}$ and $T \in \mathcal{T}$,
	the map $\FAbl{(b \circ (M_T, \gamma))} \in \CcF{\UF}{\Lin{\SX}{\SZ}}{k}$
	and that
	\[
		\mathcal{T} \times \CcF{\UF}{\SY_{2}}{k + 1} \to \CcF{\UF}{\Lin{\SX}{\SZ}}{k}
		:
		\gamma \mapsto \FAbl{(b \circ (M_T, \gamma))}
	\]
	is continuous.
	Using \refer{lem:Ableitung_Kompo_mit_multilinearer_Abb} we get
	\[
		\FAbl{(b \circ (M_T, \gamma))}
		= b^{(1)} \circ (\FAbl{M_T}, \gamma) +  b^{(2)} \circ (M_T, \FAbl{\gamma}),
	\]
	with $b^{(i)}$ as in \refer{susec:Multilineare_Abb}.
	So by applying the inductive hypothesis to the maps
	$b^{(1)} \circ (\FAbl{M_T}, \gamma)$ and $b^{(2)} \circ (M_T, \FAbl{\gamma})$,
	we see that $\FAbl{(b \circ (M_T, \gamma))}$ is in $\CcF{\UF}{\Lin{\SX}{\SZ}}{k}$
	and the map \eqref{bilineare_Wirkung_Multiplier_gewichtete_Abb} is continuous.
	\\
	$k = \infty$:
	\BeweisschrittCkCinfty%
	{ \mathcal{T} \times \CcF{\UF}{\SY_{2}}{\infty}  }
	{ b_{*, \infty} }
	{ \CcF{\UF}{\SZ}{\infty} }
	{}
	{}
	{ \mathcal{T} \times \CcF{\UF}{\SY_{2}}{n} }
	{ b_* }
	{ \CcF{\UF}{\SZ}{n} }
	{n}
\end{proof}

\subsection{Contravariant composition on weighted functions}
\newcommand{\logbound}[1]{ \ifthenelse{ \equal{#1}{} }{}{\ensuremath{#1}-}logarithmically bounded}
Here we prove sufficient conditions that make \eqref{Komposition-gewAbbinfty-Liegruppe} smooth.
Since the second factor of the domain of this map in general is not contained in a vector space,
we have to wrestle with certain technical difficulties,
leading to the definition of a notion of
\emph{\logbound{}} identity neighborhoods in Lie groups.

\begin{lem}
	Let $\G$ be a Lie group and $\omega : \G \times M \to M$
	a smooth action of $\G$ on the smooth manifold $M$.
	\begin{enumerate}
		\item\label{enum1:verschiedenes_glatte_Wirkungen_2}
		For any $g \in \G$, the identity
		\[
			\Tang{ \omega }
			= \Tang{ \omega_g } \circ \Tang{ \omega } \circ (\Tang{ \lambda_{g^{-1}}} \times \id{ \Tang{M} })
		\]
		holds, where $\lambda_{g^{-1}} : \G \to \G$ denotes the left multiplication with $g^{-1}$.
	\end{enumerate}
	In the following, let $S, T \in \G$ and $W : [0,1] \to \G$ be a smooth curve with
	$W(0) = S$ and $W(1) = T$.
	\begin{enumerate}[resume]
		\item\label{enum1:verschiedenes_glatte_Wirkungen_3}
		Let $N$ be another smooth manifold and $\gamma : M \to N$ a $\ConDiff{}{}{1}-map$.
		Then for $t \in [0,1]$ and $x \in M$, we have
		\begin{equation*}\tag{\ensuremath{\dagger}}\label{id:Ableitung_Funktion-Wirkung-Weg,id}
			\Tang{ (\gamma \circ \omega \circ (W \times \id{M}) ) }(t, 1, 0_x)
			= \Tang{\gamma} \circ \Tang{ \omega_{W(t)} } (- \AblAct{\omega}( \LLA{W}(t) )(x)).
		\end{equation*}

		\item\label{enum1:verschiedenes_glatte_Wirkungen_4}
		Let $\SX$ and $\SY$ be normed spaces. Assume that $M$ is an open nonempty subset of $\SX$.
		Then for $\gamma, \eta \in \ConDiff{M}{\SY}{1}$ and $x \in M$, we have
		\begin{equation}
		\label{id:Differenz-Wirkung_Gruppe_von_hinten-gewAbb}
			\begin{aligned}
			&(\gamma \circ \omega_T)(x) - (\eta \circ \omega_S)(x)
			\\
			=& ((\gamma - \eta)\circ \omega_T)(x)
			- \Rint{0}{1}{ \FAbl{\eta}(\omega_{W(t)}(x) ) \MaMu \FAbl{\omega_{W(t)}}(x)%
					\eval\AblAct{\omega}( \LLA{W}(t) )(x)
				}{t}.
			\end{aligned}
		\end{equation}
	\end{enumerate}
\end{lem}
\begin{proof}
	\refer{enum1:verschiedenes_glatte_Wirkungen_2}
	We calculate for $h \in \G$ and $m \in M$ that
	\[
		\omega(h, m)
		= \omega(g g^{-1} h, m)
		= \omega(g, \omega(g^{-1} h, m))
		= \omega_g( \omega( \lambda_{g^{-1}} (h), m ) ).
	\]
	Applying the tangent functor gives the assertion.

	\refer{enum1:verschiedenes_glatte_Wirkungen_3}
	We calculate
	\begin{multline*}
		\Tang{ (\gamma \circ \omega \circ (W \times \id{M}) ) }(t, 1, 0_x)
		= \Tang{ \gamma} \circ \Tang{\omega}( W'(t),  0_x)
		\\
		= \Tang{ \gamma} \circ \Tang{ \omega_{W(t)} } \circ \Tang{ \omega } (W(t)^{-1} \cdot W'(t), 0_x)
		= \Tang{ \gamma} \circ \Tang{ \omega_{W(t)} }( - \AblAct{\omega}( W(t)^{-1} W'(t) )(x)).
	\end{multline*}
	Here we used \refer{enum1:verschiedenes_glatte_Wirkungen_2}.

	\refer{enum1:verschiedenes_glatte_Wirkungen_4}
	By adding $0 = \eta \circ \omega_T - \eta \circ \omega_T$, we get
	\[
		(\gamma \circ \omega_T)(x) - (\eta \circ \omega_S)(x)
		= ((\gamma - \eta)\circ \omega_T)(x) + (\eta \circ \omega_T)(x) - (\eta \circ \omega_S)(x)
	\]
	We elaborate on the second summand:
	\begin{align*}
		(\eta \circ \omega_T)(x) - (\eta \circ \omega_S)(x)
		&= \eta(\omega(W(1), x)) - \eta(\omega(W(0), x))
		\\
		&= \Rint{0}{1}{ \FAbl{(\eta \circ \omega \circ (W \times \id{\UF}))}(t, x) \eval (1,0)
		}{t}
		\\
		&= - \Rint{0}{1}{ \FAbl{\eta}(\omega_{W(t)}(x)) \MaMu \FAbl{\omega_{W(t)}}(x)%
				\eval\AblAct{\omega}( \LLA{W}(t) )(x)
			}{t}.
	\end{align*}
	Here we used \refer{id:Ableitung_Funktion-Wirkung-Weg,id}.
\end{proof}

\begin{defi}
	Let $\G$ be a Lie group and $\UF \subseteq \G$, $\VF \subseteq \LieAlg{\G}$ sets.
	We call a path $W \in \ConDiff{[0,1]}{\G}{1}$ \emph{ \logbound{\VF} }
	if $\LLA{W}([0,1]) \subseteq \VF$.
	The set $\UF$ is called \emph{ \logbound{\VF} } if
	for all $g, h\in \UF$ there exists an \logbound{\VF} $W \in \ConDiff{[0,1]}{\VF}{\infty}$
	with $W(0) = g$ and $W(1) = h$.
\end{defi}

\newcommand{\compo}{\ensuremath{\mathfrak c}}

\begin{prop}\label{prop:Diffbarkeit-Komposition-gewAbb-Liegruppe}
	Let $\SX$ and $\SY$ be normed spaces, $\UF \subseteq \SX$ an open nonempty set,
	$k \in \cl{\N}$, $\GewFunk \subseteq \cl{\R}^\UF$ a nonempty set of weights,
	$\G$ a locally convex Lie group and $\omega : \G \times \UF \to \UF$ a smooth action.
	Assume that there exists an open neighborhood $\Omega$ of $\one$ in $\G$ such that
	\begin{equation}
		\label{bed:Stetigkeit-Komposition-gewAbb_LieGr}
		\begin{multlined}[0.8\columnwidth]
			(\forall f \in \GewFunk, T \in \Omega)
			\exists g \in \ExtWeights{\GewFunk}
			(\forall \eps > 0)
			\\
			\exists \VF \in \neigh[\LieAlg{\G}]{0}, \widetilde{\Omega} \in \neigh[\Omega]{T} \text{ \logbound{\VF}}
			\\
			(\forall S \in \widetilde{\Omega}, v \in \VF)
			:
			\abs{f} \cdot \norm{\FAbl{\omega_S} \eval \AblAct{\omega}(v)}
			< \eps \abs{g \circ \omega_S}.
		\end{multlined}
	\end{equation}
	Further assume that $\GewFunk \circ \omega_{\Omega}^{-1} \subseteq \ExtWeights{\GewFunk}$,
	and that for all $m \in \N$ with $m < k$ and normed spaces $\SZ$, the map
	\begin{equation}
		\label{pktw-Kompo-lineare-Abb_gewAbb-und-LieG-(k-1)}
		\CcF{\UF}{ \Lin{\SX}{\SZ} }{m} \times \Omega \to \CcF{\UF}{ \Lin{\SX}{\SZ} }{m}
		:
		(\Gamma, T) \mapsto \Gamma \MaMu \FAbl{\omega_T}
	\end{equation}
	is defined and continuous.
	\begin{enumerate}
		\item
		\label{enum1:stetigkeit_Kompo_von_hinten}
		Then the map
		\[
			\CcF{\UF}{\SY}{k + 1} \times \Omega \to \CcF{\UF}{\SY}{k}
			: (\gamma, T) \mapsto \gamma \circ \omega_T
		\]
		is well-defined and continuous.
		\index{composition!of weighted maps und certain subsets of Lie groups}%

		\item
		\label{enum1:diffbarkeit_Kompo_von_hinten}
		Let $\ell \in \N^*$. Additionally assume that the maps
		\begin{equation}
			\label{pktw-Kompo-lineare-Abb_gewAbb-und-LieG-k}
			\CcF{\UF}{ \Lin{\SX}{\SY} }{k} \times \Omega \to \CcF{\UF}{ \Lin{\SX}{\SY} }{k}
			:
			(\Gamma, T) \mapsto \Gamma \MaMu \FAbl{\omega_T}
		\end{equation}
		and
		\begin{equation}
			\label{pktw-Auswerung-lineare-Abb-gewAbb_Vektorfelder_LieAlg}
			\CcF{\UF}{\Lin{\SX}{\SY} }{k} \times \LieAlg{\G} \to \CcF{\UF}{ \SY }{k}
			:
			(\Gamma, v) \mapsto \Gamma \eval \AblAct{\omega}(v)
		\end{equation}
		are well-defined and $\ConDiff{}{}{\ell - 1}$.
		Then the map
		\[
			\compo:
			\CcF{\UF}{\SY}{k + \ell + 1} \times \Omega \to \CcF{\UF}{\SY}{k}
			: (\gamma, T) \mapsto \gamma \circ \omega_T
		\]
		is $\ConDiff{}{}{\ell}$ with the derivative
		\[
			\tag{\ensuremath{\dagger}}\label{id:Ableitung_Kompo-gewAbb-LieGruppe}
			\dA{\compo}{(\gamma, S)}{(\gamma_1, S_1)}
			= - (\FAbl{\gamma} \circ \omega_{S} ) \MaMu \FAbl{\omega_{S }}%
					\eval\AblAct{\omega}( S^{-1} \cdot S_1 ) + \gamma_1 \circ \omega_{S}.
		\]
	\end{enumerate}
\end{prop}
\begin{proof}
	\refer{enum1:stetigkeit_Kompo_von_hinten}
	For $k < \infty$, this is proved by induction on $k$.

	$k = 0$:
	Let $\gamma, \eta \in \CcF{\UF}{\SY}{1}$, $T \in \Omega$ and $f\in \GewFunk$.
	Let $g\in \ExtWeights{\GewFunk}$ as in \refer{bed:Stetigkeit-Komposition-gewAbb_LieGr}.
	Given $\eps > 0$, we find a neighborhood $\widetilde{\Omega}$ of $T$
	and $\VF \in \neigh[\LieAlg{\G}]{0}$
	such that \refer{bed:Stetigkeit-Komposition-gewAbb_LieGr} is satisfied.
	Using \refer{id:Differenz-Wirkung_Gruppe_von_hinten-gewAbb},
	we calculate for $S \in \widetilde{\Omega}$,
	a \logbound{\VF} path $W : [0,1] \to \widetilde{\Omega}$ connecting $S$ and $T$,
	and $x\in\UF$ that
	\begin{align*}
		&\abs{f(x)}\, \norm{(\gamma \circ \omega_T)(x) - (\eta \circ \omega_S)(x)}
		\\
		\leq&
			\abs{f(x)}
			\left(
			\norm{
				((\gamma - \eta)\circ \omega_T)(x)
			}
			+
			\left\norm{
				\Rint{0}{1}{ \FAbl{\eta}(\omega_{W(t)}(x) ) \MaMu \FAbl{\omega_{W(t)}}(x)%
					\eval\AblAct{\omega}( \LLA{W}(t) )(x)
				}{t}
			\right}
			\right)
		\\
		\leq&
			\hn{\gamma - \eta}{f\circ\omega_T^{-1}}{0}
			+
			\Rint{0}{1}{
				\abs{f(x)}\,
				\Opnorm{\FAbl{\eta}(\omega_{W(t)}(x) )}
				\cdot
				\norm{\FAbl{\omega_{W(t)}}(x) \eval\AblAct{\omega}( \LLA{W}(t) )(x)}
			}{t}
		\\
		\leq&
			\hn{\gamma - \eta}{f\circ\omega_T^{-1}}{0}
			+
			\eps
			\Rint{0}{1}{ \abs{(g \circ \omega_{W(t)}) (x)} \, \Opnorm{\FAbl{\eta}(\omega_{W(t)}(x))}
				}{t}
		\\
		\leq&
			\hn{\gamma - \eta}{f\circ\omega_T^{-1}}{0}
			+
			\eps
			\hn{\eta}{g}{1}.
	\end{align*}
	The continuity at $(\gamma, \eta)$ follows from this estimate.

	$k \to k + 1$:
	By \refer{prop:topologische_Zerlegung_von_CFk} and the inductive hypothesis, we just need to check
	that the map
	\[
		\CcF{\UF}{\SY}{k + 2} \times \Omega \to \CcF{\UF}{\Lin{\SX}{\SY}}{k}
		:
		(\gamma, T) \mapsto \FAbl{(\gamma \circ \omega_T)}
	\]
	is well-defined and continuous.
	For $\gamma \in \CcF{\UF}{\SY}{k + 2}$ and $T\in\Omega$, we have
	\[
		\FAbl{(\gamma \circ \omega_T)}
		= ( \FAbl{\gamma} \circ \omega_T ) \MaMu \FAbl{\omega_T}.
	\]
	Hence by the inductive hypothesis and
	the continuity of \eqref{pktw-Kompo-lineare-Abb_gewAbb-und-LieG-(k-1)},
	the induction is finished.

	$k = \infty$:
	This is an easy consequence of the case $k < \infty$
	and \refer{cor:Topologie_von_CinfF}.

	\refer{enum1:diffbarkeit_Kompo_von_hinten}
	We prove this by induction on $\ell$.

	$\ell = 1$:
	Let $\gamma, \gamma_1 \in \CcF{\UF}{\SY}{k + \ell + 1}$,
	$S \in \Omega$ and $S_1 \in \Tang[S]{\Omega}$.
	Further, let $\Gamma : \,]\!-\!\delta, \delta[\, \to \Omega$ be a smooth curve
	with $\Gamma(0) = S$ and $\Gamma'(0) = S_1$.
	Then we calculate for a sufficiently small $t \neq 0$:
	\[
		\frac{1}{t} ( (\gamma + t \gamma_1) \circ \omega_{\Gamma(t)} - \gamma \circ \omega_S )
		= \frac{1}{t} (\gamma \circ w_{\Gamma(t)} - \gamma \circ \omega_S)
			+ \gamma_1 \circ \omega_{\Gamma(t)}.
	\]
	Using \refer{id:Differenz-Wirkung_Gruppe_von_hinten-gewAbb} we elaborate on the first summand:
	\begin{align*}
		\frac{1}{t} (\gamma \circ w_{\Gamma(t)} - \gamma \circ \omega_S)(x)
		=
		- \frac{1}{t} \Rint{0}{1}{ \FAbl{\gamma}(\omega_{ \Gamma(s t)}(x) ) \MaMu \FAbl{\omega_{\Gamma(s t) }}(x)%
					\eval\AblAct{\omega}(t \LLA{\Gamma}( s t) )(x)
				}{s}
		.
	\end{align*}
	Hence
	\begin{align*}
		\frac{1}{t} (\gamma \circ w_{\Gamma(t)} - \gamma \circ \omega_S)
		= - \Rint{0}{1}{ (\FAbl{\gamma} \circ \omega_{ \Gamma(s t)}) \MaMu \FAbl{\omega_{\Gamma(s t) }}%
					\eval\AblAct{\omega}(\LLA{\Gamma}( s t) )
				}{s};
	\end{align*}
	note that the integral on the right hand side exists by \refer{lem:Kriterium_Integrierbarkeit_in_CW}
	since the curve
	\[
		[0,1] \to \CcF{\UF}{\SY}{k}
		:
		s \mapsto (\FAbl{\gamma} \circ \omega_{ \Gamma(s t)}) \MaMu \FAbl{\omega_{\Gamma(s t) }}%
					\eval\AblAct{\omega}(\LLA{\Gamma}( s t))
	\]
	is well-defined and continuous by \refer{enum1:stetigkeit_Kompo_von_hinten}
	and the continuity of \eqref{pktw-Kompo-lineare-Abb_gewAbb-und-LieG-k}
	and \eqref{pktw-Auswerung-lineare-Abb-gewAbb_Vektorfelder_LieAlg}.
	Hence by \refer{prop:Stetigkeit_parameterab_Int},
	\[
		\lim_{t\to 0}
		\frac{1}{t} ( (\gamma + t \gamma_1) \circ \omega_{\Gamma(t)} - \gamma \circ \omega_S )
		= - (\FAbl{\gamma} \circ \omega_{S} ) \MaMu \FAbl{\omega_{S }}%
					\eval\AblAct{\omega}( S^{-1} \cdot S_1 ) + \gamma_1 \circ \omega_S,
	\]
	so the directional derivatives of $\compo$ exist,
	are of the form \eqref{id:Ableitung_Kompo-gewAbb-LieGruppe}
	and depend continuously on the directions
	by \refer{enum1:stetigkeit_Kompo_von_hinten}
	and the continuity of \eqref{pktw-Kompo-lineare-Abb_gewAbb-und-LieG-k}
	and \eqref{pktw-Auswerung-lineare-Abb-gewAbb_Vektorfelder_LieAlg}.

	$\ell \to \ell + 1$:
	Since \eqref{pktw-Kompo-lineare-Abb_gewAbb-und-LieG-k}
	and \eqref{pktw-Auswerung-lineare-Abb-gewAbb_Vektorfelder_LieAlg}
	are $\ConDiff{}{}{\ell}$ by assumption,
	we conclude from \eqref{id:Ableitung_Kompo-gewAbb-LieGruppe}
	and the inductive hypothesis that $\dA{\compo}{}{}$ is $\ConDiff{}{}{\ell}$,
	whence $\compo$ is $\ConDiff{}{}{\ell + 1}$.
\end{proof}
\section{Conclusion and Examples}
Finally, we prove a sufficient criterion
for the smoothness of the conjugation action of a Lie group $\G$ acting on $\SX$ and $\DiffWz$.
\begin{satz}\label{satz:Hinreichend_Konjugation_Liegruppe-DiffWz_glatt}
	Let $\SX$ be a Banach space, $\G$ a Lie group, $\omega : \G \times \SX \to \SX$ a smooth action
	and $\GewFunk \subseteq \cl{\R}^\SX$ with $1_\SX \in \GewFunk$. Assume that 
	$
		\set{f \circ \omega_T}{f \in \GewFunk, T \in \G} \subseteq \ExtWeights{\GewFunk}
	$
	(we defined $\ExtWeights{\GewFunk}$ in \refer{def:alle_Gewichte,die_gleichen_Raum_erzeugen}),
	$
		\set{\FAbl{\omega_T} }{ T \in \G } \subseteq \BC{\SX}{\Lin{\SX}{\SX}}{\infty},
	$
	the maps
	\[
		\label{Differential_Liegruppe_in_BCinf-Abb_glatt}
		\tag{\ensuremath{\dagger}}
		\FAbl{} : \G \to \BC{\SX}{\Lin{\SX}{\SX}}{\infty} : T \mapsto \FAbl{\omega_T}
	\]
	and \eqref{pktw-Auswerung-lineare-Abb-gewAbb_Vektorfelder_LieAlg} are well-defined and smooth
	and \refer{bed:Stetigkeit-Komposition-gewAbb_LieGr} is satisfied.
	Then the map
	\[
		\G \times \DiffWz \to \DiffWz : (T, \phi) \mapsto \omega_T \circ \phi \circ \omega_T^{-1}
	\]
	is well-defined and smooth.
	\index{semidirect product!of $\DiffWz$ and a Lie group acting on $\SX$}%
	\index{diffeomorphisms!groups of!semidirect product with|see{semidirect product}}%
\end{satz}
\begin{proof}
	Since \eqref{Differential_Liegruppe_in_BCinf-Abb_glatt} is well-defined and smooth,
	we can apply \refer{cor:Auswertung_linearer_abb_und_CF} to see that
	\eqref{Auswertung-gewAbbinfty-Ableitung_Liegruppe} is also well-defined and smooth.
	Similarly, using \refer{cor:Komposition_linearer_Abb_und_CF}, we see that \eqref{pktw-Kompo-lineare-Abb_gewAbb-und-LieG-(k-1)}
	and \eqref{pktw-Kompo-lineare-Abb_gewAbb-und-LieG-k}
	are well-defined and smooth.
	Hence \refer{prop:Diffbarkeit-Komposition-gewAbb-Liegruppe}
	shows that \eqref{Komposition-gewAbbinfty-Liegruppe} is smooth.
	The assertion follows from \refer{lem:Glatte_Wirkung_Lie-Gruppe_auf_DiffWz}.
\end{proof}
Finally, we give a positive and a negative example.
The first example shows that we can form the semidirect product
$\Diff{\SX}{}{\mathcal{S}}_0 \rtimes \GL{\SX}$ with respect to the conjugation.
\begin{lem}\label{lem:Multiplication_LineareAbb_gewAbb}
	Let $\SX$, $\SY$ and $\SZ$ be normed spaces,
	$\UF \sub \SX$ an open nonempty set,
	$\GewFunk \sub \cl{\R}^\UF$ nonempty
	such that for each $f \in \GewFunk$, $f \norm{\cdot} \in \ExtWeights{\GewFunk}$.
	Further, let $k \in \cl{\N}$ and $b : \SY \times \SX \to \SZ$
	a continuous bilinear map.
	Then
	\[
		\CcF{\UF}{\SY}{k} \times \Lin{\SX}{\SX} \to \CcF{\UF}{\SZ}{k}
		: (\gamma, T) \mapsto b \circ (\gamma, T)
		\tag{\ensuremath{\dagger}}\label{bilineareWirkung_gewAbb_lineareOp}
	\]
	is well-defined and smooth.
\end{lem}
\begin{proof}
	The assertion holds for $k = \infty$ if it holds for all $k \in \N$.
	For $k \neq \infty$, the proof is by induction on $k$.

	$k=0$:
	Since \eqref{bilineareWirkung_gewAbb_lineareOp}
	is bilinear, it is smooth iff it is continuous in $0$.
	So we only prove that.
	Let $f \in\GewFunk$, $\gamma \in \CcF{\UF}{\SY}{k}$, $T \in \Lin{\SX}{\SX}$
	and $x \in \UF$.
	Then
	\[
		\abs{f(x)}\,\norm{b(\gamma(x), T(x))}
		\leq \Opnorm{b}\abs{f(x)}\,\norm{x}\,\norm{\gamma(x)}\,\Opnorm{T}
		\leq \Opnorm{b} \hn{\gamma}{f \norm{\cdot}}{0} \Opnorm{T}.
	\]
	We conclude that $b \circ (\gamma, T) \in \CcF{\UF}{\SZ}{0}$
	and that \eqref{bilineareWirkung_gewAbb_lineareOp} is continuous in $0$.

	$k \to k+1$:
	By \refer{lem:Ableitung_Kompo_mit_multilinearer_Abb},
	we have for $\gamma \in \CcF{\UF}{\SY}{k}$ and $T \in \Lin{\SX}{\SX}$ that
	\[
		\FAbl{(b \circ (\gamma, T))}
		= b^{(1)} \circ (\FAbl{\gamma}, T) + b^{(2)} \circ (\gamma, \FAbl{T}).
	\]
	Since $\FAbl{T} \in \BC{\SX}{\Lin{\SX}{\SX}}{\infty}$,
	by \refer{prop:multilineare_Abb_und_CF}
	$b^{(2)} \circ (\gamma, \FAbl{T}) \in \CcF{\UF}{\Lin{\SX}{\SZ} }{k + 1}$
	and the map $(\gamma, T) \mapsto b^{(2)} \circ (\gamma, \FAbl{T})$ is smooth
	(here we use that $\Lin{\SX}{\SX} \to \BC{\SX}{\Lin{\SX}{\SX}}{\infty} : T \mapsto \FAbl{T}$ is smooth).
	By the induction hypothesis, the same holds for
	$(\gamma, T) \mapsto b^{(1)} \circ (\FAbl{\gamma}, T)$.
	So using \refer{prop:topologische_Zerlegung_von_CFk}, the proof is finished.
\end{proof}
\begin{lem}\label{lem:Berechnungen_Wirkung_GL_Schwartzraum}
	Let $\SX$ be a Banach space and $\G \ndef \GL{\SX}$.
	We define the action
	\[
		\omega : \G \times \SX \to \SX : (g, x) \mapsto g(x)
	\]
	and set $\GewFunk \ndef \set{x \mapsto \norm{x}^n}{n\in\N}$.
	Then
	\begin{enumerate}
		\item\label{enum1:Berechnungen_Wirkung_GL_SchwartzraumA}
		The map \eqref{pktw-Auswerung-lineare-Abb-gewAbb_Vektorfelder_LieAlg} is smooth.

		\item\label{enum1:Berechnungen_Wirkung_GL_SchwartzraumB}
		The \refer{bed:Stetigkeit-Komposition-gewAbb_LieGr} is satisfied.
	\end{enumerate}
\end{lem}
\begin{proof}
	We easily see that $\AblAct{\omega} = -\id{\Lin{\SX}{\SX}}$ (since $\LieAlg{\G} = \Lin{\SX}{\SX}$),
	and for each $S \in \G$ and $x\in\SX$, $\omega_S = S$ and $\FAbl{S}(x) = S$.
	For \refer{enum1:Berechnungen_Wirkung_GL_SchwartzraumA}, we give two different proofs.
	The first one uses \refer{lem:Multiplication_LineareAbb_gewAbb},
	the second uses a topology on the multiplier space $\Lin{\SX}{\SX}$.
	
	\refer{enum1:Berechnungen_Wirkung_GL_SchwartzraumA}
	\emph{First variant:}
	Let $\SY$ be another normed space.
	Since for $\Gamma \in \CcF{\SX}{\Lin{\SX}{\SY}}{k}$ and $S \in \LieAlg{\G}$,
	$\Gamma \cdot \AblAct{\omega}(S) = \evTwo_{\Lin{\SX}{\SY}} \circ (\Gamma, -S)$
	and $\evTwo_{\Lin{\SX}{\SY}}$ is bilinear and continuous,
	this is a consequence of \refer{lem:Multiplication_LineareAbb_gewAbb}.
	
	\emph{Second variant:}
	Obviously $\AblAct{\omega}(\LieAlg{\G}) = \Lin{\SX}{\SX}$
	consists of \mur{}s.
	Further, \refer{bed:Bedingung_Stetigkeit_bilineare_OperationMultiplier}
	is satisfied (where $\mathcal{T} = \Lin{\SX}{\SX}$ and the family of \mur{}s is given by $\id{\Lin{\SX}{\SX}}$)
	since for $A, B \in \Lin{\SX}{\SX}$ and $x \in \SX$
	\[
		\norm{(A - B)(x)} \leq \Opnorm{A - B} \norm{x}
	\]
	and
	\[
		\norm{\FAbl{(A - B)}(x)} = \Opnorm{A - B}
	\]
	and $\norm{\FAbl[k]{(A - B)}} = \norm{0} = 0$ for $k > 1$.
	Hence we can apply \refer{lem:Kriterium_simultane_Stetigkeit_bilineare_Operation_Multiplier-gewAbb}
	to see that \eqref{pktw-Auswerung-lineare-Abb-gewAbb_Vektorfelder_LieAlg} is smooth.

	\refer{enum1:Berechnungen_Wirkung_GL_SchwartzraumB}
	Let $f = \norm{\cdot}^n \in \GewFunk$, $T\in\G$ and $\eps > 0$.
	There exists an open convex $\UF \in \neigh[\G]{T}$ such that for all $S \in \UF$,
	\begin{itemize}
		\item
		$\Opnorm{S - T} < \eps$

		\item
		$\Opnorm{S^{-1}} < 2 \Opnorm{T^{-1}}$

		\item
		$\Opnorm{S} < 2 \Opnorm{T}$.
	\end{itemize}
	Then the path
	$W : [0,1] \to \G : t \mapsto t T + (1 - t) S$ has the left logarithmic derivative
	$\LLA{W}(t) = W(t)^{-1} (T - S)$, hence $\UF$ is
	\logbound{\clBall[\Lin{\SX}{\SX}]{0}{2 \Opnorm{T} \eps}}.
	We calculate for $x \in \SX$, $S \in \UF$ and $A \in \clBall[\Lin{\SX}{\SX}]{0}{2 \Opnorm{T} \eps}$ that
	\begin{multline*}
		\abs{f(x)}\, \norm{\FAbl{\omega_S}(x) \eval \AblAct{\omega}(A)(x)}
		= \norm{x}^n \norm{(S \circ A)(x)}
		\leq \Opnorm{S} \Opnorm{A} \norm{x}^{n+1}
		\\
		\leq 4 \Opnorm{T}^2 \eps \norm{S^{-1} S x}^{n+1}
		\leq \eps 2^{n + 3} \Opnorm{T}^2 \Opnorm{T^{-1}}^{n + 1} \norm{S x}^{n+1}.
	\end{multline*}
	Since $x \mapsto 2^{n + 3} \Opnorm{T}^2 \Opnorm{T^{-1}}^{n + 1} \norm{x}^{n+1} \in \ExtWeights{\GewFunk}$,
	we see that \refer{bed:Stetigkeit-Komposition-gewAbb_LieGr}
	is satisfied.
\end{proof}
\begin{beisp}\label{beisp:GL_operiert_glatt_auf_DiffS}
	Let $\SX$, $\G$, $\omega$ and $\GewFunk$ be as in \refer{lem:Berechnungen_Wirkung_GL_Schwartzraum}.
	For each $S \in \G$ and $x\in\SX$, $\FAbl{S}(x) = S$. Hence the map
	\[
		\FAbl{} : \G \to \BC{\SX}{\Lin{\SX}{\SX}}{\infty} : S \mapsto \FAbl{S}
	\]
	is smooth.
	By \refer{lem:Berechnungen_Wirkung_GL_Schwartzraum}, the assumptions of \refer{satz:Hinreichend_Konjugation_Liegruppe-DiffWz_glatt} hold
	(since $\GewFunk \circ \G \subseteq \ExtWeights{\GewFunk}$ is obviously true),
	hence the map
	\[
		\GL{\SX} \times \DiffWz \to \DiffWz : (T, \phi) \mapsto T \circ \phi \circ T^{-1}
	\]
	is smooth.
	So using \refer{lem:Semidirektes_Produkt_Liegruppen-glatteWirkung-ist_Liegruppe},
	we can form the semidirect product
	\[
		\DiffWz \rtimes \GL{\SX}
	\]
	with respect to the inner automorphisms on $\DiffWz$ that are induced by $\GL{\SX}$.
\end{beisp}
Finally, we show that the the conjugation of $\GL{\R}$ on $\DiffGeWz{\sset{1_\R}}$,
if it was defined, could not be continuous.
\begin{beisp}\label{beisp:GL_operiert_NICHT_auf_DiffBC}
	For each $n\in\N$, $\sin((1 + \frac{1}{2 n}) n \pi) = \pm 1$,
	but $\sin(n \pi) = 0$.
	Hence
	\[
		\left\hn{ \sin(t_n \cdot) -  \sin\right}{1_\R}{0} \geq 1
	\]
	for each $n \in \N$, where $t_n \ndef 1 + \frac{1}{2 n}$.
	We see with \refer{lem:Offenheit_H_UV} that
	\[
		\frac{1}{2} \sin \in \karte{\sset{1_\R}}^{-1}(\Diff{\R}{}{\sset{1_\R}}),
	\]
	and obviously $\karte{\sset{1_\R}}(\frac{1}{2} \sin) \in \DiffGeWz{\sset{1_\R}}$.
	If the conjugation of $\GL{\R}$ on $\DiffGeWz{\sset{1_\R}}$ was defined and continuous,
	then the map
	\[
		\R\setminus\sset{0} \times \BC{\R}{\R}{\infty} \to \BC{\R}{\R}{\infty}
		: (t, \gamma) \mapsto t^{-1} \gamma(t \cdot)
	\]
	would be continuous in $(1, \frac{1}{2} \sin)$.
	But it is not since for $ t > 0$ and $x \in \R$
	\begin{multline*}
		 \norm{t^{-1} \sin(t x) - \sin(x)}
		\geq
		t^{-1} \norm{\sin(t x) - \sin(x)} - \norm{(t^{-1} -  1) \sin(x)}
		\\
		\geq
		t^{-1} \norm{\sin(t x) - \sin(x)} - \abs{t^{-1} -  1};
	\end{multline*}
	hence we can calculate that for sufficiently large $n$,
	\[
		\left\hn{ \tfrac{1}{2} t_n^{-1} \sin(t_n \cdot) - \tfrac{1}{2} \sin\right}{1_\R}{0}
		\geq \tfrac{1}{4}.
	\]
\end{beisp}
\chapter{Lie group structures on weighted mapping groups}
In this chapter we will use the weighted function spaces
discussed in \refer{chp:Modellraeume} for the construction
of locally convex Lie groups, the \emph{weighted mapping groups}.
These groups arise as subgroups of $\G^{\UF}$,
where $\G$ is a suitable Lie group and $\UF$ is an open nonempty subset of a normed space.
First, we give some definitions that are used throughout this chapter.
\setcounter{COUNT}{0}
\begin{defi}\label{defi:Gruppenoperationen_Abbildungsgruppen}
	Let $\UF$ be a nonempty set and $\G$ be a group with the multiplication map $m_\G$
	and the inversion map $I_\G$.
	Then $\G^\UF$ can be endowed with a group structure:
	The multiplication is given by
	\[
		((g_u)_{u \in \UF}, (h_u)_{u \in \UF}) \mapsto (m_\G(g_u, h_u))_{u \in \UF} = m_\G \circ ((g_u)_{u \in \UF}, (h_u)_{u \in \UF})
	\]
	and the inversion by
	\[
		(g_u)_{u \in \UF} \mapsto (I_\G(g_u))_{u \in \UF} = I_\G \circ (g_u)_{u \in \UF}.
	\]
	Further we call a set $A \subseteq \G$ \emph{symmetric} if
	\[
		A = I_\G(A).
	\]
	Inductively, for $n \in \N$ with $n \geq 1$ we define
	\[
		A^{n + 1} \ndef m_\G(A^{n} \times A),
	\]
	where $A^{1} \ndef A$.
\end{defi}

\begin{defi}
	Let $\G$ be a Lie group and $\phi : \VF \to \LieAlg{\G}$ a chart.
	We call the pair $(\phi,\VF)$ \emph{centered around $\one$} or just \emph{centered}
	if $\VF \subseteq \G$ is an open identity neighborhood and $\phi(\one)= 0$.
	\index{centered chart}%
\end{defi}

\section{Weighted maps into Banach Lie groups}
\label{sec:gewAbbildungsgruppen_Banach-Lie}
In this section, we discuss certain subgroups of $\G^\UF$, where $\G$ is a Banach Lie group
and $\UF$ an open subset of a normed space $\SX$.
We construct a subgroup $\CcF{\UF}{\G}{k}$ consisting of \emph{weighted mappings} that can be turned into a (connected) Lie group.
Its modelling space is $\CcF{\UF}{\LieAlg{\G}}{k}$, where $k \in \cl{\N}$
and $\GewFunk$ is a set of weights on $\UF$ containing $1_\UF$.
Later we prove that these groups are regular Lie groups.
Finally, we discuss the case when $\UF = \SX$. Then $\DiffW$ acts on $\CcF{\SX}{\G}{\infty}$,
and this we can turn the semidirect product of these groups into a Lie group.

\subsection{Construction of the Lie group}
We construct the Lie group from local data using \refer{lem:Erzeugung_von_Liegruppen_aus_lokalen}.
For a chart $(\phi, \VF)$ of $\G$, we can endow the set $\phi^{-1} \circ \CcFo{\UF}{\phi(\VF)}{k} \sub \G^\UF$
with the manifold structure that turns the superposition operator $\phi_\ast$ into a chart.
We need to check whether the local multiplication and inversion on this set
are smooth with respect to this manifold structure.
The group operations on $\G^\UF$ arise as the composition of
the corresponding operations on $\G$ with the mappings (see \refer{defi:Gruppenoperationen_Abbildungsgruppen}).
Since the group operations of Banach Lie groups are analytic,
we will use the results of \refer{susec:Composing_weighted_functions_with_analytic_maps} as our main tools.
The use of this tools allows to construct $\CcF{\UF}{\G}{k}$
when $\G$ is an analytic Lie group modelled on an arbitrary normed space.

\begin{bem}
	We call a Lie group $\G$ \emph{normed} if $\LieAlg{\G}$ is a normable space.
	A \emph{normed analytic} Lie group is a normed Lie group which is an analytic Lie group.
\end{bem}

\paragraph{Local multiplication}
The treatment of the group multiplication is a simple application of \refer{prop:Operation_analytischer_Abb_auf_CWk}.
\begin{lem}\label{lem:Multiplikation_von_gewichteten_Abbildungsraumen}
	Let $\SX$ be a normed space, $\UF \subseteq \SX$ an open nonempty subset,
	$\GewFunk \subseteq \cl{\R}^\UF$ with $1_\UF \in \GewFunk$, $\ell \in \cl{\N}$,
	$\G$ an normed analytic Lie group with the group multiplication $m_\G$
	and $(\phi, \VF)$ a centered chart of $\G$.
	Then there exists an open identity neighborhood $\WF \subseteq \VF$ such that the map
	\[
		\CcFo{\UF}{\phi(\WF)}{\ell} \times \CcFo{\UF}{\phi(\WF)}{\ell} \to \CcFo{\UF}{\phi(\VF)}{\ell}
		: (\gamma, \eta) \mapsto \phi \circ m_\G  \circ (\phi^{-1} \circ \gamma, \phi^{-1} \circ \eta)
		\tag{\ensuremath{\dagger}}\label{lokale_Gruppen-Multiplikation_in_Koordinaten}
	\]
	 is defined and analytic.
\end{lem}
\begin{proof}
	By \refer{lem:gewichtete,verschwindende_Abb_Produktisomorphie-lokalkonvex}, the map \eqref{lokale_Gruppen-Multiplikation_in_Koordinaten} is defined and analytic
	iff there exists an open identity neighborhood $\WF \subseteq \G$ such that
	\[
		\Hom{(\phi \circ m_\G \circ (\phi^{-1} \times \phi^{-1}) )}{\cW}{}
		: \CcFo{\UF}{\phi(\WF) \times \phi(\WF)}{\ell} \to \CcFo{\UF}{\phi(\VF)}{\ell}
	\]
	is so.
	There exists an open bounded zero neighborhood $\widetilde{\WF_L} \subseteq \LieAlg{\G}$ such that
	$\widetilde{\WF_L} + \widetilde{\WF_L} \subseteq \phi(\VF)$.
	By the continuity of the multiplication $m_\G$
	there exists an open $\one$-neighborhood $\WF$ with
	$m_\G(\WF \times \WF) \subseteq \phi^{-1}(\widetilde{\WF_L})$.
	We may assume w.l.o.g. that $\phi(\WF)$ is star-shaped with center~$0$. Then
	\[
		(\phi \circ m_\G \circ (\phi^{-1} \times \phi^{-1}))(\phi(\WF) \times \phi(\WF))
		\subseteq \widetilde{\WF_L}.
	\]
	Further the restriction of $\phi \circ m_\G \circ (\phi^{-1} \times \phi^{-1})$ to $\phi(\WF) \times \phi(\WF)$ is analytic,
	takes $(0,0)$ to $0$ and has bounded image, since $\phi$ is centered and $\widetilde{\WF_L} $ is bounded.
	In the real case, using \refer{lem:Verkleinerung_reell_analytischer_Abb_haben_gute_Komplexifizierung}
	we can choose $\phi(\WF)$ sufficiently small such that
	the restriction of $\phi \circ m_\G \circ (\phi^{-1} \times \phi^{-1})$
	to $\phi(\WF)$ has a good complexification.
	Hence we can apply \refer{prop:Operation_analytischer_Abb_auf_CWk} to see that
	\[
		(\phi \circ m_\G \circ (\phi^{-1} \times \phi^{-1}))
		\circ \CcFo{\UF}{\phi(\WF) \times \phi(\WF)}{\ell}
		\in \CcF{\UF}{\widetilde{\WF_L}}{\ell}
	\]
	and that the map $\Hom{(\phi \circ m_\G \circ (\phi^{-1} \times \phi^{-1})) }{\cW}{}$
	is analytic.
	But
	\[
		\CcF{\UF}{\widetilde{\WF_L}}{\ell}
		\subseteq \CcFo{\UF}{\phi(\VF)}{\ell}
	\]
	by the definition of $\widetilde{\WF_L}$, and this gives the assertion.
\end{proof}

\paragraph{Local inversion}
The discussion of the inversion is more delicate.
For a short explanation, let $(\phi, \widetilde{\VF})$
be a chart for $\G$, $\VF \sub \widetilde{\VF}$ a symmetric open identity neighborhood
and $I_\G$ the inversion of $\G$.
Then the superposition of $\phi \circ I_\G \circ \phi^{-1}$ described in \refer{prop:Operation_analytischer_Abb_auf_CWk}
does not necessarily map $\CcFo{\UF}{\phi(\VF)}{\ell}$ into itself;
hence we have to work to construct symmetrical open subsets.

\begin{lem}\label{lem:kleine_symmetrische_offene_Mengen_in_lokalen_Top_Gruppen}
	Let $\G$ be a group, $\UF \subseteq \G$ a topological space and
	$\VF \subseteq \UF$ a symmetric subset with $\one \in \interior{\VF}$
	such that the inversion $I_\G : \VF \to \VF$ is continuous.
	Then
	\[
		\interior{\VF} \cap I_\G(\interior{\VF})
	\]
	is a symmetric set that is open in $\UF$ and contains $\one$.
\end{lem}
\begin{proof}
	Let $\WF \ndef \interior{\VF} \cap I_\G(\interior{\VF})$. Then $\one\in \WF$, and since
	\[
		\WF^{-1} = I_\G(\WF) = I_\G(\interior{\VF} \cap I_\G(\interior{\VF}))
		= I_\G(\interior{\VF}) \cap I_\G(I_\G(\interior{\VF}))
		= I_\G(\interior{\VF}) \cap \interior{\VF}
		= \WF ,
	\]
	it is a symmetric set. Since $I_\G$ is a homeomorphism,
	$I_\G(\interior{\VF})$ is an open subset of $\VF$.
	Hence $\WF = I_\G(\interior{\VF}) \cap \interior{\VF}$
	is an open subset of $\interior{\VF}$ and hence of $\UF$.
\end{proof}

\begin{lem}\label{lem:Inversion_auf_gewichteten_Abbildungsraumen}
	Let $\SX$ be a normed space, $\UF \subseteq \SX$ an open nonempty subset,
	$\GewFunk \subseteq \cl{\R}^\UF$ with $1_\UF \in \GewFunk$, $\ell \in \cl{\N}$,
	$\G$ an normed analytic Lie group with the group inversion $I_\G$, $(\phi, \VF)$ a centered chart of $\G$
	such that $\phi(\VF)$ is bounded and $\VF$ is symmetric.
	Then the following statements hold:
	\begin{enumerate}
		\item\label{enum1:Lokale_Inversion_analytisch_und_wirkt_analytisch_auf_Funktionenraeumen}
		The map
		\[
			I_L \ndef \phi \circ I_\G \circ \phi^{-1} : \phi(\VF) \to \phi(\VF)
		\]
		is an analytic bijective involution.
		Hence for any open and star-shaped set $\WF \subseteq \phi(\VF)$ with center~$0$,
		the map
		\[
			\CcFo{\UF}{\WF}{\ell} \to \CcF{\UF}{\phi(\VF)}{\ell} : \gamma \mapsto I_L \circ \gamma
		\]
		is analytic, assuming in the real case that $\rest{I_L}{\WF}$ has a good complexification.

		\item\label{enum1:Konstruktion_von_symmetrischen_Mengen_in_Abbildungsgruppen}
		Let $\Omega \subseteq \CcFo{\UF}{\phi(\VF)}{\ell}$. Then
		$
			\phi^{-1} \circ (\Omega \cap I_L \circ \Omega)
		$
		is a symmetric subset of $\G^\UF$.

		\item\label{enum1:Inneres_Bild_Inversion-mapgrps_nichtleer}
		For any open zero neighborhood $\widetilde{\WF} \subseteq \phi(\VF)$ there
		exists an open convex zero neighborhood $\WF \sub \widetilde{\WF} $ such that
		\[
			\CcFo{\UF}{\WF}{\ell} \subseteq
			\CcFo{\UF}{\widetilde{\WF}}{\ell} \cap I_L \circ \CcFo{\UF}{\widetilde{\WF}}{\ell}.
		\]

		\item\label{enum1:Inversion-mapgrps_lokal_glatt}
		There exists an open convex zero neighborhood $\WF \subseteq \phi(\VF)$
		and a zero neighborhood $C_\GewFunk^\ell \subseteq \CcFo{\UF}{\phi(\VF)}{\ell}$
		such that
		\[
			\CcFo{\UF}{\WF}{\ell} \subseteq \interior{(C_\GewFunk^\ell)} \cap I_L \circ \interior{(C_\GewFunk^\ell)},
		\]
		$
			\phi^{-1} \circ C_\GewFunk^\ell
		$
		is symmetric in $\G^\UF$, the map
		\[
			C_\GewFunk^\ell \to C_\GewFunk^\ell : \gamma \mapsto I_L \circ \gamma
		\]
		is continuous and its restriction to $\interior{(C_\GewFunk^\ell)}$ is analytic.
		The set $\WF$ can be chosen independently of $\ell$ and $\GewFunk$.
	\end{enumerate}
\end{lem}
\begin{proof}
	\refer{enum1:Lokale_Inversion_analytisch_und_wirkt_analytisch_auf_Funktionenraeumen}
	The assertions concerning $I_L$ follow from the fact that $\VF$ is symmetric and $\G$
	is an analytic Lie group.
	\\
	The assertion on the superposition map of $I_L$ is a consequence of
	\refer{prop:Operation_analytischer_Abb_auf_CWk} since $\WF$ is star-shaped with center~$0$
	and $\phi(\VF)$ is bounded.

	\refer{enum1:Konstruktion_von_symmetrischen_Mengen_in_Abbildungsgruppen}
	This is an easy computation.

	\refer{enum1:Inneres_Bild_Inversion-mapgrps_nichtleer}
	By the continuity of the addition, we find an open zero neighborhood $H$
	with $H + H \subseteq \widetilde{\WF}$.
	Since $I_L$ is continuous in $0$ there exists an open convex zero neighborhood $\WF$
	with $I_L(\WF) \subseteq H$ and $\WF \subseteq \widetilde{\WF}$.
	Then
	\[
		\CcFo{\UF}{\WF}{\ell} \subseteq \CcFo{\UF}{\widetilde{\WF}}{\ell}
	\]
	and by \refer{enum1:Lokale_Inversion_analytisch_und_wirkt_analytisch_auf_Funktionenraeumen}
	\[
		I_L \circ \CcFo{\UF}{\WF}{\ell} \subseteq \CcF{\UF}{H}{\ell} \subseteq \CcFo{\UF}{\widetilde{\WF}}{\ell}.
	\]
	The fact that $I_L \circ I_L = \id{\phi(V)}$ completes the argument.

	\refer{enum1:Inversion-mapgrps_lokal_glatt}
	Let $\WF_3 \subseteq \phi(\VF)$ be an open convex zero neighborhood.
	Then by \refer{enum1:Inneres_Bild_Inversion-mapgrps_nichtleer}
	we find open convex zero neighborhoods $\WF_1, \WF_2 \subseteq \phi(\VF)$ such that
	\[
		\CcFo{\UF}{\WF_i}{\ell} \subseteq
			\CcFo{\UF}{\WF_{i + 1}}{\ell} \cap I_L \circ \CcFo{\UF}{\WF_{i + 1}}{\ell}
	\]
	for $i = 1, 2$. So
	\[
		C_\GewFunk^\ell \ndef \CcFo{\UF}{\WF_{3}}{\ell} \cap I_L \circ \CcFo{\UF}{\WF_{3}}{\ell}
	\]
	is a zero neighborhood, and by \refer{enum1:Konstruktion_von_symmetrischen_Mengen_in_Abbildungsgruppen},
	$\phi^{-1} \circ C_\GewFunk^\ell$ is symmetric.
	Hence the superposition of $I_L$ maps $C_\GewFunk^\ell$ into itself and
	is continuous on $C_\GewFunk^\ell$ and analytic on $\interior{(C_\GewFunk^\ell)}$
	(see \refer{enum1:Lokale_Inversion_analytisch_und_wirkt_analytisch_auf_Funktionenraeumen}).
	Further
	\[
		\interior{(C_\GewFunk^\ell)} \cap I_L \circ \interior{(C_\GewFunk^\ell)}
		\supseteq
		\CcFo{\UF}{\WF_{2}}{\ell} \cap I_L \circ \CcFo{\UF}{\WF_{2}}{\ell}
		\supseteq
		\CcFo{\UF}{\WF_{1}}{\ell},
	\]
	whence \refer{enum1:Inversion-mapgrps_lokal_glatt} is established with $\WF \ndef \WF_{1}$.
\end{proof}

\paragraph{Construction of the Lie group structure}
After discussing the group operations locally,
we turn a subgroup of $\G^\UF$ into a Lie group for each centered chart of $\G$.
We will also show that the identity component of this group
does not depend on the chart.

\begin{lem}\label{lem:Konvexe_Mengen_im_gewichteten_Funktionenraum}
	Let $\SX$ and $\SY$ be normed spaces, $\UF \subseteq \SX$ an open nonempty subset,
	$\GewFunk \subseteq \cl{\R}^\UF$ with $1_\UF \in \GewFunk$, $\ell \in \cl{\N}$
	and $\VF \subseteq \SY$ convex.
	Then the set $\CcFo{\UF}{\VF }{\ell}$ is convex.
\end{lem}
\begin{proof}
	It is obvious that $\CcF{\UF}{\VF }{\ell}$ is convex since $\VF$ is so.
	The set $\CcFo{\UF}{\VF }{\ell}$ is the interior of $\CcF{\UF}{\VF }{\ell}$
	with respect to the norm $\hn{\cdot}{1_\UF}{0}$, hence it is convex.
\end{proof}

\renewcommand{\H}{H}
\begin{prop}\label{prop:Liegruppenstruktur_auf_gewichteten_Abbildungsgruppen;Zusammenhangskomponente}
	Let $\SX$ be a normed space, $\UF \subseteq \SX$ an open nonempty subset,
	$\GewFunk \subseteq \cl{\R}^\UF$ with $1_\UF \in \GewFunk$, $\ell \in \cl{\N}$,
	$\G$ an normed analytic Lie group
	and $(\phi, \VF)$ a centered chart.
	There exist a subgroup $(\G, \phi)^\UF_{\GewFunk, \ell}$ of $\G^\UF$
	that can be turned into an analytic Lie group which is modelled on $\CcF{\UF}{\LieAlg{\G}}{\ell}$;
	and an open $\one$-neighborhood $\WF \subseteq \VF$ which is independent of $\GewFunk$ and $\ell$ such that
	\[
		\CcFo{\UF}{\phi(\WF)}{\ell} \to (\G, \phi)^\UF_{\GewFunk, \ell}
		: \gamma \mapsto \phi^{-1} \circ \gamma
	\]
	is an analytic embedding onto an open set.
	Moreover, for any convex open zero neighborhood $\widetilde{\WF} \subseteq \phi(\WF)$,
	the set
	$
		\phi^{-1} \circ \CcFo{\UF}{\widetilde{\WF}}{\ell}
	$
	generates the identity component of $(\G, \phi)^\UF_{\GewFunk, \ell}$ as a group.
\end{prop}
\begin{proof}
	Using \refer{lem:Multiplikation_von_gewichteten_Abbildungsraumen}
	we find an open $\one$-neighborhood $\widetilde{\WF} \subseteq \VF$ such that
	\[
		\CcFo{\UF}{\phi(\widetilde{\WF})}{\ell} \times \CcFo{\UF}{\phi(\widetilde{\WF})}{\ell}
			\to \CcFo{\UF}{\phi(\VF)}{\ell}
		: (\gamma, \eta) \mapsto \phi \circ m_\G  \circ (\phi^{-1} \circ \gamma, \phi^{-1} \circ \eta)
	\]
	is analytic.
	We may assume w.l.o.g. that $\widetilde{\WF}$ is symmetric. With
	\refer{lem:Inversion_auf_gewichteten_Abbildungsraumen}~\refer{enum1:Inversion-mapgrps_lokal_glatt}
	and \refer{lem:kleine_symmetrische_offene_Mengen_in_lokalen_Top_Gruppen},
	we find an open zero neighborhood $\H \subseteq \CcFo{\UF}{\phi(\widetilde{\WF})}{\ell}$
	such that $\phi^{-1} \circ \H$ is symmetric, the map
	\[
		\H \to \H : \gamma \mapsto I_L \circ \gamma
	\]
	is analytic and $\CcFo{\UF}{\phi(\WF)}{\ell} \subseteq \H$
	for some open $\one$-neighborhood $\WF \subseteq \VF$, which is independent of $\GewFunk$ and $\ell$.
	We endow $\phi^{-1} \circ \H$ with the differential structure
	which turns the bijection
	\[
		\phi^{-1} \circ \H \to \H : \gamma \mapsto \phi \circ \gamma
	\]
	into an analytic diffeomorphism.
	Then we can apply \refer{lem:Erzeugung_von_Liegruppen_aus_lokalen}
	to construct an analytic Lie group structure on the subgroup $(\G, \phi)^\UF_{\GewFunk, \ell}$ of $\G^\UF$
	which is generated by $\phi^{-1} \circ \H$ such that $\phi^{-1} \circ \H$ becomes an open subset of $(\G, \phi)^\UF_{\GewFunk, \ell}$.

	Since we may assume w.l.o.g. that $\phi(\WF)$ is convex,
	$\CcFo{\UF}{\phi(\WF)}{\ell}$ is open and convex
	(see \refer{lem:Konvexe_Mengen_im_gewichteten_Funktionenraum}),
	hence the set
	\[
		\phi^{-1} \circ \CcFo{\UF}{\phi(\WF)}{\ell}
	\]
	is connected and open by the construction of the differential structure of $(\G, \phi)^\UF_{\GewFunk, \ell}$.
	Furthermore it obviously contains the unit element, whence it generates the identity component.
\end{proof}

\begin{lem}\label{lem:Zusammenhangskomponente_gewichtete_Abbildungsgruppen_kartenunabhaengig}
	Let $\SX$ be a normed space, $\UF \subseteq \SX$ an open nonempty subset,
	$\GewFunk \subseteq \cl{\R}^\UF$ with $1_\UF \in \GewFunk$, $\ell \in \cl{\N}$
	and $\G$ be an normed analytic Lie group.
	Then for centered charts $(\phi_1, \VF_1)$, $(\phi_2, \VF_2)$,
	the identity component of $(\G, \phi_1)^\UF_{\GewFunk, \ell}$ coincides with the one of $(\G, \phi_2)^\UF_{\GewFunk, \ell}$,
	and the identity map between them is an analytic diffeomorphism.
\end{lem}
\begin{proof}
	We may assume w.l.o.g. that $\phi_1(\VF_1)$ and $\phi_2(\VF_2)$ are bounded.
	Using \refer{prop:Liegruppenstruktur_auf_gewichteten_Abbildungsgruppen;Zusammenhangskomponente},
	we find open $\one$-neighborhoods $\WF_1 \subseteq \VF_1$, $\WF_2 \subseteq \VF_2$
	such that the identity component of $(\G, \phi_i)^\UF_{\GewFunk, \ell}$ is generated
	by $\phi^{-1}_i \circ \CcFo{\UF}{\phi_i(\WF_i)}{\ell}$ for $i \in\{1,2\}$.
	Since $\phi_1 \circ \phi_2^{-1}$ is analytic, we find open zero neighborhoods
	$\widetilde{\WF}^L_1 \subseteq \phi_1(\WF_1)$ and $\widetilde{\WF}^L_2 \subseteq \phi_2(\WF_2)$
	such that
	\[
		(\phi_1 \circ \phi_2^{-1})(\widetilde{\WF}^L_2) \subseteq \widetilde{\WF}^L_1
		\text{ and }
		\widetilde{\WF}^L_1 + \widetilde{\WF}^L_1 \subseteq \phi_1(\WF_1)
	\]
	and $\widetilde{\WF}^L_2$ is convex.
	Then by \refer{prop:Liegruppenstruktur_auf_gewichteten_Abbildungsgruppen;Zusammenhangskomponente},
	the identity component of $(\G, \phi_2)^\UF_{\GewFunk, \ell}$ is generated by
	$
		\phi^{-1}_2 \circ \CcFo{\UF}{ \widetilde{\WF}^L_2 }{\ell},
	$
	and in the real case we may assume that $\rest{\phi_1 \circ \phi_2^{-1} }{\widetilde{\WF}^L_2}$
	has a good complexification.
	By \refer{prop:Operation_analytischer_Abb_auf_CWk} the map
	\[
		\CcFo{\UF}{\widetilde{\WF}^L_2}{\ell} \to \CcFo{\UF}{\phi_1(\WF_1)}{\ell}
		:
		\gamma \mapsto \phi_1 \circ \phi_2^{-1} \circ \gamma
	\]
	is defined and analytic, and this implies that
	\[
		\phi_2^{-1} \circ \CcFo{\UF}{\widetilde{\WF}^L_2}{\ell}
		\subseteq
		\phi_1^{-1} \circ \CcFo{\UF}{\phi_1(\WF_1)}{\ell}.
	\]
	Hence the identity component of $(\G, \phi_2)^\UF_{\GewFunk, \ell}$
	is contained in the one of $(\G, \phi_1)^\UF_{\GewFunk, \ell}$,
	and the inclusion map of the former into the latter is analytic.

	Exchanging the roles of $\phi_1$ and $\phi_2$ in the preceding argument, we get the assertion.
\end{proof}

\begin{defi}\label{def:Zsghd_gewichtete_Abb_Gruppe}
	Let $\SX$ be a normed space, $\UF \subseteq \SX$ an open nonempty subset,
	$\GewFunk \subseteq \cl{\R}^\UF$ with $1_\UF \in \GewFunk$, $\ell \in \cl{\N}$
	and $\G$ be an normed analytic Lie group.
	We write $\glstext{gew_Abbildungsgruppe_BL-Gruppe}$
	\index{weighted maps!into Banach Lie groups}%
	\index{mapping groups!with values in a Banach Lie group}
	for  the connected Lie group that was constructed
	in \refer{prop:Liegruppenstruktur_auf_gewichteten_Abbildungsgruppen;Zusammenhangskomponente}.
	There and in \refer{lem:Zusammenhangskomponente_gewichtete_Abbildungsgruppen_kartenunabhaengig}
	it was proved that for any centered chart $(\phi, \VF)$ of $\G$ and $\WF \sub \VF$ such that $\phi(\WF)$ is convex,
	the inverse map of
	\[
		\CcFo{\UF}{\phi(\WF)}{\ell} \to \CcF{\UF}{\G}{\ell} : \gamma \mapsto \phi^{-1} \circ \gamma
	\]
	is a chart.
\end{defi}

\subsection{Regularity}
We show that for a Banach Lie group $\G$, the Lie group $\CcF{\UF}{\G}{\ell}$
is regular.

\begin{lem}\label{lem:LLA_Komposition_mit_Lie-Morph}
Let $G, H$ be Lie groups and $\phi : G \to H$ a Lie group morphism.
	\begin{enumerate}
		\item\label{enum1:Tangentialabbildung_Gruppenmorphismus}
		For each $g \in G$ and $v \in \Tang[g]{G}$, we have
		$
			\Tang[g]{\phi}(v) = \phi(g) \cdot \LieAlg{\phi}(g^{-1}\cdot v).
		$

		\item\label{enum1:LLA_Komposition_mit_Lie-Morph}
		Let $\gamma \in \ConDiff{[0,1]}{\G}{1}$.
		Then $\LLA{\phi \circ \gamma} = \LieAlg{\phi} \circ \LLA{\gamma}$.
	\end{enumerate}
\end{lem}
\begin{proof}
	The proof of \refer{enum1:Tangentialabbildung_Gruppenmorphismus}
	being straightforward, we turn to \refer{enum1:LLA_Komposition_mit_Lie-Morph}.
	We calculate the derivative of $\phi \circ \gamma$ using \refer{enum1:Tangentialabbildung_Gruppenmorphismus}
	and the fact that $\phi$ is a Lie group morphism:
	\[
		(\phi \circ \gamma)'(t)
		= \Tang{(\phi \circ \gamma)} (t,1)
		= \Tang[\gamma(t)]{\phi} (\gamma'(t))
		= \phi(\gamma(t)) \cdot \LieAlg{\phi}(\gamma(t)^{-1} \cdot \gamma'(t)).
	\]
	From this we derive
	\[
		\LLA{\phi \circ \gamma}(t)
		=(\phi \circ \gamma)(t)^{-1} \cdot (\phi \circ \gamma)'(t)
		= \LieAlg{\phi}(\gamma(t)^{-1} \cdot \gamma'(t))
		= \LieAlg{\phi}(\LLA{\gamma}(t)),
	\]
	and the proof is finished.
\end{proof}

The following is well known from the theory of Banach Lie groups.
\begin{lem}\label{lem:Nette_Reg-DGL_Abschaetzung_BLG}
	Let $\G$ be a Banach Lie group and $\VF \in \neigh{\one}$.
	Then there exists a balanced open $\WF \in \neigh[\LieAlg{\G}]{0}$
	such that
	\begin{equation}\label{Nette_Reg-DGL_Abschaetzung_BLG-LinksEvol}
		\gamma \in \ConDiff{[0,1]}{\WF}{0}
		\implies
		\lEvol[\G]{\gamma} \in \ConDiff{[0,1]}{\VF}{0}.
	\end{equation}
	Furthermore, the map $\levol[\G]{} : \ConDiff{[0,1]}{\WF}{0} \to \G$ is continuous.
\end{lem}

We define some terminology needed for the proof.

\begin{defi}\label{def:Inklusion_Diagramme_und_Lie-Funktoren}
	Let $\SX$ be a normed space, $\UF \subseteq \SX$ an open nonempty set,
	$\GewFunk \subseteq \cl{\R}^\UF$ with $1_\UF \in \GewFunk$, $k\in\cl{\N}$
	and $\G$ be a Banach Lie group.
	Further, let $\cF_1, \cF_2 \subseteq \GewFunk$ such that $1_\UF \in \cF_1 \subseteq \cF_2$
	and $\ell_1, \ell_2 \in \cl{\N}$ such that $\ell_1 \leq \ell_2 \leq k$.
	We denote the inclusion
	\[
		\CF{\UF}{\LieAlg{\G}}{\cF_2}{\ell_2} \to \CF{\UF}{\LieAlg{\G}}{\cF_1}{\ell_1}.
	\]
	by $\iota_{(\cF_2, \ell_2), (\cF_1, \ell_1)}^L$ and the inclusion
	\[
		\CF{\UF}{\G}{\cF_2}{\ell_2} \to \CF{\UF}{\G}{\cF_1}{\ell_1}
	\]
	by $\iota_{(\cF_2, \ell_2), (\cF_1, \ell_1)}^\G$.
	Further, we define $\iota_{\cF_1, \ell_1}^L \ndef \iota_{(\GewFunk, k), (\cF_1, \ell_1)}^L$
	and $\iota_{\cF_1, \ell_1}^\G \ndef \iota_{(\GewFunk, k), (\cF_1, \ell_1)}^\G$.
	Then for a suitable centered chart $(\phi, \VF)$ of $\G$, the diagram
	\[
		\xymatrix{
		{ \CFo{\UF}{\phi(\VF)}{\cF_2}{\ell_2} } \ar[rr]^-{ \phi^{-1}_* } \ar[d]|{ \iota^L_{(\cF_2, \ell_2), (\cF_1, \ell_1)} }%
		&& { \CF{\UF}{\G}{\cF_2}{\ell_2} } \ar[d]|{ \iota_{(\cF_2, \ell_2), (\cF_1, \ell_1)}^\G }
		\\
		{ \CFo{\UF}{\phi(\VF)}{\cF_1}{\ell_1} } \ar[rr]_-{ \phi^{-1}_* }%
		&& { \CF{\UF}{\G}{\cF_1}{\ell_1} }
		}
	\]
	commutes. Hence we derive the identity
	\[
		\LieAlg{\iota_{(\cF_2, \ell_2), (\cF_1, \ell_1)}^\G}
		= \Tang[0]{\phi^{-1}_*} \circ \Tang[0]{\iota_{(\cF_2, \ell_2), (\cF_1, \ell_1)}^L} \circ \Tang[\one]{\phi_*}.
	\]
	Let $x \in \UF$. We let $\evTwo_x^\G$ resp. $\evTwo_x^L$ denote the maps
	\[
		\evTwo_x^\G : \CFo{\UF}{\G}{\cF_1}{\ell_1} \to \G : \gamma \mapsto \gamma(x)
		\qquad
		\evTwo_x^L : \CFo{\UF}{\LieAlg{\G}}{\cF_1}{\ell_1}  \to \LieAlg{\G} : \gamma \mapsto \gamma(x).
	\]
	Obviously, the diagram
	\[
		\xymatrix{
		{ \CFo{\UF}{\phi(\VF)}{\cF_1}{\ell_1} } \ar[rr]^-{ \phi^{-1}_* } \ar[d]|{ \evTwo_x^L }%
		&& { \CF{\UF}{\G}{\cF_1}{\ell_1} } \ar[d]^{ \evTwo_x^\G }
		\\
		{ \phi(\VF) } \ar[rr]_-{ \phi^{-1} }%
		&& { \G }
		}
	\]
	commutes, so we derive the identity
	\[
		\LieAlg{\evTwo_x^\G}
		= \Tang[0]{\phi^{-1}} \circ \Tang[0]{\evTwo_x^L} \circ \Tang[\one]{\phi_*}.
	\]
\end{defi}

\begin{bem}
In the following, if $E$ is a locally convex vector space, we shall frequently identity
$\Tang[0]{E} = \sset{0} \times E$ with $E$ in the obvious way.
Then for a Banach Lie group $\G$ and a centered chart $(\phi, \VF)$ of $\G$
such that $\rest{\dA{\phi}{}{} }{\LieAlg{\G}} = \id{\LieAlg{\G}}$,
we can identify $\CcF{\UF}{\LieAlg{\G} }{k}$ with $\LieAlg{ \CcF{\UF}{\G}{k} }$
via $\Tang[0]{\phi^{-1}_*}$ and $\Tang[\one]{\phi_*}$, respectively.
\end{bem}

\begin{lem}\label{lem:Evolution_Abb-Gruppen-vs-punktweise_Evol}
	Let $\SX$ be a normed space, $\UF \subseteq \SX$ an open nonempty set,
	$\GewFunk \subseteq \cl{\R}^\UF$ with $1_\UF \in \GewFunk$, $k\in\cl{\N}$
	$\G$ a Banach Lie group and $(\phi, \VF)$ a centered chart for $\G$ such that $\rest{\dA{\phi}{}{} }{\LieAlg{\G}} = \id{\LieAlg{\G}}$.
	Further, let $x \in \UF$ and $\Gamma : [0,1] \to \CcF{\UF}{\LieAlg{\G} }{k}$
	a smooth curve whose left evolution exists.
	Then $\evTwo_x^\G \circ \lEvol{\Tang[0]{\phi^{-1}_*} \circ \Gamma}$ is the left evolution
	of $\evTwo_x^L \circ \Gamma$.
\end{lem}
\begin{proof}
	We set $\eta \ndef \lEvol{\Tang[0]{\phi^{-1}_*} \circ \Gamma}$ and calculate using \refer{lem:LLA_Komposition_mit_Lie-Morph}
	and \refer{def:Inklusion_Diagramme_und_Lie-Funktoren} that
	\[
		\LLA{\evTwo_x^\G \circ \eta}
		= \LieAlg{\evTwo_x^\G} \circ \LLA{\eta}
		= \Tang[0]{\phi^{-1}} \circ \Tang[0]{\evTwo_x^L} \circ \Tang[\one]{\phi_*} \circ \Tang[0]{\phi^{-1}_*} \circ \Gamma
		= \evTwo_x^L \circ \Gamma.
	\]
	This shows the assertion.
\end{proof}

\begin{prop}
	\index{regularity!of $\CcF{\UF}{\G}{k}$}
	Let $\SX$ be a normed space, $\UF \subseteq \SX$ an open nonempty set,
	$\GewFunk \subseteq \cl{\R}^\UF$ with $1_\UF \in \GewFunk$, $k\in\cl{\N}$
	and $\G$ a Banach Lie group.
	Then the following assertions hold:
	\begin{enumerate}
		\item\label{enum1:AbbGruppe_BL-Gruppe_regular}
		$\CcF{\UF}{\G}{k}$, endowed with the Lie group structure
		described in \refer{def:Zsghd_gewichtete_Abb_Gruppe}, is regular.
		
		\item\label{enum1:AbbGruppe_BL-Gruppe_Exponentialfunktion}
		The exponential function of $\CcF{\UF}{\G }{k}$ is given by
		\[
			\CcF{\UF}{\LieAlg{\G} }{k} \to \CcF{\UF}{\G}{k}
			: \gamma \mapsto \ex[\G] \circ \gamma ,
		\]
		where we identify $\CcF{\UF}{\LieAlg{\G} }{k}$ with $\LieAlg{ \CcF{\UF}{\G}{k} }$.
	\end{enumerate}
\end{prop}
\begin{proof}
	\refer{enum1:AbbGruppe_BL-Gruppe_regular}
	Let $(\phi, \widetilde{\VF})$ be a centered chart of $\G$
	such that $\rest{\dA{\phi}{}{} }{\LieAlg{\G}} = \id{\LieAlg{\G}}$.
	We set
	\[
		\mathbf{F} \ndef \{\cF \subseteq \GewFunk : 1_\UF \in \cF, \abs{\cF} < \infty\}.
	\]
	After shrinking $\widetilde{\VF}$, we may assume that the inverse map of
	\[
		\CFo{\UF}{\widetilde{\VF}}{\cF}{\ell} \to \CF{\UF}{\G}{\cF}{\ell} : \Gamma \mapsto \phi^{-1} \circ \Gamma
	\]
	is a chart around the identity for $\cF \in \mathbf{F}$ and $\ell \in \N$ with $\ell \leq k$  (see \refer{def:Zsghd_gewichtete_Abb_Gruppe}).
	Let $\VF \subseteq \widetilde{\VF}$ an open $\one$-neighborhood such that
	$\phi(\VF) + \phi(\VF) \subseteq \phi(\widetilde{\VF})$.
	We choose an open zero neighborhood $\WF \sub \phi(\widetilde{\VF})$
	such that the implication \eqref{Nette_Reg-DGL_Abschaetzung_BLG-LinksEvol} holds.
	Let $\Gamma : [0,1] \to \CcFo{\UF}{\WF}{k}$ be a smooth curve.
	Then $\Gamma_{\cF, \ell} \ndef \iota_{\cF, \ell}^L \circ \Gamma$ is smooth,
	and since $\CF{\UF}{\G}{\cF}{\ell}$ is a Banach Lie group,
	the curve $\Tang[0]{\phi^{-1}_*} \circ \Gamma_{\cF, \ell}$
	has a smooth left evolution $\eta_{\cF, \ell} : [0,1] \to \CF{\UF}{\G}{\cF}{\ell}$.
	Then, for each $x \in \UF$,
	$\evTwo_x^\G \circ \eta_{\cF, \ell}$ is the left evolution of $\evTwo_x^L \circ \Gamma_{\cF, \ell}$ by \refer{lem:Evolution_Abb-Gruppen-vs-punktweise_Evol}.
	Since we assumed that \eqref{Nette_Reg-DGL_Abschaetzung_BLG-LinksEvol} holds,
	we conclude that for each $t \in [0,1]$, the image of $\eta_{\cF, \ell}(t)$ is contained in $\VF$.

	Further, for $\cF_1, \cF_2 \in \mathbf{F}$ such that $\cF_1 \subseteq \cF_2$
	and $\ell_1, \ell_2 \in \N$ such that $\ell_1 \leq \ell_2 \leq k$,
	\begin{multline*}
		\LLA{ \iota^\G_{(\cF_2, \ell_2), (\cF_1, \ell_1)} \circ \eta_{\cF_2, \ell_2} }
		= \LieAlg{\iota^\G_{(\cF_2, \ell_2), (\cF_1, \ell_1)} } \circ \LLA{\eta_{\cF_2, \ell_2} }
		\\
		= \Tang[0]{\phi^{-1}_*} \circ \Tang[0]{\iota_{(\cF_2, \ell_2), (\cF_1, \ell_1)}^L} \circ \Tang[\one]{\phi_*}
		\circ \LLA{\eta_{\cF_2, \ell_2} }
		= \Tang[0]{\phi^{-1}_*} \circ \Gamma_{\cF_1, \ell_1}
		= \LLA{\eta_{\cF_1, \ell_1} }.
	\end{multline*}
	Hence $\eta_{\cF_1, \ell_1} = \iota^\G_{(\cF_2, \ell_2), (\cF_1, \ell_1)} \circ \eta_{\cF_2, \ell_2}$.
	So the family $(\phi_* \circ \eta_{\cF, \ell})_{\cF \in \mathbf{F},\ell \leq k}$
	is compatible with the inclusion maps,
	hence using \refer{prop:CW_projektiver_LB-Raum} and \refer{prop:Differenzierbarkeit_Abb_in_projektiven_Limes},
	we derive a smooth curve $\widetilde{\eta} : [0,1] \to \CcFo{\UF}{\phi(\widetilde{\VF})}{k}$
	such that for all $\cF \in \mathbf{F}$ and $\ell \in \N$ with $\ell \leq k$,
	we have $\iota_{\cF, \ell}^L \circ \widetilde{\eta} = \phi_* \circ \eta_{\cF, \ell}$.
	We set $\eta \ndef \phi^{-1}_* \circ \widetilde{\eta}$.
	Then
	\[
		\Tang[0]{\phi^{-1}_*} \circ \Tang[0]{\iota_{\cF, \ell}^L} \circ \Tang[\one]{\phi_*} \circ \LLA{ \eta }
		= \LieAlg{\iota_{\cF, \ell}^\G } \circ \LLA{ \eta }
		= \LLA{ \eta_{\cF, \ell} }
		= \Tang[0]{\phi^{-1}_*} \circ \Gamma_{\cF, \ell}
		= \Tang[0]{\phi^{-1}_*} \circ \iota_{\cF, \ell}^L \circ \Gamma,
	\]
	and since $\cF$ and $\ell$ were arbitrary, we conclude (using \refer{prop:CW_projektiver_LB-Raum})
	that $\Tang[\one]{\phi_*} \circ \LLA{ \eta } = \Gamma$ and thus
	\[
		\LLA{ \eta } = \Tang[0]{\phi^{-1}_*} \circ \Gamma .
	\]
	
	It remains to show that the left evolution is smooth.
	To this end, we denote the left evolution of $\CF{\UF}{\G}{\cF}{\ell}$ with $\mathrm{evol}_{\cF, \ell}$
	and the one of $\CF{\UF}{\G}{\GewFunk}{k}$ with $\mathrm{evol}$.
	From our results above and \refer{def:Inklusion_Diagramme_und_Lie-Funktoren}, we derive the commutative diagram
	\[
		\xymatrix{
		{ \ConDiff{[0,1]}{ \CFo{\UF}{\WF}{\GewFunk}{k} }{\infty} } \ar[rrr]^-{ \mathrm{evol} \circ \Tang[0]{\phi^{-1}_*} } \ar[d]|{ \iota_{\cF, \ell}^L }%
		&&& { \phi^{-1}_\ast \circ \CFo{\UF}{\phi(\widetilde{\VF}) }{\GewFunk}{k} } \ar[d]^{ \iota_{\cF, \ell}^\G }
		\\
		{ \ConDiff{[0,1]}{ \CFo{\UF}{\WF}{\cF}{\ell} }{\infty} } \ar[rrr]_-{ \mathrm{evol}_{\cF, \ell} \circ \Tang[0]{\phi^{-1}_*} }%
		&&& { \phi^{-1}_\ast \circ \CFo{\UF}{\phi(\widetilde{\VF}) }{\cF}{\ell} }
		}
	\]
	Since the three lower arrows represent smooth maps, the map
	\[
		 \phi_\ast \circ \iota_{\cF, \ell}^\G \circ \mathrm{evol} \circ \Tang[0]{\phi^{-1}_*}
		= \iota_{\cF, \ell}^L \circ \phi_\ast \circ \mathrm{evol} \circ \Tang[0]{\phi^{-1}_*}
	\]
	is smooth on $\ConDiff{[0,1]}{ \CFo{\UF}{\WF}{\GewFunk}{k} }{\infty}$.
	Using \refer{prop:Differenzierbarkeit_Abb_in_projektiven_Limes} and
	\refer{susec:Projektive_Limiten_Gewichtete_Funktionen},
	we conclude that $\phi_\ast \circ \mathrm{evol} \circ \Tang[0]{\phi^{-1}_*}$ is smooth,
	and since $\phi_\ast$ and $\Tang[0]{\phi^{-1}_*}$ are diffeomorphisms,
	using \refer{lem:Lie-Gruppe_lokal_reg-->global_reg} we deduce that $\mathrm{evol}$ is smooth.

	\refer{enum1:AbbGruppe_BL-Gruppe_Exponentialfunktion}
	Let $(\phi, \VF)$ be a centered chart of $\G$
	such that $\rest{\dA{\phi}{}{} }{\LieAlg{\G}} = \id{\LieAlg{\G}}$.
	We denote the exponential function of $\CcF{\UF}{\G }{k}$ by $\ex[\GewFunk]$.
	Let $x \in \UF$ and $\gamma \in \CcF{\UF}{\LieAlg{\G} }{k}$.
	We denote the constant, $\gamma$-valued curve from $[0,1]$ to $\CcF{\UF}{\LieAlg{\G} }{k}$ by $\Gamma$.
	We proved in \refer{lem:Evolution_Abb-Gruppen-vs-punktweise_Evol}
	that $\evTwo_x^\G \circ \lEvol{\phi^{-1}_* \circ \Gamma}$ is the left evolution of $\evTwo_x^L \circ \Gamma$.
	On the other hand, since $\Gamma$ is constant, the left evolution of $\evTwo_x^L \circ \Gamma$
	is the restriction of the $1$-parameter group $\R \to \G : t \mapsto \ex[\G] (t \evTwo_x^L(\gamma))$. Hence
	\[
		\ex[\G] (\evTwo_x^L(\gamma))
		=
		(\evTwo_x^\G \circ \lEvol{\phi^{-1}_* \circ\Gamma})(1)
		= \evTwo_x^\G \circ\, \levol{\phi^{-1}_* \circ\Gamma}
		= \evTwo_x^\G \circ \ex[\GewFunk](\phi^{-1}_* (\gamma)).
	\]
	Thus $\ex[\GewFunk](\phi^{-1}_* (\gamma))(x) = \ex[\G](\gamma(x))$, from which we conclude the assertion
	since $x \in \UF$ was arbitrary.
\end{proof}

\subsection{Semidirect products with weighted diffeomorphisms}
In this subsection we discuss an action of the diffeomorphism group
$\DiffW$ on the Lie group $\CcF{\SX}{\G}{\infty}$, where $\G$ is a Banach Lie group.
This action can be used to construct the semidirect product $\CcF{\SX}{\G}{\infty} \rtimes \DiffW$
and turn it into a Lie group.
For technical reasons, we first discuss the following action of $\DiffW$ on $\G^\SX$.
\begin{defi}
	Let $\SX$ be a Banach space, $\G$ a Banach Lie group and $\GewFunk \subseteq \cl{\R}^\SX$
	with $1_\SX \in \GewFunk$. We define the map
	\[
		\widetilde{\omega}
		: \DiffW \times \G^\SX \to \G^\SX
		: (\phi, \gamma) \mapsto \gamma \circ \phi^{-1} .
	\]
\end{defi}
It is easy to see that $\widetilde{\omega}$ is in fact a group action, and moreover that it is a group
morphism in its second argument:
\begin{lem}\label{lem:Eigenschaften_der_Gruppenwirkung_von_DiffW(X)_auf_G^X}
	\begin{enumerate}
	\item
	$\widetilde{\omega}$ is a group action of $\DiffW$ on $\G^\SX$.
	\item
	For each $\phi \in \DiffW$ the partial map $\widetilde{\omega}(\phi,\cdot)$ is a
	group homomorphism.
	\end{enumerate}
\end{lem}
\begin{proof}
	These are easy computations.
\end{proof}
We show that this action leaves $\CcF{\SX}{\G}{\infty}$ invariant.
Since we proved in \refer{lem:Eigenschaften_der_Gruppenwirkung_von_DiffW(X)_auf_G^X}
that $\widetilde{\omega}$ is a group morphism in its second argument, it suffices to show
that it maps a generating set of $\CcF{\SX}{\G}{\infty}$ into this space. 
\begin{lem}\label{lem:Bild_der_Wirkung_DiffW(X)_auf_G^X}
\label{lem:Wirkung_DiffW_auf_C(X,G)_lokal_glatt}
	Let $\SX$ be a Banach space, $\G$ a Banach Lie group, $\GewFunk \subseteq \cl{\R}^\SX$
	with $1_\SX \in \GewFunk$, $(\phi, \widetilde{\VF})$ a centered chart of $\G$
	and $\VF$ an open identity neighborhood such that $\phi(\VF)$ is convex.
	Then
	\[
		\widetilde{\omega}\bigl(\DiffW \times (\phi^{-1} \circ \CcFo{\SX}{\phi(\VF)}{\infty})\bigr)
		\subseteq
		\phi^{-1} \circ \CcFo{\SX}{\phi(\VF)}{\infty} ,
	\]
	and the map
	\[
		\DiffW \times \CcFo{\SX}{\phi(\VF)}{\infty}
		\to \CcFo{\SX}{\phi(\VF)}{\infty}
		:
		(\psi, \gamma)  \mapsto \phi \circ \widetilde{\omega}(\psi, \phi^{-1} \circ \gamma)
	\]
	is smooth.
	Moreover,
	\[
		\widetilde{\omega}(\DiffW \times \CcF{\SX}{\G}{\infty})
		\subseteq
		\CcF{\SX}{\G}{\infty}.
	\]
\end{lem}
\begin{proof}
	Let $\psi$ be an element of $\DiffW$ and
	$\gamma \in \CcFo{\SX}{\phi(\VF)}{\infty}$. Then
	\[
		\widetilde{\omega}(\psi, \phi^{-1} \circ \gamma)
		= \phi^{-1} \circ (\gamma \circ \psi^{-1}),
	\]
	and using \refer{prop:Kompo_Koord_glatt} this identity proves the first
	and the second assertion.
	The final assertion follows immediately from the first assertion
	since we proved in \refer{lem:Eigenschaften_der_Gruppenwirkung_von_DiffW(X)_auf_G^X}
	that $\widetilde{\omega}$ is a group morphism in its second argument,
	and in \refer{def:Zsghd_gewichtete_Abb_Gruppe} that
	that $\CcF{\SX}{\G}{\infty}$ is generated by $\phi^{-1} \circ \CcFo{\SX}{\phi(\VF)}{k}$.
\end{proof}
So by restricting $\widetilde{\omega}$ to $\DiffW \times \CcF{\SX}{\G}{\infty}$, we
get a group action of $\DiffW$ on $\CcF{\SX}{\G}{\infty}$.
\begin{defi}
	We define
	\[
		\omega \ndef \rest{\widetilde{\omega}}{\DiffW \times \CcF{\SX}{\G}{\infty}}
		: \DiffW \times \CcF{\SX}{\G}{\infty} \to \CcF{\SX}{\G}{\infty}
		: (\phi, \gamma) \mapsto \gamma \circ \phi^{-1} .
	\]
\end{defi}

Finally, we are able to turn the semidirect product $\CcF{\SX}{\G}{\infty} \rtimes_\omega \DiffW$
into a Lie group.
\begin{satz}
	Let $\SX$ be a Banach space, $\G$ a Banach Lie group and $\GewFunk \subseteq \cl{\R}^\SX$
	with $1_\SX \in \GewFunk$.
	Then $\CcF{\SX}{\G}{\infty} \rtimes_\omega \DiffW$ can be turned into a
	Lie group modelled on $\CcF{\SX}{\LieAlg{\G}}{\infty} \times \CcF{\SX}{\SX}{\infty}$.
	\index{semidirect product!of $\CcF{\SX}{\G}{\infty}$ and $\DiffW$}%
\end{satz}
\begin{proof}
	We proved in \refer{lem:Wirkung_DiffW_auf_C(X,G)_lokal_glatt}
	that $\omega$ is smooth on a neighborhood of $(\id{\SX}, \one)$,
	and since this neighborhood is the product of generators
	of $\DiffW$ resp. $\CcF{\SX}{\G}{\infty}$,
	we can use
	\refer{lem:Kriterium-fuer-glatte-Gruppenwirkung}
	to see that $\omega$ is smooth.
	Hence we can apply
	\refer{lem:Semidirektes_Produkt_Liegruppen-glatteWirkung-ist_Liegruppe}
	and are home.
\end{proof}

\section{Weighted maps into locally convex Lie groups}
\label{sec: Weighted_maps_into_locally_convex_Lie_groups}
In this section, we discuss certain subgroups of $\G^\UF$, where $\G$ is a Lie group
and $\UF$ an open subset of a finite dimensional space $\SX$.
We construct a subgroup $\CcFvanK{\UF}{\G}{k}$ consisting of \emph{weighted decreasing mappings}
that can be turned into a (connected) Lie group.
After that, we extend this group to a Lie group $\CcFvanKBig{\UF}{\G}{k}$
which contains $\CcFvanK{\UF}{\G}{k}$ as an open normal subgroup,
and discuss its relation with \enquote{rapidly decreasing mappings}.

The modelling space of these groups is $\CcFvanK{\UF}{\LieAlg{\G}}{k}$,
where $k \in \cl{\N}$ and $\GewFunk$ is a set of weights on $\UF$ containing $1_\UF$.
These spaces are introduced in \refer{sec:GewAbb_Bildbereich-lokalkonvex}.

\subsection{Construction of the Lie group}
We construct the Lie group from local data using \refer{lem:Erzeugung_von_Liegruppen_aus_lokalen}.
For a chart $(\phi, \VF)$ of $\G$, we can endow the set $\phi^{-1} \circ \CcFvanK{\UF}{\phi(\VF)}{k} \sub \G^\UF$
with the manifold structure that turns the superposition operator $\phi_\ast$ into a chart.
We then need to check whether the multiplication and inversion on $\G^\UF$
are smooth with respect to this manifold structure.
The group operations on $\G^\UF$ arise as the composition of
the corresponding group operations on $\G$ with the mappings in $\G^\UF$ (see \refer{defi:Gruppenoperationen_Abbildungsgruppen}).
The main tool used in this subsection is the superposition with smooth maps
that we discussed in \refer{prop:Superposition_glatter_Abb_auf_weig-van_Abb}.

\paragraph{Local group operations}
We first discuss the local multiplication.
\begin{lem}\label{lem:Multiplikation_von_gewichteten_Abbildungsraumen-lokalkonvex}
	Let $\SX$ be a finite-dimensional space, $\UF \subseteq \SX$ an open nonempty subset,
	$\GewFunk \subseteq \cl{\R}^\UF$ with $1_\UF \in \GewFunk$, $\ell \in \cl{\N}$,
	$\G$ a locally convex Lie group with the group multiplication $m_\G$ and $(\phi, \VF)$ a centered chart of $\G$.
	Then there exists an open identity neighborhood $\WF \subseteq \VF$ such that the map
	\[
		\CcFvanK{\UF}{\phi(\WF)}{\ell} \times \CcFvanK{\UF}{\phi(\WF)}{\ell} \to \CcFvanK{\UF}{\phi(\VF)}{\ell}
		: (\gamma, \eta) \mapsto \phi \circ m_\G  \circ (\phi^{-1} \circ \gamma, \phi^{-1} \circ \eta)
		\tag{\ensuremath{\dagger}}\label{lokale_Gruppen-Multiplikation_in_Koordinaten-lokalkonvex}
	\]
	 is defined and smooth.
\end{lem}
\begin{proof}
	By \refer{lem:gewichtete,verschwindende_Abb_Produktisomorphie-lokalkonvex},
	the map \eqref{lokale_Gruppen-Multiplikation_in_Koordinaten-lokalkonvex} is defined and smooth
	iff there exists an open neighborhood $\WF \subseteq \G$ such that
	\[
		\Hom{ (\phi \circ m_\G \circ (\phi^{-1} \times \phi^{-1})) }{\cW}{}
		: \CcFvanK{\UF}{\phi(\WF) \times \phi(\WF)}{\ell} \to \CcFvanK{\UF}{\phi(\VF)}{\ell}
	\]
	is so.
	By the continuity of the multiplication $m_\G$ there exists an open subset
	$\WF \subseteq \VF$ such that $m_\G(\WF \times \WF) \subseteq \VF$.
	We may assume that $\phi(\WF)$ is star-shaped with center~$0$.
	Since the map $\phi \circ m_\G \circ (\phi^{-1} \times \phi^{-1})$ is smooth and
	maps $(0,0)$ to $0$, we can apply \refer{prop:Superposition_glatter_Abb_auf_weig-van_Abb} to see that
	\[
		(\phi \circ m_\G \circ (\phi^{-1} \times \phi^{-1}))
		\circ \CcFvanK{\UF}{\phi(\WF) \times \phi(\WF)}{\ell}
		\subseteq \CcFvanK{\UF}{\phi(\VF)}{\ell}
	\]
	and that the map $\Hom{ (\phi \circ m_\G \circ (\phi^{-1} \times \phi^{-1})) }{\cW}{}$
	is smooth.
\end{proof}
Now, we turn to the local inversion.
\begin{lem}\label{lem:Inversion_auf_gewichteten_Abbildungsraumen-lokalkonvex}
	Let $\SX$ be a finite-dimensional space, $\UF \subseteq \SX$ an open nonempty subset,
	$\GewFunk \subseteq \cl{\R}^\UF$ with $1_\UF \in \GewFunk$, $\ell \in \cl{\N}$,
	$\G$ a locally convex Lie group with the group inversion $I_\G$ and $(\phi, \VF)$ a centered chart such that $\VF$ is symmetric.
	Further let $\WF \subseteq \VF$ be a symmetric open $\one$-neighborhood such that
	there exists an open star-shaped set $\WF_{L}$ with center~$0$ and
	$\phi(\WF) \subseteq \WF_{L} \subseteq \phi(\VF)$.
	Then for each $\gamma \in \CcFvanK{\UF}{\phi(\WF)}{\ell}$,
	\[
		(\phi \circ I_\G \circ \phi^{-1}) \circ \gamma \in \CcFvanK{\UF}{\WF}{\ell},
	\]
	and the map
	\[
		\CcFvanK{\UF}{\phi(\WF)}{\ell} \to \CcFvanK{\UF}{\phi(\WF)}{\ell}
		: \gamma \mapsto (\phi \circ I_\G \circ \phi^{-1}) \circ \gamma
	\]
	is smooth.
\end{lem}
\begin{proof}
	Since $I_L \ndef \phi \circ I_\G \circ \phi^{-1} : \phi(\VF) \to \phi(\VF)$
	is smooth and $I_{L}(0) = 0$, 
	we conclude with \refer{prop:Superposition_glatter_Abb_auf_weig-van_Abb} that
	\[
		\CcFvanK{\UF}{\WF_{L}}{\ell} \to \CcFvanK{\UF}{\phi(\VF)}{\ell}
		: \gamma \mapsto I_L \circ \gamma
	\]
	is smooth.
	Since we proved in \refer{lem:CFvan_kompakt_mit_Werten_in_offener_Menge_offen-Normierte_Version}
	that $\CcFvanK{\UF}{\phi(\WF)}{\ell}$ is an open subset of $\CcFvanK{\UF}{\WF_{L}}{\ell}$,
	the restriction of this map is also smooth, and since $\WF$ is symmetric, it takes values in this set.
\end{proof}

\paragraph{Conclusion}
We put everything together to obtain a Lie group for each centered chart of $\G$.
We show that the identity component does not depend on the used chart.
\begin{lem}\label{lem:Liegruppenstruktur_auf_gewichteten_Abbildungsgruppen;Zusammenhangskomponente-lokalkonvex}
	Let $\SX$ be a finite-dimensional space, $\UF \subseteq \SX$ an open nonempty subset,
	$\GewFunk \subseteq \cl{\R}^\UF$ with $1_\UF \in \GewFunk$, $\ell \in \cl{\N}$,
	$\G$ a locally convex Lie group and $(\phi, \VF)$ a centered chart.
	Then there exists a subgroup $(\G, \phi)^\UF_{\GewFunk, \ell}$ of $\G^\UF$
	that can be turned into a Lie group.
	It is modelled on $\CcFvanK{\UF}{\LieAlg{\G}}{\ell}$ in such a way that there exists
	an open $\one$-neighborhood $\WF \subseteq \VF$ such that
	\[
		\CcFvanK{\UF}{\phi(\WF)}{\ell} \to (\G, \phi)^\UF_{\GewFunk, \ell} : \gamma \mapsto \phi^{-1} \circ \gamma
	\]
	becomes a smooth embedding and its image is open.
	Further, for any subset $\widetilde{\WF} \subseteq \WF$ such that
	$\phi(\widetilde{\WF})$ is an open convex zero neighborhood,
	\[
		\phi^{-1} \circ \CcFvanK{\UF}{\phi(\widetilde{\WF})}{\ell}
	\]
	generates the identity component of $(\G, \phi)^\UF_{\GewFunk, \ell}$.
\end{lem}
\begin{proof}
	Using \refer{lem:Multiplikation_von_gewichteten_Abbildungsraumen-lokalkonvex}
	we find an open $\one$-neighborhood $\WF \subseteq \VF$ such that
	\[
		\CcFvanK{\UF}{\phi(\WF)}{\ell} \times \CcFvanK{\UF}{\phi(\WF)}{\ell}
			\to \CcFvanK{\UF}{\phi(\VF)}{\ell}
		: (\gamma, \eta) \mapsto \phi \circ m_\G  \circ (\phi^{-1} \circ \gamma, \phi^{-1} \circ \eta)
	\]
	is smooth.
	We may assume w.l.o.g. that $\WF$ is symmetric and
	that there exists an open convex set $H$ such that
	$\phi(\WF) \subseteq H \subseteq \phi(\VF)$.
	We know from \refer{lem:Inversion_auf_gewichteten_Abbildungsraumen-lokalkonvex}
	that the set
	\[
		\phi^{-1} \circ \CcFvanK{\UF}{\phi(\WF)}{\ell} \subseteq \G^\UF
	\]
	is symmetric and
	\[
		\CcFvanK{\UF}{\phi(\WF)}{\ell} \to \CcFvanK{\UF}{\phi(\WF)}{\ell}
		: \gamma \mapsto \phi \circ I_\G \circ \phi^{-1} \circ \gamma
	\]
	is smooth.
	We endow $\phi^{-1} \circ \CcFvanK{\UF}{\phi(\WF)}{\ell} $ with the differential structure
	which turns the bijection
	\[
		\phi^{-1} \circ \CcFvanK{\UF}{\phi(\WF)}{\ell} \to \CcFvanK{\UF}{\phi(\WF)}{\ell}
		: \gamma \mapsto \phi \circ \gamma
	\]
	into a smooth diffeomorphism.
	Then we can apply \refer{lem:Erzeugung_von_Liegruppen_aus_lokalen}
	to construct a Lie group structure on the subgroup $(\G, \phi)^\UF_{\GewFunk, \ell}$ of $\G^\UF$
	which is generated by $\phi^{-1} \circ \CcFvanK{\UF}{\phi(\WF)}{\ell}$,
	such that $\phi^{-1} \circ \CcFvanK{\UF}{\phi(\WF)}{\ell}$ becomes an open subset.

	Moreover, for each open $\one$-neighborhood $\widetilde{\WF} \subseteq \WF$
	such that $\phi(\widetilde{\WF})$ is convex, the set $\CcFvanK{\UF}{\phi(\widetilde{\WF})}{\ell}$
	is convex (\refer{lem:Konvexe_Mengen_im_gewichteten_Funktionenraum-lokalkonvex_van}).
	Hence $\phi^{-1} \circ \CcFvanK{\UF}{\phi(\widetilde{\WF})}{\ell}$ is connected,
	and it is open by the construction of the differential structure of $(\G, \phi)^\UF_{\GewFunk, \ell}$.
	Further it obviously contains the unit element, hence it generates the identity component.
\end{proof}

\begin{lem}\label{lem:Zusammenhangskomponente_gewichtete_Abbildungsgruppen_kartenunabhaengig-lokalkonvex}
	Let $\SX$ be a finite-dimensional space, $\UF \subseteq \SX$ an open nonempty subset,
	$\GewFunk \subseteq \cl{\R}^\UF$ with $1_\UF \in \GewFunk$, $\ell \in \cl{\N}$
	and $\G$ a locally convex Lie group.
	Then for centered charts $(\phi_1, \VF_1)$ and $(\phi_2, \VF_2)$,
	the identity component of $(\G, \phi_1)^\UF_{\GewFunk, \ell}$ coincides with the one of $(\G, \phi_2)^\UF_{\GewFunk, \ell}$,
	and the identity map between them is a smooth diffeomorphism.
\end{lem}
\begin{proof}
	Using \refer{lem:Liegruppenstruktur_auf_gewichteten_Abbildungsgruppen;Zusammenhangskomponente-lokalkonvex},
	we find open $\one$-neighborhoods $\WF_1 \subseteq \VF_1$, $\WF_2 \subseteq \VF_2$
	such that the identity component of $(\G, \phi_i)^\UF_{\GewFunk, \ell}$ is generated
	by $\phi^{-1}_i \circ \CcFvanK{\UF}{\phi_i(\WF_i)}{\ell}$ for $i \in\{1,2\}$.
	Since $\phi_1 \circ \phi_2^{-1}$ is smooth, we find an open convex zero neighborhood
	$\widetilde{\WF}^L_2 \subseteq \phi_2(\WF_1 \cap \WF_2)$.
	By \refer{prop:Superposition_glatter_Abb_auf_weig-van_Abb}, the map
	\[
		\CcFvanK{\UF}{\widetilde{\WF}^L_2}{\ell} \to \CcFvanK{\UF}{\phi_1(\WF_1)}{\ell}
		:
		\gamma \mapsto \phi_1 \circ \phi_2^{-1} \circ \gamma
	\]
	is defined and smooth. This implies that
	\[
		\phi_2^{-1} \circ \CcFvanK{\UF}{\widetilde{\WF}^L_2}{\ell}
		\subseteq
		\phi_1^{-1} \circ \CcFvanK{\UF}{\phi_1(\WF_1)}{\ell}.
	\]
	Hence the identity component of $(\G, \phi_2)^\UF_{\GewFunk, \ell}$
	is contained in the one of $(\G, \phi_1)^\UF_{\GewFunk, \ell}$,
	and the inclusion map of the former into the latter is smooth.

	Exchanging the roles of $\phi_1$ and $\phi_2$ in the preceding argument, we get the assertion.
\end{proof}

\begin{defi}\label{def:Zsghd_gewichtete_Abb_Gruppe-lokalkonvex}
	Let $\SX$ be a finite-dimensional space, $\UF \subseteq \SX$ an open nonempty subset,
	$\GewFunk \subseteq \cl{\R}^\UF$ with $1_\UF \in \GewFunk$, $\ell \in \cl{\N}$
	and $\G$ a locally convex Lie group.
	Henceforth, we write $\glstext{gew_Abbildungsgruppe_abfallend-kompakt}$
	\index{weighted maps!into locally convex Lie groups}
	\index{mapping groups!with values in a locally convex Lie group}
	for the connected Lie group that was constructed in
	\refer{lem:Liegruppenstruktur_auf_gewichteten_Abbildungsgruppen;Zusammenhangskomponente-lokalkonvex}.
	There and in \refer{lem:Zusammenhangskomponente_gewichtete_Abbildungsgruppen_kartenunabhaengig-lokalkonvex}
	it was proved that for any centered chart $(\phi, \VF)$ of $\G$
	there exists an open $\one$-neighborhood $\WF$ such that the inverse map of
	\[
		\CcFvanK{\UF}{\phi(\WF)}{\ell} \to \CcF{\UF}{\G}{\ell} : \gamma \mapsto \phi^{-1} \circ \gamma
	\]
	is a chart, and that for any convex zero neighborhood $\widetilde{\WF} \subseteq \phi(\WF)$, the set
	\[
		\phi^{-1} \circ \CcFvanK{\UF}{\widetilde{\WF}}{\ell}
	\]
	generates $\CcFvanK{\UF}{\G}{\ell}$.
\end{defi}

\subsection{A larger Lie group of weighted mappings}
\label{susec:LArger_group_CWvanK}
We extend the Lie group described in \refer{def:Zsghd_gewichtete_Abb_Gruppe-lokalkonvex}.
Generally, it is possible using \refer{lem:Erzeugung_von_Liegruppen_aus_lokalen}
to extend a Lie group $\G$ that is a subgroup of a larger group $H$
by looking at its \enquote{smooth normalizer}\index{smooth normalizer}, that is all $h \in H$ that normalize $\G$
and for which the inner automorphism, restricted to suitable $\one$-neighborhoods, is smooth.
This approach has the disadvantage that we do not really know which maps are contained in the smooth normalizer.
So in the following, we will define a subset of $\G^\UF$ and show that it is a group
contained in the smooth normalizer of $\CcFvanK{\UF}{\G}{\ell}$.

Further, we show that this bigger group contains certain groups of \emph{rapidly decreasing mappings}
constructed in \cite{MR654676} as open subgroups.
\subsubsection{A group of mappings}
We define a set of mappings.
\begin{defi}
	Let $\G$ be a locally convex Lie group, $\SX$ a finite-dimensional vector space,
	$\UF \subseteq \SX$ a nonempty open subset, $\GewFunk \subseteq \cl{\R}^\UF$ nonempty and $k\in\cl{\N}$.
	Then for any centered chart $(\phi, \VF_\phi)$ of $\G$,
	compact set $K \subseteq \UF$ and $h \in \DCcInf{\UF}{\R}$ with $h \equiv 1_\UF$ on a neighborhood of $K$
	we define $M((\phi, \VF_\phi), K, h)$ as the set
	\[
		\{
		\gamma \in \ConDiff{\UF}{\G}{k} :
		\gamma(\UF\setminus K)\subseteq \VF_\phi \text{ and }
		\rest{(1_\UF - h) \cdot (\phi \circ \gamma)}{\UF\setminus K} \in \CcFvanK{\UF\setminus K}{\LieAlg{\G}}{k}
		\}.
	\]
	Further we define
	\[
		\glstext{Dicke_gew_Abbildungsgruppe_abfallend-kompakt}
		\ndef \bigcup_{(\phi, \VF_\phi), K, h} M((\phi, \VF_\phi), K, h).
	\]
	\index{weighted maps!into locally convex Lie groups}
\end{defi}
In the following, we show that $\CcFvanKBig{\UF}{\G}{k}$ is a subgroup of $\G^\UF$.
In order to do this, we provide some technical tools.
First, we show that we can use a cutoff technique to shrink the domain of a decreasing function.
\begin{lem}\label{lem:Kriterium_fuer_Bleiben_in_van-weigh_maps_bei_Einschraenken}
	Let $\SX$ be a finite-dimensional space,  $\UF \subseteq \SX$ an open nonempty subset,
	$\SY$ a locally convex space
	and $\GewFunk \subseteq \cl{\R}^\UF$ nonempty.
	Let $k \in \cl{\N}$ and $\gamma \in \ConDiff{\UF}{\SY}{k}$.
	\begin{enumerate}
		\item\label{enum1:Kriterium_fuer_Bleiben_in_van-weigh_maps_bei_Einschraenken-a}
		Suppose that $\gamma \in \CcFvanK{\UF}{\SY}{k}$.
		Let $A \subseteq \UF$ be a closed nonempty set such that $\rest{\gamma}{\UF\setminus A} \equiv 0$
		and $\VF \sub \UF$ an open neighborhood of $A$.
		Then $\rest{\gamma}{\VF} \in \CcFvanK{\VF}{\SY}{k}$.

		\item\label{enum1:Kriterium_fuer_Bleiben_in_van-weigh_maps_bei_Einschraenken-b}
		Let $K_1 \subseteq K_2 \subseteq \UF$ be closed sets such that
		$\rest{\gamma}{\UF\setminus K_1} \in \CcFvanK{\UF\setminus K_1}{\SY}{k}$
		and $h \in \BC{\UF}{\R}{\infty}$ such that $h \equiv 1$ on a neighborhood
		of $K_2$. Then
		\[
			\rest{(1_\UF - h) \cdot \gamma}{\UF\setminus K_2} \in \CcFvanK{\UF\setminus K_2}{\SY}{k}.
		\]
	\end{enumerate}
\end{lem}
\begin{proof}
	\refer{enum1:Kriterium_fuer_Bleiben_in_van-weigh_maps_bei_Einschraenken-a}
	It is obvious that $\rest{\gamma}{\VF} \in \CcF{\VF}{\SY}{k}$.
	Let $f \in \GewFunk$ and $\ell \in \N$ with $\ell \leq k$.
	For $\eps > 0$ and $p \in \normsOn{\SY}$
	there exists a compact set $K \subseteq \UF$
	such that $\hn{\rest{\gamma}{\UF\setminus K}}{p, f}{\ell} < \eps$.
	The set $\widetilde{K} \ndef K \cap A$ is compact and contained in $\VF$.
	Further $\hn{\rest{\gamma}{\VF\setminus \widetilde{K}}}{p, f}{\ell} < \eps$
	since $\FAbl[\ell]{\rest{\gamma}{\UF \setminus A}} = 0$.

	\refer{enum1:Kriterium_fuer_Bleiben_in_van-weigh_maps_bei_Einschraenken-b}
	Let $\VF \supseteq K$ be open in $\UF$ such that $\rest{h}{\VF} \equiv 1$.
	Then by \refer{cor:multilineare_Abb_und_CFvan_endl-dim},
	\[
		\rest{(1_\UF - h) \cdot \gamma}{\UF\setminus K_1} \in \CcFvanK{\UF\setminus K_1}{\SY}{k}.
	\]
	Further $\rest{(1_\UF - h) \cdot \gamma}{\UF \setminus (\UF \setminus \VF)} \equiv 0$.
	Since $\UF \setminus K_2$ is an open neighborhood of $\UF \setminus \VF$,
	an application of \refer{enum1:Kriterium_fuer_Bleiben_in_van-weigh_maps_bei_Einschraenken-a}
	finishes the proof.
\end{proof}
Now we examine $\CcFvanKBig{\UF}{\G}{k}$. We show that for a mapping in this set,
we can change the chart of $\G$, shrink the $\one$-neighborhood
and enlarge the compact set.
\begin{lem}\label{lem:Kartenwechsel_bei_normalteilenden_weigh.-van._Abb}
	Let $\SX$ be a finite-dimensional vector space,  $\UF \subseteq \SX$ an open nonempty subset,
	$\G$ a locally convex Lie group,
	$\GewFunk \subseteq \cl{\R}^\UF$ with $1_\UF \in \GewFunk$ and $k \in \cl{\N}$.
	Further, let $\gamma \in M((\phi, \VF_\phi), K, h)$.
	\begin{enumerate}
		\item\label{enum1:Kartenwechsel_bei_normalteilenden_weigh.-van._Abb-1}
		Then for each $\one$-neighborhood $\VF \subseteq \VF_\phi$,
		there exists a compact set $K_\VF \sub \UF$ such that
		for each map $h_\VF \in \DCcInf{\UF}{\R}$ with $h_\VF \equiv 1$ on a neighborhood of $K_\VF$,
		the map
		$
			\gamma \in M((\rest{\phi}{\VF}, \VF), K_\VF, h_\VF)
		$.

		\item\label{enum1:Kartenwechsel_bei_normalteilenden_weigh.-van._Abb-2}
		Let $(\psi, \VF_\psi)$ be a centered chart. Then there exists
		a compact set $K_\psi \subseteq \UF$ such that
		$
			\gamma \in M((\psi, \VF_\psi), K_\psi, h_\psi)
		$
		for each $h_\psi \in \DCcInf{\UF}{\R}$ with $h_\psi \equiv 1$ on a neighborhood of $K_\psi$.
		
		\item\label{enum1:Kartenwechsel_bei_normalteilenden_weigh.-van._Abb-3}
		Let $\eta \in M((\phi, \VF_\phi), \widetilde{K}, \widetilde{h})$.
		There exists a compact set $L$ such that
		for each $g \in \DCcInf{\UF}{\R}$ with $g \equiv 1$ on a neighborhood of $L$, we have
		$
			\gamma, \eta \in M((\phi, \VF_\phi), L, g).
		$
	\end{enumerate}
\end{lem}
\begin{proof}
	\refer{enum1:Kartenwechsel_bei_normalteilenden_weigh.-van._Abb-1}
	Since $\rest{(1_\UF - h) \cdot (\phi \circ \gamma)}{\UF\setminus K} \in \CcFvanK{\UF\setminus K}{\LieAlg{\G}}{k}$ and $1_\UF \in \GewFunk$,
	there exists a compact set $\widetilde{K} \subseteq \UF$ such that
	\[
		(1_\UF - h) \cdot (\phi \circ \gamma) ((\UF\setminus K)\setminus \widetilde{K}) \subseteq \phi(\VF).
	\]
	We define the compact set $K_\VF \ndef \widetilde{K} \cup \supp{h}$ and choose
	$h_\VF \in \DCcInf{\UF}{\R}$ with $h_\VF \equiv 1$ on a neighborhood of $K_\VF$.
	Using \refer{lem:Kriterium_fuer_Bleiben_in_van-weigh_maps_bei_Einschraenken}
	and the fact that $h \equiv 0$ on $\UF\setminus K_\VF$, we see that
	\[
		\rest{(1_\UF - h_\VF) \cdot (\phi \circ \gamma)}{\UF\setminus K_\VF}
		=
		\rest{(1_\UF - h_\VF) (1_\UF - h) \cdot (\phi \circ \gamma)}{\UF\setminus K_\VF}
		\in \CcFvanK{\UF\setminus K_\VF}{\LieAlg{\G}}{k}.
	\]
	Further we calculate using again that $h \equiv 0$ on $\UF\setminus K_\VF$:
	\[
		(\phi \circ \gamma)(\UF \setminus K_\VF) = (1_\UF - h) \cdot (\phi \circ \gamma) ((\UF\setminus K)\setminus K_\VF)
		\subseteq \phi(\VF).
	\]
	
	\refer{enum1:Kartenwechsel_bei_normalteilenden_weigh.-van._Abb-2}
	There exists an open $\one$-neighborhood $\VF \subseteq \VF_\phi \cap \VF_\psi$
	such that $\phi(\VF)$ is star-shaped with center~$0$.
	We know from \refer{enum1:Kartenwechsel_bei_normalteilenden_weigh.-van._Abb-1}
	that there exist a compact set $\widetilde{K} \subseteq \UF$ and
	a map $\widetilde{h} \in \DCcInf{\UF}{[0,1]}$ with $\widetilde{h} \equiv 1$
	on a neighborhood of $\widetilde{K}$ such that
	\[
		\gamma \in M((\rest{\phi}{\VF}, \VF), \widetilde{K}, \widetilde{h}).
	\]
	We conclude with \refer{prop:Superposition_glatter_Abb_auf_weig-van_Abb} that
	\[
		(\psi \circ \phi^{-1}) \circ (\rest{(1_\UF - \widetilde{h}) \cdot (\phi \circ \gamma)}{\UF\setminus \widetilde{K}} \in \CcFvanK{\UF\setminus \widetilde{K}}{\LieAlg{\G}}{k}.
	\]
	Let $h_\psi \in \DCcInf{\UF}{\R}$ such that $h_\psi \equiv 1$ on a neighborhood of $K_\psi$,
	where $K_\psi \ndef \widetilde{K} \cup \supp{\widetilde{h}}$.
	We conclude with \refer{lem:Kriterium_fuer_Bleiben_in_van-weigh_maps_bei_Einschraenken} that
	\[
		(1_\UF - h_\psi) \cdot (\psi \circ \phi^{-1}) \circ (\rest{(1_\UF - \widetilde{h}) \cdot (\phi \circ \gamma)}{\UF\setminus K_\psi} \in \CcFvanK{\UF\setminus K_\psi}{\LieAlg{\G}}{k}.
	\]
	Since $(1_\UF - \widetilde{h}) \equiv 1_\UF$ on $\UF\setminus K_\psi$, the proof is finished.
	
	\refer{enum1:Kartenwechsel_bei_normalteilenden_weigh.-van._Abb-3}
	We set $L \ndef \supp{h} \cup \supp{\widetilde{h}}$.
	Then
	\[
		\gamma(\UF \setminus L) \subseteq \gamma(\UF \setminus K) \subseteq \VF_\phi,
	\]
	and for $g \in \DCcInf{\UF}{\R}$ with $g \equiv 1$ on a neighborhood of $L$
	we conclude using \refer{lem:Kriterium_fuer_Bleiben_in_van-weigh_maps_bei_Einschraenken} that
	\[
		\rest{(1_\UF - g) \cdot (\phi \circ \gamma)}{U\setminus L}
		= \rest{(1_\UF - g) \cdot (1_\UF - h) \cdot (\phi \circ \gamma)}{U\setminus L}
		\in \CcFvanK{\UF\setminus L}{\LieAlg{\G}}{k}.
	\]
	Since the argument for $\eta$ is the same, we are home.
\end{proof}
Now we are ready to show that $\CcFvanKBig{\UF}{\G}{k}$ is a group.
\begin{lem}\label{lem:Dicke_Abbildungsgruppe_ist_Gruppe}
	Let $\SX$ be a finite-dimensional vector space,  $\UF \subseteq \SX$ an open nonempty subset,
	$\G$ a locally convex Lie group,
	$\GewFunk \subseteq \cl{\R}^\UF$ with $1_\UF \in \GewFunk$ and $k \in \cl{\N}$.
	Then the set $\CcFvanKBig{\UF}{\G}{k}$ is a subgroup of $\G^\UF$.
\end{lem}
\begin{proof}
	Let $(\phi, \VF_\phi)$ be a centered chart for $\G$
	and $\VF \subseteq \VF_\phi$ an open neighborhood of $\one$ such that
	$m_\G (\VF \times I_\G(\VF)) \subseteq \VF_\phi$ and $\phi(\VF)$ is star-shaped.
	We define the map
	\[
		H_\G : \VF \times \VF \to \VF_\phi : (x, y) \mapsto m_\G(x, I_\G(y)).
	\]
	Let $\gamma, \eta \in \CcFvanKBig{\UF}{\G}{k}$.
	Using \refer{lem:Kartenwechsel_bei_normalteilenden_weigh.-van._Abb}
	we find a compact set $K\subseteq \UF$ and a map $h \in \DCcInf{\UF}{[0,1]}$
	with $h \equiv 1_\UF$ on $K$ such that
	\[
		\gamma, \eta \in M((\rest{\phi}{\VF}, \VF), K, h).
	\]
	We define $H_\phi \ndef \rest{\phi \circ H_\G \circ (\phi^{-1} \times \phi^{-1})}{\VF\times\VF}$
	and want to show that there exists a compact set $\widetilde{K}$
	and $\widetilde{h} \in \DCcInf{\UF}{\R}$ with
	$\widetilde{h} \equiv 1$ on a neighborhood of $\widetilde{K}$
	such that $H_\G \circ (\gamma, \eta) \in M((\phi, \VF_\phi), \widetilde{K}, \widetilde{h})$.
	It is obvious that
	\[
		(H_\G \circ (\gamma, \eta)) (\UF\setminus K) \subseteq m_\G (\VF \times I_\G(\VF)) \subseteq \VF_\phi.
	\]
	Since we know with \refer{lem:gewichtete,verschwindende_Abb_Produktisomorphie-lokalkonvex}
	that
	\[
		(1_\UF - h)\cdot (\phi \circ\gamma, \phi \circ\eta)
		= ((1_\UF - h)\cdot (\phi \circ\gamma), (1_\UF - h)\cdot (\phi \circ\eta))
		\in \CcFvanK{\UF\setminus K}{\LieAlg{\G} \times \LieAlg{\G}}{k},
	\]
	we conclude using \refer{prop:Superposition_glatter_Abb_auf_weig-van_Abb} that
	\[
		H_\phi \circ ((1_\UF - h)\cdot (\phi \circ\gamma, \phi \circ\eta)) \in \CcFvanK{\UF\setminus K}{\LieAlg{\G}}{k}.
	\]
	Further, $\widetilde{K} \ndef K \cup \supp{h}$ is a compact set, so by
	\refer{lem:Kriterium_fuer_Bleiben_in_van-weigh_maps_bei_Einschraenken}
	\[
		(1_\UF-\widetilde{h}) \cdot H_\phi \circ ((1_\UF - h)\cdot (\phi \circ\gamma, \phi \circ\eta))
		\in \CcFvanK{\UF\setminus \widetilde{K}}{\LieAlg{\G}}{k}
	\]
	for any $\widetilde{h} \in \DCcInf{\UF}{\R}$ with
	$\widetilde{h} \equiv 1$ on a neighborhood of $\widetilde{K}$.
	Since $(1_\UF - h) \equiv 0$ on $\UF\setminus \widetilde{K}$,
	$
	(1_\UF-\widetilde{h}) \cdot \rest{(\phi \circ H_\G \circ (\gamma,\eta))}{\UF\setminus \widetilde{K}}
		\in \CcFvanK{\UF\setminus \widetilde{K}}{\LieAlg{\G}}{k}
	$
	and hence
	\[
		H_\G \circ (\gamma, \eta) \in M((\phi, \VF_\phi), \widetilde{K}, \widetilde{h}).
	\]
	The proof is complete.
\end{proof}
\subsubsection{Inclusion in the smooth normalizer}
We show that $\CcFvanKBig{\UF}{\G}{k}$ is contained in the
smooth normalizer of $\CcFvanK{\UF}{\G}{k}$.
To this end, we show that each $\gamma \in \CcFvanKBig{\UF}{\G}{k}$ can be written
as a product of a compactly supported $\ConDiff{}{}{k}$-map and a $\ConDiff{}{}{k}$-map
that takes values in a chosen chart domain.
After that, we show that these two classes of mappings
are contained in the smooth normalizer of $\CcFvanK{\UF}{\G}{k}$.

We start with the following technical lemma about extending decreasing functions.
\begin{lem}\label{lem:Kriterium_fuer_Ausdehnen_von_van-weigh_maps}
	Let $\SX$ be a finite-dimensional space, $\UF \subseteq \SX$ an open nonempty subset,
	$A \subseteq \UF$ a closed subset,
	$\SY$ a locally convex space,
	$\GewFunk \subseteq \cl{\R}^\UF$ with $1_\UF \in \GewFunk$, $k \in \cl{\N}$
	and $\gamma \in \CcFvanK{\UF\setminus A}{\SY}{k}$.
	Then the map
	\[
		\widetilde{\gamma} : \UF \to \SY
		: x \mapsto \begin{cases}
		\gamma(x) & \text{if } x\in\UF\setminus A,\\
		0 & else
		\end{cases}
	\]
	is in $\CcFvanK{\UF}{\SY}{k}$.
\end{lem}
\begin{proof}
	Obviously, the assertion holds on $\UF\setminus A$ and $\interior{A}$,
	since $\widetilde{\gamma}$ and its derivatives vanish on $\interior{A}$.
	We show that $\widetilde{\gamma}$ is $\ConDiff{}{}{k}$ on $\partial A$
	and it and its derivatives also vanish there.
	Since this is true iff for each $p \in \normsOn{\SY}$,
	the map $\HomQuot{\widetilde{\gamma}}{p}$ is $\ConDiff{}{}{k}$ on $\partial A$
	and it and its derivatives vanish there, and the identity
	$\HomQuot{ \widetilde{\gamma} }{p} = \widetilde{\HomQuot{\gamma}{p} }$ holds,
	we may assume w.l.o.g. that $\SY$ is normable.
	
	Since $1_\UF \in \GewFunk$, for each $\ell \in \N$ with $\ell \leq k$,
	the map $\widetilde{ \FAbl[\ell]{\gamma} }$ is continuous and hence
	\[
		\widetilde{ \FAbl[\ell]{\gamma} }
		\in \CcFvanK{\UF}{\Lin[\ell]{\SX}{\SY}}{0}.
	\]
	Using \refer{lem:Topologie_auf_CFk_ist_InitialTop_der_Ableitungen}, it remains to show that $\widetilde{\gamma}$ is $\ConDiff{}{}{k}$ with
	$\FAbl[\ell]{ \widetilde{\gamma} } = \widetilde{ \FAbl[\ell]{\gamma} }$
	for all $\ell \in \N$ with $\ell \leq k$.
	We show the assertion by an induction over $\ell$.
	
	$\ell = 1$:
	Let $x \in \partial A$ and $h \in \SX$. If there exists $\delta > 0$ such that
	$x + ]\!-\!\delta, 0]\, h \sub A$ or $x + [0, \delta[\, h \sub A$,
	then $D_h \widetilde{\gamma}(x) = 0 = \widetilde{ \FAbl{\gamma} }(x) h $.
	\\
	Otherwise, there exists a null sequence $(t_n)_{n\in\N}$ in $]-\infty, 0[$ or $]0,\infty[$
	such that for each $n \in \N$, $x + t_n h \in \UF \setminus A$.
	After replacing $h$ by $-h$ if necessary, we may assume w.l.o.g. that all $t_n$ are positive.
	Since $1_\UF \in \GewFunk$, $\widetilde{ \FAbl{\gamma} }$ is continuous and $\widetilde{ \FAbl{\gamma} }(x) = 0$,
	given $\eps > 0$ we find $\delta > 0$ such that
	for all $s \in ]\!-\!\delta, \delta[$,
	\[
		\Opnorm{\widetilde{ \FAbl{\gamma} }(x + s h) } < \eps.
	\]
	We find an $n \in \N$ such that $t_n \in ]\!-\!\delta, \delta[$.
	Then we define
	\[
		t \ndef \inf \set{\tau > 0}{]\tau, t_n] \sub \UF\setminus A} > 0.
	\]
	We calculate for $\tau \in ]t, t_n[$:
	\begin{multline*}
		\left\norm{\tfrac{\widetilde{\gamma}(x + t_n h) - \widetilde{\gamma}(x + \tau h)}{t_n} \right}
		< \left\norm{\tfrac{\widetilde{\gamma}(x + t_n h) - \widetilde{\gamma}(x + \tau h)}{t_n - \tau}\right}
		\\
		=  \left\norm{ \Rint{0}{1}{ \FAbl{\gamma}\bigl(x + (s t_n + (1 - s) \tau) h\bigr) \eval \tfrac{t_n - \tau}{t_n - \tau} h }{s} \right}
		< \eps \norm{h}.
	\end{multline*}
	But $\widetilde{\gamma}(x + \tau h) \to 0$ as $\tau \to t$, and hence
	\[
		\left\norm{\tfrac{\widetilde{\gamma}(x + t_n h) - \widetilde{\gamma}(x)}{t_n} \right}
		= \left\norm{\tfrac{\widetilde{\gamma}(x + t_n h) }{t_n} \right}
		\leq \eps \norm{h}.
	\]
	Since $\eps$ was arbitrary, we conclude that
	$D_h \widetilde{\gamma}(x) = 0 = \widetilde{ \FAbl{\gamma} }(x) h $.

	$\ell \to \ell + 1$:
	Using the inductive hypothesis, we conclude that
	$\widetilde{\FAbl{\gamma}}$ is $\FC{}{}{\ell}$,
	and $\FAbl[\ell]{ \widetilde{\FAbl{\gamma}} } = \widetilde{ \FAbl[\ell]{ \FAbl{\gamma} }}$.
	Hence $\widetilde{\gamma}$ is $\FC{}{}{\ell + 1}$, so by \refer{lem:Ableitungen_von_D}
	$\FAbl[\ell + 1]{ \widetilde{\gamma} } = \widetilde{\FAbl[\ell + 1]{\gamma}}$.
\end{proof}

\begin{prop}\label{prop:Zerlegung_von_dicken_Abb-gruppen-Elementen}
	Let $\SX$ be a finite-dimensional space, $\UF \subseteq \SX$ an open nonempty subset,
	$\G$ a locally convex Lie group,
	$\GewFunk \subseteq \cl{\R}^\UF$ with $1_\UF \in \GewFunk$, $k \in \cl{\N}$,
	$(\phi, \VF_\phi)$ a centered chart of $\G$
	and $\gamma \in \CcFvanKBig{\UF}{\G}{k}$.
	Then there exist maps $\eta \in M((\phi, \VF_\phi), \emptyset, 0_\UF)$
	and $\chi \in \DCc{\UF}{\G}{k}$ such that
	\[
		\gamma = \eta \cdot \chi.
	\]
\end{prop}
\begin{proof}
	Using \refer{lem:Kartenwechsel_bei_normalteilenden_weigh.-van._Abb}
	we find a compact set $K$ and $h \in \DCcInf{\UF}{[0,1]}$ such that
	$\gamma \in M((\phi, \VF_\phi), K, h)$.
	Using \refer{lem:Kriterium_fuer_Ausdehnen_von_van-weigh_maps} we see
	that
	\[
		\eta \ndef \phi^{-1} \circ \widetilde{\rest{(1_\UF - h)\cdot(\phi \circ \gamma)}{\UF\setminus K} }
		\in M((\phi, \VF_\phi), \emptyset, 0_\UF),
	\]
	and it is obvious that $\rest{\eta}{\UF \setminus \supp{h}} = \rest{\gamma}{\UF \setminus \supp{h}}$.
	Hence
	\[
		\chi \ndef  \eta^{-1} \cdot \gamma \in \DCc{\UF}{\G}{k},
	\]
	and obviously $\gamma = \eta \cdot \chi$.
\end{proof}
We now show that the weighted maps that take values in a suitable chart domain
are contained in the smooth normalizer.
\begin{lem}\label{lem:Normalteiler-Operation_ist_lokal_glatt-lokalkonvex}
	Let $\SX$ be a finite-dimensional space, $\UF \subseteq \SX$ an open nonempty subset,
	$\G$ a locally convex Lie group,
	$\GewFunk \subseteq \cl{\R}^\UF$ with $1_\UF \in \GewFunk$, $k \in \cl{\N}$
	and $(\phi, \VF_\phi)$ a centered chart of $\G$.
	Further let $\WF_\phi \subseteq \VF_\phi$ be an open $\one$-neighborhood such that
	\[
		\WF_\phi \cdot \WF_\phi \cdot \WF_\phi^{-1} \subseteq \VF_\phi
	\]
	and $\phi(\WF_\phi)$ is star-shaped with center~$0$. Then for each
	$\eta \in M((\phi, \WF_\phi), \emptyset, 0_\UF)$, the map
	\[
		\CcFvanK{\UF}{\phi(\WF_\phi)}{k} \to \CcFvanK{\UF}{\phi(\VF_\phi)}{k}
		: \gamma \mapsto \phi \circ (\eta \cdot (\phi^{-1} \circ \gamma) \cdot \eta^{-1})
	\]
	is smooth.
\end{lem}
\begin{proof}
	As a consequence of \refer{prop:Superposition_glatter_Abb_auf_weig-van_Abb}
	and \refer{lem:gewichtete,verschwindende_Abb_Produktisomorphie-lokalkonvex}, the map
	\begin{align*}
			&\CcFvanK{\UF}{\phi(\WF_\phi)}{k} \times \CcFvanK{\UF}{\phi(\WF_\phi)}{k} \times \CcFvanK{\UF}{\phi(\WF_\phi)}{k}
			\to \CcFvanK{\UF}{\phi(\VF_\phi)}{k}
			\\
			:& (\gamma_1, \gamma_2, \gamma_3)
			\mapsto \phi \circ ((\phi^{-1} \circ \gamma_1) \cdot (\phi^{-1} \circ \gamma_2) \cdot (\phi^{-1} \circ \gamma_3)^{-1})
	\end{align*}
	is smooth. We easily deduce the desired assertion.
\end{proof}
\paragraph{Normalization with compactly supported mappings}
While the treatment of $\ConDiff{}{}{k}$-maps with values in a suitable chart domain was straightforward,
we need to develop other tools to deal with the compactly supported mappings.
The main problem is that a compactly supported map may not take values in any chart domain.
To get around this problem, we need more technical machinery.
As motivation for the following, let $\chi \in \DCc{\UF}{\G}{k}$ and $(\phi, V_\phi)$ be a centered chart of $\G$.
Using that $\chi(\UF)$ is compact, we can find a symmetrical neighborhood $O$ of $\chi(\UF)$
and an open $\one$-neighborhood $W_\phi \sub V_\phi$ such that $O \cdot \WF_{\phi} \cdot O^{-1} \subseteq \VF_{\phi}$.
Then we can define the \enquote{normalization map in charts}
\[
	N : O \times \phi(\WF_{\phi}) \to \phi(\VF_\phi)
	: (g, y) \mapsto \phi(g\cdot \phi^{-1}(y) \cdot g^{-1}).
\]
We can calculate that for $\gamma \in \phi(\WF_{\phi})^\UF$,
we have the identity
\[
	\phi \circ (\chi \cdot \gamma \cdot \chi^{-1})
	= N \circ (\chi \times \id{\phi(\WF_\phi)}) \circ (\id{\UF}, \gamma).
\]
In the following two lemmas, we will examine the properties of maps of the form
$N \circ (\chi \times \id{\phi(\WF_\phi)})$
and whether they induce a kind of superposition operator for decreasing weighted functions.
\begin{lem}\label{lem:Funktionen_zum_Sternen}
	Let $\SX$, $\SY$ and $\SZ$ be locally convex spaces, $\UF \subseteq \SX$, $\VF \subseteq \SY$ and $\WF \subseteq \SZ$
	open nonempty subsets, $M$ a locally convex manifold and $k\in\cl{\N}$.
	Let $\Gamma \in \ConDiff{M \times \VF}{\WF}{\infty}$ and $\eta \in \ConDiff{\UF}{M}{k}$.
	Then the map
	\[
		\Xi \ndef \Gamma \circ (\eta \times \id{\VF}) : \UF \times \VF \to \WF
	\]
	has the following properties:
	\begin{enumerate}
		\item\label{enum1:Funktionen_zum_Sternen-1}
		The second partial derivative of $\Xi$ is
		\[
			\dA{_{2}\Xi}{}{} = (\pi_{2} \circ \Tang[2]{\Gamma}) \circ (\eta \times \id{\VF \times \SY})
		\]
		and if $k \geq 1$, the first partial derivative of $\Xi$ is
		\[
			\dA{_{1}\Xi}{}{} = (\pi_{2} \circ \Tang[1]{\Gamma}) \circ (\Tang{\eta} \times \id{\VF}) \circ S,
		\]
		where $\pi_{2}$ denotes the projection $\WF \times \SZ \to \SZ$ on the second component,
		and
		$S : \UF \times \VF \times \SX \to \UF \times \SX \times \VF : (x,y,h)\mapsto (x,h,y)$ denotes the swap map.
		
		\item\label{enum1:Funktionen_zum_Sternen-2}
		For all $x\in\UF$, the partial map $\Xi(x,\cdot) : \VF \to \WF$ is smooth,
		and for all $\ell\in\N$ the map
		$\dA[\ell]{_{2}\Xi}{}{} : \UF \times \VF \times \SY^{\ell} \to \WF$
		is $\ConDiff{}{}{k}$.
		
		\item\label{enum1:Funktionen_zum_Sternen-3}
		Assume that $\SX$ has finite dimension. Then for
		\[
			A_{1} : \UF \times \VF \to \Lin{\SX}{\SZ}
			: (x, y) \mapsto (h\mapsto \dA{_{1}\Xi}{x, y}{h})
		\]
		(which is only defined if $k \geq 1$) and
		\[
			A_{2} : \UF \times \VF \times \Lin{\SX}{\SY} \to \Lin{\SX}{\SZ}
			: (x, y, T) \mapsto (h\mapsto \dA{_{2}\Xi}{x, y}{T\eval h}),
		\]
		all partial maps $A_{1}(x,\cdot)$ and $A_{2}(x,\cdot)$ are smooth and all partial derivatives
		$\dA[\ell]{_{2}A_{1}}{}{}$ and $\dA[\ell]{_{2}A_{2}}{}{}$ are $\ConDiff{}{}{k-1}$,
		respectively $\ConDiff{}{}{k}$.
	\end{enumerate}
\end{lem}
\begin{proof}
	\refer{enum1:Funktionen_zum_Sternen-1}
	We calculate for $x \in \UF$, $y \in \VF$ and $h \in \SY$ that
	\[
		\dA{_{2}\Xi}{(x, y)}{h}
		= \lim_{t \to 0} \tfrac{\Xi(x, y + t h) - \Xi(x, y)}{t}
		= \lim_{t \to 0} \tfrac{\Gamma(\eta(x), y + t h) - \Gamma(\eta(x), y)}{t}
		= (\pi_{2} \circ \Tang[2]{\Gamma})(\eta(x), y, h).
	\]
	This shows the desired identity for $\dA{_{2}\Xi}{}{}$.
	If $k > 0$, we get using the chain rule that
	\[
		\dA{\Xi}{}{} \circ P = \pi_{2} \circ \Tang{\Xi} \circ P = \pi_{2} \circ \Tang{\Gamma} \circ (\Tang{\eta} \times \id{\Tang{\VF}}),
	\]
	where $P : \UF \times \SX \times \VF \times \SY \to \UF \times \VF \times \SX \times \SY$
	permutes the middle arguments.
	Since $\dA{_{1}\Xi}{(x,y)}{h_{x}} = \dA{\Xi}{(x,y)}{(h_{x}, 0)}$,
	we get the assertion for $\dA{_{1}\Xi}{}{}$.

	\refer{enum1:Funktionen_zum_Sternen-2}
	It is obvious that the partial maps are smooth. We prove the second assertion by induction on $\ell$:
	
	$\ell = 0:$
	This is obvious.
	
	$\ell \to \ell + 1:$
	In \refer{enum1:Funktionen_zum_Sternen-1} we proved that $\dA{_{2}\Xi}{}{}$ is of the same form as $\Xi$.
	By the inductive hypothesis,
	\[
		\dA[\ell]{_{2} (\dA{_{2}\Xi}{}{}) }{}{} : \UF \times \VF \times \SY \times (\SY \times \SY)^{\ell} \to \WF
	\]
	is a $\ConDiff{}{}{k}$-map. But
	\[
		\dA[\ell + 1]{_{2}\Xi}{x,y}{h_{1},h_{2},\dotsc,h_{\ell + 1}}
		= \dA[\ell]{_{2} (\dA{_{2}\Xi}{}{}) }{x,y,h_{1}}{(h_{2},0),\dotsc,(h_{\ell + 1},0)},
	\]
	so $\dA[\ell + 1]{_{2}\Xi}{}{}$ is $\ConDiff{}{}{k}$.

	\refer{enum1:Funktionen_zum_Sternen-3}
	The partial maps $A_{1}(x,\cdot)$ and $A_{2}(x,\cdot)$ are smooth and the maps
	$\dA[\ell]{_{2}A_{1}}{}{}$ and $\dA[\ell]{_{2}A_{2}}{}{}$ are $\ConDiff{}{}{k-1}$
	respective $\ConDiff{}{}{k}$ iff for each $h\in\SX$, the maps
	$A_{1}(x,\cdot)\eval h$ and $A_{2}(x,\cdot)\eval h$
	have the corresponding properties. By \refer{enum1:Funktionen_zum_Sternen-1},
	\[
		A_{1}(x,y) \eval h = \dA{_{1}\Xi}{x, y}{h}
		= (\pi_{2} \circ \Tang{_{1}\Gamma}) \circ (\Tang{\eta} \times \id{\VF}) \circ S(x,y,h)
	\]
	and
	\begin{align*}
		A_{2}(x,y,T) \eval h = \dA{_{2}\Xi}{x, y}{T\eval h}
		&= (\pi_{2} \circ \Tang{_{2}\Gamma}) \circ (\eta \times \id{\VF \times \SY})(x,y,T\eval h)
		\\
		&= (\pi_{2} \circ \Tang{_{2}\Gamma} \circ S_{1}) \circ (\eta \times \mathrm{ev}_{h} \times \id{\VF}) \circ S_{2} (x,y,T).
	\end{align*}
	Here $S_1$ and $S_2$ denote the swap maps
	\[
		M \times \SY \times \VF \to M \times \VF \times \SY,
	\]
	and
	\[
		\UF \times \VF \times \Lin{\SX}{\SY}
		\to \UF \times \Lin{\SX}{\SY} \times \VF
	\]
	respectively.
	Since $S$, $S_1$ and $S_{2}$ are restrictions of continuous linear maps,
	\refer{enum1:Funktionen_zum_Sternen-2} applies to both
	$A_{1}(x,\cdot)\eval h$ and $A_{2}(x,\cdot)\eval h$.
\end{proof}

\begin{lem}\label{lem:Sternoperation_auf_gew.van.Abb}
	Let $\SX$ be a finite-dimensional space, $\UF \subseteq \SX$ an open nonempty subset,
	$\SY$ and $\SZ$ locally convex spaces, $M$ a locally convex manifold,
	$\VF \subseteq \SY$ an open zero neighborhood that is star-shaped with center~$0$,
	$\GewFunk \subseteq \cl{\R}^\UF$ with $1_\UF \in \GewFunk$ and $k \in \cl{\N}$.
	Further, let $\Gamma \in \ConDiff{M \times \VF}{\SZ}{\infty}$, and $\theta \in \ConDiff{\UF}{M}{k}$
	such that the map
	\[
		\Xi \ndef \Gamma \circ (\theta \times \id{\VF}) : \UF \times \VF \to \SZ
	\]
	satisfies
	\begin{itemize}
		\item
		$\Xi(\UF \times \sset{0}) = \sset{0}$,

		\item
		There exists a compact set $K \subseteq \UF$ such that
		$\Xi((\UF\setminus K) \times \VF) = \sset{0}$.
	\end{itemize}
	Then for any $\gamma \in \CcFvanK{\UF}{\VF}{k}$
	\[
		\label{SternOperationNormalteiler-definiert}
		\tag{\ensuremath{\dagger}}
		\Xi \circ (\id{\UF}, \gamma) \in \CcFvanK{\UF}{\SZ}{k},
	\]
	and the map
	\[
		\Xi_{\ast}:
		\CcFvanK{\UF}{\VF}{k} \to \CcFvanK{\UF}{\SZ}{k}
		: \gamma \mapsto \Xi \circ (\id{\UF}, \gamma)
	\]
	is smooth.
\end{lem}
\begin{proof}
	We first prove that $\Xi_\ast$ is defined and continuous, by induction on $k$:
	
	$k=0:$
	Let $\gamma, \eta \in \CcFvanK{\UF}{\VF}{k}$ such that the line segment
	$\{t \gamma + (1-t) \eta : t \in [0,1]\} \subseteq \CcFvanK{\UF}{\VF}{k}$.
	We easily prove using \refer{lem:Bild_einer_verschwindenden_gewichteten_Abb_ist_kompakt} that the set
	\[
		\widetilde{K} \ndef \{t \gamma(x) + (1-t) \eta(x) : t \in [0,1], x \in \UF\}
	\]
	is relatively compact in $\VF$. Since $\dA{_{2}\Xi}{}{}$ is continuous by \refer{lem:Funktionen_zum_Sternen} \refer{enum1:Funktionen_zum_Sternen-2}
	and satisfies $\dA{_{2}\Xi}{}{}(\UF \times \VF \times \{0\}) = \{0\}$,
	we conclude using the Wallace Lemma that for each $p\in\normsOn{\SZ}$,
	there exists  $q\in\normsOn{\SY}$ such that
	\[
		\dA{_{2}\Xi}{}{}(K \times \widetilde{K} \times \Ball[q]{0}{1}) \subseteq \Ball[p]{0}{1}.
	\]
	This relation implies that
	\[
		\forall x \in K, y \in \widetilde{K}, h \in \SY:
		\norm{\dA{_{2}\Xi}{x,y}{h}}_{p} \leq \norm{h}_{q}.
	\]
	For each $x \in \UF$, we calculate
	\begin{equation*}
		\Xi(x, \gamma(x)) - \Xi(x, \eta(x))
		= \Rint{0}{1}{ \dA{_{2}\Xi}{x, t \gamma(x) + (1-t) \eta(x)}{\gamma(x) - \eta(x)} }{t}.
	\end{equation*}
	Hence for each $f\in\GewFunk$, we have
	\[
		\abs{f(x)}\, \norm{\Xi(x, \gamma(x)) - \Xi(x, \eta(x))}_{p}
		\leq \abs{f(x)}\, \norm{\gamma(x) - \eta(x)}_{q}.
	\]
	Taking $\eta = 0$, this estimate implies \eqref{SternOperationNormalteiler-definiert}.
	Further, since we proved in \refer{lem:CWvan_(endl-dim/normiert)_Abstand_Rand}
	that $\CcFvanK{\UF}{\VF}{k}$ is open,
	$\gamma$ has a convex neighborhood in $\CcFvanK{\UF}{\VF}{k}$;
	hence the estimate also implies the continuity of $\Xi_{\ast}$ in $\gamma$.

	$k \to k +1:$
	For each $x \in \UF$, $h \in \SX$ and $\gamma \in \CcFvanK{\UF}{\VF}{k + 1}$, we calculate
	\begin{align*}
		\dA{(\Xi\circ (\id{\UF}, \gamma))}{x}{h}
		&= \dA{\Xi}{x, \gamma(x)}{h, D\gamma(x)\eval h}
		\\
		&= \dA{_{1}\Xi}{x, \gamma(x)}{h}
		+ \dA{_{2}\Xi}{x, \gamma(x)}{D\gamma(x)\eval h}.
	\end{align*}
	Recall the maps $A_{1}$ and $A_{2}$ defined in \refer{lem:Funktionen_zum_Sternen}\refer{enum1:Funktionen_zum_Sternen-3}.
	We get the identity
	\[
		D(\Xi\circ (\id{\UF}, \gamma))(x) = (A_{1}\circ (\id{\UF}, \gamma))(x)
		+ (A_{2} \circ (\id{\UF}, \gamma, D\gamma))(x).
	\]
	We prove that $A_{1}$ and $A_{2}$ satisfy the same properties
	as $\Xi$ does:
	For $x \in \UF$, $y\in\VF$, $h\in\SX$, we have
	\[
		A_{1}(x, 0)\eval h = \dA{_{1}\Xi}{x, 0}{h} = \lim_{t\to 0}\frac{\Xi(x + t h, 0) - \Xi(x, 0) }{t} = 0,
	\]
	whence $A_{1}(x, 0) = 0$. Let $x\in \UF\setminus K$. Then
	\[
		A_{1}(x, y)\eval h = \dA{_{1}\Xi}{x, y}{h} = \lim_{t\to 0}\frac{\Xi(x + t h, y) - \Xi(x, y) }{t} = 0
	\]
	since $\UF\setminus K$ is open, hence $A_{1}(x, y) = 0$.
	\\
	As to $A_{2}$, for $x \in \UF$, $y\in\VF$ and $h\in\SX$ we calculate	
	\[
		A_{2}(x, y, 0)\eval h = \dA{_{2}\Xi}{x, y}{0\eval h} = 0,
	\]
	whence $A_{2}(x, y, 0) = 0$. Let $x\in \UF\setminus K$ and $T \in \Lin{\SX}{\SY}$. Then
	\[
		A_{2}(x, y, T)\eval h = \dA{_{2}\Xi}{x, y}{T\eval h}
		= \lim_{t\to 0}\frac{\Xi(x, y  + t T\eval h) - \Xi(x, y) }{t} = 0,
	\]
	hence $A_{2}(x, y, T) = 0$.
	
	So we can apply the inductive hypothesis to $A_{1}$ and $A_{2}$ and conclude that
	\[
		A_{1}\circ (\id{\SX}, \gamma),\,
		A_{2} \circ (\id{\SX}, \gamma, \FAbl{\gamma}) \in \CcFvanK{\UF}{\Lin{\SX}{\SZ}}{k}
	\]
	and the maps $\CcFvanK{\UF}{\VF}{k + 1} \to \CcFvanK{\UF}{\Lin{\SX}{\SZ}}{k}$
	\[
		\gamma \mapsto A_{1}\circ (\id{\SX}, \gamma)\text{ and }
		\gamma \mapsto A_{2} \circ (\id{\SX}, \gamma, \FAbl{\gamma})
	\]
	are continuous.
	In view of \refer{prop:topologische_Zerlegung_von_CFkvan-kompakte_Version},
	the continuity of $\Xi_\ast$ is established.

	We pass on to prove the smoothness of $\Xi_{\ast}$.
	In order to do this, we have to examine $\dA{_{2}\Xi}{}{}$.
	By \refer{lem:Funktionen_zum_Sternen} \refer{enum1:Funktionen_zum_Sternen-1},
	$\dA{_{2}\Xi}{}{} = \pi_2 \circ \Tang[2]{\Gamma} \circ (\theta \times \id{\VF \times \SY})$,
	and we easily see that
	\[
		\dA{_{2}\Xi}{}{}(\UF\times\sset{0}\times\sset{0})
		= \dA{_{2}\Xi}{}{}((\UF\setminus K) \times \VF \times \SY)
		= \sset{0}.
	\]
	Hence by the results already established, the map
	\[
		(\dA{_{2}\Xi}{}{})_{\ast}
		: \CcFvanK{\UF}{\VF \times \SY}{k} \to \CcFvanK{\UF}{\SZ}{k}
		: (\gamma) \mapsto \dA{_{2}\Xi}{}{} \circ (\id{\UF}, \gamma)
	\]
	is defined and continuous.
	Now let $\gamma \in \CcFvanK{\UF}{\VF}{k}$ and $\gamma_{1} \in \CcFvanK{\UF}{\SY}{k}$.
	Since $\CcFvanK{\UF}{\VF}{k}$ is open,
	there exists an $r > 0$ such that $\set{\gamma + s \gamma_{1} }{ s \in \Ball[\K]{0}{r}} \subseteq \CcFvanK{\UF}{\VF}{k}$.
	We calculate for $x \in \UF$ and $t \in \Ball[\K]{0}{r}[\setminus\sset{0}$ (using \refer{lem:gewichtete,verschwindende_Abb_Produktisomorphie-lokalkonvex} implicitly) that
	\begin{multline*}
		\frac{\Xi_\ast(\gamma + t\gamma_{1})(x) - \Xi_\ast(\gamma)(x) }{t}
		= \frac{\Xi(x, \gamma(x) + t\gamma_{1}(x)) - \Xi(x, \gamma(x))}{t}\\
		= \Rint{0}{1}{ \dA{_{2}\Xi}{(x, \gamma(x) + s t \gamma_{1}(x))}{\gamma_{1}(x)} }{s}
		= \Rint{0}{1}{( \dA{_{2}\Xi}{}{})_{\ast}(\gamma + s t \gamma_1, \gamma_1)(x) }{s}.
	\end{multline*}
	Hence by \refer{lem:Kriterium_Integrierbarkeit_in_CW_lkvx}
	and \refer{prop:Stetigkeit_parameterab_Int}, $\Xi_{\ast}$ is $\ConDiff{}{}{1}$ with
	\[
		\dA{\Xi_{\ast}}{\gamma}{\gamma_{1}}
		= (\dA{_{2}\Xi}{}{})_{\ast}(\gamma, \gamma_{1}).
	\]
	So using an easy induction argument we conclude from this identity
	that $\Xi_{\ast}$ is $\ConDiff{}{}{\ell}$ for each $\ell \in \N$ and hence smooth.
\end{proof}
Now we are ready to deal with the inner automorphism induced by a compactly supported map.
\begin{lem}\label{lem:Normalteileroperation_kompakt_getragene_Abb-lkvxAbbGruppen}
	Let $\SX$ be a finite-dimensional space, $\UF \subseteq \SX$ an open nonempty subset,
	$\G$ a locally convex Lie group,
	$\GewFunk \subseteq \cl{\R}^\UF$ with $1_\UF \in \GewFunk$, $k \in \cl{\N}$
	and $(\phi, \VF_\phi)$ a centered chart for $\G$.
	Let $\chi \in \DCc{\UF}{\G}{k}$. Then there exists an open $\one$-neighborhood $\WF_\phi \subseteq \VF_{\phi}$
	such that the map
	\[
		\tag{\ensuremath{\dagger}}
		\label{Abb-Normalteileroperation_kompakt_getragene_Abb-lkvxAbbGruppen}
		\CcFvanK{\UF}{\phi(\WF_\phi)}{k} \to \CcFvanK{\UF}{\LieAlg{\G}}{k}
		: \gamma \mapsto \phi \circ (\chi \cdot (\phi^{-1} \circ \gamma) \cdot \chi^{-1})
	\]
	is defined and smooth.
\end{lem}
\begin{proof} 
	Since $\chi(\UF)$ is compact, we can find an open $\one$-neighborhood
	$\WF_\phi \subseteq \VF_{\phi}$ and an open symmetrical neighborhood $O$ of $\chi(\UF)$ such that
	\[
		O \cdot \WF_{\phi} \cdot O^{-1} \subseteq \VF_{\phi};
	\]
	we may assume w.l.o.g. that $\phi(\WF_\phi)$ is star-shaped with center~$0$.
	We define the smooth map
	\[
		N : O \times \phi(\WF_{\phi}) \to \LieAlg{\G}
		: (g, y) \mapsto \phi(g\cdot \phi^{-1}(y) \cdot g^{-1}) - y.
	\]
	Then it is easy to see that
	\[
		N \circ (\chi \times \id{\phi(\WF_{\phi})})
		: \UF \times \phi(\WF_{\phi}) \to \LieAlg{\G}
	\]
	satisfies the assumptions of \refer{lem:Sternoperation_auf_gew.van.Abb},
	and that for $\gamma \in \CcFvanK{\UF}{\phi(\WF_\phi)}{k}$
	\[
		(N \circ (\chi \times \id{\phi(\WF_{\phi})})) \circ (\id{\UF},\gamma)
		= \phi \circ (\chi \cdot (\phi^{-1} \circ \gamma) \cdot \chi^{-1}) - \gamma .
	\]
	Hence the map
	\[
		\CcFvanK{\UF}{\phi(\WF_\phi)}{k} \to \CcFvanK{\UF}{\LieAlg{\G}}{k} :
		\gamma \mapsto \phi \circ (\chi \cdot (\phi^{-1} \circ \gamma) \cdot \chi^{-1}) - \gamma
	\]
	is smooth. Since the vector space addition is smooth, \eqref{Abb-Normalteileroperation_kompakt_getragene_Abb-lkvxAbbGruppen}
	is defined and smooth.
\end{proof}

\paragraph{Conclusion and the Lie group structure}
Finally, we put everything together and show that $\CcFvanKBig{\UF}{\G}{k}$
is contained in the smooth normalizer of $\CcFvanK{\UF}{\G}{k}$.
As mentioned above, we this allows the construction of a Lie group structure on $\CcFvanKBig{\UF}{\G}{k}$.
\begin{lem}\label{lem:Normalteiler-Operation_dicke_Gruppen_ist_lokal_glatt-lokalkonvex}
	Let $\SX$ be a finite-dimensional space, $\UF \subseteq \SX$ an open nonempty subset,
	$\G$ a locally convex Lie group,
	$\GewFunk \subseteq \cl{\R}^\UF$ with $1_\UF \in \GewFunk$, $k \in \cl{\N}$
	and $(\phi, \VF_\phi)$ a centered chart for $\G$.
	Let $\theta \in \CcFvanKBig{\UF}{\G}{k}$. Then there exists an open $\one$-neighborhood $\WF_\phi \subseteq \VF_{\phi}$
	such that the map
	\[
		\tag{\ensuremath{\dagger}}
		\CcFvanK{\UF}{\phi(\WF_\phi)}{k} \to \CcFvanK{\UF}{\phi(\VF_\phi)}{k}
		: \gamma \mapsto \phi \circ (\theta \cdot (\phi^{-1} \circ \gamma) \cdot \theta ^{-1})
	\]
	is defined and smooth.
\end{lem}
\begin{proof}
	Let $\widetilde{\VF_\phi} \subseteq \VF_\phi$ be an open $\one$-neighborhood such that
	\[
		\widetilde{\VF_\phi} \cdot \widetilde{\VF_\phi} \cdot \widetilde{\VF_\phi}^{-1} \subseteq \VF_\phi
	\]
	and $\phi(\widetilde{\VF_\phi})$ is star-shaped with center~$0$.
	According to \refer{prop:Zerlegung_von_dicken_Abb-gruppen-Elementen} there exist
	$\eta \in M((\phi, \widetilde{\VF_\phi}), \emptyset, 0_\UF)$ and $\chi \in \DCc{\UF}{\G}{k}$
	such that $\theta = \eta \cdot \chi$.
	By \refer{lem:Normalteileroperation_kompakt_getragene_Abb-lkvxAbbGruppen},
	there exists an open $\one$-neighborhood $\WF_\phi \subseteq \VF_{\phi}$
	such that
	\[
		\CcFvanK{\UF}{\phi(\WF_\phi)}{k} \to \CcFvanK{\UF}{\phi(\widetilde{\VF_\phi})}{k}
		: \gamma \mapsto \phi \circ (\chi \cdot (\phi^{-1} \circ \gamma) \cdot \chi^{-1})
	\]
	is smooth, and by \refer{lem:Normalteiler-Operation_ist_lokal_glatt-lokalkonvex}
	the map
	\[
		\CcFvanK{\UF}{\phi(\widetilde{\VF_\phi})}{k} \to \CcFvanK{\UF}{\phi(\VF_\phi)}{k}
		: \gamma \mapsto \phi \circ (\eta \cdot (\phi^{-1} \circ \gamma) \cdot \eta^{-1})
	\]
	is also smooth. Composing these two maps, we obtain the assertion.
\end{proof}

\begin{satz}
	Let $\SX$ be a finite-dimensional space, $\UF \subseteq \SX$ an open nonempty subset,
	$\G$ a locally convex Lie group,
	$\GewFunk \subseteq \cl{\R}^\UF$ with $1_\UF \in \GewFunk$ and $k \in \cl{\N}$.
	Then $\CcFvanKBig{\UF}{\G}{k}$ can be made into a Lie group that contains
	$\CcFvanK{\UF}{\G}{k}$ as an open normal subgroup.
	\index{weighted maps!into locally convex Lie groups}
	\index{mapping groups!with values in a locally convex Lie group}
\end{satz}
\begin{proof}
	We showed in \refer{def:Zsghd_gewichtete_Abb_Gruppe-lokalkonvex} that $\CcFvanK{\UF}{\G}{k}$ can be turned into a Lie group
	such that there exists a centered chart $(\phi, \VF_\phi)$ for which
	\[
		\CcFvanK{\UF}{\phi(\VF_\phi)}{k} \to \CcFvanK{\UF}{\G}{k}
		: \gamma \mapsto \phi^{-1} \circ \gamma
	\]
	is an embedding and its image generates $\CcFvanK{\UF}{\G}{k}$.
	Further, we proved in \refer{lem:Dicke_Abbildungsgruppe_ist_Gruppe}
	and \refer{lem:Normalteiler-Operation_dicke_Gruppen_ist_lokal_glatt-lokalkonvex}
	that $\CcFvanKBig{\UF}{\G}{k}$ is a subgroup of $\G^\UF$ and for each
	$\theta \in \CcFvanKBig{\UF}{\G}{k}$ there exists an open $\one$-neighborhood $\WF_\phi \subseteq \VF_\phi$
	such that the conjugation operation
	\[
		\CcFvanK{\UF}{\phi(\WF_\phi)}{k} \to \CcFvanK{\UF}{\phi(\VF_\phi)}{k}
		: \gamma \mapsto \phi \circ (\theta \cdot (\phi^{-1} \circ \gamma) \cdot \theta ^{-1})
	\]
	is smooth. Hence \refer{lem:Erzeugung_von_Liegruppen_aus_lokalen} gives the assertion.
\end{proof}
\subsubsection{Comparison with groups of \emph{rapidly decreasing mappings}}
In the book \cite[Section~4.2.1, pages 111-117]{MR654676},
for weights that satisfy conditions described below in \refer{def:BCR-weights},
certain \emph{$\Gamma$-rapidly decreasing functions} with values in locally convex spaces are defined
and used to construct \emph{$\Gamma$-rapidly decreasing mappings} that take values in Lie groups.
We compare these function spaces with our weighted decreasing functions
and will see that they coincide.
Further, we will show that the $\Gamma$-rapidly decreasing mappings are open subgroups
of a certain $\CcFvanKBig{\UF}{\G}{k}$.
\paragraph{$\boldsymbol{\GewFunk}$-rapidly decreasing functions}
We give the definition of the $\GewFunk$-rapidly decreasing functions.
\begin{defi}[BCR-weights]\label{def:BCR-weights}
	Let $\SX$ be a finite-dimensional vector space and $\GewFunk \subseteq [1,\infty]^{\SX}$ such that
	\begin{enumerate}
		\item[\ensuremath{(\GewFunk 1)}]
		for all $f, g \in \GewFunk$, the sets
		$f^{-1}(\infty)$ and $g^{-1}(\infty) \defn M_\infty$ coincide,

		\item[\ensuremath{(\GewFunk 2)}]
		$\GewFunk$ is directed upwards and contains a smallest element $f_{\text{min}}$
		defined by
		\[
			f_{\text{min}}(x)
			= \begin{cases}
			  	1 & x \not\in M_\infty\\
			  	\infty & \text{else},
			  \end{cases}
		\]

		\item[\ensuremath{(\GewFunk 3)}]\label{BCR-Bed3}
		and for each $f_1 \in \GewFunk$ there exists $f_2 \in \GewFunk$
		such that
		\[
			(\forall \eps > 0)(\exists n \in \N)
			\,
			\norm{x} \geq n \text{ or } f_1(x) \geq n
			\implies f_1(x) \leq \eps \cdot f_2(x).
		\]
	\end{enumerate}
	Furthermore each $f \in \GewFunk$ has to be continuous on the complement of $M_\infty$.
\end{defi}

\newcommand{\BCRspace}[3]{\ensuremath{S(#1, #2; #3)}}
\newcommand{\BCRnorm}[4]{\ensuremath{\norm{#1}_{#2,#3}^{#4}}}

\begin{defi}[$\GewFunk$-rapidly decreasing functions]
	Let $\GewFunk$ be a set of weights as in \refer{def:BCR-weights},
	$\UF \subseteq \R^m$ open and nonempty and $\SY$ a locally convex space.
	A smooth function $\gamma : \UF \to \SY$ is called \emph{$\GewFunk$-rapidly decreasing}
	if for each $f \in \GewFunk$ and $\beta \in \N^m$
	we have $\rest{\parAbl{\beta}{\gamma}}{\UF \cap M_\infty} \equiv 0$, and the function
	\[
		f \cdot \parAbl{\beta}{\gamma} : \UF \to \SY
	\]
	is continuous and bounded, where $\infty \cdot 0 = 0$.
	The set
	\[
		\BCRspace{\UF}{\SY}{\GewFunk}
		\ndef
		\set{\gamma \in \ConDiff{\UF}{\SY}{\infty} }{ \text{$\gamma$ is $\GewFunk$-rapidly decreasing}}
	\]
	endowed with the seminorms
	\[
		\BCRnorm{\gamma}{q}{f}{k}
		\ndef
		\sup \set{ q( f \cdot \parAbl{\beta}{\gamma} (x))}{ x \in \UF, \abs{\beta} \leq k}
	\]
	(where $q \in \normsOn{\SY}$, $k \in \N$ and $f \in \GewFunk$)
	becomes a locally convex space.
\end{defi}

\paragraph{Comparison of $\boldsymbol{\BCRspace{\UF}{\SY}{\GewFunk}}$ and $\boldsymbol{\CcF{\UF}{\SY}{\infty}}$}
We show that these function spaces coincide as topological vector spaces.
To this end, we need the following technical lemma.
\begin{lem}\label{lem:BCR-Gamma3-beschraenkt_impliziert_stetig}
	Let $\GewFunk$ be a set of weights as in \refer{def:BCR-weights},
	$\UF \subseteq \R^m$ open and nonempty, $F$ a locally convex space,
	$\gamma : \UF \to F$ a smooth function
	and $\beta \in \N^m$.
	Suppose that $\rest{\parAbl{\beta}{\gamma}}{\UF \cap M_\infty} \equiv 0$ and that
	for each $f \in \GewFunk$ the function
	\[
		f \cdot \parAbl{\beta}{\gamma}: \UF \to F
	\]
	is bounded. Then for each $f \in \GewFunk$,
	the function $f \cdot \parAbl{\beta}{\gamma}$ is continuous.
\end{lem}
\begin{proof}
	Let $f \in \GewFunk$ and $x \in \UF$.
	If $x \not\in \cl{M_\infty \cap \UF}$, $f \cdot \parAbl{\beta}{\gamma}$
	is continuous on a suitable neighborhood of $x$ since $f$ is so.
	\\
	Otherwise, $\parAbl{\beta}{\gamma}(x) = 0$ because $\parAbl{\beta}{\gamma}$ is continuous.
	If there exists $\VF \in \neigh{x}$ such that $f$ is bounded on $\VF \setminus M_\infty$,
	the map $f\cdot\parAbl{\beta}{\gamma}$ is continuous on $\VF$ because for $y \in \VF \setminus M_\infty$
	and $q \in \normsOn{F}$
	\[
		\norm{f(y) \parAbl{\beta}{\gamma}(y) - f(x) \parAbl{\beta}{\gamma}(x)}_q
		= \norm{f(y) \parAbl{\beta}{\gamma}(y) }_q
		\leq \noma{\rest{f}{\VF \setminus M_\infty}} \norm{\parAbl{\beta}{\gamma}(y) }_q,
	\]
	and this estimate is valid for $y \in M_\infty$.

	Otherwise, we choose $g \in \GewFunk$ such that (\ensuremath{\GewFunk 3}) holds.
	Let $\eps > 0$. There exists an $n \in \N$ such that
	\[
		(\forall y \in \UF).\,
		f(y) \geq n
		\implies
		f(y) \leq \frac{\eps}{\BCRnorm{\gamma}{q}{g}{ \abs{\beta} } + 1} g(y).
	\]
	For $q \in \normsOn{F}$ there exists $\VF \in \neigh{x}$ such that for $y \in \VF$
	\[
		\norm{\parAbl{\beta}{\gamma}(y)}_q
		<
		\frac{\eps}{n}.
	\]
	Let $y \in \VF$. If $f(y) \geq n$, we calculate
	\[
		\norm{f(y) \parAbl{\beta}{\gamma}(y)}_q
		= f(y) \norm{\parAbl{\beta}{\gamma}(y)}_q
		\leq \frac{\eps}{\BCRnorm{\gamma}{q}{g}{ \abs{\beta} } + 1} g(y) \norm{\parAbl{\beta}{\gamma}(y)}_q
		< \eps.
	\]
	Otherwise
	\[
		\norm{f(y) \parAbl{\beta}{\gamma}(y)}_q
		\leq n \norm{\parAbl{\beta}{\gamma}(y)}_q
		< \eps.
	\]
	So the assertion holds in all cases.
\end{proof}

\begin{lem}\label{lem:BCR_Abb-sind-CcF_Abb}
	Let $\GewFunk$ be a set of weights as in \refer{def:BCR-weights},
	$\UF \subseteq \R^m$ open and nonempty and $F$ a locally convex space.
	Then $\CcF{\UF}{\SY}{\infty} = \BCRspace{\UF}{\SY}{\GewFunk}$
	as a topological vector space.
\end{lem}
\begin{proof}
	We first prove that $\CcF{\UF}{\SY}{\infty} = \BCRspace{\UF}{\SY}{\GewFunk}$
	as set. To this end, let $\gamma \in \CcF{\UF}{\SY}{\infty}$, $f \in \GewFunk$
	and $\beta \in \N^m$.
	We set $k \ndef \abs{\beta}$. We know that for $p \in \normsOn{\SY}$,
	the map $\FAbl[k]{(\HomQuot{\gamma}{p})}$ vanishes on $M_\infty$, and
	\[
		f \cdot \FAbl[k]{(\HomQuot{\gamma}{p})} : \UF \to \Lin[k]{\R^m}{\SY_p}
	\]
	is bounded. Since the evaluation $\Lin[k]{\R^m}{\SY_p} \to \SY_p$
	at a fixed point is continuous linear, the map
	$f \cdot \parAbl{\beta}{ (\HomQuot{\gamma}{p})} =  \HomQuot{(f \cdot \parAbl{\beta}{\gamma}}{p})
	: \UF \to \SY_p$
	is also bounded. Hence $f \cdot \parAbl{\beta}{\gamma}$ is bounded,
	so an application of \refer{lem:BCR-Gamma3-beschraenkt_impliziert_stetig}
	gives $\gamma \in \BCRspace{\UF}{\SY}{\GewFunk}$.

	On the other hand, let $\gamma \in \BCRspace{\UF}{\SY}{\GewFunk}$ and $k \in \N$.
	For each $p \in \normsOn{\SY}$, we get with \refer{id:D^k_und_partielle_Abl}
	\[
		\FAbl[k]{(\HomQuot{\gamma}{p})}
		= \sum_{\substack{\alpha \in \N^m\\\abs{\alpha}=k}}
		S_\alpha \cdot \parAbl{\alpha}{(\HomQuot{\gamma}{p})}
		= \sum_{\substack{\alpha \in \N^m\\\abs{\alpha}=k}}
		S_\alpha \cdot (\HomQuot{ \parAbl{\alpha}{\gamma}}{p})
	\]
	Hence for $f \in \GewFunk$
	\[
		\tag{\ensuremath{\dagger}}\label{est:BCR-norm_richtige-Norm}
		\hn{\gamma}{p, f}{k}
		\leq \BCRnorm{\gamma}{p}{f}{k} \cdot
			\sum_{\substack{\alpha \in \N^n\\\abs{\alpha}=k}} \Opnorm{S_\alpha} < \infty.
	\]
	So $\gamma \in \CcF{\UF}{\SY}{\infty}$.

	We see from \eqref{est:BCR-norm_richtige-Norm} that for each $p \in \normsOn{\SY}$,
	$f\in \GewFunk$ and $k\in\N$ the seminorm $\hn{\cdot}{p, f}{k}$
	is continuous on $\BCRspace{\UF}{\SY}{\GewFunk}$.
	Since the seminorms $\BCRnorm{\cdot}{p}{f}{k}$ are obviously continuous
	on $\CcF{\UF}{\SY}{\infty}$, the spaces are the same as topological vector spaces.
\end{proof}
\begin{bem}
	Let $\GewFunk$ be a set of weights as in \refer{def:BCR-weights}.
	Then $1_\UF \in \GewFunk \iff M_\infty = \emptyset$.
	But obviously
	$\CcF{\UF}{\SY}{k} = \CF{\UF}{\SY}{\GewFunk \cup \sset{1_\UF}}{k}$
	and
	$\CcFvanK{\UF}{\SY}{k} = \CFvanK{\UF}{\SY}{\GewFunk \cup \sset{1_\UF}}{k}$
	as topological vector spaces.
\end{bem}
\paragraph{Rapidly decreasing mappings}
In \cite[Section~4.2.1, page 117--118]{MR654676},
the set of \emph{$\Gamma$-rapidly decreasing mappings} is defined.
We will show that these mappings are open subgroups of $\CcFvanKBig{\R^m}{\G}{\infty}$.
\begin{defi}[$\GewFunk$-rapidly decreasing mappings]
	Let $m \in \N$, $\G$ a locally convex Lie group and
	$\GewFunk$ a set of weights as in \refer{def:BCR-weights}.
	We define $\BCRspace{\R^m}{\G}{\GewFunk}$
	as the set of smooth functions $\gamma : \R^m \to \G$
	such that
	\begin{itemize}
		\item
		$\gamma(x) = \one$ for each $x \in M_\infty$, and
		$\gamma(x) \to \one$ if $\norm{x} \to \infty$.
		
		\item
		For any centered chart $(\phi, \widetilde{\VF})$ of $\G$
		and each open $\one$-neighborhood $\VF$ with $\cl{\VF} \subseteq \widetilde{\VF}$,
		$\phi \circ \rest{\gamma}{\gamma^{-1}(\VF)} \in \BCRspace{\gamma^{-1}(\VF)}{\LieAlg{\G}}{\GewFunk}$.
	\end{itemize}
\end{defi}
In the next lemmas, we provide tools needed for the further discussion.
First, we show that for weights as in \refer{def:BCR-weights},
the product of a weighted function with an suitable cutoff function
is a weighted decreasing function.
We use this result to prove a superposition lemma for the spaces $\CcF{\UF}{\SY}{k}$.
\begin{lem}\label{lem:Abschneiden_CcF_liefert_CcFvan(Kompakta)}
	Let $K$ be a compact subset of the finite-dimensional vector space $\SX$,
	$\SY$ a locally convex space, $k\in\N$,
	$\GewFunk$ a set of weights as in \refer{def:BCR-weights},
	$\gamma \in \CcF{\UF}{\SY}{k}$ (where $\UF \ndef \SX \setminus K$)
	and $h \in \DCcInf{\SX}{\R}$ such that
	$h \equiv 1$ on a neighborhood $\VF$ of $K$.
	Then
	\[
		\rest{(1 - h)}{\UF} \cdot \gamma \in \CcFvanK{\UF}{\SY}{k}.
	\]
\end{lem}
\begin{proof}
	We prove this by induction on $k$.
	
	$k = 0$:
	Let $f \in \GewFunk$, $p \in \normsOn{\SY}$ and $\eps > 0$.
	We use (\ensuremath{\GewFunk 3}) to see that there exists a $n \in \N$ such that
	\[
		\hn{\rest{\gamma}{\UF \setminus \clBall{0}{n} }}{p, f}{0}
		< \frac{\eps}{ 1 + \noma{1 - h} }.
	\]
	Further, the set
	\[
		A \ndef
		\left\set{x \in \SX}{\abs{(1 - h)(x)} \geq \frac{\eps}{ \hn{\gamma}{p, f}{0} + 1 }\right}
		\cap
		\clBall{0}{n}
	\]
	is compact and contained in $\UF$ since $(1 - h) \equiv 0$ on $\VF$.
	Using this two estimates, we easily calculate that
	$\hn{\rest{(1 - h)\cdot\gamma}{\UF \setminus A }}{p, f}{0} < \eps$.

	$k \to k + 1$:
	We have
	\[
		\FAbl{ (\rest{(1 - h)}{\UF} \cdot \gamma)}
		= \rest{(1 - h)}{\UF} \cdot \FAbl{\gamma} - \rest{\FAbl{h}}{\UF} \cdot \gamma .
	\]
	By the inductive hypothesis,
	$\rest{(1 - h)}{\UF} \cdot \FAbl{\gamma} \in \CcFvanK{\UF}{\Lin{\SX}{\SY}}{k}$,
	and since $\rest{\FAbl{h}}{\UF} \in \DCcInf{\UF}{\Lin{\SX}{\R}}$,
	we use \refer{cor:multilineare_Abb_und_CFvan_endl-dim}
	and \refer{prop:topologische_Zerlegung_von_CFkvan-kompakte_Version}
	to finish the proof.
\end{proof}
\begin{lem}\label{lem:Superposition_von_CcF-Abb(kompakt)}
	Let $m \in \N$, $k \in \cl{\N}$, $\GewFunk$ a set of weights as in \refer{def:BCR-weights},
	$\SY$ and $\SZ$ locally convex spaces, $\Omega \sub \SY$ open and balanced,
	$\phi : \Omega \to \SZ$ a smooth map with $\phi(0) = 0$
	and $\UF \sub \R^m$ open and nonempty such that $\R^m\setminus \UF$ is compact
	and $\cl{M_\infty} \sub \UF$.
	Further, let $\gamma \in \CcF{\UF}{\SY}{k}$
	such that $\gamma(\UF) \sub \Omega$.
	Then there exists an open set $\VF \sub \UF$ such that $\R^m\setminus \VF$ is compact,
	$\cl{M_\infty} \sub \VF$ and
	$\rest{\phi \circ \gamma}{\VF} \in \CcF{\VF}{\SZ}{k}$.
\end{lem}
\begin{proof}
	By our assumptions, there exists $h \in \DCcInf{\R^m}{[0,1]}$ with $h \equiv 1$ on
	a neighborhood of $\R^m\setminus \UF$
	and $h \equiv 0$ on a neighborhood of $\cl{M_\infty}$.
	Using \refer{lem:Abschneiden_CcF_liefert_CcFvan(Kompakta)}
	and \refer{prop:Superposition_glatter_Abb_auf_weig-van_Abb}
	we see that
	\[
		\phi \circ ((1 - h)\cdot\gamma) \in \CcFvanK{\UF}{\SZ}{k},
	\]
	so $\rest{\phi \circ \gamma}{\VF} \in \CcF{\VF}{\SZ}{k}$,
	where $\VF \ndef \R^m \setminus \supp{h}$.
	Further, $\R^m\setminus \VF$ is compact and $\cl{M_\infty} \sub \VF$, so the proof is finished.
\end{proof}
To complete our preparations, we prove a kind of extension lemma for weighted functions.
\begin{lem}\label{lem:Fortsetzen_von_CcF-Abb(kompakt)}
	Let $m \in \N$, $k \in \cl{\N}$, $\GewFunk$ a set of weights as in \refer{def:BCR-weights},
	$\SY$ a locally convex space,
	$\VF \sub \UF$ open and nonempty subsets of $\R^m$ such that $\R^m\setminus \VF$ is compact
	and $\cl{M_\infty} \sub \VF$.
	Further, let $\gamma \in \ConDiff{\UF}{\SY}{k}$
	such that $\rest{\gamma}{\VF} \in \CcF{\VF}{\SY}{k}$.
	Then for any open set $\WF$ with $\cl{\WF} \sub \UF$,
	the map $\rest{\gamma}{\WF}$ is in $\CcF{\WF}{\SY}{k}$.
\end{lem}
\begin{proof}
	Obviously
	$\cl{ \WF \setminus \VF } \sub \cl{ \WF } \cap (\R^m\setminus \VF)$,
	hence $\cl{\WF \setminus \VF}$ is compact and does not meet $M_\infty$.
	So for each $f \in \GewFunk$ and $\ell \in \N$ with $\ell \leq k$,
	the map $f\cdot \FAbl[\ell]{\gamma}$ is bounded on $\cl{\WF \setminus \VF}$
	since $f$ is continuous on this set.
	But $f\cdot \FAbl[\ell]{\gamma}$ is bounded on $\VF$ by our assumption.
	Hence $f\cdot \FAbl[\ell]{\gamma}$ is bounded on all of $\WF$
	and the proof is finished.
\end{proof}
Now we are able to prove the main results.
\begin{prop}\label{prop:Vergleich_Boseck-Gruppen_meine-Gruppen}
	Let $m \in \N$, $\G$ a locally convex Lie group and
	$\GewFunk$ a set of weights as in \refer{def:BCR-weights}.
	Then the following assertions hold:
	\begin{enumerate}
		\item\label{enum1:BCR-Gruppen_sind_Gruppen}
		$\BCRspace{\R^m}{\G}{\GewFunk}$ is a group.
		
		\item\label{enum1:BCR-Gruppen_enthalten_ZSHGs-Komponente}
		$\CcFvanK{\R^m}{\G}{\infty} \subseteq \BCRspace{\R^m}{\G}{\GewFunk}$.
		
		\item\label{enum1:BCR-Gruppen_enthalten_in_dicken_Gruppen}
		$\BCRspace{\R^m}{\G}{\GewFunk} \subseteq \CcFvanKBig{\R^m}{\G}{\infty}$.
	\end{enumerate}
	\index{mapping groups!with values in a locally convex Lie group}%
\end{prop}
\begin{proof}
	\refer{enum1:BCR-Gruppen_sind_Gruppen}
	Let $\gamma_1, \gamma_2 \in \BCRspace{\R^m}{\G}{\GewFunk}$.
	We set $\gamma \ndef \gamma_1 \cdot \gamma_2^{-1}$.
	Then for $x \in M_\infty$, we have $\gamma(x) = \gamma_1(x) \cdot \gamma_2^{-1}(x) = \one$,
	and it is easy to see that $\gamma(x) \to \one$ if $\norm{x} \to \infty$.
	\\
	Let $(\phi, \widetilde{\VF})$ be a centered chart of $\G$ and $\VF \sub \widetilde{\VF}$
	an open $\one$-neighborhood with $\cl{\VF} \sub \widetilde{\VF}$.
	There exist centered charts $(\phi_1, \VF_1)$ and $(\phi_2, \VF_2)$
	such that
	$\phi_i \circ \gamma_i \in \BCRspace{\gamma_i^{-1}(\VF_i)}{\LieAlg{\G}}{\GewFunk}$,
	where $i\in \sset{1,2}$;
	we may assume w.l.o.g. that $\VF_1 \cdot \VF_2^{-1} \sub \VF$, $\VF_2 \sub \VF$ and
	$\phi_1(\VF_1)$ and $\phi_2(\VF_2)$ are balanced.
	We define $\WF \ndef \bigcap_{i\in\sset{1,2}}\gamma_i^{-1}(\VF_i)$.
	Then by \refer{lem:gewichtete,verschwindende_Abb_Produktisomorphie-lokalkonvex}
	and \refer{lem:BCR_Abb-sind-CcF_Abb}
	\[
		(\phi_1 \circ \rest{\gamma_1}{\WF}, \phi_2 \circ \rest{\gamma_2}{\WF})
		\in \CcF{\WF}{ \phi_1(\VF_1) \times \phi_2(\VF_2) }{\infty}.
	\]
	Further $\R^m \setminus \WF$ is compact, and since
	there exist closed $A_i \in \neigh[\G]{\one}$ with $A_i \sub \VF_i$ ($i\in\sset{1,2}$),
	we have $\cl{M_\infty} \sub \bigcap_{i\in\sset{1,2}}\gamma_i^{-1}(A_i) \sub \WF$.
	We now apply \refer{lem:Superposition_von_CcF-Abb(kompakt)} to
	$(\phi_1 \circ \rest{\gamma_1}{\WF}, \phi_2 \circ \rest{\gamma_2}{\WF})$
	and the map
	\[
		\phi \circ \widetilde{m_\G} \circ (\phi_1^{-1} \times \phi_2^{-1})
		: \phi_1(\VF_1) \times \phi_2(\VF_2) \to \LieAlg{\G}
	\]
	(where $\widetilde{m_\G}$ denotes the map $\G \times \G \to \G : (g, h) \mapsto g \cdot h^{-1}$)
	and find an open set $\WF' \sub \WF$ such that $\cl{M_\infty} \sub \WF'$,
	$\R^m \setminus \WF'$ is compact
	and $\rest{\phi \circ \gamma}{\WF'} \in \CcF{\WF' }{ \LieAlg{\G} }{\infty}$.
	Applying \refer{lem:Fortsetzen_von_CcF-Abb(kompakt)}
	with the open sets $\WF' \sub \gamma^{-1}(\widetilde{\VF})$ and $\gamma^{-1}(\VF) \sub \gamma^{-1}(\widetilde{\VF})$,
	we obtain
	\[
		\rest{\phi \circ \gamma}{ \gamma^{-1}(\VF)}
		\in \CcF{ \gamma^{-1}(\VF) }{ \LieAlg{\G} }{\infty}
		= \BCRspace{\gamma^{-1}(\VF)}{\LieAlg{\G}}{\GewFunk}.
	\]

	\refer{enum1:BCR-Gruppen_enthalten_ZSHGs-Komponente}
	Since we proved that $\BCRspace{\R^m}{\G}{\GewFunk}$ is a group,
	we just have to show that it contains a generating set of $\CcFvanK{\R^m}{\G}{\infty}$.
	We know from \refer{def:Zsghd_gewichtete_Abb_Gruppe-lokalkonvex}
	that $\CcFvanK{\R^m}{\G}{\infty}$ is generated by
	$
		\phi^{-1} \circ \CcFvanK{\R^m}{\WF}{\infty},
	$
	where $(\phi, \widetilde{\WF})$ is a centered chart of $\G$ and
	$\WF \subseteq \phi(\widetilde{\WF})$ is an open convex zero neighborhood.
	Let $\gamma \in \CcFvanK{\R^m}{\WF}{\infty}$.
	Then $\rest{\gamma}{M_\infty} \equiv 0$,
	hence $\rest{\phi^{-1} \circ \gamma}{M_\infty} \equiv \one$.
	Further, since $1_{\R^m} \in \GewFunk$, $\gamma(x) \to 0$ if $\norm{x} \to \infty$,
	and thus $(\phi^{-1} \circ \gamma)(x) \to \one$ if $\norm{x} \to \infty$.
	Now let $(\psi, \widetilde{\VF})$ be a centered chart of $\G$ and $\VF \subseteq \widetilde{\VF}$
	an open $\one$-neighborhood with $\cl{\VF} \sub \widetilde{\VF}$.
	There exists an open balanced set $\Omega \sub \WF$ such that $\phi^{-1}(\Omega) \sub \VF$. 
	We set $\UF \ndef \gamma^{-1}(\Omega)$.
	Then $\rest{\gamma}{\UF} \in \CcF{\UF}{\LieAlg{\G}}{\infty}$, $\R^m \setminus \UF$
	is compact, and $\cl{M_\infty} \sub \gamma^{-1}(\sset{0}) \sub \UF$.
	Hence we can apply \refer{lem:Superposition_von_CcF-Abb(kompakt)}
	to $\rest{\gamma}{\UF}$ and $\rest{\psi \circ \phi^{-1}}{\Omega}$ to see that
	$
	 	\rest{\psi \circ \phi^{-1} \circ \gamma}{\UF}
		\in \CcF{\UF}{\LieAlg{\G}}{\infty}
	$
	Applying \refer{lem:Fortsetzen_von_CcF-Abb(kompakt)}
	with the open sets $\UF \sub (\psi \circ \phi^{-1} \circ \gamma)^{-1}(\widetilde{\VF})$
	and $(\psi \circ \phi^{-1} \circ \gamma)^{-1}(\VF) \sub (\psi \circ \phi^{-1} \circ \gamma)^{-1}(\widetilde{\VF})$,
	we obtain
	\[
		\rest{\psi \circ \phi^{-1} \circ \gamma}{ (\psi \circ \phi^{-1} \circ \gamma)^{-1}(\VF)}
		\in \CcF{ (\psi \circ \phi^{-1} \circ \gamma)^{-1}(\VF) }{ \LieAlg{\G} }{\infty}
		= \BCRspace{(\psi \circ \phi^{-1} \circ \gamma)^{-1}(\VF)}{\LieAlg{\G}}{\GewFunk}.
	\]

	\refer{enum1:BCR-Gruppen_enthalten_in_dicken_Gruppen}
	Let $\gamma \in \BCRspace{\R^m}{\G}{\GewFunk}$, $(\phi, \widetilde{\VF})$ be a centered chart of $\G$
	and $\VF$ an open $\one$-neighbor\-hood with $\cl{\VF} \subseteq \widetilde{\VF}$.
	Then the set $K \ndef \R^m \setminus \gamma^{-1}(\VF)$ is closed and bounded,
	hence compact, and
	\[
		\rest{\phi \circ \gamma}{\R^m\setminus K} \in \BCRspace{\R^m\setminus K}{\LieAlg{\G}}{\GewFunk}
		= \CcF{\R^m\setminus K}{\LieAlg{\G}}{\infty};
	\]
	the last identity is by \refer{lem:BCR_Abb-sind-CcF_Abb}.
	Let $h \in \DCcInf{\R^m}{\R}$ such that $h \equiv 1$ on a neighborhood of $K$.
	Then by \refer{lem:Abschneiden_CcF_liefert_CcFvan(Kompakta)}
	\[
		\rest{(1_{\R^m} - h) \cdot \phi \circ \gamma}{\R^m\setminus K}
		\in \CcFvanK{\R^m\setminus K}{\LieAlg{\G}}{\infty}.
	\]
	Hence $\gamma \in \CcFvanKBig{\R^m}{\G}{\infty}$.
\end{proof}
We characterize when $\CcFvanKBig{\R^m}{\G}{\infty}$
consists entirely of $\GewFunk$-rapidly decreasing mappings.
\begin{lem}
	Let $m \in \N$, $\G$ a locally convex Lie group and
	$\GewFunk$ a set of weights as in \refer{def:BCR-weights}.
		The following equivalence holds:
		\[
			\CcFvanKBig{\R^m}{\G}{\infty} = \BCRspace{\R^m}{\G}{\GewFunk}
			\iff
			M_\infty = \emptyset
			.
		\]
\end{lem}
\begin{proof}
	Suppose that $M_\infty = \emptyset$. Let $\gamma \in \CcFvanKBig{\R^m}{\G}{\infty}$,
	$(\psi, \widetilde{\VF})$ a centered chart of $\G$
	and $\VF$ a $\one$-neighborhood with $\cl{\VF} \sub \widetilde{\VF}$.
	By \refer{lem:Kartenwechsel_bei_normalteilenden_weigh.-van._Abb}, 
	there exist a compact set $K \subseteq \R^m$ and $h \in \DCcInf{\R^m}{\R}$ with $h \equiv 1$ on a neighborhood of $K$
	such that $\gamma(\R^m\setminus K)\subseteq \widetilde{\VF}$ and
	$\rest{(1 - h) \cdot (\psi \circ \gamma)}{\R^m\setminus K} \in \CcFvanK{\R^m\setminus K}{\LieAlg{\G}}{\infty}$.
	Since $1_{\R^m} \in \GewFunk$ and $K$ and $\supp{h}$ are compact,
	$(\psi \circ \gamma)(x) \to 0$ if $\norm{x} \to \infty$,
	hence $\gamma(x) \to \one$ if $\norm{x} \to \infty$.
	Further
	$\rest{\psi \circ \gamma}{\R^m\setminus \supp{h}} \in \CcF{\R^m\setminus \supp{h}}{\LieAlg{\G}}{\infty}$,
	so we apply \refer{lem:Fortsetzen_von_CcF-Abb(kompakt)}
	with the open sets $\R^m\setminus \supp{h} \sub \gamma^{-1}(\widetilde{\VF})$ and $\gamma^{-1}(\VF) \sub \gamma^{-1}(\widetilde{\VF})$
	and get
	$\rest{\psi \circ \gamma}{\gamma^{-1}(\VF)} \in \CcF{\gamma^{-1}(\VF)}{\LieAlg{\G}}{\infty}.$
	Hence $\gamma \in \BCRspace{\R^m}{\G}{\GewFunk}$, so in view of \refer{prop:Vergleich_Boseck-Gruppen_meine-Gruppen},
	the implication holds.

	Now let $M_\infty \neq \emptyset$.
	By definition, $\DCcInf{\R^m}{\G} \sub \CcFvanKBig{\R^m}{\G}{\infty}$,
	so there exists a $\gamma \in \CcFvanKBig{\R^m}{\G}{\infty}$
	such that $\gamma \not\equiv \one$ on $M_\infty$. Then $\gamma \not\in \BCRspace{\R^m}{\G}{\GewFunk}$.
\end{proof}

\begin{bem}\label{bem:BCR}
	In the book \cite{MR654676}, the groups $\BCRspace{\R^m}{\G}{\GewFunk}$
	are only defined if $\G$ is a so-called \emph{LE-Lie group}.
	Since we do not need this concept, we do not discuss it further.
	In \refer{prop:Vergleich_Boseck-Gruppen_meine-Gruppen} we proved that
	$\BCRspace{\R^m}{\G}{\GewFunk}$ is an open subgroup of $\CcFvanKBig{\R^m}{\G}{\infty} $
	and hence a Lie group.
	Further, for a set $\GewFunk$ of weights as in \refer{def:BCR-weights}
	obviously $\CcFvanK{\R^m}{\LieAlg{\G}}{\infty} = \CcF{\R^m}{\LieAlg{\G}}{\infty}$,
	whence the results derived by \cite{MR654676} concerning the Lie group structure of $\BCRspace{\R^m}{\G}{\GewFunk}$
	are special cases of our more general construction.

	It should be noted that the proof of \cite[Lemma 4.2.1.9]{MR654676}
	(whose assertion resembles \refer{prop:Superposition_glatter_Abb_auf_weig-van_Abb})
	is not really complete:
	The boundedness of $\gamma \cdot \parAbl{\beta}{(g \circ f)}$, where $\abs{\beta} > 0$,
	is hardly discussed. In the finite-dimensional case, compactness arguments
	simular to the one in \refer{lem:Bild_einer_verschwindenden_gewichteten_Abb_ist_kompakt}
	and the Fa\`a di Bruno-formula should save the day,
	but the infinite-dimensional case requires more work.
\end{bem}
\appendix
\chapter{Differential calculus}
In this chapter, we present the tools of Michal-Bastiani and Fr\'echet differential
calculus used in this work.
For proofs of the assertions, we refer the reader to
\cite{MR830252}, \cite{MR656198}, or \cite{MR583436}.
Further, we state some facts about ordinary differential equations.

In the following, let $\SX$, $\SY$ and $\SZ$ denote 
locally convex topological vector spaces over the same field $\K \in \sset{\R,\C}$.
\section{Differential calculus of maps between locally convex spaces}
\label{sec:Helge_diffbar}
\subsection{Curves and integrals}
\begin{defi}[Curves]
	A continuous map $\gamma: I \to \SX$ that is defined
	on a proper interval $I\subseteq\R$ is called a $\ConDiff{}{}{0}$-curve.
	A $\ConDiff{}{}{0}$-Kurve $\gamma: I \to \SX$ is called a $\ConDiff{}{}{1}$-curve
	if the limit
	\[
		\gamma^{(1)}(s)\ndef \lim_{t\to 0}\frac{\gamma(s+t)-\gamma(s)}{t}
	\]
	exists for all $s\in I$ and the map $\gamma^{(1)} : I \to \SX$
	is a $\ConDiff{}{}{0}$-curve.
	
	Inductively, for $k\in\N$ a map $\gamma: I \to \SX$ is called a
	$\ConDiff{}{}{k}$-curve if it is a $\ConDiff{}{}{1}$-curve
	and the map $\gamma^{(1)}$ is a $\ConDiff{}{}{k - 1}$-curve.
	We then define $\gamma^{(k)}\ndef (\gamma^{(1)})^{(k-1)}$.
	
	If $\gamma$ is a $\ConDiff{}{}{k}$-curve for each $k\in\N$,
	we call $\gamma$ a $\ConDiff{}{}{\infty}$- or \emph{smooth} curve.
\end{defi}

\begin{defi}[Weak integral]
	Let $\gamma: [a,b] \to \SX$ be a map. If there exists
	$x\in\SX$ such that
	\[
		\lambda(x) = \Rint{a}{b}{(\lambda\circ\gamma)(t)}{t}
		\text{\qquad for all $\lambda\in\SX'$,}
	\]
	we call $\gamma$ \emph{weakly integrable} with the \emph{weak integral} $x$
	and write
	\[
		 \Rint{a}{b}{\gamma(t)}{t}\ndef x .
	\]
\end{defi}

\begin{defi}[Line integral]
	Let $\gamma : [a,b] \to \SX$ be a $\ConDiff{}{}{1}$-curve and
	$f:\gamma([a,b])\to\SY$ a continuous map.
	We define the line integral of $f$ on $\gamma$ by
	\[
		\Lint{\gamma}{f(\zeta)}{\zeta} \ndef
		\Rint{a}{b}{f(\gamma(t))\cdot \gamma^{(1)}(t)}{t}
	\]
	if the weak integral on the right hand side exists.
\end{defi}

We record some properties of weak integrals.
\begin{lem}\label{lem:stet_lin_Abb_wirken_auf_int_Kurven}
	Let $\gamma: [a,b] \to \SX$ be a weakly integrable curve
	and $A: \SX \to \SY$ a continuous linear map.
	Then the map $A\circ\gamma$ is weakly integrable with the integral
	\[
		\Rint{a}{b}{(A\circ\gamma)(t)}{t}
		= A\left(\Rint{a}{b}{\gamma(t)}{t}\right) .
	\]
\end{lem}

\begin{prop}[Fundamental theorem of calculus]
\label{prop:Hauptsatz}
	Let $\gamma: [a,b] \to \SX$ be a $\ConDiff{}{}{1}$-curve.
	Then $\gamma^{(1)}$ is weakly integrable with the integral
	\[
		\Rint{a}{b}{\gamma^{(1)}(t)}{t} = \gamma(b) - \gamma(a) .
	\]
\end{prop}

\begin{lem}\label{lem:vollstaendig_stetig_integrierbar}
	If $\SX$ is sequentially complete, each continuous curve in
	$\SX$ is weakly integrable.
\end{lem}

\begin{lem}\label{lem:Integral_stetiges_Funktional}
	We endow the set of weakly integrable continuous curves
	from $[a,b]$ to $\SX$ with the topology of uniform convergence.
	The weak integral defines a continuous linear map between
	this space and $\SX$.
	In particular, for each continuous seminorm	$p : \SX \to \R$
	and each weakly integrable continuous curve $\gamma : [a,b] \to \SX$
	\[
		\left\norm{\Rint{a}{b}{\gamma(t)}{t}\right}_{p}
		\leq
		\Rint{a}{b}{\norm{\gamma(t)}_{p}}{t}
		,
	\]
	where we define $\norm{\cdot}_{p} \ndef p$.
\end{lem}

\begin{prop}[Continuity of parameter-dependent integrals]
\label{prop:Stetigkeit_parameterab_Int}
	Let $P$ be a topological space, $I\subseteq\R$ a proper interval
	and $a, b \in I$. Further, let $f:P\times I \to \SX$ be a continuous map
	such that the weak integral
	\[
		\Rint{a}{b}{f(p,t)}{t} \defn g(p)
	\]
	exists for all $p\in P$. Then the map $g:P\to\SX$ is continuous.
\end{prop}
\paragraph{Evaluation of curves}
We prove that the (simultaneous) evaluation of smooth curves is smooth.
\begin{lem}\label{lem:Auswertung_von_Kurven_glatt}
	Let $\SY$ be a locally convex topological vector space and
	$m \in \cl{\N}$. Then the \emph{evaluation function}
	\[
		\evTwo :
		\ConDiff{[0,1]}{ \SY }{m}
		\times
		[0,1]
		\to
		\SY
		:
		(\Gamma, t)\mapsto \Gamma(t)
	\]
	is a $\ConDiff{}{}{m}$-map. For $m \geq 1$, we have
	\begin{equation*}\label{id:Differential_Auswertung_glatte_Kurven}
		\tag{\ensuremath{\dagger}}
		\dA{\evTwo}{(\Gamma, t) }{ (\Gamma_{1}, s) }
		= s \cdot \evTwo(\Gamma', t) + \evTwo(\Gamma_{1}, t)
	\end{equation*}
	(using the same symbol, $\evTwo$, for the evaluation of $\ConDiff{}{}{m-1}$-curves).
\end{lem}
\begin{proof}
	The proof is by induction:
	
	$m = 0$: Let $\Gamma \in \ConDiff{[0,1]}{ \SY }{0}$ and $t \in [0,1]$.
	For a continuous seminorm $\norm{\cdot}$ on $\SY$ and $\eps > 0$ let $U$ be a
	neighborhood of $\Gamma$ in $\ConDiff{[0,1]}{ \SY }{0}$ such that for all $\Phi \in U$
	\[
		\norm{\Phi - \Gamma}_{\infty} < \frac{\eps}{2},
	\]
	where $\norm{\cdot }_\infty$  is defined by
	\[
		\ConDiff{[0,1]}{ \SY }{0} \to \R : \Phi \mapsto \sup_{t\in[0,1]} \norm{\Phi(t)}.
	\]
	By the continuity of $\Gamma$, there exists $\delta > 0$ such that
	for all $s \in [0,1]$ with $\abs{s - t} < \delta$ the estimate
	\[
		\norm{\Gamma(s) - \Gamma(t)} < \frac{\eps}{2}
	\]
	holds. Then
	\[
		\norm{\evTwo(\Gamma,t) - \evTwo(\Phi,s)}
		\leq
		\norm{\Gamma(t) - \Gamma(s)} + \norm{\Gamma(s) - \Phi(s)}
		<
		\eps ,
	\]
	whence $\evTwo$ is continuous in $(\Gamma, t)$.
	
	$m = 1$: Let $\Gamma, \Gamma_{1} \in \ConDiff{[0,1]}{\SY}{1}$,
	$t \in ]0,1[$, $h \in \R^\ast$ and $s \in \R$ such that $t + h s \in [0,1]$.
	Then
	\[
		\frac{\evTwo( (\Gamma, t) + h (\Gamma_{1}, s) ) - \evTwo(\Gamma, t)}{h}
		= \frac{\Gamma(t + h s) - \Gamma(t)}{h} + \evTwo(\Gamma_{1}, t + h s),
	\]
	and because $\Gamma$ is differentiable and $\evTwo$ is continuous,
	this term converges to
	\begin{equation*}
		s \cdot \evTwo(\Gamma', t) + \evTwo(\Gamma_{1}, t)
	\end{equation*}
	for $h \to 0$. Since this term has an obvious continuous extension to
	$\ConDiff{[0,1]}{ \SY }{1} \times [0,1] \times \ConDiff{[0,1]}{ \SY }{1}  \times \R$,
	$\evTwo$ is differentiable with the directional derivative
	\eqref{id:Differential_Auswertung_glatte_Kurven},
	which is continuous.
	
	$m \to m + 1$: The map
	\[
		\ConDiff{[0,1]}{ \SY }{m + 1}
		\to
		\ConDiff{[0,1]}{ \SY }{m}
		:
		\Gamma \mapsto \Gamma'
	\]
	is continuous linear and thus smooth.
	Using the inductive hypothesis, we therefore deduce from \eqref{id:Differential_Auswertung_glatte_Kurven}
	that $\dA{\evTwo}{}{}$ is $\ConDiff{}{}{m}$. Hence $\evTwo$ is $\ConDiff{}{}{m + 1}$.
\end{proof}

\subsection{Differentiable maps}
We give a short introduction on a differential calculus for maps
between locally convex spaces. It was first developed by
A.~Bastiani in the work \cite{MR0177277} and is also known as
Keller's $C^k_c$-theory.

Recall the definitions given in \refer{sec:Helge_diffbar--vorne}.
In the following, let $\SX$ and $\SY$ be locally convex spaces and
$\UF \subseteq \SX$ an open nonempty set.

\begin{prop}[Mean value theorem]\label{prop:MWS_C1_Abb}
	Let $f \in \ConDiff{\UF}{\SY}{1}$ and $v, u \in \UF$ such that the
	line segment $\{ t u + (1-t) v : t\in [0,1]\}$ is contained in $\UF$. Then
	\[
		f(v) - f(u) = \Rint{0}{1}{\dA{f}{u + t (v-u)}{v-u} }{t} .
	\]
\end{prop}

\begin{prop}[Chain rule]\label{prop:Kettenregel_Helge}
	Let $k\in\cl{\N}$, $f \in \ConDiff{\UF}{\SY}{k}$ and
	$g \in \ConDiff{\VF}{\SZ}{k}$ such that	$f(\UF)\subseteq\VF$.
	Then the composition $g\circ f:\UF\to\SZ$ is a $\ConDiff{}{}{k}$-map
	with
	\[
		\dA{(g\circ f)}{u}{x} = \dA{g}{f(u)}{ \dA{f}{u}{x} }
		\text{\qquad for all $(u,x)\in \UF\times\SX$.}
	\]
\end{prop}

\begin{prop}\label{prop:Differenzierbarkeit_Abb_in_projektiven_Limes}
	Let $\SX$ and $\SY$ be locally convex spaces, $\UF\subseteq\SX$ be open and nonempty
	and $k\in\cl{\N}$.
	\begin{enumerate}
		\item
		A map
		\[
			f = (f_i)_{i\in I} : \UF \to \prod_{i\in I} \SY_i
		\]
		to a direct product of locally convex spaces ist $\ConDiff{}{}{k}$
		iff each component $f_i$ is $\ConDiff{}{}{k}$.

		\item
		A map $f : \UF \to \SY$ with values in a closed vector subspace $\SZ$
		is $\ConDiff{}{}{k}$ iff $\rest[\SZ]{f}{} : \UF \to \SZ$ is $\ConDiff{}{}{k}$.

		\item
		If $\SY$ is the projective limit of locally convex spaces $\{\SY_i : i \in I\}$
		with limit maps $\pi_i : \SY\to \SY_i$,
		then a map $f:\UF \to \SY$ is $\ConDiff{}{}{k}$ iff
		$\pi_i \circ f : \UF \to \SY_i$ is $\ConDiff{}{}{k}$ for all $i\in I$.
	\end{enumerate}
\end{prop}

\paragraph{Characterization of differentiability of higher order}
In \refer{prop:hohe_Ableitungen_d}, we stated that a map is $\ConDiff{}{}{k}$ iff
all iterated directional derivatives up to order $k$ exist and depend continuously
on the directions.
Here, we present some facts about the iterated directional derivatives.

\begin{bem}
	We give a more explicit formula for the $k$-th derivative.
	Obviously, $\dA[1]{f}{u}{x_1} = \dA{f}{u}{x_1}$ and
	\[
		\dA[k]{f}{u}{x_1,\dotsc,x_k}
		= \lim_{t\to 0}
		\frac{\dA[k - 1]{f}{u + t x_k}{x_1,\dotsc,x_{k-1}}
			- \dA[k - 1]{f}{u}{x_1,\dotsc,x_{k-1})}
			}
			{t} .
	\]
\end{bem}

The Schwarz theorem extends to the present situation:
\begin{prop}[Schwarz' theorem]
	Let $r\in \cl{\N}$, $f \in \ConDiff[\K]{\UF}{\SY}{r}$,
	$k\in\N$ with $k\leq r$ and $u\in\UF$. The map
	\[
		\dA[k]{f}{u}{\cdot}:\SX^k \to \SY
		: (x_1,\dotsc,x_k) \mapsto \dA[k]{f}{u}{x_1,\dotsc,x_k}
	\]
	is continuous, symmetric and $k$-linear (over the field $\K$).
\end{prop}
\paragraph{Examples}
We give some examples of $\ConDiff{}{}{k}$-maps and calculate the higher-order
differentials of some maps.
\begin{beisp}\label{beisp:Helge_diffbare_Abb}
	\begin{enumerate}
	\item
		A map $\gamma : I \to \SX$ is a $\ConDiff{}{}{k}$-curve
		iff it is a $\ConDiff[\R]{}{}{k}$-map, and
		$\dA{\gamma}{x}{h} = h \cdot \gamma^{(1)}(x)$.
	\item
		A continuous linear map $A:\SX \to \SY$ is smooth with
		$\dA{A}{x}{h} = A\eval h$.
	\item
		More general, a $k$-linear continuous map
		$b:\SX_1\times\dotsb\times\SX_k \to \SY$
		is smooth with
		\[
			\dA{b}{x_1,\dotsc,x_k}{h_1,\dotsc,h_k}
			= \sum_{i=1}^k b(x_1,\dotsc,x_{i-1},h_i,x_{i+1},\dotsc,x_k).
		\]
	\end{enumerate}
\end{beisp}
We can calculate higher differentials of $f \circ g$ if one of the maps is linear.
\begin{lem}\label{lem:Hoehere_Ableitungen_kovariante_komposition_linearer_Abb}
	Let $\SX$, $\SY$ and $\SZ$ be locally convex topological vector spaces,
	$\UF \subseteq \SX$ an open nonempty set,
	$k \in \cl{\N}$ and $A:\SY \to \SZ$ a continuous linear map.
	Then for $\gamma \in \ConDiff{\UF}{\SY}{k}$
	\[
		A \circ \gamma \in \ConDiff{\UF}{\SZ}{k}.
	\]
	Moreover, for each $\ell \in \N$ with $\ell \leq k$
	\[
		\dA[\ell]{(A \circ \gamma)}{}{} = A \circ \dA[\ell]{\gamma}{}{}.
		\tag{\ensuremath{\dagger}}
		\label{id:hoeheres_Differential_Komposition_linearer_Abb}
	\]
\end{lem}
\begin{proof}
	This is proved by induction on $\ell$:
	\\
	The chain rule (\refer{prop:Kettenregel_Helge}) assures
	$A\circ\gamma \in \ConDiff{\UF}{\SZ}{k}$ and
	\[
		\dA{(A \circ \gamma)}{x}{h} = \dA{A}{\gamma(x)}{\dA{\gamma}{x}{h}}
		= A (\dA{\gamma}{x}{h})
	\]
	for $x \in \UF$ and $h \in \SX$,
	hence \eqref{id:hoeheres_Differential_Komposition_linearer_Abb}
	is satisfied for $\ell = 1$.
	\\
	If we assume that \eqref{id:hoeheres_Differential_Komposition_linearer_Abb}
	holds for an $\ell \in \N$, we conclude for
	$x\in\UF$ and $h_{1}, \dotsc, h_{\ell},h_{\ell + 1} \in \SX$
	\begin{align*}
		&\dA[\ell + 1]{(A \circ \gamma)}{x}{h_{1}, \dotsc, h_{\ell}, h_{\ell + 1}}
		\\
		=&\lim_{t \to 0}
		\frac{\dA[\ell]{(A \circ \gamma)}{x + t h_{\ell + 1}}{h_{1}, \dotsc, h_{\ell}}
			- \dA[\ell]{(A \circ \gamma)}{x}{h_{1}, \dotsc, h_{\ell}}}{t}
		\\
		=& \lim_{t \to 0}
		\frac{A (\dA[\ell]{\gamma}{x + t h_{\ell + 1}}{h_{1}, \dotsc, h_{\ell}})
			- A (\dA[\ell]{\gamma}{x}{h_{1}, \dotsc, h_{\ell}})}{t}
		\\
		=& A\left(\lim_{t \to 0}
		\frac{\dA[\ell]{\gamma}{x + t h_{\ell + 1}}{h_{1}, \dotsc, h_{\ell}}
			- \dA[\ell]{\gamma}{x}{h_{1}, \dotsc, h_{\ell}}}{t}\right)
		\\
		=& (A\circ\dA[\ell + 1]{\gamma)}{x}{h_{1}, \dotsc, h_{\ell}, h_{\ell + 1}},
	\end{align*}
	so \eqref{id:hoeheres_Differential_Komposition_linearer_Abb}
	holds for $\ell + 1$ as well.
\end{proof}

\begin{lem}\label{lem:Hoehere_Ableitungen_kontravariante_komposition_linearer_Abb}
	Let $\SX$, $\SY$ and $\SZ$ be locally convex topological vector spaces,
	$k \in \cl{\N}$ and $A:\SX \to \SY$ a continuous linear map.
	Then for $\gamma \in \ConDiff{\SY}{\SZ}{k}$
	\[
		\gamma \circ A \in \ConDiff{\SX}{\SZ}{k}.
	\]
	Moreover, for each $\ell \in \N$ with $\ell \leq k$
	\[
		\dA[\ell]{(\gamma \circ A)}{}{} = \dA[\ell]{\gamma}{}{} \circ \mathop{\Pi}_{j=1}^{\ell + 1} A.
		\tag{\ensuremath{\dagger}}
		\label{id:hoeheres_Differential_kontravariante_Komposition_linearer_Abb}
	\]
\end{lem}
\begin{proof}
	This is proved by induction on $\ell$:
	\\
	The chain rule (\refer{prop:Kettenregel_Helge}) assures
	$\gamma \circ A \in \ConDiff{\UF}{\SZ}{k}$ and
	\[
		\dA{(\gamma \circ A)}{x}{h} = \dA{\gamma}{A(x)}{\dA{A}{x}{h}}
		= \dA{\gamma}{A(x)}{A(h)}
	\]
	for $x \in \SX$ and $h \in \SX$,
	hence \eqref{id:hoeheres_Differential_kontravariante_Komposition_linearer_Abb}
	is satisfied for $\ell = 1$.
	\\
	If we assume that \eqref{id:hoeheres_Differential_Komposition_linearer_Abb}
	holds for an arbitrary $\ell \in \N$, we conclude that for
	$x\in\SX$ and $h_{1}, \dotsc, h_{\ell},h_{\ell + 1} \in \SX$
	\begin{align*}
		&\dA[\ell + 1]{(\gamma \circ A)}{x}{h_{1}, \dotsc, h_{\ell}, h_{\ell + 1}}
		\\
		=&\lim_{t \to 0}
		\frac{\dA[\ell]{(\gamma \circ A)}{x + t h_{\ell + 1}}{h_{1}, \dotsc, h_{\ell}}
			- \dA[\ell]{(\gamma \circ A)}{x}{h_{1}, \dotsc, h_{\ell}}}{t}
		\\
		=& \lim_{t \to 0}
		\frac{\dA[\ell]{\gamma}{A(x + t h_{\ell + 1})}{A\eval h_{1}, \dotsc, A\eval h_{\ell}}
			- \dA[\ell]{\gamma}{A(x)}{A\eval h_{1}, \dotsc, A\eval h_{\ell}}}{t}
		\\
		= &\lim_{t \to 0}
		\frac{1}{t} \Rint{0}{1}{\dA[\ell + 1]{\gamma}{A(x) + s t A(h_{\ell + 1})}%
		{A\eval h_{1}, \dotsc, A\eval h_{\ell}, t A \eval h_{\ell + 1}}}{s}
		\\
		=&\dA[\ell + 1]{\gamma}{A(x)}{A\eval h_{1}, \dotsc, A\eval h_{\ell}, A \eval h_{\ell + 1}}
	\end{align*}
	so \eqref{id:hoeheres_Differential_kontravariante_Komposition_linearer_Abb}
	holds for $\ell + 1$ as well.
\end{proof}
Another example for the computation of directional derivatives follows.
\begin{lem}\label{lem:DifferenzierbarkeitseigenschaftenKomposition_gut,hinten_linear}
	Let $\SX$, $\SY$ and $\SZ$ be locally convex spaces,
	$\VF \subseteq \SY$ an open nonempty set, $k \in \cl{\N}$,
	$\gamma : \VF \to \SZ$ a map and $A \in \Lin{\SX}{\SY}$ surjective such that
	\[
		\gamma \circ A \in \ConDiff{\UF}{\SZ}{k},
	\]
	where $\UF \ndef A^{-1}(\VF)$. Then all directional derivatives of $\gamma$ up to order
	$k$ exist and satisfy the identity
	\[
		\dA[\ell]{\gamma}{}{} \circ \mathop{\Pi}_{i = 1}^{\ell + 1} A = \dA[\ell]{(\gamma \circ A)}{}{}
	\]
	for all $\ell \in \N$ with $\ell \leq k$.
\end{lem}
\begin{proof}
	This is proved by induction on $\ell$:

	$\ell = 0$: This is obvious.

	$\ell \to \ell + 1$:
	Let $y \in \VF$ and $h_1, \dotsc, h_{\ell}, h_{\ell + 1} \in \SY$. By the surjectivity of $A$
	there exist $x \in \UF$ and $v_1, \dotsc, v_{\ell}, v_{\ell + 1} \in \SX$ with $A\eval x = y$
	and $A\eval v_i = h_i$ for $i = 1,\dotsc,\ell,\ell+1$. Then for all suitable $t \neq 0$
		\begin{align*}
		&\lim_{t\to 0}
			\frac{\dA[\ell]{\gamma}{y + t h_{\ell + 1}}{h_1,\dotsc, h_{\ell}}
			- \dA[\ell]{\gamma}{y}{h_1,\dotsc, h_{\ell}} }{t}
		\\
		=&
		\lim_{t\to 0}
			\frac{\dA[\ell]{\gamma}{ A (x + t v_{\ell + 1}) }{A\eval v_1,\dotsc, A\eval v_{\ell}}
			- \dA[\ell]{\gamma}{A\eval x}{A\eval v_1,\dotsc, A\eval v_{\ell}} }{t}
		\\
		=&
		\lim_{t\to 0}
			\frac{(\dA[\ell]{\gamma}{}{} \circ \mathop{\Pi}_{i = 1}^{\ell + 1} A)(x + t v_{\ell + 1}, v_1,\dotsc,  v_{\ell})
			- (\dA[\ell]{\gamma}{}{} \circ \mathop{\Pi}_{i = 1}^{\ell + 1} A)(x, v_1,\dotsc,  v_{\ell}) }{t}
		\\
		=&
		\dA[\ell + 1]{(\gamma \circ A)}{x}{v_1,\dotsc, v_{\ell}, v_{\ell + 1}},
	\end{align*}
	and this completes the proof.
\end{proof}

We give a specialization of \refer{prop:Stetigkeit_parameterab_Int}.
\begin{prop}[Differentiability of parameter-dependent integrals]
\label{prop:Glattheit_parameterab_Int}
	Let $P$ be an open subset of a locally convex space, $I\subseteq\R$ a proper interval,
	$a, b \in I$ and $k \in \cl{\N}$. Further, let $f:P\times I \to \SX$ be a $\ConDiff{}{}{k}$-map
	such that the weak integral
	\[
		\Rint{a}{b}{f(p,t)}{t} \defn g(p)
	\]
	exists for all $p\in P$. Then the map $g:P\to\SX$ is $\ConDiff{}{}{k}$.
\end{prop}
\subsubsection{Analytic maps}
Complex analytic maps will be defined as maps which can locally be approximated by polynomials.
Real analytic maps are maps that have a \emph{complexification}.
\index{analytic maps}
\paragraph{Polynomials and symmetric multilinear maps}
For the definition of complex analytic maps we need to define polynomials.
\begin{defi}
	Let $k\in\N$. A \emph{homogenous polynomial of degree $k$} from
	$\SX$ to $\SY$ is a map for which there exists a $k$-linear map
	$\beta:\SX^k\to\SY$ such that
	\[
		p(x) = \beta(\underbrace{x,\dotsc,x}_{k})
	\]
	for all $x \in \SX$.
	In particular, a homogenous polynomial of degree $0$ is a constant map.
	
	A \emph{polynomial of degree $\leq k$} is a sum of homogenous polynomials
	of degree $\leq k$.
\end{defi}
There is a bijection between the set of homogenous polynomials and that
of symmetric multilinear maps. In this article, we just need that one
can reconstruct a symmetric multilinear map from its homogenous polynomial.
\begin{prop}[Polarization formula]\label{prop:Polarisierungsformel}
	Let $\beta:\SX^k\to\SY$ be a symmetric $k$-linear map,
	$p:\SX\to\SY:x \mapsto \beta(x,\dotsc,x)$ its homogenous polynomial
	and $x_0\in\SX$. Then
	\[
		\beta(x_1,\dotsc,x_k) = \frac{1}{k!}
		\sum^1_{\eps_1,\dotsc,\eps_k = 0} (-1)^{k - (\eps_1 + \dotsb + \eps_k)}
		p(x_0 + \eps_1 x_1 + \dotsb + \eps_k x_k)
	\]
	for all $x_1,\dotsc,x_k\in\SX$.
\end{prop}
\paragraph{Complex analytic maps}
Now we can define complex analytic maps.
\begin{defi}[Complex analytic maps]
	Let $\SX$, $\SY$ be complex locally convex topological vector spaces and
	$\UF \subseteq \SX$ an open nonempty set.
	A map $f:\UF\to\SY$ is called \emph{complex analytic}
	if it is continuous and, for each $x\in\UF$ there exists
	a sequence $(p_k)_{k\in\N}$ of continuous homogenous polynomials
	$p_{k} : \SX \to \SY$ of degree $k$ such that
	\[
		f(x + v) = \sum_{k=0}^\infty p_k(v)
	\]
	for all $v$ in some zero neighborhood $\VF$ such that $x + \VF \subseteq\UF$.
\end{defi}

\begin{defi}
	Let $\SX$, $\SY$ be complex locally convex topological vector spaces and
	$\UF \subseteq \SX$ an open nonempty set.
	A map $f:\UF\to\SY$ is called \emph{Gateaux analytic}
	if its restriction on each affine line is complex analytic; that is,
	for each $x\in\UF$ and $v\in\SX$ the map
	\[
		\set{z \in \C}{x + z v \in \UF} \to \SY : z\mapsto f(x + z v)
	\]
	is complex analytic.
\end{defi}

\begin{satz}\label{satz:Char_komplex_analytischer_Abb}
	Let $\SX$, $\SY$ be complex locally convex topological vector spaces and
	$\UF \subseteq \SX$ an open nonempty set.
	Then for a map $f:\UF\to\SY$ the following assertions are equivalent:
	\begin{enumerate}
		\item
			$f$ is $\ConDiff[\C]{}{}{\infty}$,
			
		\item
			$f$ is complex analytic,
			
		\item
			$f$ is continuous and Gateaux analytic.
	\end{enumerate}
\end{satz}
We state a few results concerning analytic curves. These share many properties
with holomorphic functions. Using \refer{satz:Char_komplex_analytischer_Abb}, we see that
some of these properties carry over to general analytic functions.
\begin{defi}
	Let $\SY$ be a complex locally convex topological vector space and
	$\UF \subseteq \C$ an open nonempty set.
	A continuous map $f:\UF\to\SY$ is called a \emph{$\ConDiff[\C]{}{}{0}$-curve}.
	A $\ConDiff[\C]{}{}{0}$-curve $f:\UF\to\SY$ is called a
	$\ConDiff[\C]{}{}{1}$-curve if for all $z \in \UF$ the limit
	\[
		f^{(1)}(z)\ndef \lim_{w\to 0}\frac{f(z+w)-f(z)}{w}
	\]
	exists and the curve $f^{(1)} : \UF \to \SX$ is a $\ConDiff[\C]{}{}{0}$-curve.
	\\
	Inductively, for $k\in\N$ a curve $f$ is called a $\ConDiff[\C]{}{}{k}$-curve
	if it is a $\ConDiff[\C]{}{}{1}$-curve and $f^{(1)}$
	is a $\ConDiff[\C]{}{}{k - 1}$-curve.
	In this case, we define $f^{(k)}\ndef (f^{(1)})^{(k-1)}$.
	\\
	If $f$ is a $\ConDiff[\C]{}{}{k}$-curve for all $k\in\N$,
	$f$ is called a $\ConDiff[\C]{}{}{\infty}$-curve.
\end{defi}

\begin{lem}[Cauchy integral formula]\label{lem:Komplexe_Kurven_Ableitung}
	Let $\SY$ be a complex locally convex topological vector space,
	$\UF \subseteq \C$ an open nonempty set and $f:\UF\to\SY$ a map.
	Then
	\[
		\text{$f$ is a $\ConDiff[\C]{}{}{k}$-curve}
		\iff
		f \in \ConDiff[\C]{\UF}{\SY}{k}
	\]
	and furthermore
	\[
		\dA[k]{f}{x}{h_1,\dotsc,h_k}
		= h_1 \cdot\dotsm\cdot h_k \cdot f^{(k)}(x).
	\]
	A $\ConDiff[\C]{}{}{\infty}$-curve is complex analytic, and for each
	$x\in\UF$, $k \in \N_0$ and $r>0$ with $\clBall{x}{r} \subseteq \UF$
	the \emph{Cauchy integral formula}
	\[
		f^{(k)}(z) = \frac{k !}{2\pi i}
			\Lint{\abs{\zeta - x} = r}{\frac{f(\zeta)}{(\zeta - z)^{k + 1}} }{\zeta}
	\]
	holds, where $z \in \Ball{x}{r}$.
\end{lem}
The Cauchy integral formula implies the Cauchy estimates.
\begin{cor}\label{cor:Cauchy_Abschaetzungen}
	Let $\SY$ be a complex locally convex topological vector space,
	$\UF \subseteq \C$ an open nonempty set, $f:\UF\to\SY$ a complex analytic map,
	$x\in\UF$, $r>0$ such that $\clBall{x}{r} \subseteq \UF$, $\sigma \in ]0, 1[$
	and $p$ a continuous seminorm on $\SY$.
	Then for each $k\in\N$ and $z \in \Ball{x}{ r }$ with $\abs{z - x} = \sigma r$,
	we get the estimate
	\[
		\norm{f^{(k)}(z)}_{p} \leq \frac{k!}{(1 - \sigma )^{k + 1} r^k}
		\sup_{\abs{\zeta - x} = r}  \norm{f(\zeta)}_{p}.
	\]
\end{cor}

\paragraph{Real analytic maps}
\begin{defi}[Real analytic maps]
	\label{def:reell_analytische_Abb}
	Let $\SX$, $\SY$ be real locally convex topological vector spaces and
	$\UF \subseteq \SX$ an open nonempty set. Let $\SX_\C$ resp.
	$\SY_\C$ denote the complexifications of $\SX$ resp. $\SY$.
	A map $f:\UF\to\SY$ is called \emph{real analytic}
	if there is an extension $\widetilde{f} : \VF \to \SY_\C$ of $f$
	to an open neighborhood $\VF$ of $\UF$ in $\SX_\C$ that is complex analytic.
	Such a map $\widetilde{f}$ will be refered to as a \emph{complexification} of $f$.
	\index{complexification!of maps}
\end{defi}
\subsubsection{Lipschitz continuous maps between locally convex spaces and induced maps on normed spaces}
\label{sususec:Lipschitz-stetige_Abb}
We define and discuss Lipschitz continuous maps between locally convex spaces.
To this end, we define some terms concerning seminorms and the quotient maps they induce.
\begin{defi}\label{def:lokalkonvexe_Faktorraeume}
	Let $\SX$ be a locally convex space and $p : \SX \to \R$ a continuous seminorm.
	We denote the Hausdorff space $\SX /p^{-1}(0)$ with $\glstext{normierter_Faktorraum}$
	and the quotient map with $\glstext{Projektion_normierter_Faktorraum} : \SX \to \SX_p$.
	More general, for any subset $A \subseteq \SX$ we set
	$A_p \ndef \morQuot{p}(A)$.
	\\
	Further, we let $\glstext{Menge_der_stetigen_Halbnormen}$
	denote the set of continuous seminorms on $\SX$.
	\\
	Let $p \in \normsOn{\SX}$. We call $\UF \subseteq \SX$ \emph{open with respect to $p$}
	if for each $x \in \UF$ there exists $r > 0$ such that
	$\{y \in \SX : \norm{y - x}_{p} < r\} \subseteq \UF$.
\end{defi}

\begin{bem}
	For any locally convex space $\SX$ and each $p \in \normsOn{\SX}$,
	the norm induced by $p$ on $\SX_p$ will also be denoted by $p$.
	Note that this leads to the identity $p = \HomQuot{p}{p}$, in particular
	$p$ is a norm and generates the topology on $\SX_p$.
	No confusion will arise.
\end{bem}

\begin{defi}[Lipschitz continuous maps]
	Let $\SX$ and $\SY$ be locally convex spaces, $\UF \subseteq \SX$ an open nonempty set,
	$k \in \cl{\N}$, $p\in\normsOn{\SY}$ and $q\in\normsOn{\SX}$.
	We call $\gamma: \UF\to\SY$
	\emph{Lipschitz up to order $k$ with respect to $p$ and $q$}
	if $\gamma \in \ConDiff{\UF}{\SY}{k}$,
	\begin{equation}
		\label{est:Lipschitz_Stetigkeit_der_l-ten_Ableitung_bzgl_Halbnormen}
		\norm{ \dA[\ell]{\gamma}{y}{h_1,\dotsc,h_\ell} - \dA[\ell]{\gamma}{x}{h_1,\dotsc,h_\ell} }_{p}
		\leq \norm{y - x}_{q} \prod_{i=1}^\ell\norm{ h_i}_{q}
	\end{equation}
	and
	\begin{equation}
		\label{est:Stetigkeit_der_l-ten_Ableitung_bzgl_Halbnormen}
		\norm{ \dA[\ell]{\gamma}{x}{h_1,\dotsc,h_\ell} }_{p}
		\leq \prod_{i=1}^\ell\norm{ h_i}_{q}.
	\end{equation}
	for all $\ell \in \N$ with $\ell \leq k$, $x, y \in \UF$ and $h_1,\dotsc,h_\ell \in\SX$.
	We write $\LipCon{q}{p}{\UF}{\SY}{k}$ for the set of maps that are
	Lipschitz up to order $k$ with respect to $p$ and $q$.
\end{defi}
As for differentiable maps between normed spaces, differentiable maps always are at least locally Lipschitz.
\begin{lem}\label{lem:stetig_diffbar_impliziert_lokal_Lipschitz}
	Let $\SX, \SY$ be locally convex spaces, $\UF \subseteq \SX$ an open nonempty set,
	$k \in \cl{\N}$, $\gamma \in \ConDiff{\UF}{\SY}{k + 1}$ and $\ell \in \N$ with $\ell \leq k$.
	Then for each $p \in \normsOn{\SY}$ and $x_0 \in \UF$ there exist $q \in \normsOn{\SX}$
	and a convex neighborhood $U_{x_0} \subseteq \UF$ of $x$ with respect to $q$
	such that $\rest{\gamma}{U_{x_0}} \in \LipCon{q}{p}{U_{x_0}}{\SY}{k}$.
\end{lem}
\begin{proof}
	Since $\dA[\ell]{\gamma}{}{}$ and  $\dA[\ell + 1]{\gamma}{}{}$ are continuous in $(x_0, 0, \dotsc, 0)$
	and multilinear in their last $\ell$ resp. $\ell + 1$ arguments, for each $p \in \normsOn{\SY}$
	there exist $q \in \normsOn{\SX}$
	and an open ball $U_{x_0} \ndef\Ball[q]{x_0}{r} \subseteq \UF$ such that
	\[
		1 \geq
		\sup \{\norm{ \dA[\ell + 1]{\gamma}{y}{h_1,\dotsc, h_{\ell + 1}}}_{p}
		: y \in \Ball[q]{x_0}{r}, \norm{h_1}_{q}, \dotsc, \norm{h_{\ell + 1}}_{q} \leq 1 \}
	\]
	and
	\[
		1 \geq
		\sup \{\norm{ \dA[\ell]{\gamma}{y}{h_1,\dotsc, h_{\ell}}}_{p}
		: y \in \Ball[q]{x_0}{r}, \norm{h_1}_{q}, \dotsc, \norm{h_{\ell}}_{q} \leq 1 \}.
	\]
	This implies that for each $y \in \Ball[q]{x_0}{r}$ and $h_1,\dotsc, h_{n} \in \SX$
	\[
		\tag{\ensuremath{\dagger}}
		\label{est:beschraenktheit_n-tes_Differential_geeignete_Halbnormen}
		\norm{ \dA[n]{\gamma}{y}{h_1,\dotsc, h_{n}}}_{p} \leq 1\cdot \prod_{i=1}^n\norm{ h_i}_{q},
	\]
	where $n \in \{\ell, \ell + 1\}$; this proves \refer{est:Stetigkeit_der_l-ten_Ableitung_bzgl_Halbnormen}.
	\\
	To prove \refer{est:Lipschitz_Stetigkeit_der_l-ten_Ableitung_bzgl_Halbnormen}, we see that
	for $x, y \in \Ball[q]{x_0}{r}$ and $h_1,\dotsc, h_{\ell + 1} \in \SX$
	\[
		\dA[\ell]{\gamma}{y}{h_1,\dotsc, h_\ell} - \dA[\ell]{\gamma}{x}{h_1,\dotsc, h_\ell}
		=
		\Rint{0}{1}{
		\dA[\ell + 1]{\gamma}{t y + (1 - t) x}{h_1,\dotsc, h_\ell, y - x}
		}%
		{t}.
	\]
	We apply \refer{lem:Integral_stetiges_Funktional} to the right hand side and get using
	\eqref{est:beschraenktheit_n-tes_Differential_geeignete_Halbnormen} with $n = \ell + 1$.
	\[
		\norm{ \dA[\ell]{\gamma}{y}{h_1,\dotsc, h_\ell} - \dA[\ell]{\gamma}{x}{h_1,\dotsc, h_\ell} }_{p}
		\leq
		\norm{h_1}_{q} \dotsm \norm{h_{\ell}}_{q} \cdot \norm{y - x}_{q}
	\]
	which finishes the proof.
\end{proof}
We show that each Lipschitz map induces another Lipschitz map between the respective (normed) quotient spaces.
\begin{lem}\label{lem:Faktorisieren_Lipschitz-stetiger_Abb_lokalkonvexer_Raeume}
	Let $\SX$ and $\SY$ be locally convex spaces, $\UF \subseteq \SX$ an open nonempty set,
	$k \in \cl{\N}$, $p\in\normsOn{\SY}$, $q\in\normsOn{\SX}$
	and $\gamma \in \LipCon{q}{p}{\UF}{\SY}{k}$.
	Then there exists a map $\FakLC{q}{p}{\gamma} \in \LipCon{q}{p}{\UF_q}{\SY_p}{k}$ that makes the diagram
	\begin{equation*}
		\xymatrix{
			{\UF }
			\ar[rr]^{\gamma}
			\ar[d]|{\morQuot{q} }
			&&
			{\SY}
			\ar[d]|{\morQuot{p}}
			\\
			{\UF_q}
			\ar[rr]^{\FakLC{q}{p}{\gamma}}
			&&
			{\SY_p}
		}
	\end{equation*}
	commutative (using notation as in \refer{def:lokalkonvexe_Faktorraeume}).
\end{lem}
\begin{proof}
	Let $\ell \in \N$ with $\ell \leq k$.
	Since $\gamma \in \LipCon{q}{p}{\UF}{\SY}{k}$, the map
	\[
		\HomQuot{\dA[\ell]{\gamma}{}{}}{p} : (\UF,q)\times(\SX,q)^{\ell} \to \SY_p
	\]
	is continuous.
	Hence by the universal property of the separation there exists a continuous map
	$\FakLC{q}{p}{\gamma}_\ell$ such that the diagram
	\begin{equation*}
		\xymatrix{
			{\UF \times \SX^\ell}
			\ar[rr]^{\dA[\ell]{\gamma}{}{}}
			\ar[dd]|{\mathop{\Pi}\limits_{i = 1}^{\ell + 1} \morQuot{q}}
			\ar[rd]
			&&
			{\SY}
			\ar[dd]|{\morQuot{p}}
			\\
			&
			(\UF,q)\times(\SX,q)^{\ell}
			\ar[rd]
			\ar@{>>}[ld]
			&
			\\
			{\UF_q\times \SX_q^{\ell}}
			\ar[rr]^{\FakLC{q}{p}{\gamma}_\ell}
			&
			&
			{\SY_p}
		}
	\end{equation*}
	commutes, where we denote $\rest{\morQuot{q}}{\UF}$ with $\morQuot{q}$.
	The diagram for $\ell = 0$ implies that
	$\FakLC{q}{p}{\gamma} \circ \morQuot{q} = \HomQuot{\gamma}{p} \in \ConDiff{\UF}{\SY_p}{k}$,
	where $\FakLC{q}{p}{\gamma} \ndef \FakLC{q}{p}{\gamma}_0$.
	We proved in \refer{lem:DifferenzierbarkeitseigenschaftenKomposition_gut,hinten_linear}
	that the $\ell$-th directional derivative of $\FakLC{q}{p}{\gamma}$ exists and satisfies the identity
	\[
		\dA[\ell]{\FakLC{q}{p}{\gamma}}{}{} \circ \mathop{\Pi}_{i = 1}^{\ell + 1} \morQuot{q}
		= \dA[\ell]{(\FakLC{q}{p}{\gamma} \circ \morQuot{q})}{}{}
		= \dA[\ell]{(\morQuot{p} \circ \gamma)}{}{}
		= \morQuot{p} \circ \dA[\ell]{\gamma}{}{}
		= \FakLC{q}{p}{\gamma}_\ell \circ \mathop{\Pi}_{i = 1}^{\ell + 1}\morQuot{q}.
	\]
	Since $\mathop{\Pi}_{i = 1}^{\ell + 1}\morQuot{q}$ is surjective, this implies that
	$\dA[\ell]{\FakLC{q}{p}{\gamma}}{}{} = \FakLC{q}{p}{\gamma}_\ell$, so the former is continuous.
	From this we conclude that $\FakLC{q}{p}{\gamma} \in \ConDiff{\UF_q}{\SY_p}{k}$
	and that the estimates \eqref{est:Lipschitz_Stetigkeit_der_l-ten_Ableitung_bzgl_Halbnormen}
	and \eqref{est:Stetigkeit_der_l-ten_Ableitung_bzgl_Halbnormen}
	are satisfied by $\FakLC{q}{p}{\gamma}$.
\end{proof}
Finally, we prove that for each compact set, each $\ConDiff{}{}{k + 1}$-map
defined on it and each seminorm on the image there exists a seminorm on the domain
such that the quotient map, and its differentials, are bounded.
For that, we need to use a lemma about the relationship between differentiability and Fréchet differentiability
that is proved later in \refer{sec:Frechetbarkeit}.
\begin{lem}\label{lem:stetig_diffbar_impliziert_lokal_Lipschitz_und_beschraenkt}
	Let $\SX$ and $\SY$ be locally convex spaces, $\UF \subseteq \SX$ an open nonempty set,
	$k \in \N$, $\gamma \in \ConDiff{\UF}{\SY}{k + 1}$, $p \in \normsOn{\SY}$ and $K$ a compact subset of $\UF$.
	Then there exists a seminorm $q \in \normsOn{\SX}$ and an open set $\VF$ w.r.t. $q$ such that
	$K \subseteq \VF \subseteq \UF$ and $\FakLC{q}{p}{\gamma} \in \BC{\VF_q}{\SY_p}{k}$
	(For the definition of $\FakLC{q}{p}{\gamma}$ see \refer{lem:Faktorisieren_Lipschitz-stetiger_Abb_lokalkonvexer_Raeume}).
\end{lem}
\begin{proof}
	Using \refer{lem:stetig_diffbar_impliziert_lokal_Lipschitz} and standard compactness arguments,
	we find $q \in \normsOn{\SX}$ and a neighborhood $\widetilde{\VF}$ w.r.t. $q$ of $K$ in $\UF$
	such that
	\refer{est:Lipschitz_Stetigkeit_der_l-ten_Ableitung_bzgl_Halbnormen}
	and \refer{est:Stetigkeit_der_l-ten_Ableitung_bzgl_Halbnormen}
	hold for $\gamma$ on $\widetilde{\VF}$ and all $\ell\in\N$ with $\ell \leq k$.
	We proved in \refer{lem:Faktorisieren_Lipschitz-stetiger_Abb_lokalkonvexer_Raeume}
	that this implies that $\FakLC{q}{p}{\gamma} \in \LipCon{q}{p}{\widetilde{\VF}_q}{\SY_p}{k}$,
	and with \refer{prop:Char_Frechet_Diff} we can conclude that
	$\FakLC{q}{p}{\gamma} \in \FC{\widetilde{\VF}_q}{\SY_p}{k}$.
	Further, since $D^{(\ell)}\FakLC{q}{p}{\gamma}(K_q)$ is compact for all $\ell \leq k$,
	there exists a neighborhood $\VF_q$ of $K_q$ such that
	$\FakLC{q}{p}{\gamma}$ and all its derivatives up to degree $k$ are bounded on $\VF_q$.
\end{proof}
\section{Fr\'echet differentiability}
\label{sec:Frechetbarkeit}
For maps between normed spaces, there is the classical notion of \emph{Fr\'echet differentiability}.
This concept relies on the existence of a well-behaved topology on the space of
($k$-)linear maps between normed spaces.
\paragraph{Spaces of multilinear maps between normed spaces}
We provide the details about the norm topology of multilinear operators.
\begin{defi}
	Let $\SX$, $\SY$ be normed spaces. For each $k \in \N^\ast$ we define
	\[
		\glstext{space_of_k-linear-maps-normed}
		\ndef
		\{\Xi:\SX^k \to \SY :
			\Xi\text{ is $k$-linear and continuous}\} .
	\]
	For $k=1$ we define
	\[
		\Lin{\SX}{\SY} \ndef \Lin[1]{\SX}{\SY}
		\text{ and }
		\Lin{\SX}{\SX} \ndef \Lin[1]{\SX}{\SX},
	\]
	and furthermore
	\[
		\Lin[0]{\SX}{\SY} \ndef \SY .
	\]
\end{defi}
The set of multilinear continuous maps can be turned into a normed vector space:
\begin{prop}
	Let $\SX$, $\SY$ be normed spaces and $k \in \N^\ast$.
	A $k$-linear map $\Xi : \SX^{k} \to \SY$ is continuous iff
	\[
		\glsdisp{Operatornorm_einer_k-linear-maps-normed}{\Opnorm{\Xi}} \ndef
		\sup\{ \norm{\Xi(v_1,\ldots,v_k)} :
			\norm{v_1},\ldots,\norm{v_k}\leq 1 \} < \infty .
	\]
	$\Opnorm{\Xi}$ is called the \emph{operator norm} of $\Xi$.	
	$\Opnorm{\cdot}$ is a norm on $\Lin[k]{\SX}{\SY}$. The space $\Lin[k]{\SX}{\SY}$,
	endowed with this norm, is complete if $\SY$ is so.
\end{prop}
\begin{proof}
	The (elementary) proof can be found in
	\cite[Chapter \RN{5}, \S 7]{MR0120319}.
\end{proof}

\begin{lem}\label{lem:Auswertung_linearer_Abb_ist_stetig}
	Let $\SX$, $\SY$ be normed spaces and $k \in \N^\ast$.
	Then the evaluation map
	\[
		\Lin[k]{\SX}{\SY}\times \SX^k
		: (\Xi, v_1,\ldots,v_k)\mapsto \Xi(v_1,\ldots,v_k)
	\]
	is $(k+1)$-linear and continuous.
\end{lem}
\begin{proof}
	This is trivial.
\end{proof}

\begin{lem}
	Let $\SX$ and $\SY$ be normed spaces, $k \in \N^\ast$, $\Xi \in \Lin[k]{\SX}{\SY}$
	and $h_1,\dotsc,h_k$, $v_1,\dotsc, v_k \in \SX$.
	Then
	\[
		\norm{\Xi(h_1,\dotsc,h_n) - \Xi(v_1,\dotsc,v_k)}
		\leq \sum_{i=1}^k \norm{\Xi(v_1,\dotsc,v_{i-1}, h_i - v_i, h_{i+1},\dotsc,h_k)}.
	\]
\end{lem}
\begin{proof}
	This estimate is derived by an iterated application of the triangle inequality.
\end{proof}

The following lemma helps to deal with higher derivatives of Fr\'echet-differentiable maps.

\begin{lem}\label{lem:Identifizierung_seltsamer_Raeume}
	Let $\SX$, $\SY$ be normed spaces and $n,k \in \N^\ast$.
	Then the map
	\begin{multline*}
		\EmbA{k}{n}
		:\Lin[k]{\SX}{ \Lin[n]{\SX}{\SY} }
		\to \Lin[k + n]{\SX}{\SY}
		\\
		\EmbA{k}{n}(\Xi)(h_1,\dotsc,h_n, v_1,\dotsc,v_k)
		\ndef \Xi(v_1,\dotsc,v_k) (h_1,\dotsc,h_n)
	\end{multline*}
	is an isometric isomorphism.
	In some cases, we will denote $\EmbA{k}{n}$ by $\EmbA[\SY]{k}{n}$.
\end{lem}
\begin{proof}
	Obviously $\EmbA{k}{n}$ is linear and injective. Furthermore
	\begin{multline*}
		\norm{\EmbA{k}{n}(\Xi)(h_1,\dotsc,h_n, v_1,\dotsc,v_k)}
		= \norm{\Xi(v_1,\dotsc,v_k) (h_1,\dotsc,h_n)}
		\\
		\leq \Opnorm{\Xi(v_1,\dotsc,v_k)} \prod_{i=1}^n\norm{h_i}
		\leq \Opnorm{\Xi}\prod_{i=1}^k\norm{v_i} \prod_{i=1}^n\norm{h_i},
	\end{multline*}
	and hence
	\[
		\Opnorm{\EmbA{k}{n}(\Xi)} \leq \Opnorm{\Xi}.
	\]	
	On the other hand, for
	$\norm{v_1},\dotsc,\norm{v_k}, \norm{h_1},\dotsc,\norm{h_n}\leq 1$
	we have
	\[
		\norm{\Xi(v_1,\dotsc,v_k) (h_1,\dotsc,h_n)}
		\leq \Opnorm{\EmbA{k}{n}(\Xi)}.
	\]
	Hence
	\[
		\Opnorm{\Xi(v_1,\dotsc,v_k)}\leq \Opnorm{\EmbA{k}{n}(\Xi)} ,
	\]
	which leads to
	\[
		\Opnorm{\Xi}\leq \Opnorm{\EmbA{k}{n}(\Xi)},
	\]
	so $\EmbA{k}{n}$ is an isometry. It remains to show that $\EmbA{k}{n}$ is surjective.
	To this end, for a $M\in\Lin[k + n]{\SX}{\SY}$ we define the map
	$\overline{M} \in \Lin[k]{\SX}{ \Lin[n]{\SX}{\SY} }$ by
	\[
		\overline{M}(v_1,\dotsc,v_k)(h_1,\dotsc,h_n)
		\ndef M(h_1,\dotsc,h_n, v_1,\dotsc,v_k) .
	\]
	Clearly, $\EmbA{k}{n}(\overline{M}) = M$.
	Since $M$ was arbitrary, $\EmbA{k}{n}$ is surjective.
\end{proof}

\begin{lem}\label{lem:Produkte_und_Lin^k}
	Let $\SX$, $\SY$ and $\SZ$ be normed spaces and $k \in \N$. Then the map
	\begin{equation}\label{map:Isomorphie_multilineare_Abb_Produkt}
		\Lin[k]{\SX}{\SY \times \SZ} \to \Lin[k]{\SX}{\SY} \times \Lin[k]{\SX}{\SZ}
		:
		\Xi \mapsto (\pi_{\SY} \circ \Xi, \pi_{\SZ} \circ \Xi) ,
	\end{equation}
	where $\pi_{\SY}$ respective $\pi_{\SZ}$ denotes the canonical projection
	from $\SY \times \SZ$ to $\SY$ respective $\SZ$, is an
	isomorphism of topological vector spaces.
\end{lem}
\begin{proof}
	The map in \eqref{map:Isomorphie_multilineare_Abb_Produkt}
	is linear since its component maps $\Xi \mapsto \pi_{\SY} \circ \Xi$
	and $\Xi \mapsto \pi_{\SZ} \circ \Xi$ are so.
	The injectivity of \eqref{map:Isomorphie_multilineare_Abb_Produkt} is clear,
	and the surjectivity can also be shown by an easy computation.
	\\
	To see that \eqref{map:Isomorphie_multilineare_Abb_Produkt}
	is an isomorphism we denote it by $\mathfrak{i}$ and compute
	for $x_{1},\dotsc,x_{k} \in \SX$
	\begin{multline*}
		\bigl(
			(\pi_{ \Lin[k]{\SX}{\SY} } \circ \mathfrak{i})(\Xi)(x_{1},\dotsc,x_{k}),
			(\pi_{ \Lin[k]{\SX}{\SZ} } \circ \mathfrak{i})(\Xi)(x_{1},\dotsc,x_{k})
		\bigr)
		\\
		=
		\bigl(
			(\pi_{\SY} \circ \Xi)(x_{1},\dotsc,x_{k}),
			(\pi_{\SZ} \circ \Xi)(x_{1},\dotsc,x_{k})
		\bigr)
		=
		\Xi(x_{1},\dotsc,x_{k}) .
	\end{multline*}
	From this one can easily derive that $\mathfrak{i}$ and its inverse are
	continuous since depending on the norm we chose on the products, $\mathfrak{i}$ is an isometry.
\end{proof}

\paragraph{The calculus}
In the following, let $\SX$, $\SY$ and $\SZ$ denote normed spaces and
$\UF$ be an open nonempty subset of $\SX$.
Recall the definition of Fr\'echet differentiability given in \refer{def:Frechet_Diffbarkeit}.

We give some examples of Fr\'echet differentiable maps.
\begin{beisp}\label{beisp:Frechet_diffbare_Abb}
	\begin{enumerate}
	\item
		A continuous linear map $A:\SX \to \SY$ is smooth with $D A(x) = A$.
	\item
		More generally, a continuous $k$-linear map
		$b:\SX_1\times\dotsb\times\SX_k \to \SY$
		is smooth with
		\[
			D b(x_1,\dotsc,x_k) (h_1,\dotsc,h_k)
			= \sum_{i=1}^k b(x_1,\dotsc,x_{i-1},h_i,x_{i+1},\dotsc,x_k).
		\]
	\end{enumerate}
\end{beisp}

We prove the Chain Rule and the Mean Value Theorem for Fr\'echet differentiable maps.
Beforehand, we need the following
\begin{lem}\label{lem:Kettenregel_linearer_Abb}
	Let $\SX$, $\SY$ and $\SZ$ be normed spaces,
	$\UF \subseteq \SX$ an open nonempty set,
	$k \in \cl{\N}$ and $A:\SY \to \SZ$ a continuous linear map.
	Then for $\gamma \in \FC{\UF}{\SY}{k}$
	\[
		A \circ \gamma \in \FC{\UF}{\SZ}{k}.
	\]
\end{lem}
\begin{proof}
	We prove this by induction over $k$.
	The assertion is obviously true for $k = 0$.
	If $k = 1$, then $A \circ \gamma$ is $\ConDiff{}{}{1}$ by \refer{prop:Kettenregel_Helge} with
	\[
		\dA{(A \circ \gamma)}{x}{\cdot}
		= \dA{A}{\gamma(x)}{\cdot} \MaMu \dA{\gamma}{x}{\cdot}
		= A \circ \dA{\gamma}{x}{\cdot}.
	\]
	Since the composition of linear maps is continuous, we conclude that $A \circ \gamma$ is $\FC{}{}{1}$
	with $\FAbl{(A \circ \gamma)} = A \circ \FAbl{\gamma}$.

	$k \to k + 1$:
	The map $\FAbl{\gamma}$ is $\FC{}{}{k}$, hence by the induction hypothesis, so is
	$A \circ \FAbl{\gamma} = \FAbl{(A \circ \gamma)}$.
	Hence $A \circ \gamma$ is $\FC{}{}{k + 1}$, which finishes the induction.
\end{proof}

\begin{lem}\label{lem:Zusammensetzen_Frechet-diffbarer_Abb-Produkt}
	Let $k \in \cl{\N}$, $\eta \in \FC{\UF}{\SY}{k}$ and $\gamma \in \FC{\UF}{\SZ}{k}$.
	Then the map
	\[
		(\gamma, \eta) : \UF \to \SY \times \SZ : x \mapsto (\gamma(x), \eta(x))
	\]
	is contained in $\FC{\UF}{\SY \times \SZ}{k}$.
\end{lem}
\begin{proof}
	For $k = 0$ the assertion is obviously true.
	If $k = 1$, we easily calculate that $(\gamma, \eta)$ is $\ConDiff{}{}{1}$ with
	\[
		\dA{(\gamma, \eta)}{x}{h} = (\dA{\gamma}{x}{h}, \dA{\eta}{x}{h}).
	\]
	Hence
	\[
		\dA{(\gamma, \eta)}{x}{\cdot} = \mathfrak{i}^{-1}(\dA{\gamma}{x}{\cdot}, \dA{\eta}{x}{\cdot}),
	\]
	where $\mathfrak{i}$ denotes the isomorphism \eqref{map:Isomorphie_multilineare_Abb_Produkt}
	from \refer{lem:Produkte_und_Lin^k}. We conclude that $(\gamma, \eta)$ is $\FC{}{}{1}$.
	
	For $k > 1$, the assertion is proved with an easy induction using \refer{lem:Kettenregel_linearer_Abb}.
\end{proof}

\begin{prop}[Chain Rule]\label{prop:Kettenregel_Frechet}
	Let $k \in \cl{\N}$, $\eta \in \FC{\UF}{\SY}{k}$ and
	$\gamma \in \FC{\VF}{\SZ}{k}$ such that $\eta(\UF)\subseteq\VF$.
	Then
	$\gamma\circ \eta \in \FC{\UF}{\SZ}{k}$ and
	\begin{equation*}\label{id:Differential_komponierter_Abb}
		\tag{\ensuremath{\ast}}
		\FAbl{(\gamma\circ \eta)}(u) = ( \FAbl{\gamma} \circ \eta)(u) \MaMu  \FAbl{\eta}(u)
	\end{equation*}
	for all $u \in \UF$.
\end{prop}
\begin{proof}
	The proof is by induction on $k$:\\
	$k=1:$
	We apply the chain rule for $\ConDiff{}{}{1}$-maps (\refer{prop:Kettenregel_Helge})
	to see that $\gamma \circ \eta$ is $\ConDiff{}{}{1}$,
	and for $(u,x)\in \UF\times\SX$ we have
	\[
		\dA{(\gamma\circ \eta)}{u}{x}
		= \dA{\gamma}{\eta(u)}{\dA{\eta}{u}{x}}.
	\]
	From this identity we conclude
	that \eqref{id:Differential_komponierter_Abb} holds.
	Finally we obtain the continuity of $\FAbl{(\gamma\circ \eta)}$
	from the one of $\MaMu$, $\FAbl{\gamma}$, $\FAbl{\eta}$ and $\eta$.
	
	$k\to k+1:$
	By the inductive hypothesis, the maps $\FAbl{\gamma}$ and $\FAbl{\eta}$ are $\FC{}{}{k}$.
	We already proved in the case $k = 1$ that \eqref{id:Differential_komponierter_Abb} holds.
	By the inductive hypothesis, $\FAbl{\gamma} \circ \eta \in \FC{}{}{k}$.
	Since $\MaMu$ is smooth (see \refer{beisp:Frechet_diffbare_Abb}),
	we conclude using \refer{lem:Zusammensetzen_Frechet-diffbarer_Abb-Produkt}
	and the inductive hypothesis that $\FAbl{(\gamma\circ \eta)}$ is $\FC{}{}{k}$.
	Hence $\gamma\circ \eta$ is $\FC{}{}{k + 1}$.
\end{proof}

\begin{prop}[Mean Value Theorem]\label{prop:MWS_FC1_Abb}
	Let $f \in \FC{\UF}{\SY}{1}$. Then
	\[
		f(v) - f(u) = \Rint{0}{1}{\FAbl{f}(u + t (v-u)) \eval (v-u)}{t}
	\]
	for all $v, u \in \UF$ such that the line segment
	$\{ t u + (1-t) v : t\in [0,1]\}$ is contained in $\UF$.
	In particular
	\[
		\norm{f(v) - f(u)}
		\leq \sup_{t\in[0,1]} \Opnorm{\FAbl{f}(u + t (v-u))} \norm{v - u}.
	\]
\end{prop}
\begin{proof}
	The identity is a reformulation of \refer{prop:MWS_C1_Abb},
	hence the estimate is a direct consequence of
	\refer{lem:Integral_stetiges_Funktional}.
\end{proof}
The isomorphisms provided by \refer{lem:Identifizierung_seltsamer_Raeume}
can be used to characterize Fr\'echet differentiability of higher order.
\begin{bem}\label{bem:Notation_Frechet_Identifizierung}
	We define inductively
	\[
		L^{0}_{\SX,\SY} \ndef \SY
		\text{ and }
		L^{k + 1}_{\SX,\SY} \ndef \Lin{\SX}{L^{k}_{\SX,\SY}}.
	\]
\end{bem}

\begin{defi}[Higher derivatives]\label{def:Frechet_hohe_Abl}
	Let $n\in\N$. For each $k\in\N$ with $k\leq n$
	we define a linear map
	\[
		\FAbl[k]{} : \FC{\UF}{\SY}{n}
		\to \FC{\UF}{\Lin[k]{\SX}{\SY}}{n-k}
	\]
	by $\FAbl[0]{} \ndef \id{\FC{\UF}{\SY}{n}}$ for $k=0$,
	$\FAbl[1]{}\ndef \FAbl{}$ for $k = 1$ and for $1 < k \leq n$ by
	\[
		\FAbl[k]{} \gamma \ndef
		\EmbA[\SY]{k-1}{1} \circ \dotsb \circ
		\EmbA[L^{k-3}_{\SX,\SY}]{2}{1} \circ \EmbA[L^{k-2}_{\SX,\SY}]{1}{1}
		\circ (\underbrace{\FAbl{} \circ \dotsb \circ \FAbl{}}_{k \text{ times}}) (\gamma).
	\]
	Here we used the notations introduced in
	\refer{bem:Notation_Frechet_Identifizierung}.
	Note that the image of $\FAbl[k]{}$ is contained in
	$\FC{\UF}{\Lin[k]{\SX}{\SY}}{n-k}$
	because the maps
	$\EmbA[\SY]{k-1}{1}$, \ldots, $\EmbA[L^{k-3}_{\SX,\SY}]{2}{1}$, $\EmbA[L^{k-2}_{\SX,\SY}]{1}{1}$
	are continuous linear maps and hence smooth
	(see \refer{beisp:Frechet_diffbare_Abb});
	so the chain rule (\refer{prop:Kettenregel_Frechet}) gives the result.
	\\
	We call $\FAbl[k]{}$ the \emph{$k$-th derivative operator}.
\end{defi}
The $(k + 1)$-st derivative of a map $\gamma$ is closely related to
the $k$-th derivative of $\FAbl{\gamma}$:
\begin{lem}\label{lem:Ableitungen_von_D}
	Let $n\in \cl{\N}^\ast$, $\gamma \in \FC{\UF}{\SY}{n}$ and
	$k \in \N$ with $k < n$. Then
	\[
		\FAbl[k + 1]{\gamma}
		= \EmbA[\SY]{k}{1}\circ (\FAbl[k]{ (\FAbl{\gamma}) }) .
	\]
\end{lem}
\begin{proof}
	This follows directly from the definition of $\FAbl[k + 1]{\gamma}$.
\end{proof}
\subsection{The Lipschitz inverse function theorem}
\label{susec:InverseFuncThm-Lipschitz}
We discuss an inverse function theorem for functions of the sort $T + \eta$,
where $T$ is a linear Operator and $\eta$ is a \enquote{small} perturbation map.
We derive an estimate for the Lipschitz constant of $(T + \eta)^{-1}$,
and consequently another for the size of the image of $T + \eta$.
Further we discuss families of such functions to derive a parametrized inverse function theorem.
The next lemma discusses a special case.
The main tool for proving it is a parameterized version of the Banach fixed point theorem
which can be found in \cite[Appendix C]{MR586942}
\begin{lem}\label{lem:Lipschitz_inv_auf_Kugeln}
	Let $\SX$ be a normed space, $T \in \Lin{\SX}{\SX}$ a linear homeomorphism,
	$D \sub \SX$ a nonempty set, and $\eta : D \to \SX$
	Lipschitz with a constant $L$ such that $L \Opnorm{T^{-1}} < 1$.
	\begin{enumerate}
		\item\label{enum1:Lip_Stoerung_injective}
		The map $T + \eta$ is injective.
	\end{enumerate}	
	Suppose that $D = \clBall{0}{r}$ with $r > 0$ and $\eta(0) = 0$.
	Further, we set $r' \ndef r(\frac{1 - L \Opnorm{T^{-1}}}{\Opnorm{T^{-1}}})$.
	\begin{enumerate}[resume]
		\item\label{enum1:LipStoerung_Kontraktionen}
		The map
		\[
			H : \clBall{0}{r} \times \clBall{0}{r'} \to \clBall{0}{r}
			: (x, y) \mapsto  T^{-1} (y - \eta(x))
		\]
		is defined and a contraction in the first argument. For $y \in \clBall{0}{r'}$
		and $x \in D$, we have
		\[
			y = (T + \eta)(x)
			\iff
			x = H(x, y).
		\]
	\end{enumerate}
	Suppose that $\SX$ is a Banach space.
	\begin{enumerate}[resume]
		\item\label{enum1:Lip_Inv_Kugel_hard_work}
		Then $\image{(T + \eta)} = \clBall{0}{r'}$,
		and $(T + \eta)^{-1}$ is Lipschitz with constant $\frac{\Opnorm{T^{-1}}}{1 - L \Opnorm{T^{-1}}}$.
		In particular, $\Ball{0}{r'} \sub (T + \eta)(\Ball{0}{r})$.
		
		\item\label{enum1:faserweise_Invertierung_T+Lip_FCk}
		Additionally, let $\SY$ be a normed space, $\UF \sub \SY$ an open nonempty set and $k \in \cl{\N}$.
		Further, let $\Xi \in \FC{\UF \times D}{\SX}{k}$ such that $\Xi(\UF \times \sset{0}) = \sset{0}$
		and for each $p \in \UF$, the map $\Xi_p \ndef \Xi(p, \cdot) : D \to \SX$ is $L$-Lipschitz.
		Then
		\[
			\tag{\ensuremath{\dagger}}
			\label{map:faserweise_Invertierung_T+Lip}
			\UF \times \Ball{0}{r'} \to \Ball{0}{r} : (p, y) \mapsto (T + \Xi_p)^{-1}(y)
		\]
		is $\FC{}{}{k}$.
	\end{enumerate}
\end{lem}
\begin{proof}
	\refer{enum1:Lip_Stoerung_injective}
	Let $x, y \in D$ such that $(T + \eta)(x)  = (T + \eta)(y)$. Then
	\begin{equation*}
		\norm{x - y} = \norm{T^{-1}(\eta(y) - \eta(x))} \leq \Opnorm{T^{-1}} L \norm{y - x}.
	\end{equation*}
	Since $\Opnorm{T^{-1}} L < 1$, we conclude that $\norm{x - y} \leq 0$, and hence $x = y$.

	\refer{enum1:LipStoerung_Kontraktionen}
	Let $x \in \clBall{0}{r}$ and $y \in \clBall{0}{r'}$. We calculate
	\begin{multline*}
		\norm{T^{-1} (y - \eta(x))} \leq \Opnorm{T^{-1}} \norm{y - \eta(x)} \leq \Opnorm{T^{-1}} (\norm{y} + \norm{\eta(x) - \eta(0)})
		\\
		\leq \Opnorm{T^{-1}} (\norm{y} + L \norm{x})
		\leq r \Opnorm{T^{-1}} (\tfrac{1 - L \Opnorm{T^{-1}}}{\Opnorm{T^{-1}}} + L)
		= r.
	\end{multline*}
	Thus $H$ is defined. We show that $H$ is a contraction in the first argument. To this end, let 
	$x, z \in \clBall{0}{r}$ and $y \in \clBall{0}{r'}$. Then
	\[
		\norm{H(x, y) - H(z, y)} = \norm{T^{-1}(\eta(z) - \eta(x) } \leq \Opnorm{T^{-1}} L \norm{z - x}.
	\]
	Hence the map $H(\cdot, y) : \clBall{0}{r} \to \clBall{0}{r}$ is a contraction.
	The stated characterization is proved by an easy calculation.
	
	\refer{enum1:Lip_Inv_Kugel_hard_work}
	Since $\clBall{0}{r}$ is complete, by the Banach Fixed Point Theorem $H(\cdot, y)$ has a fixed point $g(y)$.
	We can use \cite[Theorem (C.7), p.~241-242]{MR586942} to see that $g$ is continuous, and moreover,
	we see that $g$ is Lipschitz with a constant not greater then $\frac{\Opnorm{T^{-1}}}{1 - L \Opnorm{T^{-1}}}$
	since for $y, z \in \clBall{0}{r'}$ and $x \in \clBall{0}{r}$,
	\[
		\norm{H(x, y) - H(x, z)} = \norm{T^{-1}(y - z)} \leq \Opnorm{T^{-1}} \norm{y - z}.
	\]
	(Notice that the Lipschitz constant of the fixed point map is implicitly calculated in the proof of
	\cite[Theorem (C.7)]{MR586942}).
	Furthermore, we calculate for $y \in \clBall{0}{r'}$ that
	\[
		g(y) = H(g(y), y) = T^{-1}(y - \eta(g(y))),
	\]
	and hence $y = (T + \eta)(g(y))$. This shows that $\rest{(T + \eta)}{\clBall{0}{r}}^{-1} = g$
	since we proved that $T + \eta$ is injective.
	To prove the last assertion, let $y \in \Ball{0}{r'}$. Then
	\[
		\norm{g(y)} = \norm{g(y) - g(0)} \leq \frac{\Opnorm{T^{-1}}}{1 - L \Opnorm{T^{-1}}} r' < r.
	\]
	Hence $\Ball{0}{r'} \sub (T + \eta)(\Ball{0}{r})$.
	
	\refer{enum1:faserweise_Invertierung_T+Lip_FCk}
	The map \eqref{map:faserweise_Invertierung_T+Lip} is defined by \refer{enum1:Lip_Inv_Kugel_hard_work},
	and we see with \refer{enum1:LipStoerung_Kontraktionen} that it arises as the restriction of the fixed point map for
	\[
		\widetilde{H} : \clBall{0}{r} \times \clBall{0}{r'} \times \UF \to \clBall{0}{r}
		: (x, y, p) \mapsto  T^{-1} (y - \Xi(p, x)).
	\]
	Hence we derive the assertion from \cite[Theorem (C.7)]{MR586942}.
\end{proof}
We use this lemma to prove two theorems on inverse functions that are better suited for citation.
\begin{satz}[Parameterized Lipschitz inverse function theorem]\label{satz:parameterized_Lipschitz_inverse-function_theorem}
	Let $\SX$ be a Banach space, $\SY$ a normed space, $T \in \Lin{\SX}{\SX}$ invertible,
	$\UF \sub \SX$ and $\VF \sub \SY$ open nonempty sets, $k \in \cl{\N}$
	and $\Xi \in \FC{ \VF \times \UF }{\SX}{k}$ such that for each $p \in \VF$,
	the map $\Xi_p \ndef \Xi(p, \cdot) : \UF \to \SX$ is $L$-Lipschitz, where $L \Opnorm{T^{-1}} < 1$.
	Then for each $p \in \VF$, $T + \Xi_p$ is a homeomorphism on its image, which is an open subset of $\SX$.
	More precisely, for each $x \in \UF$ and $r > 0$ such that $\Ball{x}{r} \sub \UF$,
	we have that $\Ball{(T + \Xi_p)(x)}{r'} \sub (T + \Xi_p)(\Ball{x}{r})$, where $r' \ndef r(\frac{1 - L \Opnorm{T^{-1}}}{\Opnorm{T^{-1}}})$.
	Further, $\rest{(T + \Xi_p)^{-1}}{\Ball{(T + \Xi_p)(x)}{r'}}$ is Lipschitz
	with constant $\frac{\Opnorm{T^{-1}}}{1 - L \Opnorm{T^{-1}}}$, and the map
	\[
		\bigcup_{p \in \VF} \sset{p} \times \Ball{(T + \Xi_p)(x)}{r'} \to \Ball{x}{r}
		: (p, y) \mapsto (T + \Xi_p)^{-1}(y)
	\]
	is $\FC{}{}{k}$.
\end{satz}
\begin{proof}
	By \refer{lem:Lipschitz_inv_auf_Kugeln}, for each $p \in \VF$ the map $T + \Xi_p$ is injective.
	To prove the other assertions,
	let $x \in \UF$ and $r > 0$ such that $\Ball{x}{r} \sub \UF$.
	Since each $\Xi_p$ is uniformly continuous, it can be extended to $\clBall{0}{r}$;
	and the extension also is $L$-Lipschitz.
	Then we can apply \refer{lem:Lipschitz_inv_auf_Kugeln} to $T$ and the map
	\[
		\Xi_p^x : \clBall{0}{r}\to \SX
		: y \mapsto \Xi_p(x + y) - \Xi_p(x) = (\tau_{- \Xi_p(x)} \circ \eta \circ \tau_x )(y)  .
	\]
	We derive that $T + \Xi_p^x$ is a homeomorphism, $\Ball{0}{r'} \sub (T + \Xi_p^x)(\Ball{0}{r})$,
	and its inverse map is Lipschitz with constant $\frac{\Opnorm{T^{-1}}}{1 - L \Opnorm{T^{-1}}}$.
	Thus using the identity
	\[
		(T + \Xi_p^x)^{-1} = \tau_{-x} \circ (T + \Xi_p)^{-1} \circ \tau_{(\Xi_p + T)(x)},
	\]
	we derive all assertions.
\end{proof}

\begin{cor}[Lipschitz inverse function theorem]\label{cor:Lipschitz_inverse-function_theorem}
	Let $\SX$ be a Banach space, $T \in \Lin{\SX}{\SX}$ invertible,
	$\UF \sub \SX$ an open nonempty set, and $\eta : \UF \to \SX$
	Lipschitz with constant $L$ such that $L \Opnorm{T^{-1}} < 1$.
	Then $T + \eta$ is a homeomorphism on its image, which is an open subset of $\SX$.
	If $\eta$ is $\FC{}{}{k}$, so is $(T + \eta)^{-1}$.
	More precisely, for each $x \in \UF$ and $r > 0$ such that $\Ball{x}{r} \sub \UF$,
	we have that $\Ball{(T + \eta)(x)}{r'} \sub (T + \eta)(\Ball{x}{r})$, where $r' \ndef r(\frac{1 - L \Opnorm{T^{-1}}}{\Opnorm{T^{-1}}})$.
	Further, $\rest{(T + \eta)^{-1}}{\Ball{(T + \eta)(x)}{r'}}$ is Lipschitz
	with constant $\frac{\Opnorm{T^{-1}}}{1 - L \Opnorm{T^{-1}}}$.
\end{cor}
\begin{proof}
	The assertions follow immediately from \refer{satz:parameterized_Lipschitz_inverse-function_theorem}.
\end{proof}

\section{Relation between the differential calculi}
We show that the two calculi presented are closely related.
First we prove that each $\FC{}{}{k}$-map is a $\ConDiff{}{}{k}$-map
and that the higher differentials are in a close relation.
\begin{lem}\label{lem:Frechet_und_Helge}
	Let $k \in \N^\ast$ and $\gamma \in \FC{\UF}{\SY}{k}$.
	Then $\gamma$ is a $\ConDiff{}{}{k}$-map
	(in the sense of \refer{sec:Helge_diffbar}),
	and for each $x\in\UF$ we have
	\[
		D^{(k)}\gamma(x) = \dA[k]{\gamma}{x}{\cdot}.
	\]
\end{lem}
\begin{proof}
	We prove this by induction.
	\\
	$k=1$:
	It follows directly from \refer{def:Frechet_Diffbarkeit} that
	$\gamma$ is a $\ConDiff{}{}{1}$ map and that the identity
	\[
		D^{(1)}\gamma(x) = D\gamma(x)
		= \dA{\gamma}{x}{\cdot} = \dA[1]{\gamma}{x}{\cdot}
	\]
	holds.
	
	$k\to k+1$:
	Let $x\in \UF$ and $h_1,\ldots,h_{k+1}\in\SX$.
	We know from \refer{lem:Ableitungen_von_D} that
	\begin{align*}
		&(D^{(k+1)}\gamma) (x)(h_1,\dotsc,h_{k+1})\\
		=& (\EmbA{k}{1}\circ (D^{(k)} D\gamma))(x)(h_1,\dotsc,h_{k+1})\\
		=& (D^{(k)} D\gamma(x) (h_2,\dotsc,h_{k+1}))\eval h_{1}.\\
		\intertext{
		The inductive hypothesis gives
		}
		=& (\dA[k]{D\gamma}{x}{h_2,\dotsc,h_{k+1}})\eval h_{1}\\
		=& \left(\lim_{t\to 0}
			\frac{\dA[k-1]{(D\gamma)}{x + t h_{k+1}}{h_2,\dotsc, h_{k}}
			- \dA[k-1]{(D\gamma)}{x}{h_2,\dotsc, h_{k}} }{t}\right)
			\eval h_1 .\\
		\intertext{
		Another application of the inductive hypothesis, together with
		the continuity of the evaluation of linear maps
		(\refer{lem:Auswertung_linearer_Abb_ist_stetig})
		and \refer{lem:Ableitungen_von_D},
		gives		
		}
		=& \lim_{t\to 0}
			\frac{D^{(k-1)}(D\gamma)(x + t h_{k+1})(h_2,\dotsc, h_{k})
			\eval h_1
			- D^{(k-1)}(D\gamma)(x)(h_2,\dotsc, h_{k})\eval h_1 }{t}\\
		=& \lim_{t\to 0}
			\frac{ (\EmbA{k-1}{1} \circ D^{(k-1)}(D\gamma))
			(x + t h_{k+1})(h_1, \dotsc, h_{k})
			- (\EmbA{k-1}{1} \circ D^{(k-1)}(D\gamma))
			(x)(h_1, \dotsc, h_{k})}{t}\\
		=& \lim_{t\to 0}
			\frac{ D^{(k)}\gamma(x + t h_{k+1})(h_1,\dotsc, h_{k})
			- D^{(k)}\gamma(x)(h_1, \dotsc, h_{k})}{t}.\\
		\intertext{
		Another application of the inductive hypothesis finally gives
		}
		=& \lim_{t\to 0}
			\frac{ \dA[k]{\gamma}{x + t h_{k+1}}{h_1, \dotsc, h_{k}}
			- \dA[k]{\gamma}{x}{h_1, \dotsc, h_{k}}}{t} .
	\end{align*}
	Hence $\dA[k + 1]{\gamma}{}{}$ exists and satisfies the identity
	\[
		\dA[k + 1]{\gamma}{x}{h_1, \dotsc, h_{k+1}}
		= D^{(k+1)}\gamma (x)(h_1,\dotsc,h_{k+1}) .
	\]
	Since $D^{(k+1)}\gamma$ and the evalution of multilinear maps
	are continuous (see \refer{lem:Auswertung_linearer_Abb_ist_stetig}),
	$\dA[k + 1]{\gamma}{}{}$ is so.
	In \refer{prop:hohe_Ableitungen_d} we stated that this
	(and the inductive hypothesis) assure that
	$\gamma$ is a $\ConDiff{}{}{k + 1}$ map.
\end{proof}
The preceding can be used to give a characterization of Fr\'echet differentiable maps.
\begin{prop}\label{prop:Char_Frechet_Diff}
	Let $\gamma : \UF \to \SY$ be a continuous map.
	Then $\gamma \in \FC{\UF}{\SY}{k}$
	iff
	$\gamma$ is a $\ConDiff{}{}{k}$-map and the map
	\begin{equation*}\label{Hohe_Abl_stetig}
		\UF \to \Lin[\ell]{\SX}{\SY} : x \mapsto \dA[\ell]{\gamma}{x}{\cdot}
		\tag{\ensuremath{*_k}}
	\end{equation*}
	is continuous for each $\ell \in \N$ with $\ell \leq k$.
\end{prop}
\begin{proof}
	For $\gamma \in \FC{\UF}{\SY}{k}$ we stated in \refer{lem:Frechet_und_Helge}
	that $\gamma \in \ConDiff{\UF}{\SY}{k}$ and
	\[
		\dA[\ell]{\gamma}{x}{\cdot}
		= \FAbl[\ell]{\gamma} (x)
	\]
	for each $x \in \UF$ and $\ell \in \N$ with $\ell \leq k$.
	Since $\FAbl[\ell]{\gamma}$ is continuous by its
	definition (\ref{def:Frechet_hohe_Abl}),
	\eqref{Hohe_Abl_stetig} is satisfied.
	
	We have to prove the other direction. This is done by induction on $k$:
	\\
	$k = 1$:
	This follows directly from the definition of $\FC{\UF}{\SY}{1}$.
	\\
	$k \to k + 1$:
	We have to show that $\gamma \in \FC{\UF}{\SY}{k + 1}$,
	and this is clearly the case if $\FAbl{\gamma} \in \FC{\UF}{\Lin{\SX}{\SY}}{k}$.
	By the inductive hypothesis this is the case if
	$
	\FAbl{\gamma} \in \ConDiff{\UF}{\Lin{\SX}{\SY}}{k}
	$
	and it satisfies \eqref{Hohe_Abl_stetig}.
	Since $\gamma \in \FC{\UF}{\SY}{k}$ by the inductive hypothesis
	and hence
	$
	\FAbl{\gamma} \in \FC{\UF}{\Lin{\SX}{\SY}}{k - 1}
	$,
	we just have to show that $\FAbl{\gamma}$ is $\ConDiff{}{}{k}$ and
	\[
		\UF \to \Lin[k]{\SX}{ \Lin{\SX}{\SY} } : x \mapsto \dA[k]{(\FAbl{\gamma})}{x}{\cdot}
	\]
	is continuous. To this end, let $x \in \UF$,
	$h, v_1,\dotsc, v_{k-1}, v_k \in \SX$ and $t \in \K$ such that
	the line segment
	$\{x + s t v_k : s \in [0,1]\} \subseteq \UF$.
	We calculate using \refer{lem:Ableitungen_von_D}, the mean value theorem
	and two applications of \refer{lem:Frechet_und_Helge}:
	\begin{align*}
		&\left(\frac{\dA[k-1]{(\FAbl{\gamma})}{x + t v_k}{v_1,\dotsc, v_{k-1}}
			- \dA[k-1]{(\FAbl{\gamma})}{x}{v_1,\dotsc, v_{k-1}} }{t}\right)
			\eval h
		\\
		=& \frac{\dA[k]{\gamma}{x + t v_k}{h, v_1,\dotsc, v_{k-1}}
			- \dA[k]{\gamma}{x}{h, v_1,\dotsc, v_{k-1}}}{t}
		\\
		=& \Rint{0}{1}{
			\dA[k+1]{\gamma}{x + s t v_k}{h, v_1,\dotsc, v_{k-1}, v_k}}
		{s}.
	\end{align*}
	Since $x \mapsto \dA[k + 1]{\gamma}{x}{\cdot}$ is continuous by hypothesis,
	the left hand side of this identity converges for $t \to 0$
	with respect to the topology of uniform convergence on bounded sets
	to the linear map
	\begin{equation*}
		h \mapsto \dA[k+1]{\gamma}{x}{h, v_1,\dotsc, v_{k-1}, v_k} .
	\end{equation*}
	Hence $\FAbl{\gamma}$ is $\ConDiff{}{}{k}$ with
	\[
		\dA[k]{(\FAbl{\gamma})}{x}{v_1,\dotsc, v_{k-1}, v_k} 
		= \EmbA{k}{1}^{-1}(\dA[k+1]{\gamma}{x}{\cdot}) (v_1,\dotsc, v_{k-1}, v_k), 
	\]
	and since $x \mapsto \dA[k+1]{\gamma}{x}{\cdot}$ and $\EmbA{k}{1}^{-1}$
	are continuous (by hypothesis resp. \refer{lem:Identifizierung_seltsamer_Raeume}), 
	$x \mapsto \dA[k]{(\FAbl{\gamma})}{x}{\cdot}$ is so, too.
\end{proof}
We show that a $\ConDiff{}{}{k + 1}$ map is $\FC{\UF}{\SY}{k}$.
\begin{lem}\label{lem:Helge_diffbar-Frechet_diffbar}
	Let $f: \UF \to \SY$ be a $\ConDiff{}{}{k + 1}$ map.
	Then $f \in \FC{\UF}{\SY}{k}$.
\end{lem}
\begin{proof}
	We stated in \refer{prop:Char_Frechet_Diff} that $f$ is in
	$\FC{\UF}{\SY}{k}$ iff for each
	$\ell \in \N$ with $\ell \leq k$ the map
	\[
		\UF \to \Lin[\ell]{\SX}{\SY}
		: x \mapsto \dA[\ell]{f}{x}{\cdot}
	\] 
	is continuous; but this is a direct consequence of
	\refer{lem:stetig_diffbar_impliziert_lokal_Lipschitz}
	since it implies that \refer{est:Lipschitz_Stetigkeit_der_l-ten_Ableitung_bzgl_Halbnormen} holds.
\end{proof}
\paragraph{Differential calculus on finite-dimensional spaces}
We show that the three definitions of differentiability for maps that are defined
on a finite-dimensional space
(Fr\'echet-differentiability, Kellers $C_c^k$ theory and continuous partial differentiability)
are equivalent.

\begin{defi}
	Let $n, k \in \N^*$ and $\alpha \in \N_0^n$ a multiindex with
	$\abs{\alpha} = k$. We set
	\[
		I_\alpha \ndef \{(i_1,\dotsc,i_k) \in \{1,\dotsc,n\}^k : (\forall \ell \in \{1,\dotsc,n\})\:\alpha_\ell = \abs{\{j: i_j = \ell\}} \}
	\]
	and use this set to define the continuous $k$-linear map
	\[
		S_\alpha : (\K^n)^k \to \K :(h_1,\dotsc,h_k) \mapsto \sum_{(i_1,\dotsc,i_k)\in I_\alpha} h_{1,i_1}\dotsm h_{k,i_k} ,
	\]
	where $h_j = (h_{j,1}, \dotsc, h_{j,n})$ for $j = 1,\dotsc, k$.
\end{defi}

\begin{prop}\label{prop:Endlichdimensionale_Diffbarkeit}
	Let $\UF \subseteq \K^n$ be open and nonempty and $\gamma : \UF \to \SY$ a map.
	Then the following conditions are equivalent:
	\begin{enumerate}
		\item\label{enum1:gamma_in_FCk}
		$\gamma \in \FC{\UF}{\SY}{k}$
		
		\item\label{enum1:gamma_in_Ck}
		$\gamma \in \ConDiff{\UF}{\SY}{k}$
		
		\item\label{enum1:gamma_partiell_diffbar}
		$\gamma$ is $k$-times continuously partially differentiable.
	\end{enumerate}
	If one of these conditions is satisfied, then
	\begin{equation}\label{id:D^k_und_partielle_Abl}
		D^{(k)}\gamma(x)(h_1,\dotsc,h_k)
		= \sum_{\substack{\alpha \in \N^n_0\\\abs{\alpha}=k}}
		S_\alpha(h_1,\dotsc,h_k) \cdot \partial^\alpha \gamma(x)
	\end{equation}
	for all $x \in \UF$ and $h_{1},\dotsc,h_{k} \in \K^n$.
\end{prop}
\begin{proof}
	The assertion
	$\refer{enum1:gamma_in_FCk} \implies \refer{enum1:gamma_in_Ck}$
	is a consequence of \refer{lem:Frechet_und_Helge}; and since
	\[
		\frac{\partial^k\gamma}{\partial x_{i_1}\dotsb \partial x_{i_k}}(x)
		= \dA[k]{\gamma}{x}{e_{i_k},\dotsc,e_{i_1}}
	\]
	and $\dA[k]{\gamma}{}{}$ is continuous (\refer{prop:hohe_Ableitungen_d}),
	the implication
	$\refer{enum1:gamma_in_Ck} \implies \refer{enum1:gamma_partiell_diffbar}$
	also holds.
	
	It remains to show that 
	$\refer{enum1:gamma_partiell_diffbar} \implies \refer{enum1:gamma_in_FCk}$.
	It is well known from calculus that
	$
		D_h \gamma = \sum_{i=1}^n h_i \frac{\partial \gamma}{\partial x_i}.
	$
	Hence $\dA[\ell]{\gamma}{x}{h_1,\dotsc,h_\ell}$ exists and is given by
	\begin{align*}
		\dA[\ell]{\gamma}{x}{h_1,\dotsc,h_\ell}
		&= \sum_{i_1 = 1,\dotsc,i_\ell = 1}^n h_{1,i_1}\dotsm h_{\ell,i_\ell} \cdot \frac{\partial^k\gamma}{\partial x_{i_1}\dotsb \partial x_{i_\ell}}
		\\
		&= \sum_{\substack{\alpha \in \N^n_0\\\abs{\alpha}=\ell}}
		\Bigl(\sum_{(i_1,\dotsc,i_\ell)\in I_\alpha} h_{1,i_1}\dotsm h_{\ell,i_\ell}\Bigr) \cdot
		\partial^\alpha \gamma(x)
		\\
		&= \sum_{\substack{\alpha \in \N^n_0\\\abs{\alpha}=\ell}}
		S_\alpha(h_1,\dotsc,h_\ell) \cdot \partial^\alpha \gamma(x) .
	\end{align*}
	From this identity we derive the continuity of
	$x \mapsto \dA[\ell]{\gamma}{x}{\cdot}$, and can conclude
	using \refer{prop:Char_Frechet_Diff} that $\gamma \in \FC{\UF}{\SY}{k}$
	and \eqref{id:D^k_und_partielle_Abl} is satisfied.
\end{proof}
\section{Some facts concerning ordinary differential equations}
\label{sec:Fakten_ueber_DGLn}
We state some facts about the global solvabilty of initial value problems
and the dependence of solution on parameters.
\subsection{Maximal solutions of ODEs}
In the following, we let $J \subseteq \R$ be a nondegenerate interval and
$U$ an open subset of the Banach space $X$.
For a continuous function $f : J \times U \to X$, $x_{0} \in U$ and $t_{0} \in J$
we consider the initial value problem
\begin{align}\label{ivp:AWP}
	\begin{split}
		\gamma'(t) &= f(t, \gamma(t))\\
		\gamma(t_{0}) &= x_{0}.
	\end{split}
\end{align}
We state the famous theorem of Picard and Lindelöf:
\begin{satz}
	Let $f$ satisfy a local Lipschitz condition with respect to the second
	argument, that is, for each $(t_0, x_0) \in J \times U$ there exist
	a neighborhood $W$ of $(t_0, x_0)$ in $J\times U$ and an $K \in \R$ such that for all $(t, x), (t,\tilde{x}) \in W$
	\[
		\norm{f(t, x) -f(t,\tilde{x})} \leq K \norm{x - \tilde{x}}.
	\]
	Then, for each $(t_0, x_0) \in J \times U$ there exists a neighborhood $I$ of $t_0$ in $J$
	such that the initial value
	problem \eqref{ivp:AWP} corresponding to $t_0$ and $x_0$ has a unique solution
	that is defined on $I$.
\end{satz}
It is well-known that the local theorem of Picard and Lindelöf can be used to ensure that there exists
a maximal solution.
\begin{prop}
\label{prop:Existenz_maximale_Lsg_ODEs}
	Let $f$ satisfy a local Lipschitz condition with respect to the second
	argument and let $(t_0, x_0) \in J \times U$. Then there exists an interval
	$I \subseteq J$ and a function $\phi : I \to U$ that is a maximal
	solution to \eqref{ivp:AWP}; that is, if $\gamma : D(\gamma) \to U$
	is a solution to \eqref{ivp:AWP} defined an a connected set,
	$D(\gamma) \subseteq I$ and $\gamma = \rest{\phi}{D(\gamma)}$.
\end{prop}

\subsubsection{A criterion on global solvabilty}
\paragraph{Linearly bounded vector fields}
One class of ODEs that can be globally solved is that of linear vector fields.
This solvabilty property can be generalized to linearly bounded vector fields.
\begin{defi}
	We call $f$ \emph{linearly bounded} if there exist continuous functions
	$a,b : J \to \R$ such that
	\[
		\norm{f(t,x)} \leq a(t) \norm{x} + b(t)
	\]
	for all $(t,x) \in J \times U$.
\end{defi}
To prove that this condition on $f$ ensures globally defined solutions,
we first need to prove some lemmas.

\begin{lem}
\label{lem:linear_gebundenes_VF_impliziert_Beschraenkheit_von_Loesung_und_VF_auf_Loesung}
	Let $f$ be a linearly bounded map that satisfies a local Lipschitz condition
	with respect to the second argument.
	Let $\phi : I \to U$ be an integral curve of $f$.
	Then the following assertions hold:
	\begin{enumerate}
		\item\label{enum1:inear_gebundenes_VF_impliziert_Beschraenkheit-a}
		If $\phi$ is bounded, $\cl{I} \subseteq J$ and $\cl{I}$ is compact,
		then $f$ is bounded on the graph of $\phi$.

		\item\label{enum1:inear_gebundenes_VF_impliziert_Beschraenkheit-b}
		If $\beta \ndef \sup I \neq \sup J$,
		then $\phi$ is bounded on $[t_0, \beta[$ for each $t_0 \in J$.
		The analogous result for $\inf I$ also holds.
	\end{enumerate}
\end{lem}
\begin{proof}
	\refer{enum1:inear_gebundenes_VF_impliziert_Beschraenkheit-a}
	Let $t \in I$. Then
	\[
		\norm{f(t,\phi(t))} \leq a(t) \norm{\phi(t)} + b(t)
	\]
	since $f$ is linearly bounded. Because $a$ and $b$ are continuous and defined
	on $\cl{I}$, they are clearly bounded on $I$.

	\refer{enum1:inear_gebundenes_VF_impliziert_Beschraenkheit-b}
	For each $t \in [t_0, \beta[$ we have
	\[
		\phi(t) = \phi(t_0) + \Rint{t_0}{t}{ f(s,\phi(s)) }{s},
	\]
	and from this we deduce using that $f$ is linearly bounded:
	\begin{align*}
		\norm{\phi(t)}
		&\leq \norm{\phi(t_0)} + \left\norm{\Rint{t_0}{t}{ f(s,\phi(s)) }{s}\right}
		\\
		&\leq \norm{\phi(t_0)} + \left\abs{\Rint{t_0}{t}{ a(s) \norm{\phi(s)} + b(s)}{s}\right}
		\\
		&\leq \norm{\rest{a}{[t_0,\beta]} }_{\infty} \left\abs{\Rint{t_0}{t}{ \norm{\phi(s)} }{s}\right}
		 + \norm{\phi(t_0)} + \norm{b}_{\infty, [t_0,\beta]} \abs{\beta - t_0}.
	\end{align*}
	The assertion is proved with an application of Groenwall's lemma.
\end{proof}

\begin{lem}
\label{lem:VF-global_Lipschitz_impliziert_linear_beschrankt}
	Assume that $f$ satisfies a global Lipschitz condition
	with respect to the second argument. Then $f$ is linearly bounded.
\end{lem}
\begin{proof}
	Let $(t,x) \in J \times U$ and $x_0 \in U$. Then
	\begin{multline*}
		\norm{f(t,x)}
		\leq \norm{f(t,x) - f(t,x_0)} + \norm{f(t,x_0)}
		\\
		\leq L \norm{x - x_0} + \norm{f(t,x_0)}
		\leq L \norm{x} + L \norm{x_0} + \norm{f(t,x_0)}.
	\end{multline*}
	Defining $a(t) \ndef L$ and $b(t) \ndef L \norm{x_0} + \norm{f(t,x_0)}$
	gives the assertion.
\end{proof}

\paragraph{The criterion}

We give a sufficent condition on when an integral curve is uniformly continuous.
This can be used to extend solutions to larger domains of definition.
\begin{lem}
\label{lem:Kriterium_fuer_gleichmaessige_Stetigkeit_von_Integralkurven_fuer_ODEs}
	Let $f$ satisfy a local Lipschitz condition with respect to the second
	argument
	and let $\phi : I \to U$ be an integral curve of $f$ such that $f$ is bounded
	on the graph of $\phi$.
	Then $\phi$ is Lipschitz continuous and hence uniformly continuous.
\end{lem}
\begin{proof}
	Let $t_1, t_2 \in I$. Then
	\[
		\norm{\phi(t_2) - \phi(t_1)}
		= \left\norm{\Rint{t_1}{t_2}{\phi'(s)}{s} \right}
		= \left\norm{\Rint{t_1}{t_2}{ f(s,\phi(s)) }{s}\right}
		\leq K \abs{t_2 - t_1},
	\]
	where $K \ndef \sup_{s \in I} \norm{f(s,\phi(s))} < \infty$.
\end{proof}

\begin{satz}
\label{satz:globale_Loesbarkeit_von_AWP_kompaktes_Bild}
\label{satz:globale_Loesbarkeit_von_AWP_lineare_Gebundenheit}
	Assume that $f$ satisfies a local Lipschitz condition
	with respect to the second argument.
	Let $\phi : I \to U$ be a maximal integral curve of $f$. Assume further that
	\begin{enumerate}
		\item
		The image of $\phi$ is contained in a compact subset of $U$ or

		\item
		$f$ is linearly bounded.
	\end{enumerate}
	Then $\phi$ is a global solution, that is $I = J$.
\end{satz}
\begin{proof}
	We prove this by contradiction.
	To this end, we may assume w.l.o.g. that $\beta \ndef \sup I \neq \sup J$.
	We choose $t_0 \in I$.
	In both cases, $f$ is bounded on the graph of $\rest{\phi}{[t_0,\beta[}$:
	If the image of $\phi$ is contained in a compact set,
	we easily see that the graph of $\rest{\phi}{[t_0,\beta[}$
	is contained in a compact subset.
	If $f$ is linearly bounded, we use
	\refer{lem:linear_gebundenes_VF_impliziert_Beschraenkheit_von_Loesung_und_VF_auf_Loesung}.
	
	We apply \refer{lem:Kriterium_fuer_gleichmaessige_Stetigkeit_von_Integralkurven_fuer_ODEs}
	to see that $\rest{\phi}{[t_0,\beta[}$ is uniformly continuous,
	and thus has a continuous extension $\widetilde{\phi}$ to $[t_0,\beta]$.
	We easily calculate that $\widetilde{\phi}$ is a solution
	to \eqref{ivp:AWP} using the integral represention of an ODE.
	Since $\widetilde{\phi}$ extends $\phi$,
	we get a contradiction to the maximality of $\phi$.
\end{proof}

\subsection{Flows and dependence on parameters and initial values}
For the purpose of full generality, we need a definition.
\begin{defi}
	Let $\SX$ be a locally convex space.
	We call $P \sub \SX$ a \emph{locally convex subset with dense interior}
	if for each $x \in P$, there exists a convex neighborhood $\UF \sub P$ of $x$
	and if $P \sub \cl{\interior{P}}$.
\end{defi}
In the following, we let $J \subseteq \R$ be a nondegenerate interval,
$U$ an open subset of the Banach space $X$,
$P$ be a locally convex subset with dense interior of a locally convex space
and $k \in \cl{\N}$ with $k \geq 1$.
Further, let $f$ be in $\ConDiff{J \times U \times P }{ X}{k}$.
We consider the initial value problem
\begin{align}\label{ivp:Allgemeines_AWP}
	\begin{split}
		\gamma'(t) &= f(t, \gamma(t), p)\\
		\gamma(t_{0}) &= x_{0}
	\end{split}
\end{align}
for $t_{0} \in J$, $x_{0} \in U$ and $p \in P$.
\begin{defi}
	Let $\Omega \subseteq J \times J \times U \times P$.
	We call a map
	\[
		\phi : \Omega \to U
	\]
	a \emph{flow} for $f$
	if for all $t_0 \in J$, $x_0 \in U$ and $p \in P$ the set
	\[
		\Omega_{t_0, x_0, p} \ndef \{t \in J: (t_0, t, x_0, p) \in \Omega \}
	\]
	is connected and the partial map
	\[
		\phi(t_0, \cdot, x_0, p) : \Omega_{t_0, x_0, p} \to U
	\]
	is a solution to \eqref{ivp:Allgemeines_AWP} corresponding
	to the initial values $t_0$, $x_0$ and $p$.
	
	A flow is called \emph{maximal} if each other flow is a restriction of it.
\end{defi}
\begin{bem}
	In \cite[Theorem 10.3]{glockner-0612673}
	it was stated that for each $t_{0} \in J$, $x_{0} \in U$
	and $p_{0} \in P$
	there exist neighborhoods $J_{0}$ of $t_{0}$, $U_{0}$ of $x_{0}$
	and $P_{0}$ of $p_{0}$ such that
	for every $s \in J_{0}$, $x \in U_{0}$ and $p \in P_{0}$
	the corresponding initial value problem \eqref{ivp:Allgemeines_AWP}
	has a unique solution
	$\Gamma_{s, x, p} : J_{0} \to U$
	and the map
	\[
		\Gamma
		:
		J_{0} \times J_{0} \times U_{0} \times P_{0}
		\to U
		:
		(s, t, x, p) \mapsto \Gamma_{s, x, p}(t)
	\]
	is $\ConDiff{}{}{k}$. Therefore $\ConDiff{}{}{k}$-flows exist.
\end{bem}
The following lemma shows that two related flows can be glued together:
\begin{lem}\label{lem:Erweitern_von_Fluessen}
	Let $I \subseteq J$ be a connected set with nonempty interior and
	$\gamma : I \to U$ a solution to \eqref{ivp:Allgemeines_AWP}
	corresponding to $t_{\gamma} \in J$, $x_{\gamma} \in U$ and $p_{\gamma} \in P$.
	Further let
	\[
		\phi_{0} : J_0 \times I_0 \times U_0 \times P_0 \to U
		\text{ and }
		\phi_{1} : I_1 \times I_1 \times U_1 \times P_1 \to U
	\]
	be $\ConDiff{}{}{k}$-flows for $f$ such that
	$U_1$ is open in $X$ and
	\[
		I = I_0 \cup I_1,\,
		I_0 \cap I_1 \neq \emptyset,\,
		p_{\gamma} \in P_0 \cap P_1,\,
		(t_{\gamma}, x_{\gamma}) \in J_0 \times U_0
		\text{ and }
		\gamma(I_1) \subseteq U_1.
	\]
	Then there exist neighborhoods $J_{\gamma}$ of $t_{\gamma}$,
	$U_{\gamma}$ of $x_{\gamma}$, $P_{\gamma}$ of $p_{\gamma}$
	and a $\ConDiff{}{}{k}$-flow
	\[
		\phi : J_{\gamma} \times I \times U_{\gamma} \times P_{\gamma} \to U
	\]
	for $f$.
\end{lem}
\begin{proof}
	We choose $t_1 \in I_0 \cap I_1$. Since $\phi_{0}$ is continuous in
	$(t_{\gamma}, t_1, x_{\gamma}, p_{\gamma})$ and
	\[
		\phi_{0}(t_{\gamma}, t_1, x_{\gamma}, p_{\gamma}) = \gamma(t_1) \in U_1,
	\]
	there exist neighborhoods
	$J_{\gamma}$ of $t_{\gamma}$ in $J_0$, $U_{\gamma}$ of $x_{\gamma}$ in $U_0$ and
	$P_{\gamma} \subseteq P_0 \cap P_1$ of $p_{\gamma}$ such that
	\[
		\phi_{0}(J_{\gamma} \times \{t_1\} \times U_{\gamma} \times P_{\gamma})
			\subseteq U_1.
	\]
	Then the map
	\[
		\phi : J_{\gamma} \times I \times U_{\gamma} \times P_{\gamma} \to U
		: (t_0, x_0, p, t) \mapsto
		\begin{cases}
			\phi_{0}(t_0, t, x_0, p) & \text{if $t \in I_0$}\\
			\phi_{1}(t_1, t, \phi_{0}(t_0, t_1, x_0, p), p) & \text{if $t \in I_1$}
		\end{cases}
	\]
	is well defined since the curves $\phi_{0}(t_0, \cdot , x_0, p)$ and
	$\phi_{1}(t_1, \cdot, \phi_{0}(t_0, t_1, x_0, p), p)$
	are both solutions to the ODE \eqref{ivp:Allgemeines_AWP} that coincide in $t_1$
	and hence on $I_0 \cap I_1$.
	Since both $\phi_{0}$ and $\phi_{1}$ are $\ConDiff{}{}{k}$-flows for $f$, so is $\phi$.
\end{proof}

\begin{lem}\label{lem:Fluesse_bei_pertubierten_Parametern}
	Let $I \subseteq J$ be a connected set with nonempty interior, $t_{1} \in I$ and
	$\gamma : I \to U$ a solution to \eqref{ivp:Allgemeines_AWP}
	corresponding to $t_{\gamma} \in J$, $x_{\gamma} \in U$ and $p_{\gamma} \in P$.
	Then there exist neighborhoods $J_{\gamma}$ of $t_{\gamma}$,
	$U_{\gamma}$ of $x_{\gamma}$, $P_{\gamma}$ of $p_{\gamma}$,
	an interval $\widetilde{I} \subseteq I$ with $t_{\gamma}, t_{1} \in \widetilde{I}$
	such that $\widetilde{I}$ is a neighborhood of $t_{1}$ in $I$,
	and a $\ConDiff{}{}{k}$-flow
	\[
		\phi : J_{\gamma} \times \widetilde{I} \times U_{\gamma} \times P_{\gamma} \to U
	\]
	for $f$.
\end{lem}
\begin{proof}
	We use \cite[Theorem 10.3]{glockner-0612673} to see that for each $s \in I$
	there exist neighborhoods $J_{s}$ of $s$ in $J$, $U_{s}$ of $\gamma(s)$ in $U$,
	$P_{s}$ of $p_{0}$ in $P$ and a $\ConDiff{}{}{k}$-flow
	\[
		\phi_{s} : J_{s} \times J_{s} \times U_{s} \times P_{s}
		\to U
	\]
	for $f$; we may assume w.l.o.g.
	that $\gamma(J_{s}) \subseteq U_{s}$ since $\gamma$ is continuous and
	that $J_{s}$ is open in $I$.
	Since $I$ is connected and $\{J_{s}\}_{s \in I}$ is
	an open cover of $I$, there exist finitely many sets
	$J_{s_{1}}, \dotsc, J_{s_{n}}$
	such that $t_{\gamma} \in J_{s_{1}}$, $t_{1} \in J_{s_{n}}$ and
	$
		J_{s_{m}} \cap J_{s_{\ell}} \neq \emptyset
		\iff
		\abs{m - \ell} \leq 1.
	$
	Applying \refer{lem:Erweitern_von_Fluessen} to $\phi_{s_{1}}$ and $\phi_{s_{2}}$
	we find neighborhoods $I_{1}$ of $t_{\gamma}$, $V_{1}$ of $x_{\gamma}$, $P_{1}$ of $p_{\gamma}$
	and a $\ConDiff{}{}{k}$-flow
	\[
		\phi_{1} : I_{1} \times (J_{s_{1}} \cup J_{s_{2}}) \times V_{1} \times P_{1} \to U
	\]
	for $f$. Likewise, $\phi_1$ and $\phi_{s_3}$ lead to $\phi_2$, and iterating the argument,
	we find a $\ConDiff{}{}{k}$-flow
	\[
		\phi_{n-1} : I_{n-1} \times \bigcup_{k = 1}^{n} J_{s_{k}} \times V_{n-1} \times P_{n-1} \to U
	\]
	for $f$.
\end{proof}

Concerning maximal flows, we can state the following
\begin{satz}
	For each ODE \eqref{ivp:Allgemeines_AWP} there exists a maximal flow
	\[
		\phi : J \times J \times U \times P \supseteq \Omega \to U.
	\]
	$\Omega$ is an open subset of $J \times J \times U \times P$
	and $\phi$ is a $\ConDiff{}{}{k}$-map.
\end{satz}
\begin{proof}
	The existence of a maximal flow is a direct consequence of
	the existence of maximal solutions to ODEs without parameters,
	see \refer{prop:Existenz_maximale_Lsg_ODEs}.
	Now let $(t_{0}, t, x_{0}, p) \in \Omega$ and
	$\gamma : I \subseteq J \to U$ the maximal solution corresponding to $t_{0}$, $x_{0}$ and $p$.
	Then $t_{0}, t \in I$, and according to \refer{lem:Fluesse_bei_pertubierten_Parametern},
	there exists a $\ConDiff{}{}{k}$-flow
	\[
		\Gamma : J_{\gamma} \times \widetilde{I} \times U_{\gamma} \times P_{\gamma} \to U
	\]
	for $f$ that is defined on a neighborhood of $(t_{0}, t, x_{0}, p)$.
	Since $\phi$ is maximal,
	\[
		J_{\gamma} \times \widetilde{I} \times U_{\gamma} \times P_{\gamma} \subseteq \Omega
	\]
	and 
	\[
		\rest{\phi}{J_{\gamma} \times \widetilde{I} \times U_{\gamma} \times P_{\gamma}}
		= \Gamma .
	\]
	This gives the assertion.
\end{proof}

We examine the situation that an initial time is fixed and the initial
values depend on the parameters.
\begin{cor}\label{cor:Ck_Abhaengigkeit_LsgDGL_von_Parameter}
	Let $\alpha : P \to U$ be a $\ConDiff{}{}{k}$-map.
	Further, let $I \subseteq J$ be a nonempty interval and $t_0 \in I$ such that
	for every $p \in P$ there exists a solution
	\[
		\gamma_{p} : I \to U
	\]
	to the initial value problem \eqref{ivp:Allgemeines_AWP} corresponding to
	$p$, $t_{0}$ and the initial value $\alpha(p)$.
	Then the map
	\[
		\Gamma : I \times P \to U : (t, p) \mapsto \gamma_{p}(t)
	\]
	is $\ConDiff{}{}{k}$.
\end{cor}
\begin{proof}
	We consider a maximal flow $\phi : \Omega \to U$ for $f$.
	Since $\phi$ is maximal,
	\[
		\{t_{0}\} \times I \times \{(\alpha(p), p) : p \in P\}
		\subseteq \Omega,
	\]
	and for each $p \in P$
	\[
		\phi(t_{0}, \cdot, \alpha(p), p) = \gamma_{p}.
	\]
	Hence $\Gamma$ is the composition of $\phi$ and the $\ConDiff{}{}{k}$-map
	\[
		I \times P \to J \times I \times U \times P
		:
		(t, p) \mapsto (t_{0}, t, \alpha(p), p),
	\]
	and this gives the assertion.
\end{proof}
\chapter{Locally convex Lie groups}
The goal of this appendix mainly is to fix our conventions and notation
concerning manifolds and Lie groups modelled on locally convex spaces.
For further information see the articles
\cite{MR830252}, \cite{MR2261066} and \cite{MR2069671}.

\section{Locally convex manifolds}
Locally convex manifolds are essentially like finite-dimensional ones,
replacing the finite-dimensional modelling space by a locally convex space.
\begin{defi}[Locally convex manifolds]
	Let $M$ be a Hausdorff topological space, $k \in \cl{\N}$
	and $\SX$ a locally convex space. A \emph{$\ConDiff{}{}{k}$-atlas}
	for $M$ is a set $\mathcal{A}$ of homeomorphisms
	$\phi : U \to V$ from an open subset $U \subseteq M$ onto
	an open set $V \subseteq \SX$ whose domains cover $M$ and which are
	$\ConDiff{}{}{k}$-compatible in the sense that $\phi \circ \psi^{-1}$
	is $\ConDiff{}{}{k}$ for all $\phi, \psi \in \mathcal{A}$.
	A maximal $\ConDiff{}{}{k}$-atlas $\mathcal{A}$ on $M$ is called
	a \emph{differentiable structure} of class $\ConDiff{}{}{k}$.
	In this case, the pair $(M, \mathcal{A})$ is called (locally convex)
	$\ConDiff{}{}{k}$-manifold modelled on $\SX$.
\end{defi}

Direct products of locally convex $\ConDiff{}{}{k}$-manifolds are defined
as expected.

\begin{defi}[Tangent space and tangent bundle]
	Let $(M, \mathcal{A})$ be a $\ConDiff{}{}{k}$-manifold modelled on $\SX$,
	where $k \geq 1$.
	Given $x \in M$, let $\mathcal{A}_x$ be the set of all charts
	around $x$ (i.e. whose domain contains $x$).
	A \emph{tangent vector} of $M$ at $x$ is a family
	$y = (y_\phi)_{\phi \in \mathcal{A}_x}$ of vectors $y_\phi \in \SX$
	such that
	$y_\psi = \dA{(\psi \circ \phi^{-1}) }{\phi(x)}{y_\phi}$
	for all $\phi, \psi \in \mathcal{A}_x$.
	
	The \emph{tangent space} of $M$ at $x$ is the set $\glstext{Tangent_space_basepoint}$
	of all tangent vectors of $M$ at $x$. It has a unique structure of locally convex space
	such that the map
	$\rest{ \dA{\psi}{}{} }{ \Tang[x]{M} }
		: \Tang[x]{M} \to \SX : (y_\phi)_{\phi \in \mathcal{A}_x} \mapsto y_\psi$
	is an isomorphism of topological vector spaces
	for any $\psi \in \mathcal{A}_x$.
	
	The \emph{tangent bundle} $\glstext{Tangent_bundle}$
	of $M$ is the union of the (disjoint) tangent spaces
	$\Tang[x]{M}$ for all $x \in M$. It admits a unique structure as a
	$\ConDiff{}{}{k - 1}$-manifold modelled on $\SX \times \SX$ such that
	$\Tang{\phi} \ndef (\phi, \dA{\phi}{}{})$ is chart for each $\phi \in \mathcal{A}$.
	We let $\pi_M :\Tang{M} \to M$ be the map taking tangent vectors at $x$ to $x$
	for any $x \in M$.
\end{defi}

\begin{defi}
	A continuous map $f : M \to N$ between $\ConDiff{}{}{k}$-manifolds
	is called $\ConDiff{}{}{k}$ if the map $\psi \circ f \circ \phi^{-1}$
	is so for all charts $\psi$ of $N$ and $\phi$ of $M$.
	
	If $k \geq 1$, then we define the \emph{tangent map} of $f$ as the $\ConDiff{}{}{k - 1}$-map
	$\glstext{Tangent_map} : \Tang{M} \to \Tang{N}$
	determined by
	$\dA{\psi}{}{} \circ \Tang{f} \circ (\Tang{\phi})^{-1} = \dA{(\psi \circ f \circ \phi^{-1})}{}{}$
	for all charts $\psi$ of $N$ and $\phi$ of $M$.

	Given $x \in M$, we define
	$\glstext{Tangent_map_basepoint} \ndef \rest{\Tang{f}}[\Tang[f(x)]{N}]{\Tang[x]{M}} : \Tang[x]{M} \to \Tang[f(x)]{N}$.
\end{defi}

\begin{defi}
	Let $k > 0$, $M$, $N$ and $P$ be $\ConDiff{}{}{k}$-manifolds,
	and $f : M \times N \to P$ a $\ConDiff{}{}{k}$-map.
	We define
	\[
		\glsdisp{partial_Tangent_map}{\Tang[1]{ f}} : \Tang{M} \times N \to \Tang{P}
		: (v, n) \mapsto \Tang{\Gamma}(v, 0_n)
	\]
	and
	\[
		\Tang[2]{ f} : M \times \Tang{N} \to  \Tang{P}
		: (m, v) \mapsto \Tang{\Gamma}(0_{m}, v).
	 \]
\end{defi}

\begin{defi}[Submanifolds]
	Let $M$ be a $\ConDiff{}{}{k}$-manifold modelled on the locally convex space $\SX$
	and $\SY \sub \SX$ be a sequentially closed vector subspace.
	A \emph{submanifold of $M$ modelled on $\SY$} is a subset $N \sub M$
	such that for each $x \in N$, there exists a chart $\phi : \UF \to \VF$
	around $x$ such that $\phi(\UF \cap N) = \VF \cap \SY$.
	It is easy to see that a submanifold is also a $\ConDiff{}{}{k}$-manifold.
\end{defi}
The following lemma states that submanifolds are initial:
\begin{lem}\label{lem:SUbmanifols_sind_initial}
	Let $M$ be a $\ConDiff{}{}{k}$-manifold and $N$ a submanifold of $M$.
	Then the inclusion $\iota : N \to M$ is $\ConDiff{}{}{k}$.
	Moreover, a map $f : P \to N$ from a $\ConDiff{}{}{k}$-manifold is $\ConDiff{}{}{k}$
	iff the map $\iota \circ f : P \to M$ is so.
\end{lem}

\begin{defi}[Vector fields]
	A \emph{smooth vector field} on a manifold $M$ is a smooth map
	$\xi : M \to \Tang{M}$ such that $\pi_M \circ \xi = \id{M}$.
	We denote the set of vector fields on $M$ by $\glstext{SetOf_Vectorfields}$.
	
	A vector field $\xi$ is determined by its local representations
	$\xi_\phi \ndef \dA{\phi}{}{} \circ \xi \circ \phi^{-1} : V \to \SX$
	for each chart $\phi : U \to V$ of $M$.
	Given vector fields $\xi$ and $\eta$ on $M$, there is a unique vector field
	$[\xi, \eta]$ on $M$ such that
	$[\xi, \eta]_\phi = \dA{\eta_\phi}{}{} \circ (\id{V}, \xi_\phi) - \dA{\xi_\phi}{}{} \circ (\id{V}, \eta_\phi)$
	for all charts $\phi : U \to V$ of $M$.
\end{defi}

\begin{bem}[Analytic manifolds]
	The definition of analytic manifolds and analytic maps between them
	is literally the same as above, except that the term $\ConDiff{}{}{k}$-map
	has to be replaced by analytic map.
\end{bem}

\section{Lie groups}

\begin{defi}[Lie groups]
	A (locally convex) \emph{Lie group} is a group $G$ equipped
	with a smooth manifold structure turning the group operations into smooth maps.
	
	An analytic Lie group is a group $G$ equipped
	with an analytic manifold structure turning the group operations into analytic maps.
\end{defi}

\begin{lem}[Tangent group, action of group on TG]
	Let $G$ be a Lie group with the group multiplication $m$ and the inversion $i$.
	Then $\Tang{G}$ is a Lie group with the group multiplication
	\[
		\Tang{m} : \Tang{(G \times G)} \iso \Tang{G} \times \Tang{G} \to \Tang{G}
	\]
	and the inversion $\Tang{i}$.
	Identifying $G$ with the zero section of $\Tang{G}$,
	we obtain a smooth right action
	\[
		\Tang{G} \times G \to \Tang{G} : (v, g) \mapsto v.g \ndef \Tang{m}(v, 0_g)
	\]
	and a smooth left action
	\[
		G \times \Tang{G} \to \Tang{G} : (g, v) \mapsto g.v \ndef \Tang{m}(0_g, v).
	\]
\end{lem}

\begin{defi}[Left invariant vector fields]
	A vector field $V$ on a Lie group $G$ is called \emph{left invariant}
	if $g.V(h) = V(g h)$ for all $g, h \in G$.
	The set $\VecFieldsLinv{G}$ of left invariant vector fields is a Lie algebra
	under the bracket of vector fields defined above.
\end{defi}

\begin{defi}[Lie algebra functor]
	Let $G$ and $H$ be Lie groups.
	Using the isomorphism
	$\VecFieldsLinv{G} \to \Tang[\one]{G} : V \mapsto V(\one)$
	we transport the Lie algebra structure on $\VecFieldsLinv{G}$
	to $\LieAlg{G} \ndef \Tang[\one]{G}$.
	If $\phi : G \to H$ is a smooth homomorphism, then the map
	$\glsdisp{LieAlgebra_functor}{\LieAlg{\phi}} : \LieAlg{G} \to \LieAlg{H}$
	defined as $\rest{ \Tang{\phi} }{ \LieAlg{G} }$ is a Lie algebra homomorphism.
\end{defi}

\subsection{Generation of Lie groups}
We need the following result concerning the construction of Lie groups from local data
(compare \cite[Chapter \RN{3}, \S 1.9, Proposition 18]{MR979493}
for the case of Banach Lie groups; the general proof follows the same pattern).

\begin{lem}[Local description of Lie groups]
\label{lem:Erzeugung_von_Liegruppen_aus_lokalen}
	Let $G$ be a group, $U \subseteq G$ a subset which is equipped with
	a smooth manifold structure, and $V \subseteq U$ an open symmetric subset
	such that $\one \in V$ and $V \cdot V \subseteq U$. Consider the conditions
	\begin{enumerate}
		\item
		The group inversion restricts to a smooth self map of $V$.
		
		\item
		The group multiplication restricts to a smooth map
		$
			V \times V \to U
		$.
		
		\item
		For each $g \in G$, there exists an open $\one$-neighborhood $W \subseteq U$
		such that $g\cdot W \cdot g^{-1} \subseteq U$, and the map
		\[
			W \to U : w \mapsto g\cdot w \cdot g^{-1}
		\]
		is smooth.
	\end{enumerate}
	If (a)--(c) hold, then there exists a unique smooth manifold structure on $G$
	which makes $G$ a Lie group such that $V$ is an open submanifold of $G$.
	If (a) and (b) hold, then there exists a unique smooth manifold structure
	on the subgroup $\langle V\rangle$ generated by $V$
	which makes $\langle V\rangle$ a Lie group
	such that $V$ is an open submanifold of $\langle V\rangle$.
\end{lem}

\subsection{Regularity}
\label{susec:LieGroups-Regularity}
We recall the notion of regularity
(see \cite{MR830252} for further information).
To this end, we define left evolutions of smooth curves.
As a tool, we use the group multiplication on the tangent bundle $\Tang{G}$
of a Lie group $G$.

\begin{defi}[Left logarithmic derivative]
	Let $G$ be a Lie group, $k \in \N$ and $\eta : [0,1] \to G$ a  $\ConDiff{}{}{k + 1}$-curve.
	We define the \emph{left logarithmic derivative} of $\eta$ as
	\[
		\glsdisp{LeftLogDeriv}{\LLA{\eta}} : [0,1] \to \LieAlg{G} : t \mapsto \eta(t)^{-1} \cdot \eta'(t).
	\]
	The curve $\LLA{\eta}$ is obviously $\ConDiff{}{}{k}$.
\end{defi}

\begin{defi}[Left evolutions]
	Let $G$ be a Lie group and $\gamma : [0,1] \to \LieAlg{G}$ a smooth curve.
	A smooth curve $\eta : [0,1] \to G$ is called a \emph{left evolution} of $\gamma$ and
	denoted by $\glsdisp{Left_Evolution}{\lEvol[G]{\gamma}}$ if $\LLA{\eta} = \gamma$ and $\eta(0) = \one$.
	One can show that in case of its existence, a left evolution is uniquely determined.
\end{defi}

The existence of a left evolution is equivalent to the existence of a solution
to a certain initial value problem:
\begin{lem}
Let $G$ be a Lie group and $\gamma : [0,1] \to \LieAlg{G}$ a smooth curve.
Then there exists a left evolution $\lEvol{G}{\gamma} : [0,1] \to G$ iff
the initial value problem
\begin{align}
	\begin{split}\label{ivp:diffeq_Links-regularitaet_allgemein}
		\eta'(t) &= \eta(t) \cdot \gamma(t)\\
		\eta(0) &= \one
	\end{split}
\end{align}
has a solution $\eta$. In this case, $\eta =\lEvol[G]{\gamma}$.
\end{lem}
Now we give the definition of regularity:
\begin{defi}[Regularity]
	\index{regularity}
	A Lie group $G$ is called \emph{regular} if for each smooth curve
	$\gamma : [0,1] \to \LieAlg{G}$ there exists a left evolution
	and the map
	\[
		\glstext{Left_Evolution_endpoint} : \ConDiff{[0,1]}{\LieAlg{G}}{\infty} \to G : \gamma \mapsto\lEvol[G]{\gamma}(1)
	\]
	is smooth.
\end{defi}

\begin{lem} \label{lem:Lie-Gruppe_lokal_reg-->global_reg}
	Let $G$ be a Lie group. Suppose there exists a zero neighborhood
	$\Omega \subseteq \ConDiff{[0,1]}{\LieAlg{G}}{\infty}$
	such that for each $\gamma \in \Omega$ the left evolution
	$\lEvol[G]{\gamma}$ exists and the map
	\[
		\Omega \to G : \gamma \mapsto\lEvol[G]{\gamma}(1)
	\]
	is smooth. Then $G$ is regular.
\end{lem}

\begin{bem}
	We can define \emph{right logarithmic derivatives} and \emph{right evolutions}
	in the analogous way as we did above for the left ones.
	We denote the right logarithmic derivative by $\glsdisp{RightLogDeriv}{\RLA{}}$,
	the right evolution map by $\glsdisp{Right_Evolution}{\rEvol{}}$
	and the endpoint of the right evolution by $\glsdisp{Right_Evolution_endpoint}{\revol{}}$.
	One can show that a Lie group is left-regular
	iff it is right-regular.
	Also the equivalent of \refer{lem:Lie-Gruppe_lokal_reg-->global_reg} holds.
	In particular, \refer{ivp:diffeq_Links-regularitaet_allgemein} becomes
	\begin{align}
	\begin{split}\label{ivp:diffeq_Rechts-regularitaet_allgemein}
		\eta'(t) &= \gamma(t) \cdot \eta(t)\\
		\eta(0) &= \one
	\end{split}.
	\end{align}
\end{bem}

\begin{defi}
	Let $G$ be a Lie group. A smooth map $\glstext{Exponentialfunktion} : \LieAlg{G} \to G$
	is called an \emph{exponential map} for $G$ if $\Tang[0]{\ex[\G]} = \id{\LieAlg{G}}$
	and $\ex[G]( (s + t) v) = \ex[G](s v) \cdot \ex[G](t v)$ for all $s, t \in \R$ and $v \in \LieAlg{G}$.
\end{defi}

\subsection{Group actions}

\begin{lem}\label{lem:Kriterium_fuer_Wirkung_auf_Untergruppe}
	Let $\G$ and $H$ be groups and $\alpha : \G \times H \to H$ a group action
	that is a group morphism in its second argument.
	Further, let $\widetilde{H}$ be a subgroup of $H$ that is generated by $\UF$.
	Then
	\[
		\alpha(G \times \widetilde{H}) \subseteq \widetilde{H}
		\iff
		\alpha(G \times \UF) \subseteq \widetilde{H}.
	\]
\end{lem}
\begin{proof}
	By our assumption, $\widetilde{H} = \bigcup_{n\in\N} (\UF \cup \UF^{-1})^n$.
	So we calculate
	\begin{multline*}
		\alpha(G \times \widetilde{H})
		= \alpha(G \times \bigcup_{n\in\N} (\UF \cup \UF^{-1})^n)
		= \bigcup_{n\in\N} \alpha(G \times (\UF \cup \UF^{-1})^n)
		\\
		= \bigcup_{n\in\N} \alpha(G \times (\UF \cup \UF^{-1}))^n
		= \bigcup_{n\in\N} ( \alpha(G \times \UF) \cup \alpha(G \times \UF)^{-1} )^n
		\subseteq \widetilde{H}.
	\end{multline*}
	That's it.
\end{proof}

\begin{lem}\label{lem:Kriterium-fuer-glatte-Gruppenwirkung}
	Let $\G$ and $H$ be Lie groups and $\alpha : \G \times H \to H$ a group action
	that is a group morphism in its second argument.
	Then $\alpha$ is smooth iff the following assertions hold:
	\begin{enumerate}
		\item\label{enum1:glatte-Gruppenwirkung_simultane-lokale-Glattheit}
		it is smooth on $\UF \times \VF$,
		where $\UF$ and $\VF$ are open unit neighborhoods, respectively.

		\item\label{enum1:glatte-Gruppenwirkung_lokale-Glattheit-im-ersten-Argument}
		for each $h\in H$, there exists an open unit neighborhood $\WF$ such that
		the map $\alpha(\cdot, h) : \WF \to H$ is smooth.

		\item\label{enum1:glatte-Gruppenwirkung_Glattheit-im-zweiten-Argument}
		for each $g\in\G$ the map $\alpha(g, \cdot) : H \to H$ is smooth.
	\end{enumerate}
	If $\UF$ generates $\G$, \refer{enum1:glatte-Gruppenwirkung_lokale-Glattheit-im-ersten-Argument}
	follows from \refer{enum1:glatte-Gruppenwirkung_simultane-lokale-Glattheit}.
	If $\VF$ generates $H$, \refer{enum1:glatte-Gruppenwirkung_Glattheit-im-zweiten-Argument}
	follows from \refer{enum1:glatte-Gruppenwirkung_simultane-lokale-Glattheit}.
\end{lem}
\begin{proof}
	We first show that by our assumptions, $\alpha$ is smooth.
	To this end, let $(g, h) \in \G \times H$. Choose $\WF$ as in \refer{enum1:glatte-Gruppenwirkung_lokale-Glattheit-im-ersten-Argument}.
	Then $\UF' \ndef \UF \cap \WF \in \neigh[\G]{\one}$.
	We show that $\rest{\alpha}{g\UF' \times \VF h}$ is smooth.
	Since the map $\UF' \times \VF \to g\UF' \times \VF h : (u, v) \mapsto (g u, v h)$
	is a smooth diffeomorphism, we only need to show that the map
	\[
		\UF' \times \VF \to \H
		:
		(u, v) \mapsto \alpha(g u, h v)
	\]
	is smooth. But
	\[
		\alpha(g u, h v) = \alpha_g (\alpha(u, v h))
		= \alpha_g ( \alpha(u, v) \alpha(u, h))
		= \alpha_g(\alpha(u, v) \alpha^h(u) ),
	\]
	where we denote $\alpha(\cdot, h)$ by $\alpha^h$ and $\alpha(g, \cdot)$ by $\alpha_g$.
	Since the right hand side is obviously smooth, we are home.
	
	Now we prove the other two assertions.
	We suppose that \refer{enum1:glatte-Gruppenwirkung_simultane-lokale-Glattheit} holds.
	We let $S \sub H$ be the set of all $h \in H$ such that \refer{enum1:glatte-Gruppenwirkung_lokale-Glattheit-im-ersten-Argument}
	holds. Then $\VF \sub S$; and since $\alpha^{h^{-1}}(g) = \alpha^h(g)^{-1}$
	and $\alpha^{h h'}(g) = \alpha^{h}(g) \alpha^{h'}(g)$
	for all $g \in \G$ and $h, h' \in H$, we easily see that $S$ is a subgroup of $H$.
	Since $\VF$ is a generator, $S = H$.
	\\
	Since $\UF$ generates $\G$, for each $g \in \G$ we find $g_1, \dotsc, g_n \in \UF \cup \UF^{-1}$
	such that
	\[
		\alpha_g = \alpha_{g_n} \circ \dotsb \circ \alpha_{g_1}.
	\]
	Further, for $g' \in \G$ and $h \in H$, $\alpha_{g'^{-1}}(h) = \alpha_{g'}(h)^{-1}$,
	so each $\alpha_{g_k}$ is smooth by our assumption.
	Hence $\alpha_g$ is smooth.
\end{proof}

\begin{lem}\label{lem:Semidirektes_Produkt_Liegruppen-glatteWirkung-ist_Liegruppe}
	Let $\G$ and $H$ be Lie groups and $\omega : \G \times H \to H$ a smooth group action
	that is a group morphism in its second argument.
	Then the semidirect product $H \rtimes_\omega \G$
	can be turned into a Lie group that is modelled on
	$\LieAlg{H} \times \LieAlg{\G}$.
	\index{semidirect product}
\end{lem}
\begin{proof}
	The semidirect product $H \rtimes_\omega \G$ is endowed
	with the multiplication
	\[
		(H \times \G) \times (H \times \G) \to H \times \G
		:
		((h_1, g_1), (h_2, g_2)) \mapsto (h_1 \cdot \omega(g_1, h_2), g_1\cdot g_2)
	\]
	and the inversion
	\[
		H \times \G \to H \times \G : (h, g) \mapsto ( \omega(g^{-1}, h^{-1}), g^{-1} ),
	\]
	so the smoothness of the group operations follows from the one of $\omega$.
\end{proof}
\chapter{Quasi-inversion in algebras}
\label{app:Quasi-Inversion}
\index{quasi-inversion}
We give a short introduction to the concept of \emph{quasi-inversion}.
It is a useful tool for the treatment of algebras without a unit,
where it serves as a replacement for the ordinary inversion.
Many of the algebras we treat are without a unit.
Unless the contrary is stated, all algebras are assumed associative.

\section{Definition}
\begin{defi}[Quasi-Inversion]
	Let $A$ denote a $\K$-algebra with the multiplication $\mult$.
	An $x\in A$ is called \emph{quasi-invertible}
	if there exists a $y\in A$ such that
	\[
		x + y - x \mult y = y + x - y \mult x = 0 .
	\]
	In this case, we call $\glstext{Quasi-Inversion}(x) \ndef y$ the \emph{quasi-inverse} of $x$.
	The set that consists of all quasi-invertible elements of $A$
	is denoted by $\glsdisp{Quasi-Invertierbare}{\QInvertiblesOf{A}}$.
	The map
	$\QInvertiblesOf{A} \to \QInvertiblesOf{A} : x \mapsto \QuasiInv_A(x)$
	is called the \emph{quasi-inversion} of $A$.
	Often we will denote $\QuasiInv_A$ just by $\QuasiInv$.
\end{defi}
An interesting characterization of quasi-inversion is
\begin{lem}\label{lem:Adversion_Monoid}
	Let $A$ be a $\K$-Algebra with the multiplication $\mult$.
	Then $A$, endowed with the operation
	\[
		A \times A \to A : (x,y) \mapsto x \diamond y \ndef x + y - x \mult y ,
	\]
	is a monoid with the unit $0$ and the unit group $\QInvertiblesOf{A}$.
	The inversion map is given by $\QuasiInv_A$.
\end{lem}
\begin{proof}
	This is shown by an easy computation.
\end{proof}
In unital algebras there is a close relationship between inversion and quasi-inversion.
\begin{lem}\label{lem:Adversion_und_Inversion}
	Let $A$ be an algebra with multiplication $\mult$ and unit $e$.
	Then $x\in A$ is quasi-invertible iff
	$x-e$ is invertible. In this case
	\[
		\QuasiInv_A(x) = (x-e)^{-1} + e .
	\]
\end{lem}
\begin{proof}
	One easily computes that
	\[
		(A,\diamond) \to (A,\mult) : x \mapsto e - x
	\]
	is an isomorphism of monoids
	($\diamond$ was introduced in \refer{lem:Adversion_Monoid}),
	and from this we easily deduce the assertion.
\end{proof}

\section{Topological monoids and algebras with continuous quasi-inversion}
\label{sec:topologische_Monoide}
\label{sec:topologische_Quasi-Inversions-Algebren}
In this section, we examine algebras that are endowed with a topology.
For technical reasons we also examine monoids.
\begin{defi}\label{def:topologische_Quasi-Inversions-Algebren}
	An algebra $A$ is called a \emph{topological algebra} if
	it is a topological vector space and the multiplication is continuous.
	
	A topological algebra $A$ is called \emph{algebra with continuous quasi-inversion}
	if the set $\QInvertiblesOf{A}$ is open and the quasi-inversion $\QuasiInv$ is continuous.
	
	A monoid, endowed with a topology, is called a \emph{topological monoid} if
	the monoid multiplication is continuous.
	
	A monoid, endowed with a differential structure, is called a \emph{smooth monoid} if
	the monoid multiplication is smooth.
	\index{smooth monoid}
\end{defi}
\begin{bem}
	If $A$ is an algebra with continuous quasi-inversion,
	then $\QuasiInv$ is not only continuous, but automatically analytic, see \cite{MR1948922}.
\end{bem}

In topological monoids the unit group is open and the inversion continuous
if they are so near the unit element:
\begin{lem}\label{lem:topologische_Monoide}
	Let $M$ be a topological monoid with unit $e$ and the multiplication $\mult$.
	Then the unit group $M^\times$ is open iff there exists a neighborhood
	of $e$ that consists of invertible elements.
	The inversion map
	\[
		I : M^\times \to M^\times : x \mapsto x^{-1}
	\]
	is continuous iff it is so in $e$.
\end{lem}
\begin{proof}
	Let $U$ be a neighborhood of $e$ that consists of invertible elements
	and $m \in M^\times$. Since the map
	\begin{equation*}		\ell_{m} : M \to M : x \mapsto m \mult x
	\end{equation*}
	is a homeomorphism, $\ell_m(U)$ is open;
	and it is clear that $\ell_m(U) \subseteq M^\times$.
	Hence $M^\times = \cup_{m \in M^\times} \ell_m(U)$ is open.
	
	Let $I$ be continuous in $e$. We show it is so in $x \in M^\times$.
	For $m \in M^\times$, we have
	\[
		\label{id:Inversion_und_Inversion_in_0}\tag{\ensuremath{\dagger}}
		I(m) = m^{-1} = m^{-1} \mult x \mult x^{-1} = (x^{-1}\mult m)^{-1} \mult x^{-1}
		= (\rho_{x^{-1}} \circ I \circ \ell_{x^{-1}})(m) ,
	\]
	where $\rho_{x^{-1}}$ denotes the right multiplication by $x^{-1}$.
	Since $I$ is continuous in $e$ and $\ell_{x^{-1}}(x) = e$, we can derive
	the continuity of $I$ in $x$ from \eqref{id:Inversion_und_Inversion_in_0}.
\end{proof}
For algebras with a continuous multiplication we can deduce
\begin{lem}\label{lem:Kriterium_fuer_Stetigkeit_der_Adversion}
	Let $A$ be an algebra with the continuous multiplication $\mult$.
	Then $\QInvertiblesOf{A}$ is open if there exists a neighborhood
	of $0$ that consists of invertible elements.
	The quasi-inversion $\QuasiInv_A$ is continuous if it is so in $0$.
\end{lem}
\begin{proof}
	Since the map
	\[
		A \times A \to A : (x,y) \mapsto x + y - x \mult y
	\]
	is continuous, we derive the assertions from
	\refer{lem:Adversion_Monoid} and \refer{lem:topologische_Monoide}.
\end{proof}
\paragraph{A criterion for quasi-invertibility}
We give an criterion that ensures that an element of an algebra is quasi-invertible.
It turns out that it is quite useful in certain algebras, namely Banach algebras.
\begin{lem}\label{lem:Topologisches_Kriterium_fuer_Quasi-Invertierbarkeit}
	Let $A$ be a topological algebra and $x \in A$. If $\sum_{i=1}^\infty x^i$ exists,
	then $x$ is quasi-invertible with
	\[
		\QuasiInv_{A}(x) = - \sum_{i=1}^\infty x^i .
	\]
\end{lem}
\begin{proof}
	We just compute that $x$ is quasi-invertible:
	\[
		x + \left(- \sum_{i=1}^\infty x^i \right)
		- x \mult \left(- \sum_{i=1}^\infty x^i \right)
		= - \sum_{i=2}^\infty x^i + \sum_{i=2}^\infty x^i
		= 0 .
	\]
	The identity
	$(- \sum_{i=1}^\infty x^i) + x - (- \sum_{i=1}^\infty x^i) \mult x = 0$
	is computed in the same way.
	So the quasi-invertibility of $x$ follows direct from the definition.
\end{proof}
\paragraph{Quasi-inversion in Banach algebras}
\begin{lem}\label{lem:Einheitskugel_quasiinvertierbar}
	Let $A$ be a Banach algebra.
	Then $\Ball{0}{1} \subseteq \QInvertiblesOf{A}$. Moreover, for
	$x \in \Ball{0}{1}$
	\begin{equation*}
		\QuasiInv_{A}(x) = - \sum_{i=1}^\infty x^i  .
	\end{equation*}
\end{lem}
\begin{proof}
	For $x \in \Ball{0}{1}$ the series $\sum_{i=1}^\infty x^i$ exists
	since it is absolutely convergent and $A$ is complete.
	So the assertion follows from
	\refer{lem:Topologisches_Kriterium_fuer_Quasi-Invertierbarkeit}.
\end{proof}

\begin{lem}
	Let $A$ be a Banach algebra. Then $\QInvertiblesOf{A}$ is open in $A$
	and the quasi-inversion $\QuasiInv_{A}$ is continuous.
\end{lem}
\begin{proof}
	This is an immediate consequence of \refer{lem:Einheitskugel_quasiinvertierbar}
	and \refer{lem:Kriterium_fuer_Stetigkeit_der_Adversion} since
	\[
		x \mapsto \sum_{i=1}^\infty x^i
	\]
	is analytic (see \cite[\S 3.2.9]{MR0219078}) and hence continuous.
\end{proof}
\nocite{MR2150374}
\clearpage\phantomsection
\markboth{References}{References}
\printbibliography
\renewcommand{\glossarypreamble}{
The following list contains the symbols that are used on several occasions,
together with a short explanation of their meaning and
the page number where the respective symbol is defined.
For better overview, the entries are arranged into the categories
basic notation, spaces of weighted functions, Lie groups and manifolds,
groups and monoids of functions and further notation.
}
\setlength\glsdescwidth{0.72\linewidth}
\printglossary[type=main, title=Notation, toctitle=Notation, style=mylong3col]
\clearpage\phantomsection
\addcontentsline{toc}{chapter}{Index}
\printindex
\end{document}